\documentclass{amsart}
\usepackage[utf8]{inputenc}
\usepackage{xypic,eucal,enumerate}

\usepackage{tikz-cd}
\usepackage{amssymb}
\usepackage{mathtools}
\usepackage{hyperref}
\usepackage{yfonts}
\usepackage{placeins}
\usepackage{subcaption}

\usepackage[a4paper,margin=1.2in]{geometry}

\title[Power Towers]{Power Towers:\\
Purely Inseparable Galois Theory\\
and Foliations in Positive Characteristic}
\author{Przemysław Grabowski}
\date{\today}

\newcommand{\cO}{\mathcal{O}}
\newcommand{\C}{\mathbb{C}}
\newcommand{\Proj}{\mathbb{P}}

\newcommand{\op}{\operatorname}
\renewcommand{\phi}{\varphi}

  \newtheorem{thm}{Theorem}[section]
  \newtheorem{lemma}[thm]{Lemma}
  \newtheorem{prop}[thm]{Proposition}
  \newtheorem{cor}[thm]{Corollary}

  \theoremstyle{definition} 
  \newtheorem{defin}[thm]{Definition}
  \newtheorem{remark}[thm]{Remark}
  \newtheorem{obs}[thm]{Observation}
  \newtheorem{example}[thm]{Example}
  \newtheorem{question}[thm]{Question}
  \newtheorem{convention}[thm]{Convention}

    \newcommand{\F}{\mathcal{F}}
    \newcommand{\G}{\mathcal{G}}
    \newcommand{\E}{\textswab{E}}
    \newcommand{\D}{\mathcal{D}}
    \newcommand{\Z}{\mathbb{Z}}

\begin{document}

\begin{abstract}
    We build a purely inseparable Galois theory using non-derived commutative algebra. Our theory works on fields and on normal varieties. It says that a purely inseparable morphism corresponds to a finite (saturated) subalgebra of differential operators. Our approach unifies most of the literature about purely inseparable morphisms and shows new research directions. 

    Indeed, our theory extends to a theory that covers all (saturated) subalgebras of differential operators. The extended theory gives a correspondence between the subalgebras and a new notion of \textbf{power towers}. A power tower is an object analogous to a foliation from differential geometry. In particular, it admits its own versions of many key results about foliations, such as a ``fibrations inject into foliations'' and a ``Frobenius theorem''. The latter has a nontrivial twist: the local structure is trivial at every point, but it may differ between points! 
    
    In general, our analogy between power towers and foliations explains why purely inseparable morphisms are foliation-like, because these morphisms are power towers in an explicit way.

    Finally, we use our theory to produce a formula for pullbacks of canonical divisors for arbitrary purely inseparable morphisms. We use it to conclude some ``not-L{\"u}roth theorems'': if the characteristic is high enough, then any variety purely inseparably covered by an $n$-dimensional projective space has Iitaka dimension equal to $-\infty$ or $n$, i.e., it is not zero.
\end{abstract}

\maketitle

\tableofcontents

\newpage
\section{Introduction}

Algebraic geometry is an area of mathematics that studies systems of polynomial equations. Frequently, the coefficients of these equations are taken from a fixed field. Any field has a characteristic, which can be equal to zero, e.g., for complex numbers $\mathbb{C}$, or to a prime number $p>0$, e.g., for a finite field $\mathbb{F}_p$. In the second case, we say that the field is of a positive characteristic. Therefore, we have two types of algebraic geometries over fields: the characteristic zero and the positive characteristic. They have different flavors. The first one is analytical, the second one is arithmetical. Analysis is strongly developed. Consequently, we know plenty about the characteristic zero. On the other hand, classical calculus theorems are not true modulo $p$. Instead, we have the Frobenius morphism: $x\mapsto x^p$ on any ring satisfying $p=0$. This generates a new class of morphisms: purely inseparable morphisms.
This class is exclusive to the positive characteristic. Most of us know about them from the field theory, where they are called purely inseparable subfields.
Surprisingly, after over 100 years of studying fields abstractly, which was started in \cite{steinitz1910} by Ernst Steinitz and his influence is explained in \cite{about-steinitz}, we still know little about these morphisms - even when we restrict our attention to fields!
This paper aims to change it for both - for fields and for varieties.

\begin{defin}[Purely Inseparable Ring Map]\label{purely inseparable - rings - def}
    Let $B\to A$ be an injective morphism between commutative rings. This morphism is \emph{purely inseparable} of exponent at most $r$ if for every $a\in A$ we have $a^{p^r}\in B$. It is of exponent $r$ if the integer $r$ is the minimal such.
\end{defin}

\begin{defin}[Purely Inseparable Morphism]\label{purely inseparable - varieties - def}
   A morphism $f: X\to Y$ between normal varieties is \emph{purely inseparable} of exponent $\le r$ if it is finite, dominant, and affine locally it is purely inseparable of exponent $\le r$. It is of exponent $r$ if the integer $r$ is the minimal such.

   We say that the morphism $f$ is purely inseparable if there is a number $r$ such that $f$ is purely inseparable of exponent $\le r$.\footnote{For varieties, any purely inseparable morphism is of finite exponent, thus we only define these.}
\end{defin}

Our main goal in this paper is to establish a Galois-type description for all purely inseparable morphisms - for fields, where we work with subfields, and for varieties, where we work with morphisms. This description refers to a correspondence between these morphisms and other algebraic objects that are easier to handle. In our case, we cannot use groups of automorphisms, because these are blind to $p$-powers; thus, instead, we will use subalgebras of differential operators. We will get a correspondence between purely inseparable morphisms and finite subalgebras of differential operators. However, a natural framework for this theorem uses a new notion: \textbf{a power tower}, which is analogous to a foliation (a nice set of vector fields that resembles a fibration). With this notion, we get a correspondence between all subalgebras of differential operators and power towers that restricts to the correspondence for purely inseparable morphisms/subfields. 

As a happy consequence of the analogy between power towers and foliations, we establish new, never mentioned anywhere else, properties of purely inseparable morphisms. e.g., any such morphism naturally induces a stratification on its source!

Our secondary goal is to explore explicitly some positive characteristic phenomena. This label refers to all theorems that are unique to the positive characteristic compared to characteristic zero. Usually, they are about ``pathologies'' related to some techniques or theorems from characteristic zero not being true modulo $p$: 
vanishing theorems \cite{Vanishing_ordinary_abelian_Hacon_Zsolt},
the minimal model program for surfaces
\cite{MMP_surfaces_p_mumford},
etc. In many cases, such as the ones just cited, first we discover that they are no longer verbatim true modulo $p$, then we correct the statements. Often, the new theorems explicitly involve the Frobenius morphism or other purely inseparable morphisms. 

In particular, this paper explores the interaction between vector fields and fibrations modulo $p$. The expected result from the characteristic zero fails: vector fields tangent to the fibers do not determine the fibration! The correction is the notion of \textbf{power towers}. It emerges as the tangent vector fields plus all that must be added to that to determine the fibration.

Finally, we establish a formula for pullbacks of canonical divisors along arbitrary purely inseparable morphisms. So, we get a tool to track flows of geometric information along these morphisms. We use it to show how some unexpected unirationality of varieties may arise in positive characteristic.

Extra, this paper's informal short online presentation is available online \cite{YouTube-ME}. 

\pagebreak
\begin{question}
    Do purely inseparable morphisms explain all positive characteristic phenomena?
    
    For example, let's say we have a theorem from characteristic zero whose statement is not true in positive characteristic. Is there a purely inseparable morphism explicitly involved in this failure?
\end{question}

\subsection{A Brief Literature Overview}
The previous main approaches towards purely inseparable Galois theory are:

\begin{enumerate}
    \item In \cite{Jacobson1944}, Jacobson proved that purely inseparable subfields of exponent one correspond to $p$-Lie algebras.
    \item In \cite{Jacobson1964:FieldsAndGalois}, a book, we can find Jacobson--Bourbaki correspondence, which says that all subfields of a given field correspond to closed subalgebras of linear endomorphisms. This is based on the works of H. Cartan \cite{Cartan1947} and Jacobson \cite{Jacobson1947}. No specialization to all purely inseparable subfields was established, which was well known.
    \item In \cite{EkedahlFoliation1987}, Ekedahl proved  Jacobson correspondence for varieties. It covers exponent one. He also introduced some sketches towards higher exponents, but his ideas were not complete, nor followed by others, till now.
    \item In \cite{Sweedler-modular}, Sweedler proved that ``modular'' purely inseparable subfields correspond to sets of higher derivations. These can have arbitrary exponents, but they do not cover all subfields. In his thesis \cite{Wechter-modular-phd}, Wechter compared the Jacobson--Bourbaki correspondence with Sweedler's one. It says the higher derivations are a basis of the algebra.
    \item In \cite{brantner2023purely}, a preprint, Brantner and Waldron used ``derived $p$-Lie algebras'' called partition Lie algebroids to give a correspondence between them and purely inseparable subfields of any fixed finite exponent.
\end{enumerate}

Our theory of power towers unifies, expands, and completes the first three approaches. The fourth one is a case of our theory when extra structure of a free action of a specific infinitesimal group ($G\cong\alpha_{p^{n_1}}\oplus\ldots\oplus\alpha_{p^{n_k}}$) is present. That action induces a preferred basis for the algebra. 

Indeed, our Galois theory will show that purely inseparable subfields correspond to finite subalgebras of differential operators. This is a specialization of the Jacobson--Bourbaki correspondence. The case of exponent one is obtained by taking operators of order $1$; in this case, they generate the subalgebra. This is Jacobson's correspondence. We extend these results to varieties via saturation and get Ekedahl's correspondence. Then, we clarify his sketches using our theory. This is the first follow-up of these ideas in the literature. Finally, if a subfield is modular, then the higher derivations are a special choice of a basis for its subalgebra. But this is an extra choice; we want to work out a natural theory first.

\begin{question}
    Let $G$ be an infinitesimal finite group. Can we obtain ``a Galois theory in the form of Sweedler's one'' for the group $G$? I.e., how to describe quotients of free actions of $G$ on fields using ``higher derivations''-like objects? 
    
    The answer is: Yes, we can. However, the question asks ``how to do it nicely?'' because some choices are involved.
\end{question}

\subsection{What Is a Power Tower? Why Is It Like a Foliation?}

\begin{center}
    A \textbf{power tower} is a sequence of approximations up to $p^n$-powers. 
\end{center}

\begin{defin}[Up to $p^n$-powers.]
    Let $K$ be a field of characteristic $p>0$. Let $W\subset K$ be a subfield. An approximation of $W$ up to $p^n$-powers is the composite of the subfields $W$ and $K^{p^n}\coloneqq \{x^{p^j}: \ x\in K\}$, i.e., $W \cdot K^{p^n}$. We denote it by $W_n$.
\end{defin}

\begin{defin}[Power Tower]\label{Power Tower - Definition - intro}
    Let $K$ be a field of characteristic $p>0$. A power tower on $K$ is a sequence of subfields $W_n\subset K$ for $n=0,1,2,\ldots$ such that $W_j = W_i \cdot K^{p^j}$ for $j\le i$.
\end{defin}

In differential geometry, the notion of a foliation emerges as a product of gluing local fibrations on a manifold into a global object. This is how Ehresmann introduced it, see e.g. \cite{Ehresmann-history}. In particular, tautologically, (global) \textbf{fibrations inject into foliations}. Immediately, it was established that Lie subalgebras of tangent bundles are equivalent to foliations. Thus, this injection is just the fact that if we take a fibration, then vector fields tangent to its fibers, and then we integrate these vector fields, then in the end, we will get the initial fibration back. 

However, modulo $p$, vector fields are weaker. 

\begin{center}
    \textbf{Derivations kill $p$-powers modulo $p$}, e.g., $\frac{\partial}{\partial x}(g^p)=pg^{p-1}\frac{\partial}{\partial x}(g)=0$.
\end{center}
Thus, a single integration only gives the fibration up to $p$-powers, not the whole fibration. However, we can iterate and collect all the approximations up to $p^n$-powers together, and then they recover nice fibrations/morphisms/subfields. We provide an explicit example below, Example \ref{intro - example}.

Our Theorem \ref{infty=s refined - Corollary} covers all subfields that can be recovered from their power towers. In particular, we have:

\begin{itemize}
    \item \textbf{Purely inseparable subfields inject into power towers,} Lemma \ref{Finite Length = finite exponent - lemma}. Consequently, every such subfield admits a canonical decomposition into a sequence of subfields of exponent one - its power tower.
    \item \textbf{Purely transcendental subfields inject into $\infty$-foliations}, Proposition \ref{fibrations into infty-foliations - proposition}, a special class of power towers. 
\end{itemize}

Consequently, under certain natural assumptions, power towers allow us to reduce the study of subfields of a given field into a study of algebraic separable and purely inseparable subfields only, and their finite and infinite compositions.

For varieties, we extend the definition of power towers via normalization, Definition \ref{Power Tower on varieties - def}, and we obtain a similar result:
\begin{itemize}
    \item \textbf{Separable fibrations inject into $\infty$-foliations,} Theorem \ref{fibrations into infty-folaitions - varieties - theorem}.
\end{itemize}

\begin{question}
    Can we reduce the study of positive characteristic algebraic geometry to finite and infinite interactions between only three classes of operations: {\'e}tale morphisms, purely inseparable morphisms, and birational modifications?
\end{question}

\begin{center}
    Local structure of a power tower - a \textbf{Frobenius theorem}.
\end{center}

In (complex) differential geometry, foliations are trivial in local analytic coordinates, i.e., at every (regular) point $x$, in some formal coordinates, a foliation $(X/\mathbb{C},\mathcal{F})$ is given by 
\[
 \mathbb{C}[[t_1,\ldots,t_m]]\to \mathbb{C}[[t_1,\ldots,t_n]]\simeq \widehat{\cO_{X,x}},
\]
where $m\le n$. This is a Frobenius theorem for foliations, \cite[Theorem 2.20]{Voisin1}.

Power towers satisfy analogous theorem \ref{formal Frobenius Theorem for power towers- thm} - the  \textbf{Frobenius theorem for power towers}. Locally, in some formal coordinates, at a regular point, it is given by
\[
k[[t_1^{p^{a_1}},t_2^{p^{a_2}},t_3^{p^{a_3}},\ldots,t_d^{p^{a_d}}]]\to k[[t_1,t_2,t_3,\ldots, t_d]],
\]
where $0\le a_1\le a_2\le a_3 \le \ldots \le a_d \le \infty$. The convention is $t^{p^\infty}=0$. (The classical theorem corresponds to using only $a_i= 0, \infty$.) 

Consequently, we get a function $G$ from closed points to tuples of indices
$x \mapsto (a_1,\ldots,a_d)$. \textbf{This function does not have to be constant!} Nevertheless, it is constructible, Theorem \ref{stratification exists - thm}, and thus we get a stratification of our variety, each piece being the locus where the index is constant, e.g., Example \ref{not Ekedahl power tower}:

\[
\text{For } f: k[x,y] \supset k[x^{p^2}, y^{p}, y^{p+1}-x^p],
\text{ we get $G(f): (x,y)\mapsto (a_1,a_2)$=} \begin{cases}
(0,2) , \text{ if } y \ne 0,\\
(1,1) , \text{ if } y=0.
\end{cases}
\]

\begin{question}
    What are all possible stratifications induced by purely inseparable morphisms?
\end{question}

Power towers admit the same origin as foliations in differential geometry and have analogous formal local structures. Therefore, they are foliation-like, and anything that can be interpreted as a power tower is foliation-like too. This includes purely inseparable morphisms and nice fibrations.

\subsection{The Basic Example}
It is harder to state our Galois theory than to understand it in practice. Therefore, we will give an example and show how the theory applies to it.

\begin{example}\label{intro - example}
    Let $k$ be a perfect field of characteristic $p>0$. Let $f: \mathbb{A}^2_k\to\mathbb{A}^1_k$ be the projection corresponding to the ring inclusion $ k[x,y]\supset k[y]$ — a set of parallel lines on a plane. Here, we will compute its power tower. The following diagram summarizes it.

    \begin{center}
    \begin{tikzcd}[ampersand replacement=\&]
    {\color{blue}\ldots}\\
    {\color{blue}\frac{1}{p^2!}\frac{\partial^{p^2}}{\partial x^{p^2}}}\ar[rrrdd,end anchor={[yshift=5ex, xshift=-5ex]},mapsto, start anchor ={[yshift=2ex]}]\\
    {\color{blue} \frac{1}{p!}\frac{\partial^p}{\partial x^p}}
    \ar[rrd, mapsto,start anchor ={[yshift=1.45ex]}, end anchor ={[yshift=3.5ex, xshift=-7.5ex]}]\\
      {\color{red}k[x,y]}\ar[r, bend left =30, dash, "{\color{blue}\frac{\partial}{\partial x}}={\color{violet}\frac{\partial}{\partial x}}"] 
      \& {\color{red}k[x^p,y]}\ar[l, hook']\ar[r, bend left =30, dash, "{\color{violet}\frac{\partial}{\partial x^p}}"] 
      \& {\color{red} k[x^{p^2},y]}
      \ar[l, hook']\ar[r, bend left =30, dash, "{\color{violet}\frac{\partial}{\partial x^{p^2}}}",
end anchor={[yshift=1ex]}]
      \& {\color{red}\cdots} \ar[l, hook'] 
      \& {k[y]}\ar[l, hook']
    \end{tikzcd}
\end{center}

The relative tangent bundle of $f$, vector fields tangent to the fibers, is spanned by $\frac{\partial}{\partial x}$. Integrating $\frac{\partial}{\partial x}$ means to solve $\frac{\partial}{\partial x}(g)=0$. By definition, the solution is $k[y]$ up to $p$-powers of $k[x,y]$. It equals the ring $k[x^p,y]$, because derivations kill $p$-powers. The subring $k[x^p,y]$ and the vector field $\frac{\partial}{\partial x}$ clearly determine each other. This is a case of Jacobson--Ekedahl correspondence.

Vector fields are weaker modulo $p$. We see a manifestation $\{g\in k[x,y] \ : \frac{\partial}{\partial x}(g)=0\}\ne k[y]$. Nevertheless, we can repeat the above computation for $k[x^p,y]\hookleftarrow k[y]$. We will get $k[y]$ up to $p$-powers on $k[x^p,y]$, which is also $k[y]$ up to $p^2$-powers on $k[x,y]$. It is the subring $k[x^{p^2},y]$ that corresponds to the vector field $\frac{\partial}{\partial x^p}$.

We repeat over and over. We will get $\frac{\partial}{\partial x^{p^n}}$ and $k[x^{p^{n+1}},y]$ that determine each other. Moreover, any of this sequences determine $k[y]$, because $k[y]=\bigcap_{n\ge 0} k[x^{p^n},y]$. This is a case of nice fibrations injecting into power towers.

However, out of the {\color{blue}blue}, there is more to this story, because the iterative data we have just obtained can be packed into a subalgebra of differential operators, whose elements, after restricting them to the subrings, will become the vector fields doing approximations up to $p$-powers, e.g., 

$\frac{1}{p!}\frac{\partial^p}{\partial x^p}$ becomes $\frac{\partial}{\partial x^p}$ on $k[x,y]$, $\frac{1}{p^2!}\frac{\partial^{p^2}}{\partial x^{p^2}}$ becomes $\frac{\partial}{\partial x^{p^2}}$ on $k[x^{p^2},y]$, and so on.
\end{example}

\subsection{Galois Theory for Power Towers}

Our Galois theory for power towers, Theorems \ref{MainTheorem} and \ref{MainTheorem - var - theorem}, applied to Example \ref{intro - example}.

\begin{thm}[Galois-Type Correspondence for a Power Tower]
    Each of the following data from Example \ref{intro - example} can be computed from any other:
        \begin{itemize}
            \item {\color{red}Power Tower}: a sequence of approximations up to $p^n$-powers, e.g.,
            \[
            {\color{red} k[x,y]\supset k[x^p,y] \supset k[x^{p^2},y]\supset \ldots}.
            \]
            
            \item {\color{violet}Jacobson Sequence:} a particular sequence of sets of vector fields, $p$-Lie algebras, each making an approximation up to $p$-powers, e.g.,
            \[
            {\color{violet} \frac{\partial}{\partial x},\frac{\partial}{\partial x^{p}},\frac{\partial}{\partial x^{p^2}}, \ldots}.
            \]
            
            \item {\color{blue} Subalgebra of Differential Operators}, e.g.,
            \[
            {\color{blue}\left< 1,{\color{blue}\frac{\partial}{\partial x} }, {\color{blue}\frac{1}{p!}\frac{\partial^{p}}{\partial x^{p}}}, {\color{blue}\frac{1}{p^2!}\frac{\partial^{p^2}}{\partial x^{p^2}}},\ldots \right>} \subset \op{Diff}_k(k[x,y]).
            \]
        \end{itemize}
        Explicitly, these computations are done by the following operations:
        \begin{itemize}
            \item Annihilator: $\op{Ann}({\color{violet}\frac{\partial}{\partial x^{p^n}}})\coloneqq \{g\in {\color{red}k[x^{p^n},y]}: \ {\color{violet}\frac{\partial}{\partial x^{p^n}}}(g)=0\}={\color{red}k[x^{p^{n+1}},y]}$.
            \item Tangent Space: $T_{{\color{red}Y_n}/{\color{red}Y_{n+1}}}=\{a{\color{violet}\frac{\partial}{\partial x^{p^n}}}: \ a\in {\color{red}k[x^{p^n},y]}\}$.
            \item Subalgebra of $\op{Diff}_k(X)$: ${\color{blue}\left< 1,{\color{blue}\frac{\partial}{\partial x} }, {\color{blue}\frac{1}{p!}\frac{\partial^{p}}{\partial x^{p}}}, {\color{blue}\frac{1}{p^2!}\frac{\partial^{p^2}}{\partial x^{p^2}}},\ldots \right>}=\bigcup_{n\ge 0} \op{Diff}_{{\color{red}Y_n}}(X)$.
            \item Unpacking: $d(X/{\color{red}Y_n})({\color{blue}\frac{1}{p^n!}\frac{\partial^{p^n}}{\partial x^{p^n}}})={\color{violet}\frac{\partial}{\partial x^{p^n}}}$.
        \end{itemize}
        Above, we used the notation: $X=\op{Spec}(k[x,y]), Y_n=\op{Spec}(k[x^{p^n},y])$.
\end{thm}

Purely inseparable morphisms/subfields inject into power towers. Nice fibrations inject into power towers. Therefore, we can use the above theorem in these cases. In this way, we obtain a purely inseparable Galois theory. And, a new insight that nice fibrations are ``infinite compositions'' of purely inseparable morphisms, thus they enjoy the same descriptions via taking appropriate limits.

\subsection{Foliations, $1$-Foliations, and $\infty$-Foliations}

Power towers are analogous to foliations from the characteristic zero or differential geometry. Here we explain their relationships.

Let $X/\mathbb{Z}_p$ be a smooth variety over $p$-adic integers. Let $X_0/\mathbb{F}_p$ be its reduction modulo $p$, i.e., $X_0=X\times_{ \op{Spec}(\mathbb{Z}_p)} \op{Spec}(\mathbb{F}_p)$. Let $\F\subset T_{X/\mathbb{Z}_p}$ be a regular foliation on $X$, i.e., a subbundle that is closed under Lie brackets.

Two cases may happen with $\F$ modulo $p$:
\begin{enumerate}
    \item It might be closed under $p$-powers. If so, then $\F$ modulo $p$ is a $1$-foliation on $X_0$.
    \item It might NOT be closed under $p$-powers. Then  $\F$ modulo $p$ is not foliation-like, but there is an obstruction towards being $1$-foliation. It is called a ``$p$-curvature''. It is described in \cite[Section 5]{Katz-nilpotent-connection}. This can be used to study foliations in the characteristic zero, see \cite{Pereira-mendson-foli_codim_1}.
\end{enumerate}
We are interested in the first case, where the information flows freely. However, we lose some information anyway. Indeed, if $X=\mathbb{A}^2=\op{Spec}(\mathbb{Z}_p[x,y])$ and $\F_1,\F_2$ are the relative tangent bundles for $\mathbb{Z}_p[y]\subset \mathbb{Z}_p[x,y]$ and $\mathbb{Z}_p[y+x^p]\subset \mathbb{Z}_p[x,y]$ respectively. They determine these fibrations, but their reductions modulo $p$ are equal, i.e., these two fibrations modulo $p$ are identical up to $p$-powers. What happened? Here is a helpful diagram:

\begin{center}
    \begin{tikzcd}
        X\to Y \ar[r, mapsto] & \op{Diff}_Y(X) \ar[r, mapsto]& T_{X/Y}\\
        \text{Fibrations on $X$} \ar[r, hook]\ar[d,"\text{mod } p"] &  \text{Foliations' on $X$} \ar[r, "\simeq"]\ar[d,"\text{mod } p", dashed] & \text{Foliations on $X$}\ar[d,"\text{mod } p", dashed] \\
        \text{Fibrations on $X_0$} \ar[r, hook]& \text{$\infty$-Foliations on $X_0$}\ar[r,two heads]& \text{$1$-Foliations on $X_0$}\\
        X_0\to Y_0 \ar[r, mapsto] & \op{Diff}_{Y_0}(X_0) \ar[r, mapsto]& T_{X_0/Y_0}
    \end{tikzcd}
\end{center}

The dashed arrows are well defined only on the elements from the first case. The arrows are dashed to \emph{emphasize} it.

The observation is that over $p$-adic integers (or in characteristic zero), derivations inside a subalgebra of differential operators determine\footnote{``Determine" is not  ``generate". The generation is true in characteristic zero, but not for the $p$-adic integers.} this subalgebra. Above, we denote ``foliations given by derivations'' by Foliations, and  ``foliations given by subalgebras of differential operators'' by Foliations'. These two notions are equivalent, but this stops being true after reducing them modulo $p$. The subalgebras are richer modulo $p$.

Consequently, the better reduction of $\F$ modulo $p$ is to take the (saturated) subalgebra of differential operators generated by it, and then to take the reduction modulo $p$ of that subalgebra. The result will be an $\infty$-foliation extending $\F$ modulo $p$. And, $\infty$-foliations differentiate between fibrations, because nice fibrations inject into $\infty$-foliations, so we no longer lose this information.

This discussion can be reversed into Theorem \ref{lifting => extension - thm}. Here we provide its non-formal version.

\begin{thm}[{Slogan version of Theorem \ref{lifting => extension - thm}}]
    Let $\F_0$ be a $1$-foliation on a variety $X_0$. If there is a lifting of $\F_0$ to a foliation $\F$ on $X$, which is a lifting of $X_0$, then $\F$ induces a unique extension of $\F_0$ to an $\infty$-foliation on $X_0$.

    Explicitly, $\F$ determines a saturated subalgebra of differential operators on $X$. The reduction modulo $p$ of this subalgebra is the extension.
\end{thm}

This is a remarkable result, because from one lifting, one vector bundle, we get infinitely many vector bundles, the Jacobson sequence of the $\infty$-foliation.

\subsection{Canonical Divisor Formula}

In the positive characteristic, purely inseparable morphisms are a new type of bridge allowing for varieties' properties to flow through; new if compared to the characteristic zero. One of the basic holders of important geometric information is the canonical divisor, e.g., it is related to rationality questions. The following is our formula that extends Ekedahl's formula \cite{EkedahlFoliation1987}[Corollary 3.4] from exponent one to all purely inseparable morphisms. And after it, we apply it to unirationality.

\begin{thm}[{Canonical Divisor Formula: Theorems \ref{Canonical Divisor Formula - theorem} and \ref{canonical divisor for n-foli and ekedahl - prop}}]
    Let $f: X\to Y\to X^{(n)}$ be a purely inseparable morphism of exponent $n$ between normal varieties over a perfect field $k$ of characteristic $p>0$. We assume that the variety $X/k$ is geometrically connected.
    Let $\mathcal{F}_1=T_{X/Y}\coloneqq \Omega_{X/Y}^*$. Let $K_X, K_Y$ be canonical divisors of $X, Y$ respectively.
    
    Then, there exist explicit divisors $D_2,D_3,\ldots, D_n$ on $X$ such that
    \[
    f^* K_Y -K_X = (p^n - 1) (-c_1(\mathcal{F}_1)) + \sum_{i=2}^n (p^{n+1-i} - 1)D_i,
    \]
    where $c_1$ is the first Chern class.
    Explicitly, these divisors are given by, for $i\ge 2$, 
    \[
    D_i\coloneqq f_{i-1}^*c_1\left(\frac{T_{Y_{i-1}/k}}{\F_i+\G_{i-1}}\right),
    \]
    where $Y_\bullet$ is the power tower of $f$ on $X$ with $f_i:X\to Y_i$, $\mathcal{F}_\bullet$ is its Jacobson sequence, and $\G_{i} \coloneqq T_{Y_i/Y_{i-1}^{(1)}}$.

    Moreover, if $f: X\to Y$ is a $n$-foliation on $X$, then $D_i\ge 0$, for $i\ge 2$. 
    
    If $f: X\to Y$ is an Ekedahl $n$-foliation on $X$, then $D_i= 0$, for $i\ge 2$.
\end{thm}

\begin{question}
    What are possible collections of divisors coming from the above formula?
\end{question}

We can use the above formula to prove some ``not-L{\" u}roth theorems'' for projective spaces.

\begin{thm}[{Theorem \ref{big not-luroth theorem}}]
    Let $n>1$ be an integer. 
    Let $\mathbb{P}^n_k$ be the $n$-dimensional projective space over a perfect field $k$ of characteristic $p>0$. 
    Let $p-1$ not divide $n+1$.
    
    If $f: \mathbb{P}^n_k\to Y$ is a purely inseparable morphism between normal varieties over $k$, then the Iitaka dimension of a canonical divisor of $Y$ is equal $-\infty$ or $n$, i.e., we have
    \[
    \kappa(Y, K_Y) = -\infty, \text{or } n.
    \]
\end{thm}

This is interesting, because for $n=2$, we cannot have $\kappa(Y, K_Y)=0$, which is puzzling. Indeed, K3 surfaces are such $Y$s, and in the positive characteristic, there are unirational K3 surfaces, see e.g. \cite{Liedtke_K3_unirational_part1},\cite{Liedtke_K3_unirational_part2}, and the rational maps $\mathbb{P}^2\dashrightarrow Y$ are not separable on the generic points. Our result implies that this map cannot be obtained directly from a single purely inseparable morphism, and thus some birational modifications must be present.

\begin{question}
    I believe that if $p-1$ divides $n+1$, then $\kappa(Y, K_Y) \ne -\infty, \text{and } n$ is possible. Of course, it will equal zero then, but I don't have any example. 
    
    Can we uniformly describe all cases when $p-1$ divides $n+1$ and $\kappa(Y, K_Y)=0$? Are these morphisms ``special''?
\end{question}

\subsubsection*{Conventions}\label{Conventions}

\begin{itemize}
    \item 
    Let $X$ be a scheme over a field $k$. The scheme $X$ is a \emph{\textbf{variety}} over the field $k$ if it is integral, and the morphism $X\to \op{Spec}(k)$ is separated and of finite type.
\end{itemize}

\subsubsection*{Structure of the paper} The paper consists of $7$ main sections:
\begin{enumerate}
    \item Introduction. You have just read it.
    \item Preliminaries, where we fix notation and recall some basic facts about fields, purely inseparable subfields/morphisms, differential operators, and divided powers. Though some facts might be actually new, we prove them.
    \item We review previous methods for describing purely inseparable subfields/morphisms that directly precede the developments of this paper.
    \item We develop a theory of power towers on fields.
    \item We develop a theory of power towers on normal varieties.
    \item We discuss a subproblem of a moduli problem for power towers. Namely, a problem of extending power towers into longer power towers.
    \item We prove the formula for the pullback of canonical divisors with respect to an arbitrary purely inseparable morphism, and we give an application of this formula.
\end{enumerate}

\subsubsection*{Acknowledgment}
The results of this paper were a part of the author’s PhD
thesis (University of Amsterdam 2025), \cite{Grabowski_PhD_thesis}. The thesis is available online: \url{https://hdl.handle.net/11245.1/5c74a618-d534-4030-a083-93e68cd90fc8}. I am grateful to my supervisors, Diletta Martinelli and Mingmin Shen, for giving me this opportunity and their great support. In particular, my PhD research has been carried out at the Korteweg-de Vries Institute for Mathematics
and was funded by the MacGillavry Fellowship of Diletta Martinelli.

I would like to thank my PhD committee:  P. Cascini,
 B.J.J. Moonen,
 Z. Patakfalvi,
 S. Shadrin,
 A.L. Kret,
 J.V. Stokman for being fierce opponents.

I would like to thank an anonymous reviewer who rejected my first preprint containing some results about power towers on fields. Their work was quick and their comments immensely helpful!

I would like to thank Piotr Achinger, Fabio Bernasconi, Federico Bongiorno, Lukas Brantner, Remy van Dobben de Bruyn, Paolo Cascini, Zsolt Patakfalvi, Calum Spicer, Roberto
Svaldi, Lenny Taelman, and Joe Waldron for discussions about purely inseparable Galois theories, foliations, and related topics. These were helpful for the initial formal development of this theory, sharing its content with others, and keeping me motivated.

The work of adjusting the thesis into this paper has been carried out while employed as a Simons postdoc (2025-2028) in mathematics at Kyiv School of Economics.

\newpage
\section{Preliminaries}
This section fixes notation and recalls key definitions that are used in this paper, as well as some basic facts about them. Some propositions might actually be new. An expanded version of this preliminary can be found in my thesis \cite{Grabowski_PhD_thesis}.

\subsection{Definitions and Basic Facts about Fields}

Most of the field theory in this paper is standard and easily accessible. For example, here are some sources: Section 09FA in \cite{stacks-project}, Jacobson's book \cite{Jacobson1964:FieldsAndGalois}, and Bourbaki's book \cite{Bourbaki:Algebra2Chapter4-7}.

\subsubsection{Purely Inseparable}

\begin{defin}\label{Purely Inseparable Extension - Definition}
    Let $K$ be a field of characteristic $p>0$.

    A field extension $K/W$ is \emph{purely inseparable} if for every $x\in K$ there is a natural number $n$ such that $x^{p^n}\in W$.
\end{defin}

\begin{defin}[Subfields of $p$-Powers]\label{Subfields of $p$-Powers: definition}
Let $K$ be a field of characteristic $p>0$. 
Let $n\ge 0$ be an integer.
A subfield of $p^n$-powers of $K$ is the subfield $K^{p^n}\coloneqq \{ x^{p^n}; x\in K \}$.  
\end{defin}

\begin{defin}[Supfields of $p$-Roots]\label{Supfields of $p$-Roots: definition}
Let $K$ be a field of characteristic $p>0$. 
Let $n\ge 0$ be an integer.
A supfield of $p^n$-roots of $K$ is the subfield $K^{1/p^n}\coloneqq \{ x^{1/p^n} \in \overline{K}; x\in K \}$ of an algebraic closure $\overline{K}$ of $K$.  
\end{defin}

\begin{defin}
    Let $K$ be a field of characteristic $p>0$. 
    The field $K$ is \emph{perfect}, if $K=K^p$.
\end{defin}

The following naming ``(co-)perfection'' I took from \cite[Definition 3.4.1]{Kedlaya_prismatic}.

\begin{defin}\label{Coperfection: Definition}
    A \emph{coperfection}, or a \emph{perfect closure}, of a field $K$ is a supfield $K^{1/p^\infty}\coloneqq\bigcup_{n\ge 0} K^{1/p^n}$, where $p$-roots and the union are taken in a single algebraic closure of $K$.
\end{defin}

\begin{defin}\label{Perfection: Definition}
    Let $K$ be a field of characteristic $p>0$. 
    A \emph{perfection} of a field $K$ is the subfield $K^{p^\infty}\coloneqq \bigcap_{n\ge 0} K^{p^n}$.
\end{defin}

\begin{lemma}\label{The Largest Perfect Subfield}
    The perfection of a field is a perfect field. Moreover, it is the largest perfect subfield of this field.
\end{lemma}

\begin{proof}
    Let $K$ be a field with $\op{char}(K)=p>0$.
    We prove its perfection is perfect. Indeed, we have:
    \[
    {\left(K^{p^\infty}\right)}^p=\left(\bigcap_{n\ge 0} K^{p^n}\right)^p=\bigcap_{n\ge 0} K^{p^{n+1}}=\bigcap_{n\ge 0} K^{p^n}=K^{p^\infty}.
    \]
    
    We prove it is the largest perfect subfield. We notice that if $L$ is a perfect subfield, then, for every $n\ge 0$, we have $K^{p^n}\supset L^{p^n}=L^{p^{n-1}}=\ldots=L$. Therefore, we conclude $\bigcap_{n\ge 0} K^{p^n}\supset L$.
\end{proof}

\begin{defin}\label{Exponent: Definition}
Let $K$ be a field of characteristic $p>0$.
Let $n\ge 0$ be an integer, or $\infty$.

A field extension $K/W$ is of \emph{exponent $\le n$} if  $K\supset W\supset K^{p^n}$. And, it is of exponent $n$, if the $n$ is minimal such.
\end{defin}

\begin{remark}
    In the literature, at least two names are used to refer to exponents of subfields. For example, the first one is the \emph{exponent} used in Jacobson's book \cite{Jacobson1964:FieldsAndGalois}. For instance, the second one is a \emph{height} used in Bourbaki's book \cite{Bourbaki:Algebra2Chapter4-7}. I prefer to use exponents.
\end{remark}

\begin{lemma}\label{Purely Inseparable is of Finite Exponent}
Let $K$ be a field of characteristic $p>0$ such that $K/K^{p^\infty}$ is a finitely generated field extension. 

Then every purely inseparable subfield of $K$ of exponent $\le \infty$ is of finite exponent.
\end{lemma}

\begin{proof}
    Let $K/W/K^{p^\infty}$ be a purely inseparable extension of exponent $\le \infty$. Let $f_i$ be a finite set of generators of $K$ over $K^{p^\infty}$. Then, for every $f_i$ there is an integer $n_i>0$ such that $f_i^{p^{n_i}}\in W$. Let $N$ be the greatest integer among them. Then,
    \[
    K^{p^{N}}=K^{p^\infty}(f_i^{p^{N}})\subset W.
    \] 
    This proves that $K/W$ is of exponent $\le N$.
\end{proof}

\pagebreak
\subsubsection{$p$-basis and separable transcendental basis}

\begin{lemma}[Separable Closure, {\cite{stacks-project}[09HC, 030K]}]\label{SeparableClosure: definition}
Let $L/K$ be an extension of fields. There is a maximal subextension $L/K^s/K$ such that $K^s/K$ is separable algebraic, where the maximality means that if $L/F/K$ is a subextension such that $F/K$ is separable algebraic, then $F \subset K^s$. The field $K^s$ is called a \emph{separable closure} of $K$ in $L$.

Moreover, if $L/K$ is algebraic, then $L/K^s$ is purely inseparable.
\end{lemma}

The following two definitions are from \cite[030O (2)]{stacks-project}, or \cite[Definition on page 557]{Eisenbud:Commutative}.

\begin{defin}
    Let $K/W$ be a field extension. A subset $S\subset K$
    is a \emph{separable transcendence basis} of $K/W$ if $S$ is a transcendence basis of $K/W$, and the extension $K/W(x)_{x\in S}$ is separable algebraic, i.e., $\left(W(x)_{x\in S}\right)^s=K$.

    If $K/W$ admits a separable transcendence basis, then we say that $K$ is \emph{separably generated} over $W$.
\end{defin}

\begin{defin}\label{Separable: Definition}
    A field extension $K/W$ is a \emph{separable} extension if for every subextension $K/K'/W$ with $K'$ finitely generated as a field over $W$, the extension $K'/W$ is separably generated.
\end{defin}

\begin{cor}[MacLane 1939, \cite{MacLane:ModularFields}]\label{Over Perfect is Separable}
    Let $K/k$ be a field extension with $k$ being a perfect field. Then, this extension is separable.
\end{cor}

Some basic facts about $p$-bases, I took them from \cite[Section 07P0]{stacks-project}.

\begin{defin}\label{p-basis definition}
    Let $p$ be a prime number. Let $K/k$ be an extension of fields of characteristic $p$. Denote $k\cdot K^p$ the compositum of $k$ and $K^p$ in $K$.
    \begin{enumerate}
        \item A subset $\{x_i\}\subset K$ is called $p$-independent over $k$ if the elements $x^e\coloneqq \prod x_i^{e_i}$ where $0\le e_i <p$ are linearly independent over $k\cdot K^p$.
        \item A subset $\{x_i\}$ of $K$ is called a $p$-basis of $K$ over $k$ if the elements $x^E$ form a basis of $K$ over $k\cdot K^p$.
    \end{enumerate}
\end{defin}

\begin{lemma}\label{p-basis - characterisation}
   Let $K/k$ be a field extension. Assume $k$ has characteristic $p>0$. Let $\{x_i\}$ be a subset of $K$. The following are equivalent:
   \begin{enumerate}
       \item the elements $\{x_i\}$ are $p$-independent over $k$,
       \item the elements $dx_i$ are $K$-linearly independent in $\Omega_{K/k}$.
   \end{enumerate}

Consequently, any $p$-independent collection can be extended to a $p$-basis of $K$ over $k$. In particular, the field $K$ admits a $p$-basis over $k$. Moreover, the following are equivalent:
    \begin{enumerate}
        \item $\{x_i\}$ is a $p$-basis of $K$ over $k$,
        \item $dx_i$ is a basis of the $K$-vector space $\Omega_{K/k}$.
    \end{enumerate}
\end{lemma}

These two notions, a $p$-basis and a separable transcendence basis, will be equivalent for us. The following is Corollary \ref{Over Perfect is Separable} combined with \cite[Corollary A1.5]{Eisenbud:Commutative}].

\begin{thm}\label{p-basis is separable transcedence basis}
    Let $k$ be a perfect field of characteristic $p>0$.
    Let $K/k$ be a finitely generated field extension.
    Then, every $p$-basis of $K$ over $k$ is a separable transcendence basis of $K/k$, and vice versa.
\end{thm}

\subsection{Purely Inseparable Morphisms and Frobenius Morphisms}

We already defined purely inseparable morphisms in Definitions \ref{purely inseparable - rings - def} and \ref{purely inseparable - varieties - def}. These definitions are inspired by private notes on positive characteristic algebraic geometry by Joe Waldron and Zsolt Patakfalvi. I mention it because ``purely inseparable morphisms'' sometimes are defined as any morphisms that induce purely inseparable subfields on generic points. However, such ``purely inseparable morphisms'' are too wobbly, too shapeless; thus, we have extra ``rigidity'' conditions on what we call a purely inseparable morphism. Luckily, this rigidity makes the connection with power towers easy.

\begin{defin}[Absolute Frobenius Morphism]
    Let $X$ be a scheme over $\mathbb{F}_p$. Then taking $p$-powers in local affine charts defines a morphism $F_X: X\to X$. This is the \emph{absolute Frobenius morphism} of $X$.
\end{defin}

\begin{defin}[Relative Frobenius Morphism, {\cite[Definition 0CC9]{stacks-project}}]\label{relative Frob twist - def}
    Let $S$ be a scheme over $\mathbb{F}_p$. Let $X$ be an $S$-scheme. Then a \emph{relative Frobenius morphism} of $X/S$, denoted by $F_{X/S}$, is defined as the unique dashed arrow from the diagram below.
    \begin{center}
        \begin{tikzcd}
            X\ar[r,dashed, "F_{X/S}"]\ar[rd]\ar[rr, bend left =45,"F_X"]& X^{(1)}\ar[d]\ar[r]& X\ar[d]\\
            & S \ar[r, "F_S"]& S,
        \end{tikzcd}
    \end{center}
    where $X^{(1)}$ is a pullback of $X\to S$ and $F_S$.
\end{defin}

\begin{defin}
    We define $X^{(2)}$ to be ${X^{(1)}}^{(1)}$, $X^{(3)}$ to be ${{X^{(2)}}}^{(1)}$, and so on.
\end{defin}

The following states that, for a normal variety, purely inseparable morphisms from it and purely inseparable subfields of its generic point are in a natural bijection with each other.

\begin{prop}\label{purely inseparable equiv on fields}
   Let $X$ be a normal variety over a perfect field $k$ of characteristic $p>0$, then purely inseparable morphisms from $X$ are in a bijection with subfields $W\subset K(X)$ of finite exponent, i.e., the ones for which there exists $r$ such that $K(X)^{p^r}\subset W$.

   In particular, purely inseparable factorizations of $F^r_{X/k}$ given as $X\to Y \to X^{(r)}$ with $Y$ being a normal variety over $k$, are in bijection with purely inseparable subfields of $K(X)$ of exponent $\le r$.
\end{prop}

\begin{proof}
First, we note that we can restrict the reasoning to finite exponents by Lemma \ref{Purely Inseparable is of Finite Exponent}.

Our proof is an adaptation of the proof of \cite[Proposition 2.9.]{ZsoltJoe-1-foliations}.
Indeed, we have a sequence of equivalent notions:
\begin{itemize}
    \item purely inseparable morphisms $X\to Y \to X^{(r)}$ with $Y$ being a normal variety over $k$,
    \item coherent sheaves of normal rings $\mathcal{O}_X\supset \mathcal{A} \supset \mathcal{O}_X^{p^r}$,
    \item intermediate subfields $K(X)\supset W\supset K(X)^{p^r}$.
\end{itemize}

The equivalences are respectively given by:
\begin{itemize}
    \item $X\to Y \to X^{(r)} \mapsto \mathcal{O}_X\supset \mathcal{O}_Y \supset \mathcal{O}_X^{p^r}$
    \item $\mathcal{O}_X\supset \mathcal{O}_Y \supset \mathcal{O}_X^{p^r} \mapsto K(X)\supset K(Y) \supset K(X)^{p^r}$.
\end{itemize}

These operations are bijections, because $X, Y$ have the same underlying topological space, but different functions on it. To recover $Y$ we take the space of $X$ with the sheaf $\mathcal{O}_Y$ defined as the normalization of $\mathcal{O}_X^{p^r}=\mathcal{O}_{X^{(r)}}$ in $W$ from $K(X)\supset W\supset K(X)^{p^r}$.
\end{proof}

\subsection{Saturation}

In Proposition \ref{purely inseparable equiv on fields}, we used normalization of a ring/sheaf to get a bijection. On the other side of our Galois-type correspondence for power towers, this operation will correspond to taking a saturation.

\begin{defin}\label{Saturation: Definition}
    Let $X$ be a variety over a field $k$. Let $M$ be a torsion-free coherent sheaf on $X$. Let $\mathcal{M}=M\otimes K(X)$ be the pullback of $M$ to the generic point of $X$. Let $\mathcal{N}$ be a $K(X)$-vector subspace of $\mathcal{M}$.

    We define the \emph{saturation} of $\mathcal{N}$ in $M$ to be the following subsheaf $N$ of $M$:
    \[
    N(U) \coloneqq \mathcal{N} \cap M(U),
    \]
    where $U$ is an open subset of $X$, and $M(U)\to \mathcal{M}$ is the localization.

    Finally, we say that a subsheaf $N\subset M$ is \emph{saturated} if it is equal to its saturation. (Equivalently, it means that $M/N$ is torsion-free.)
\end{defin}

\begin{lemma}\label{Saturation is Saturated - Lemma}
    Let $X$ be a variety over a field $k$. Let $M$ be a torsion-free coherent sheaf on $X$.
    Let $\mathcal{N}$ be a subspace of $\mathcal{M}$.
    
    Then, the saturation of $\mathcal{N}$ in $M$ is a saturated coherent subsheaf of $M$.
\end{lemma}

\begin{proof}
    Trivially, the saturation of a vector subspace is a coherent subsheaf, because it is quasicoherent and inside a coherent one. It is saturated, because if not, then the quotient has a torsion, so there is a local section $x$ in $M$, and a local regular function $f$ such that $fx$ belongs to the saturation, but $x$ does not. This is a contradiction with the definition of saturation.
\end{proof}

\begin{lemma}\label{Saturated are in bij with gen subspaces - lemma}
Let $X$ be a variety over a field $k$. Let $M$ be a torsion-free coherent sheaf on $X$. There is a bijection between saturated coherent subsheaves of $M$ and vector subspaces of $\mathcal{M}=M\otimes K(X)$. Explicitly, it is given by pullbacking to the generic point, and the inverse is the saturation in $M$.
\end{lemma}

\begin{proof}
We have to show that these operations are well-defined and inverse to each other.

They are well defined as the first is a pullback, the second one is well defined by Lemma \ref{Saturation is Saturated - Lemma}.

We show that they are inverse to each other. 
Indeed, if $\mathcal{N}$ is a subspace, then its saturation $N$ in $M$ determines every $N(U)$, and these sets determine $\mathcal{N}$. If $N$ is a saturated subsheaf, then the saturation of $\mathcal{N}=N\otimes K(X)$ is equal to $N$ by the definition.
\end{proof}

\subsection{Differential Operators}\label{Differential operators - section}

This section is an exposition of an algebraic theory of differential operators based on Grothendieck's EGA IV \cite[Chapter 16]{EGAIV}.

Derivations over a polynomial ring $R=k[x_1,\ldots,x_n]$ can be given by formulas of the form $\sum_{i=1}^n a_i\frac{\partial}{\partial x_i }$, where $a_i\in R$. Suppose $k$ is a field of characteristic zero. In that case, any differential operator on $R$ is a linear combination of compositions of derivations, i.e., it can be written as $\sum_{I=(c_i)\in \mathbb{N}^n} a_I \prod_{i=1}^n\frac{\partial}{\partial x_i }^{c_i}$, where $\mathbb{N}=\{0,1,2,\ldots\}$, $a_I\in R$. However, in positive characteristic, it is not true, e.g., $\frac{\partial}{\partial x_i }^p=0$, so we cannot reach the operator $\frac{1}{p!}\frac{\partial}{\partial x_i }^p$ via compositions and sums. Yet, this operator must be included, because the natural construction of differential always gives them, or for pragmatic reasons: this operator allows us to get the coefficient next to $x^p$. Consequently, the algebraic theory of differential operators modulo $p$ initially admits more complexity, but one can quickly realize that this ``pathology'' is a blessing.

\subsubsection{Inductive Definition} Leibniz's rule definition of a derivation generalizes to the following one for differential operators.

\begin{defin}\label{Differential Operators: Explicit Definition}
Let $R$ be a ring and $k$ its subring.

    For $d\ge -1$ we define inductively the module $F_d \mathcal{D}_{R/k}$ of differential operators of order $\le d$ for $R/k$ as follows:
    \begin{itemize}
        \item $F_{-1} \mathcal{D}_{R/k} \coloneqq\{0\}\subset \op{End}_k(R)$.
        \item For $d\ge 0$, $F_d \mathcal{D}_{R/k}\coloneqq \{D\in \op{End}_k(R); \ \forall_{f\in R} \ [D,f]\in F_{d-1} \mathcal{D}_{R/k}\}$.
    \end{itemize}

The algebra of differential operators for $R/k$ is defined by
\[
\mathcal{D}_{R/k} \coloneqq \bigcup_d F_d \mathcal{D}_{R/k}\subset \op{End}_k(R).
\]
\end{defin}

\begin{example} \label{Diff Oper on Affine Line: Definiton}
    Let $R=k[x]$ be a ring of polynomials in one variable over a ring $k$. Let $a\ge 0$ be an integer. We can define a $k$-linear function $\frac{1}{a!}\frac{\partial^{a}}{\partial x ^{a}}$ from $R$ to $R$ by the formula:
    \[
    \frac{1}{a!}\frac{\partial^{a}}{\partial x ^{a}}(x^b)\coloneqq {{b}\choose{a}} x^{b-a}.
    \]
    It is a differential operator of order $a$.
    The algebra $\mathcal{D}_{R/k}$ is a free $R$-module spanned by all these operators $\frac{1}{a!}\frac{\partial^{a}}{\partial x ^{a}}$.
\end{example}

\subsubsection{Universal Construction}

The inductive definition is pretty, but it is hard to work with. Here we will describe a definition of differential operators analogous to the one for derivations via the module of K{\" a}hler differentials.

\begin{defin}\label{Set Up: Differential Operators}
    Let $k\to R$ be a morphism of rings. We define the following objects:
    \begin{itemize}
        \item A first projection $p_1: R \to R\otimes_k R$ is given by $a \mapsto a\otimes 1$.
        \item A second projection $p_2: R \to R\otimes_k R$ is given by $a \mapsto 1\otimes a$.
        \item A diagonal map $\Delta_{R/k}: R\otimes_k R\to R$ is given by $a\otimes b \mapsto ab$.
        \item The kernel of $\Delta_{R/k}$ is denoted by $J_{R/k}$.
        \item A universal ``derivation'' $d: R \to R\otimes_k R$ is equal $p_2 - p_1$, i.e., it is given by $a \mapsto 1\otimes a - a\otimes 1$. 
    \end{itemize}
\end{defin}

\begin{convention}\label{convention - diff algebra}
    Let $k\to R$ be a morphism of rings. Unless stated otherwise, the ring $R\otimes_k R$ is considered an $R$-algebra via $p_1$.
\end{convention}

\begin{defin}\label{space of principal parts: definition}
    Let $k\to R$ be a morphism of rings. A \emph{space of principal parts} of $R$ over $k$ of order at most $n$ is
    defined by
    \[
    \mathcal{P}^{\le n}_{R/k}\coloneqq R \otimes_k R/ J^{n+1}_{R/k}.
    \]
\end{defin}

\begin{remark}\label{Principal Parts - Surjection}
    For $m\ge n$, there is a natural surjective map $\mathcal{P}^{\le m}_{R/k} \to \mathcal{P}^{\le n}_{R/k}$, i.e., the quotient map.
\end{remark}

\begin{defin}\label{Differential Operators: Universal Definition}
    Let $k\to R$ be a morphism of rings. 
    
    The set of \emph{differential operators} of order at most $n$ on $R$ relative to $k$ is 
    \[
    \op{Diff}^{\le n}_{k}(R)\coloneqq \op{Hom}_R(\mathcal{P}^{\le n}_{R/k}, R).
    \]

    The set of \emph{differential operators} on $R$ relative to $k$ is 
    \[
    \op{Diff}_{k}(R)\coloneqq \op{colim}_{n\ge 0} \op{Diff}^{\le n}_{k}(R)=\bigcup_{n\ge 0} \op{Diff}^{\le n}_{k}(R),
    \]
    where the sum is taken with respect to the natural identifications.

    The filtration of $\op{Diff}_{k}(R)$ by $\op{Diff}^{\le n}_{k}(R)$ is called the (standard) order filtration.
\end{defin}

\subsubsection{Comparison of definitions}

\begin{lemma}\label{Diagonal is augmented}
    Let $k\to R$ be a morphism of rings.
    The $R$-algebra $R\otimes_k R$ naturally decomposes as an $R$-module into $R\oplus J_{R/k}$, because $p_1$ splits $0\to J_{R/k} \to R\otimes_k R \to R \to 0$.
\end{lemma}

\begin{prop}\label{E(K/W) definition}
    Let $k\to R$ be a morphism of rings.
    There is a submodule $\E(R/k)$ of $\op{Diff}_{k}(R)$ corresponding to $J_{R/k}$ in the decomposition $R\oplus J_{R/k}=R\otimes_k R$ such that  $\op{Diff}_{k}(R)=R\oplus \E(R/k)$. Namely, it is given by $\E(R/k)\coloneqq \bigcup_{n\ge 0}\op{Hom}_R( J_{R/k}/J_{R/k}^{n+1},R)$.
\end{prop}

\begin{remark}
    The above notation $\E(R/k)$ is very non-standard. I introduced it in my thesis. A naive notation for this object would be something like
    $\op{Diff}^{\ge 1}_{k}(R)$, but it is vague and super bad, so I want to avoid it. The name of the font is ``textswab''.
\end{remark}

\begin{defin}
    Let $k\to R$ be a morphism of rings. Let $\Phi$ be a differential operator of order at most $n$ on $R$ relative to $k$, i.e., a morphism of $R$-modules $\mathcal{P}^{\le n}_{R/k} \to R$. A $k$-linear operator on $R$ corresponding to $\Phi$ is given by
    \[
    D_\Phi\coloneqq \Phi\circ p_2: R\to R\otimes_k R \to \mathcal{P}^{\le n}_{R/k} \to R,
    \]
    where the middle arrow is dividing by the ideal $J_{R/k}^{n+1}$.
    The above operation defines a natural morphism of $R$-modules
    \[
    \op{Diff}_{k}(R) \to \op{End}_k(R): \Phi \mapsto D_\Phi.
    \]
\end{defin}

Finally, the comparison.

\begin{prop}\label{Diff into End}
    Let $k\to R$ be a morphism of rings. The natural map $\op{Diff}_{k}(R) \to \op{End}_k(R)$ is an injection, and the image of $\op{Diff}^{\le n}_{k}(R)$ is $F_n \mathcal{D}_{R/k}$. 
    
    Moreover, the natural decomposition $\op{Diff}_{k}(R)= R\oplus \E(R/k)$ in terms of endomorphisms is given by 
    \[
    D\mapsto (D(1),D-D(1)).
    \]
    So, the module $\E(R/k)$ consists of differential operators whose value on $1$ is zero.
\end{prop}

\subsubsection{Graded Algebras of Differential Operators}

\begin{defin}
    Let $k\to R$ be a morphism of rings. We define a \emph{graded algebra of differential operators} on $R$ over $k$, denoted by $\op{gr}\op{Diff}_k(R)$, by the following formulas:
    \begin{align*}
      \op{gr}^i \op{Diff}_k(R) &\coloneqq \op{Diff}^{\le i}_k(R)/\op{Diff}^{\le i-1}_k(R),\\
      \op{gr}\op{Diff}_k(R) &\coloneqq \bigoplus_{i=0}^\infty \op{gr}^i \op{Diff}_k(R),
    \end{align*}
    where $\op{Diff}^{\le i}_k(R)=0$ for $i<0$.
\end{defin}

\begin{defin}
    Let $k\to R$ be a morphism of rings. Let $D\in \op{Diff}_k(R)$ be a differential operator of order $i$. We define the \emph{leading form} of $D$, denoted by $[D]$, to be the image of $D$ with respect to the composition of $\op{Diff}^{\le i}_k(R)\to \op{gr}^i \op{Diff}_k(R)\to \op{gr}\op{Diff}_k(R) $.
\end{defin}

\begin{defin}
    Let $k\to R$ be a morphism of rings. Let $S$ be a subset of $\op{Diff}_k(R)$. We define a graded subset of $S$, or a gradation of $S$, by the formula $\op{gr}(S)\coloneqq \{[D]: D\in S\}$.
\end{defin}

\begin{lemma}\label{graded d in terms of diagonal - lemma}
    Let $k\to R$ be a morphism of rings. We assume $J_{R/k}/J_{R/k}^{2}$ is locally free.

    Then, we have a natural isomorphism $\op{gr}^i\op{Diff}_k(R) \simeq \op{Hom}_R (J_{R/k}^i/J_{R/k}^{i+1},R)$.
\end{lemma}

\begin{proof}
If $J_{R/k}/J_{R/k}^{2}$ is locally free, then all $R\otimes_k R/J_{R/k}^{i}=R\oplus J_{R/k}/J_{R/k}^{i}$ are locally free, so we have $\op{Ext}^1_R(R\otimes_k R/J_{R/k}^{i},R)=0$.

Consequently, we have
    \begin{align*}
    \op{gr}^i\op{Diff}_k(R)=&\op{Diff}^{\le i}_k(R)/\op{Diff}^{\le i-1}_k(R)=\\ 
    =& \op{Hom}_R (R\otimes_k R/J_{R/k}^{i+1},R)/\op{Hom}_R (R\otimes_k R/J_{R/k}^{i},R)\\ 
    \simeq& \op{Hom}_R (J_{R/k}^i/J_{R/k}^{i+1},R).
    \end{align*}
\end{proof}

\subsubsection{Explicit Formulas Using Coordinates}

The propositions and definitions below are from a short \cite[Section 16.11]{EGAIV}. It allows us to work with differential operators via coordinates. It is useful, because in this way we can check orders of our operators just by writing them down and looking at them.

\begin{defin}\label{coordinates - definition}
    Let $f:k\to R$ be a smooth morphism of rings. Let $\{x_1,x_2,\ldots,x_n \}$ be a \emph{finite} set of elements of $R$. We say that the set $\{x_i \}$ is a set of coordinates for $k\to K$ if it is a basis of $\Omega_{R/k}$.

    Thus, we say that $k\to R$ admits coordinates, if $k\to R$ is smooth, and $\Omega_{R/k}$ is free.
\end{defin}

\begin{example}
Our two cases of admitting coordinates that we use in this paper are:
\begin{itemize}
    \item A geometrically regular (i.e., smooth) variety over a perfect field admits coordinates locally. Any local basis of K{\" a}hler differentials is a set of coordinates.
    \item If $K$ is a field of characteristic $p>0$, and $K/K^{p^\infty}$ is finitely generated, then $K^{p^\infty}\to K$ admits coordinates. In this case, it is any $p$-basis.
\end{itemize}
   
\end{example}

\begin{remark}
    The above assumption about smoothness is there solely to get the following result that is REALLY what we want from ``\textbf{admitting coordinates}'', and the only thing we use from that assumption in the whole paper. Thus, in this paper, the next proposition could be taken as a definition of ``admitting coordinates''.
\end{remark}

\begin{prop}[Monomials Generate Spaces of Principal Parts]\label{Monomials in dx are a basis}
    Let a ring map $k\to R$ admit coordinates. Let $\{x_i\}$ be a set of coordinates of $R$ over $k$. Then monomials $d(x)^e\coloneqq\prod d(x_i)^{e_i}$, with $e_i \ge 0$ and $\sum e_i \le n$, form a $R$-basis of $\mathcal{P}_{R/k}^{\le n}$.
\end{prop}

\begin{defin}
     Let $k\to R$ be a morphism of rings that admits coordinates. Let $\{x_i\}$ be a finite set of coordinates of $R$ over $k$. We denote the $R$-basis of $\op{Diff}^{\le n}_k(R)$ that is dual to the $R$-basis $d(x)^e$ of $\mathcal{P}_{R/k}^{\le n}$ by $D^e$, i.e., $D^e(d(x)^f)=\begin{cases}
1, \ \text{if } e=f \\
0, \ \text{if } e\ne f
\end{cases}$. We will refer to this basis as the \emph{dual basis} of $\op{Diff}^{\le n}_k(R)$ corresponding to the set of coordinates $\{x_i\}$ for $k\to R$.
\end{defin}

\begin{remark}
The dual basis strongly depends on the choice of the set of coordinates, e.g., both sets $\{x,y\}$ and $\{x, x+y\}$ are sets of coordinates for $k[x,y]$ over $k$. Still, their dual bases are not equal, though it will often happen that they will have operators that are denoted exactly the same way, but are not actually equal. For example, we have: $\{x,y\}$ gives rise to $\frac{\partial}{\partial x}$, $\frac{\partial}{\partial y}$, and $\{x,x+y\}$ gives rise to $\frac{\partial}{\partial x}$, $\frac{\partial}{\partial x+y}$. However, $\frac{\partial}{\partial x}\ne \frac{\partial}{\partial x}$! Therefore, whenever we encounter explicit formulas, it is important to know the whole set of coordinates; otherwise, it is likely that we will make some silly mistakes.
\end{remark}

\begin{cor}\label{Dual Basis of Diff}
    Let $k\to R$ be a morphism of rings that admits coordinates. Let $\{x_i\}$ be a \emph{finite} set of coordinates of $R$ over $k$. The operators $D^e$ with $\sum e_i \ge 0$ are a basis of $\op{Diff}_k(R)$.
    We call this basis the \emph{dual basis} of $\op{Diff}_k(R)$ corresponding to the coordinates $x_i$.
\end{cor}

The following is a part of \cite[Theorem 16.11.2]{EGAIV}.
    
\begin{prop}\label{Determination1}
    Let $k\to R$ be a morphism of rings that admits coordinates. Let $\{x_i\}$ be a finite set of coordinates of $R$ over $k$. Let $D^e$ be the dual $R$-basis of $\op{Diff}^{\le n}_k(R)$ corresponding to $\{x_i\}$. These operators are determined by their values on monomials $x^f=\prod_{i} x_i^{f_i} $, where $f_i\ge 0$. These values are:
    \[
    D^e(x^f)={f \choose e}x^{f-e}=\prod_i {f_i \choose e_i} x_i^{f_i-e_i}.
    \]

    We mean $D^e : R\to R$, not $D^e: \mathcal{P}_{R/k}^{\le n} \to R$ for some $n$. (We abuse notation not to write $D_{D^e}$.)
\end{prop}

\begin{defin}\label{symbols diff operators - def}
Let $k\to R$ be a morphism of rings that admits coordinates.  Let $\{x_1,x_2,\ldots, x_n\}$ be a finite set of coordinates of $R$ over $k$.
For every $i=1,2,\ldots, n$ and $a\ge 0$ we define a differential operator $\frac{1}{a!}\frac{\partial^{a}}{\partial x_i ^{a}}$ on $R$ over $k$ by the formula\footnote{In particular, we do not actually divide by $a!$. It is just a (suggestive) symbol.}
    \[
    \frac{1}{a!}\frac{\partial^{a}}{\partial x_i ^{a}}(\prod_{j=1}^n x_j^{b_j})={{b_i}\choose{a}} x_i^{b_i -a}\prod_{j=1, j\ne i}^n x_j^{b_j} .
    \]
This operator's order is $a$.
\end{defin}

\begin{cor}\label{explicit formulas calculus - cor}
    Let $k\to R$ be a morphism of rings that admits coordinates. Let $\{x_1,x_2,\ldots, x_n\}$ be a finite set of coordinates of $R$ over $k$. Let $D^e$ be an operator from the dual basis corresponding to $x_i$ with $e=(e_1,\ldots, e_n)$, then
    \[
    D^e = \prod_{i=1}^{n} \frac{1}{e_i!}\frac{\partial^{e_i}}{\partial x_i ^{e_i}}.
    \]
\end{cor}

\begin{remark}
    One can say that the values of the operators $ \frac{1}{a!}\frac{\partial^{a}}{\partial x_i ^{a}}$ are \emph{the obvious ones}. The same is true about most relations between them because we can check any relation by comparing values on monomials. 
    
    For example, an obvious relation is that all these operators $ \frac{1}{a!}\frac{\partial^{a}}{\partial x_i ^{a}}$ commute with each other. 
    
    For example, if $p=0$, then a non-obvious relation from a differential geometry point of view is $\left(\frac{1}{a!}\frac{\partial^{a}}{\partial x_i ^{a}}\right)^{\circ p}=0$ for $a\ge 1$. This relation is easy to check on monomials.
\end{remark}

\subsubsection{Differential}

We discuss the difference between algebras of differential operators. We will define it under mild assumptions following \cite[page 23]{EGAIV}. 

A motivation for a differential comes from differential geometry. For any smooth map between two manifolds $f: X\to  Y$, we get a differential between their tangent spaces $df: TX\to f^* TY$. However, it is true that $df$ naturally extends to a map $df: \op{Diff}(X)\to f^* \op{Diff}(Y)$ between their spaces of differential operators. This is what we will define in a general algebraic setting.

In characteristic zero, the first order $TX$ determines the algebra $\op{Diff}(X)$, so there is not much gain from considering the extension. However, in general, the algebra $\op{Diff}(X)$ is a richer source of information about $X$ than $TX$ alone, e.g., this exactly happens in positive characteristic.

\begin{defin}\label{delta n - definitions}
    Let $k\to A \xrightarrow{f} R$ be morphisms of rings.
    These maps induce a morphism
    \[
    A\otimes_k A\to R\otimes_k R: a \otimes b \mapsto f(a)\otimes f(b)
    \]
    and this induces morphisms
    \[
    \Delta^{\le n}_{R/A}: \mathcal{P}^{\le n}_{A/k}=A\otimes_k A/J_{A/k}^{n+1}\rightarrow R\otimes_k R/J_{R/k}^{n+1}=\mathcal{P}^{\le n}_{R/k}.
    \]
\end{defin}

\begin{defin}\label{Natural Maps: Differential}
     Let $k\to A \to R$ be morphisms of rings. We have the following natural maps:
     \begin{enumerate}
         \item $\op{Hom}_R (\mathcal{P}^{\le n}_{R/k}, R) \rightarrow
        \op{Hom}_R (\mathcal{P}^{\le n}_{A/k}\otimes R, R):
        \phi \mapsto \phi \circ (\mathcal{P}^{\le n}_{A/k}\otimes R\xrightarrow{\Delta^{\le n}_{R/A}\otimes 1} \mathcal{P}^{\le n}_{R/k})$,
        \item $ \op{Hom}_A (\mathcal{P}^{\le n}_{A/k}, R) \rightarrow \op{Hom}_R (\mathcal{P}^{\le n}_{A/k}\otimes R, R)
       : \phi \mapsto (a \otimes x \mapsto x\phi(a))$,
        \item $ \op{Hom}_A (\mathcal{P}^{\le n}_{A/k}, A) \otimes R
        \rightarrow
        \op{Hom}_A (\mathcal{P}^{\le n}_{A/k}, R): \phi \otimes y \mapsto ((A \to R) \circ \phi) \cdot y$.
     \end{enumerate}
\end{defin}

\begin{remark}
    Item 2 from Definition \ref{Natural Maps: Differential} is always an isomorphism. It is a universal property of extending scalars from $A$ to $R$.
\end{remark}

The following definition is non-standard. I introduced it a bit ad hoc.

\begin{defin}\label{Admit a differential - definition}
    Let $k\to A \to R$ be morphisms of rings. We say that these morphisms \emph{admit the differential} if the map from item 3 of Definition \ref{Natural Maps: Differential} is an isomorphism.
\end{defin}

\begin{defin}[Differential Map for Differential Operators]\label{Differential - definition}
    Let $k\to A \to R$ be morphisms of rings that admit the differential. We define \emph{differentials} 
    \[
    d^{\le n}(R/A): \op{Diff}^{\le n}_k(R)\to \op{Diff}^{\le n}_k(A)\otimes R
    \]
    by being the composition of the map 1. with the inverses of the maps 2. and 3. from Definition \ref{Natural Maps: Differential}: 
    \begin{align*}
    \op{Diff}^{\le n}_k(R)\coloneqq \op{Hom}_R(\mathcal{P}^{\le n}_{A/k}\otimes R, R) &\xrightarrow{\text{The item 1.}} \op{Hom}_R (\mathcal{P}^{\le n}_{A/k}\otimes R, R)\\
    &\xrightarrow{\left(\text{The item 2.}\right)^{-1}} \op{Hom}_A (\mathcal{P}^{\le n}_{A/k}, R)\\
    &\xrightarrow{\left(\text{The item 3.}\right)^{-1}} \op{Hom}_A (\mathcal{P}^{\le n}_{A/k}, A) \otimes R\\
    &\eqqcolon \op{Diff}^{\le n}_k(A)\otimes R.
    \end{align*}
    These differentials $\le n$ are compatible with each other and define the \emph{differential}
    \[
    d(R/A): \op{Diff}_k(R)\to \op{Diff}_k(A)\otimes R.
    \]
\end{defin}

\begin{prop}\label{Criterion: admitting a differential}\label{Fields Admit Differentials}
    Let $k\to A \to R$ be morphisms of rings. If $k\to A$ is a morphism of rings that admits coordinates, then $k\to A\to R$ admits the differential.
\end{prop}

\begin{proof}
    From $k\to A$ admitting finite coordinates, by Proposition \ref{Monomials in dx are a basis}, we get that $\mathcal{P}^{\le n}_{A/k}$ is a free $A$-module of finite rank, say $N$. Then, there is an isomorphism $\mathcal{P}^{\le n}_{A/k}\simeq \bigoplus^N A$. The choice of this isomorphism induces the following canonical isomorphisms
    \begin{align*}
     \op{Hom}_A (\mathcal{P}^{\le n}_{A/k}, A) \otimes R
        &\simeq \op{Hom}_A (\bigoplus^N A, A) \otimes R
        \simeq \prod^N \op{Hom}_A (A, A) \otimes R\\
        &\rightarrow
         \prod^N \op{Hom}_A (A,R) \simeq  \op{Hom}_A (\bigoplus^N A, R)
         \simeq \op{Hom}_A (\mathcal{P}^{\le n}_{A/k}, R).   
    \end{align*}
    The above arrow is an isomorphism, because, in each coordinate, it is explicitly given by the following formula:
    \[
    (1\mapsto a) \otimes r \mapsto (1 \mapsto a\cdot r)
    \]
    whose explicit inverse is
    \[
    (1 \mapsto r) \mapsto (1\mapsto 1) \otimes r. 
    \]
\end{proof}

\begin{lemma}[Differential is Restriction]\label{Differential is Restriction}
    Let $k\to A \to R$ be morphisms of rings that admit the differential. Let $D$ be a differential operator on $A$ relative to $k$, then $d(R/A)(D)=D_{|A}: A\to R$, i.e., the composition $A\to R \xrightarrow{D} R$.
\end{lemma}

\begin{proof}
    It follows from the explicit formulas for the natural maps from Definition \ref{Natural Maps: Differential} that appear in the definition of the differential \ref{Differential - definition}.
\end{proof}

\begin{example}\label{symbol.differential}
    Let $K$ be a field of positive characteristic $p>0$. Let $k=K^{p^\infty}$ be its perfection. We assume that $K/k$ is a finitely generated field extension.

     Let $x_1,\ldots x_n$ form a set of coordinates for $K/k$. (It exists, because it is any $p$-basis, Theorems \ref{Over Perfect is Separable} and \ref{p-basis is separable transcedence basis}.) Their $p^m$-th powers $x_1^{p^m},\ldots, x_n^{p^m}$ form a set of coordinates for $K^{p^m}/k$. Moreover, we have that
\[
    d(K/K^p)\left(\frac{1}{p^{m}!}\frac{\partial^{p^m}}{\partial x_i ^{p^m}}\right)= \frac{1}{p^{m-1}!}\frac{\partial^{p^{m-1}}}{\partial \left(x_i^p\right)^{p^{m-1}}}.
\]
Consequently,
\[
    d(K/K^{p^m})\left(\frac{1}{p^{m}!}\frac{\partial^{p^m}}{\partial x_i ^{p^m}}\right)= \frac{\partial}{\partial \left(x_i^{p^m}\right)}.
\]
We can conclude that all these differentials $d(K^{p^m}/K^{p^n})$, for $m\le n$, are surjective, which is a special case of my theorem \ref{SES for diff - fields - theorem}. Furthermore, this surjectivity may justify an abuse of notation to use symbols $\frac{\partial}{\partial \left(x_i^{p^m}\right)}$ and $\frac{1}{p^{m}!}\frac{\partial^{p^m}}{\partial x_i ^{p^m}}$ interchangeably.
\end{example}

\subsubsection{$p$-Filtration}

From Definition \ref{Differential Operators: Explicit Definition}, every algebra of differential operators admits a filtration by orders, i.e., $\bigcup_n \op{Diff}^{\le n}_k(R)=\op{Diff}_k(R)$. It is the ``standard filtration''. However, in positive characteristic, algebras of differential operators admit one more filtration that is called a \emph{$p$-filtration} in \cite[Section 1.2.]{Haastert:p-filtration}. This filtration is induced by Frobenius morphisms. Moreover, one can easily argue that it is, in some ways, a better filtration than the standard one, because it is a filtration by subalgebras, not just by subsets. Finally, in a positive characteristic, this filtration explains why the algebra of differential operators is usually not finitely generated. The $p$-filtration is not discussed in \cite{EGAIV}.

\begin{defin}[$p$-Filtration]\label{p-filtration: definition}
    Let $k\to R$ be an injection of rings such that there is a prime number $p$ such that $p=0$ in $k$.
    
    We define a \emph{$p$-filtration} of $\op{Diff}_k(R)$ to be the collection of subalgebras $\op{Diff}_{R^{p^n}\cdot k}(R)$ for $n \ge 0$.
\end{defin}

\begin{lemma}
    The $p$-filtration is a filtration of the algebra of differential operators.
\end{lemma}

\begin{proof}
   We have to show that $\op{Diff}_{R^{p^n}}(R)$ is a subalgebra of $\op{Diff}_k(R)$, and that $\bigcup \op{Diff}_{R^{p^n}}(R)=\op{Diff}_k(R)$. Moreover, without loss of generality, we assume that $k\subset R^{p^n}$ for every $n$, so we do not have to write down the smallest ring containing both $k$ and $R^{p^n}$, denoted by $R^{p^n}\cdot k$, in the below computation.

   We show that it is a subalgebra. By Definitions \ref{space of principal parts: definition} and \ref{Differential Operators: Universal Definition} we have $\mathcal{P}^{\le m}_{R/k}=R\otimes_k R / J^{m+1}_{R/k}$ and $\op{Diff}^{\le m}_{k}(R)=\op{Hom}_R(\mathcal{P}^{\le m}_{R/k}, R)$, $\mathcal{P}^{\le m}_{R/R^{p^n}}=R\otimes_{R^{p^n}} R / J^{m+1}_{R/R^{p^n}}$, and $\op{Diff}^{\le m}_{R^{p^n}}(R)=\op{Hom}_R(\mathcal{P}^{\le m}_{R/R^{p^n}}, R)$.
   Therefore, it is enough to compare the spaces $\mathcal{P}^{\le m}_{R/R^{p^n}}$ and $\mathcal{P}^{\le m}_{R/k}$ with each other. First, we observe that there is an ideal $I^{(p^n)}$ of $R\otimes_k R$ generated by the set $\{d(x^{p^n}); \ x \in R\}$ satisfying
   $
   R\otimes_k R/I^{(p^n)}= R\otimes_{R^{p^n}}R.
  $
   So, for every $m$, we have a surjection $\mathcal{P}^{\le m}_{R/k}\to \mathcal{P}^{\le m}_{R/R^{p^n}}$ inducing the injection $\op{Diff}_{R^{p^n}}(R)\to \op{Diff}_k(R)$.

   We show that the sum of these subalgebras is the whole algebra. Again, we notice that for on $R\otimes_k R$ we have an equality of ideals $J^{m+1}_{R/R^{p^n}}=J^{m+1}_{R/k}+I^{(p^n)}$. However, if $m+1=p^n$, then we have that $I^{(p^n)}\subset J^{m+1}_{R/k}$. So, then $J^{m+1}_{R/R^{p^n}}=J^{m+1}_{R/k}$, and therefore $\op{Diff}^{\le p^n-1}_{R^{p^n}}(R)= \op{Diff}^{\le p^n-1}_k(R)$. Consequently, we get the equality $\bigcup \op{Diff}_{R^{p^n}}(R)= \op{Diff}_k(R)$.
\end{proof}

\begin{example}
    Let $k$ be a field of characteristic $p>0$ and $R=k[x_1,\ldots,x_n]$, then the $p$-filtration satisfies:
    \begin{itemize}
        \item $\op{Diff}_{R^{p}}(R)$ is generated by $\frac{\partial}{\partial x_i}$,
        \item $\op{Diff}_{R^{p^2}}(R)$ is generated by $\frac{\partial}{\partial x_i}, \frac{1}{p!}\frac{\partial^{p}}{\partial x_i ^{p}}$,
        \item $\op{Diff}_{R^{p^3}}(R)$ is generated by $\frac{\partial}{\partial x_i}, \frac{1}{p!}\frac{\partial^{p}}{\partial x_i ^{p}},\frac{1}{p^2!}\frac{\partial^{p^2}}{\partial x_i ^{p^2}}$,
        \item In general, $\op{Diff}_{R^{p^n}}(R)$ is generated by $\frac{\partial}{\partial x_i}, \frac{1}{p!}\frac{\partial^{p}}{\partial x_i ^{p}},\ldots,\frac{1}{p^{n-1}!}\frac{\partial^{p^{n-1}}}{\partial x_i ^{p^{n-1}}}$.
    \end{itemize}
Furthermore, each of the subalgebras above has a basis consisting of monomials in generators enlisted next to it with powers of the generators being smaller than $p$, e.g., for $n=3$, $\op{Diff}_{R^{p^3}}(R)$ has a $R$-basis made of $\frac{\partial}{\partial x_i}^{a_1}\circ \frac{1}{p!}\frac{\partial^{p}}{\partial x_j ^{p}}^{a_2}\circ \frac{1}{p^2!}\frac{\partial^{p^2}}{\partial x_k ^{p^2}}^{a_3}$, where $0\le a_1,a_2,a_3<p$, and $i,j,k\in \{1,2,3\}$. 

We can observe that all the inclusions $\op{Diff}_{R^{p^n}}(R)\to \op{Diff}_{R^{p^{n+1}}}(R)$ are proper inclusions, thus $\op{Diff}_{k}(R)$ is not finitely generated algebra over $R$.
\end{example}

\subsubsection{Relative Differential Operators}
We define differential operators over $k$ between $R$-modules $D:M\to N$ for given $k\to R$. The previous definitions and work were about the case $M=N=R$. We will use this notion only once in Theorem \ref{stratification exists - thm}~(The stratification exists.), where we will have $M=R, N=R/q$, where $q$ is an ideal in $R$.

\begin{defin}[{\cite[Proposition (16.8.4)]{EGAIV}}]\label{relative differential operators - def}
    Let $k\to R$ be a morphism of rings. Let $M,N$ be $R$-modules, then
    \[
    \op{Diff}_{k}(R)^{\le n}(M,N)\coloneqq \op{Hom}_{R}(\mathcal{P}^{\le n}_{R/k}\otimes M, N).
    \]
\end{defin}

Now, we could repeat enlisting all extra structures: the inductive definition, the standard filtration, the graded algebra, the $p$-filtration, etc. However, I believe in the reader, so I skip it.

\subsection{Divided Power Rings}

A divided power ring intuitively is a ring $R$ with a structure consisting of countably many functions $f_n: R \to R$, which should be thought of as the operations $x\mapsto \frac{x^n}{n!}$. However, these operations cannot be canonically defined for every ring; therefore, if we want to work with them, we have to choose a set of such functions. For us, the need for introducing them comes from differential operators, because graded algebras of differential operators (under mild assumptions) admit a \emph{canonical} divided power structure (and even if not, then they always have it at least partially!). This is useful because it makes a lot of explicit computation involving leading forms of differential operators possible and relatively easy.

\begin{remark}
    Abstract divided powers were introduced to algebraic geometry to allow Taylor expansions in positive characteristic. Indeed, if we have a polynomial $f\in k[x]$, then
    \[
    f(x+h)=f(x) + f^{(1)}(x)h+\frac{f^{(2)}(x)h^2}{2!}+\frac{f^{(3)}(x)h^2}{3!}+\ldots+\frac{f^{(n)}(x)h^n}{n!}+\ldots,
    \]
    But $\frac{1}{n!}$ could make no sense, so we must do something about it. One can add $\frac{h^n}{n!}$ abstractly. This leads to crystalline cohomology. I want to thank Jan Stienstra for sharing this origin story with me. However, we will use divided powers differently. We will show that $\op{gr}\op{Diff}_k(K)$ admits a canonical such structure, Proposition \ref{divided power for gr diff - prop}, so it will be more about $[\frac{\partial^{n}}{n!\partial x^n}]$. 
\end{remark}

\subsubsection{Divided Powers}

\begin{defin}[{\cite[Definition 07GL]{stacks-project}}]\label{PD-ring -definition}
    A \emph{divided power ring} is a triple $(A,I,\gamma)$, where $A$ is a commutative ring, $I$ is an ideal of $A$, and $\gamma$ is a collection of functions $\gamma_i: I \to A$ for $i\ge 0$ such that
    \begin{enumerate}
        \item For all $x\in I$, we have $\gamma_0(x)=1$, $\gamma_1(x)=x$, and $\gamma_i(x)\in I$ for $i>1$.
        \item For all $x,y\in I$, we have $\gamma_k(x+y)=\sum_{i+j=k}\gamma_i(x)\gamma_j(y)$.
        \item For all $\lambda \in A$, $x\in I$, we have $\gamma_k(\lambda x)=\lambda^k \gamma_k(x)$.
        \item For all $x\in I$, we have $\gamma_i(x)\gamma_j(x)={{i+j}\choose {i}}\gamma_{i+j}(x)$.
        \item For all $x\in I$, we have $\gamma_i(\gamma_j(x))=\frac{(ij)!}{i!j!^i}\gamma_{ij}(x)$.
    \end{enumerate}
\end{defin}

\begin{defin}{\cite[Section 07H4]{stacks-project}}\label{divided power polynomial algebra - def}
Let $A$ be a ring. A \emph{divided power polynomial algebra} over $A$ in $t\ge 1$ variables is the following $A$-module:
\[
A\left<x_1,x_2,\ldots,x_t\right>\coloneqq \bigoplus_{n_1,n_2,\ldots,n_t\ge 0} A x_1^{[n_1]}\cdot x_2^{[n_2]}\cdot \ldots \cdot x_t^{[n_t]}
\]
with the (commutative) multiplication determined by
\[
x_i^{[n]} \cdot x_i^{[m]} \coloneqq {{n+m}\choose n} x_i^{[n+m]}.
\]

We also set $x_1=x_i^{[1]}$. Note that $1=x_1^{[0]}\cdot x_2^{[0]}\cdot \ldots \cdot x_t^{[0]}$. There is a canonical $A$-algebra map $A\left<x_1,x_2,\ldots,x_t\right>\to A$ sending $x_i^{[n]}$ to zero for $n>0$. The kernel of this map is denoted $A\left<x_1,x_2,\ldots,x_t\right>_+$. 
\end{defin}

\begin{lemma}[{\cite[Lemma 07H5]{stacks-project} for $I=0$}]\label{divided power polynomial is PD lemma}
    Let $A$ be a ring. Let $t\ge 1$. There exists a natural divided power structure $\delta$ such that
    $(A\left<x_1,x_2,\ldots,x_t\right>,A\left<x_1,x_2,\ldots,x_t\right>_+,\delta)$ is a divided power ring. This structure is determined by
    \[
    \delta_n(x_i)=x_i^{[n]},
    \]
    where $i=1,\ldots,t$ and $n\ge 0$.
\end{lemma}

Here is a valuable observation that directly follows from the definition.

\begin{lemma}\label{divided powers - decompostion of order p^n}
    Let $A$ be a ring. Let $R=A\left<x_1,x_2,\ldots,x_t\right>$ be a divided power polynomial ring. It is a graded ring, where the order is given by the sum $n_1+\ldots +n_t$, the numbers are from the ``exponents'' $[n_i]$, i.e., $R_d =\bigoplus_{n_1,n_2,\ldots,n_t\ge 0; \sum n_i= d}A x_1^{[n_1]}\cdot x_2^{[n_2]}\cdot \ldots \cdot x_t^{[n_t]}$.

    Moreover, if the ring $A$ is an integral domain of characteristic $p=0$, then $R$ admits a $p$-filtration $F_{p^{\bullet}}$ by ``partial divided power polynomial rings'' that are the subrings
    \[
    F_{p^{n}}(R)=\bigoplus_{p^n>n_1,n_2,\ldots,n_t\ge 0} A x_1^{[n_1]}\cdot x_2^{[n_2]}\cdot \ldots \cdot x_t^{[n_t]}.
    \]
    In particular, the order $p^n$ of the divided power polynomial ring naturally decomposes
    \[
    R_{p^n} \cong \left(F_{p^{n}}(R) \cap R_{p^n} \right) \oplus \left(\gamma_{p^n}\left(R \right)\otimes_{\gamma_{p^n}(A)}A\right).
    \]
\end{lemma}
\begin{proof}
    The only thing that must be checked is that $\left(x_i^{[{n}]}\right)^p=0$ for all $i=1,2,\ldots,t$, and $n\ge 0$, if $A$ is of characteristic $p>0$. This follows from the identity \cite[Lemma 07GM (1)]{stacks-project}
    \[
    p!\gamma_p(x)=x^p
    \]
    that holds for any divided power ring $(A,I,\gamma)$ if $x\in I$. The rest follows.
\end{proof}

\subsubsection{Graded Dual Algebras of Symmetric Algebras}
We recall what graded dual algebras are in the case of symmetric algebras. Furthermore, we prove that these subalgebras naturally are divided power rings. Later, we will prove that many graded algebras of differential operators are graded dual algebras of symmetric algebras, Proposition \ref{divided power structure for graded dual - prop}, so this will prove that each of these graded algebras of differential operators admits the canonical divided power structure.

\begin{defin}[{ \cite[page 587]{Bourbaki_algebra_1-3}}]\label{graded dual - def}
    Let $A$ be a ring. Let $M$ be a finite $A$-module. The \emph{graded dual of the symmetric algebra} $S(M)$ is
    \[
    S(M)^{*gr}\coloneqq \bigoplus_{n\ge 0} \op{Hom}_A(S^{n}(M), A),
    \]
    where $S^{n}(M)$ is the $m$-th symmetric power of the module $M$.
Moreover, we introduce the following notations:
    \begin{align*}
        \left(S(M)^{*gr}\right)^n\coloneqq& \op{Hom}_A(S^{n}(M), A),\\
        \left(S(M)^{*gr}\right)^{>0}\coloneqq& \bigoplus_{n\ge 1} \op{Hom}_A(S^{n}(M), A).
    \end{align*}
\end{defin}

\begin{prop}\label{graded dual is commutative - prop}
    Let $A$ be a ring. Let $M$ be a finite $A$-module. Then
    the symmetric algebra $S(M)$ admits a natural coalgebra structure given by
    \[
    c: S(M)\xrightarrow{S(\Delta)} S(M\times M) \xrightarrow{h} S(M)\otimes_A S(M),
    \]
    where $\Delta:M\to M\times M$ is given by $x\mapsto (x,x)$, and $h$ is the ring homomorphism determined by $M\times M\ni (x,y)\mapsto x\otimes 1 + 1\otimes y \in S(M)\otimes_A S(M)$.
    
    Consequently, the graded dual algebra $S(M)^{*gr}$ is naturally a commutative ring with respect to the coalgebra structure $c$. 
\end{prop}

\begin{proof}
    The natural coalgebra structure $c$ on $S(M)$ is defined in \cite[page 575, Example (6)]{Bourbaki_algebra_1-3}. It is shown that this structure $c$ is cocommutative in \cite[page 582, Example (5)]{Bourbaki_algebra_1-3}, hence the ring $S(M)^{*gr}$ is commutative. 
\end{proof}

The following result checks that the two rings defined in two contexts are isomorphic. I do not know how well known is this particular fact.

\begin{prop}\label{divided power structure for graded dual - prop}
    Let $A$ be a ring. Let $M$ be a finite $A$-module. Then the graded dual algebra $S(M)^{*gr}$ admits a natural extension to a divided power (commutative) ring
    $
    (S(M)^{*gr},\left(S(M)^{*gr}\right)^{>0}, \gamma)
    $.
    This structure is uniquely determined by the functions
    \[
    \gamma_n: \left(S(M)^{*gr}\right)^1 \to \left(S(M)^{*gr}\right)^{n}
    \]
    given by
    \[
    \gamma_n(\phi)(a_1 \otimes a_2 \ldots \otimes a_n)=\phi(a_1)\cdot \phi(a_2) \cdot \ldots \cdot \phi(a_i),
    \]
    where $\phi : M=S^1(M) \to A$, and $a_1,\ldots,a_n\in M$.

    Moreover, if $M$ is free, and $y_1,\ldots,y_t$ is a $A$-basis of $\left(S(M)^{*gr}\right)^1$, then the natural map of divided power rings
\[ 
(A\left<x_1,x_2,\ldots,x_t\right>,A\left<x_1,x_2,\ldots,x_t\right>_+,\delta) \to (S(M)^{*gr},\left(S(M)^{*gr}\right)^{>0}, \gamma)
\]
that is given by $x_i \mapsto y_i$ is an isomorphism.
\end{prop}

\begin{proof}

    We prove the theorem for the free modules first, and then we conclude the general case from the free one.

    Let $M$ be a finite free module. Let $e_1,\ldots, e_t$ be a basis of $M$. Then monomials $e^a\coloneqq e_1^{a_1}\cdot\ldots\cdot e_n^{a_t}$ are a basis of $S^m(M)$, where $a_i\ge0$ and $\sum_{i=1}^t a_i=m$. Let $u_a$ be the dual basis of $\left(S(M)^{*gr}\right)^m$ with respect to $e^a$. By the explicit calculation in \cite[page 593, (33)]{Bourbaki_algebra_1-3},
     we have:
    \[
    u_a \cdot u_b = \prod_{i=1}^t {{a_i +b_i}\choose {a_i}} u_{a+b},
    \]
    where the multiplication is the one of $S(M)^{*gr}$. This means that the $A$-homomorphism
    \[
    A\left<x_1,x_2,\ldots,x_t\right> \to S(M)^{*gr}
    \]
    determined by
    \[
    x_1^{[n_1]}\cdot x_2^{[n_2]}\cdot \ldots \cdot x_t^{[n_t]} \mapsto u_{(n_1,n_2,\ldots,n_t)}
    \]
    is a ring isomorphism. Also, the ideals $A\left<x_1,x_2,\ldots,x_t\right>_+$ and $\left(S(M)^{*gr}\right)^{>0}$ are mapped to each other by this isomorphism. 
    Finally, the functions 
    $\gamma_n: \left(S(M)^{*gr}\right)^1 \to \left(S(M)^{*gr}\right)^{n}$ coincide with restrictions of the functions $\delta_n$ (from Lemma \ref{divided power polynomial is PD lemma}) to $\left(S(M)^{*gr}\right)^1$. Consequently, $\delta$ extends the functions $\gamma_n$ to a divided power structure. 
    And, the ideal $\left(S(M)^{*gr}\right)^{>0}$ is generated by $\gamma_n(D)$ for $D\in \left(S(M)^{*gr}\right)^{1}$, it follows from applying axioms from Definition \ref{divided power polynomial algebra - def}. So, there exists at most one such extension of the functions $\gamma_n$ to a divided power structure on $(S(M)^{*gr},\left(S(M)^{*gr}\right)^{>0})$. Since we have already found one extension, i.e., $\delta$, this extension exists and it is unique. Also, the definition of the functions $\gamma_n$ is natural. Therefore, the structure $\delta$ is natural.

    Now, let $M$ be a finite $A$-module. Let $N$ be a finite free $A$-module with a surjection $f: N\to M$. Then we have a natural map
    \[
    S(f)^{*gr}: S(M)^{*gr}\to S(N)^{*gr}
    \]
    that is an injection of rings. We observe that the functions $\gamma_n$ for $S(M)^{*gr}$ are the restrictions of the functions $\gamma_n$ for $S(N)^{*gr}$. Let $\delta$ be the unique extension of $\gamma_n$ to a divided power structure on $S(N)^{*gr}$. 
    By using the axioms from Definition \ref{divided power polynomial algebra - def}, we see that for every $D\in \left(S(M)^{*gr}\right)^{>0})$ we have $\delta_n(D)\in \left(S(M)^{*gr}\right)^{>0})$. 
    Hence, the restrictions of the functions $\delta_n$ are a divided power structure on $S(M)^{*gr}$. Finally, this structure does not depend on $N$, because if we have two surjections $f_1: N_1\to M$ and $f_2: N_2\to M$, where $N_1, N_2$ are finite free modules, then there exists a finite free module $N_3$ with surjections $g_1: N_3\to N_1$ and $g_2: N_3\to N_1$ such that the unique divided power structures on $S(N_1)^{*gr}$ and $S(N_2)^{*gr}$ are the restrictions of the unique divided power structure on $S(N_3)^{*gr}$. Therefore, the divided power structure we obtained on $S(M)^{*gr}$ does not depend on $N$, so it is natural. This finishes the proof.
\end{proof}

\subsection{Tool: Diff=End Lemma}
This subsection proves the following key lemma that will allow us to specialize the Jacobson--Bourbaki correspondence to purely inseparable subfields/morphisms. It might be well known to some experts, but I could not find any good source.

\begin{prop}[Diff=End Lemma]\label{diff=end - new}
    Let $r$ be a natural number. Let 
    \[
    A\supset B \supset A^{p^r}
    \] be inclusions of rings of characteristic $p>0$. Let $A$ be a finite $B$-module. 
    Then the natural map
    \[
    \op{Diff}_B(A)\to \op{End}_B(A)
    \]
    is an isomorphism.
\end{prop}

\begin{proof}
    By Definition \ref{Differential Operators: Universal Definition}, we have that 
    \begin{center}
      $\op{Diff}_B(A)=\bigcup_{n\ge 1} \op{Diff}^{\le n}_B(A) $ and $\op{Diff}^{\le n}_B(A)=\op{Hom}_A (A\otimes_B A /J_{A/B}^{n+1}, A)$,  
    \end{center}
     where $A\otimes_B A$ is a $A$-module with respect to $p_1$. 

    First, we prove that the ideal $J_{A/B}^{n+1}$ is zero for $n$ great enough. We recall an elementary fact that the ideal $J_{A/B}$ is generated by elements $\tilde{d}(a)\coloneqq 1\otimes a - a \otimes 1$, where $a\in A$. Moreover, if $b\in B$, then $\tilde{d}(ba)=b\tilde{d}(a)$ by a simple calculation
    \[
    \tilde{d}(ba)=1\otimes ba - ba \otimes 1=b\otimes a -ba\otimes 1=(b\otimes 1)\tilde{d}(a)=b\tilde{d}(a).
    \]
    In particular, $\tilde{d}(b)=0$, because $\tilde{d}(1)=0$.
    
    Now, by $A/B$ being finite, there exists an integer $m$ and elements $x_1,\ldots, x_m \in A$ such that $A=\sum_{i=1}^m B\cdot x_i$. Therefore, for every $a\in A$, there exist $b_i\in B$ such that $a=\sum_{i=1}^m b_i x_i$. From this, we conclude that $J_{A/B}$ is generated by $\tilde{d}(x_1),\tilde{d}(x_2),\ldots,\tilde{d}(x_m)$. Consequently, the ideal $J_{A/B}^{n+1}$ is spanned by monomials $\tilde{d}(x_1)^{e_1}\tilde{d}(x_2)^{e_2}\ldots\tilde{d}(x_m)^{e_m}$, where $\sum_{i=1}^m e_i \ge n+1$. Finally, if $n+1> mp^{r}-m$, then for at least one $i$ we have $e_i\ge p^r$, say $i=1$, hence
    $$
    \tilde{d}(x_1)^{e_1}=\tilde{d}(x_1)^{p^r}\tilde{d}(x_1)^{e_1-p^r}=
    \tilde{d}(x_1^{p^r})\tilde{d}(x_1)^{e_1-p^r}=0,
    $$
    because $x_1^{p^r} \in A^{p^r}\subset B$. Therefore, all these monomials are zero, and hence the ideal $J_{A/B}^{n+1}$ is zero for $n+1> mp^{r}-m$.
    
    We already know that the ideal $J_{A/B}^{n+1}$ is zero for $n$ great enough, say $n> N$, where $N$ is an integer. From this, we get
    \begin{align*}
      \op{Diff}_B(A)=&\bigcup_{n\ge 1} \op{Diff}^{\le n}_B(A)=
    \bigcup_{n\ge N} \op{Diff}^{\le n}_B(A)=\\
    =&\bigcup_{n\ge N}\op{Hom}_A (A\otimes_B A , A)=\op{Hom}_A (A\otimes_B A, A).
    \end{align*}
    Now, the natural map $\op{Diff}_B(A)\to \op{End}_B(A)$ is obtained by mapping
    a function $A\otimes_B A \xrightarrow{f} A $ to the composition $A \xrightarrow{p_2}A\otimes_B A \xrightarrow{f} A$.
    This is the same map from the natural isomorphism $\op{Hom}_A (A\otimes_B A, A)\to \op{Hom}_B(A,A)=\op{End}_B(A)$. This proves the lemma.
\end{proof}

\subsection{Tool: Enveloping Algebras from Divided Powers}
Here we will prove an interesting observation. Let  $A/k$ be a ring extension that admits a separable transcendental basis. Then the graded algebra $\op{gr}\op{Diff}_{k}(A)$ naturally is isomorphic to a divided power polynomial ring. Moreover, if we have a subalgebra of differential operators $\D\subset \op{Diff}_{k}(A)$ of the form $\D=\op{Diff}_C(A)$, where $k\to C\to A$, then the inclusion of their graded algebras factorizes through a divided power polynomial ring that depends only on the first order of $\D$, i.e., on the $p$-Lie algebra $\F=\D\cap T_{A/k}=\op{Der}_C(A)$ and 
\[
\op{gr}\mathcal{D}\subset {S}(\F^*)^{*gr} \subset \op{gr}\op{Diff}_k(K).
\]
Finally, these results hold for modules of relative differential operators, i.e., if we have $k\subset C\subset A\to B$. These results hold for the modules $\op{Diff}_{k}(A, B)$ and $\op{Diff}_{C}(A, B)\subset \op{Diff}_{k}(A, B)$, though the modules are not algebras - their graded modules are algebras! Even commutative ones.

These properties are key ingredients in some technical parts of the next chapters. In particular, it is used to prove that the stratification induced by regular power towers is constructible. Some experts might know these facts, but they are not known to everyone.

The precise statements follow now.

\begin{lemma}\label{graded principal is symmetric algebra - lemma}
    Let $k\to A$ be a morphism of rings that admits coordinates.

    Then, the natural map 
    \[
    S(I_{A/k}/I_{A/k}^{2}) \to \bigoplus_{i\ge 0} I_{A/k}^i/I_{A/k}^{i+1}
    \]
    is an isomorphism of graded rings, where $S (V)$ is the symmetric algebra for a $K$-vector space $V$, and $I_{K/k}$ is the ideal of the diagonal for $A/k$.
\end{lemma}

\begin{proof}
    The natural map $S (I_{A/k}/I_{A/k}^{2}) \to \bigoplus_{i\ge 0} I_{A/k}^i/I_{A/k}^{i+1}$ 
    exists by the universal property of the symmetric algebra $S(V)$ for $V=I_{K/k}/I_{K/k}^{2}$ (i.e., $S(V)$ is a polynomial ring.) applied to the inclusion 
    \[
    I_{A/k}/I_{A/k}^{2}\to \bigoplus_{i\ge 0} I_{A/k}^i/I_{A/k}^{i+1}.
    \] 
    It is an isomorphism by dimension counting in each degree. More specifically, if $x_1,\ldots, x_n$ form coordinates for $A/k$, then $dx_1^{ a_1}\cdot dx_2^{ a_2}\cdot \ldots \cdot  dx_n^{ a_n}$, where $a_i\ge 0$ and $\sum a_i =m$, is a basis for $S^m (I_{K/k}/I_{K/k}^{2})$ These elements are mapped to $\tilde{d}x_1^{ a_1} \tilde{d}x_2^{a_2} \ldots  \tilde{d}x_n^{ a_n}$ mod $I_{A/k}^{m+1}$ that are a basis of $I_{K/k}^m/I_{K/k}^{m+1}$, see Proposition \ref{Monomials in dx are a basis}.. This finishes the proof.
\end{proof}

\begin{remark}
    I do not like my definition of \textbf{``admitting coordinates''}. I have my two main examples, and they satisfy that definition, good! 
    
    However, all that I want from the notion of ``admitting coordinates'' is precisely the content of the above lemma: the map $S(I_{A/k}/I_{A/k}^{2}) \to \bigoplus_{i\ge 0} I_{A/k}^i/I_{A/k}^{i+1}$ is an isomorphism of $A$-modules and $I_{A/k}/I_{A/k}^{2}$ is free of finite rank. It is the same remark as before, but slightly restated.

    Thus, my personal taste would be to redefine ``admitting coordinates'' to mean precisely that. I believe that this notion has an ``official'' name in the literature.
\end{remark}

\begin{prop}\label{divided power for gr diff - prop}
   Let $k\to A$ be a morphism of rings that admits coordinates. Let $p$ be a prime number. Let $p=0$ in $A$.

   Then, the graded algebra of differential operators 
   $\op{gr}\op{Diff}_k(A)$ extends naturally to a divided power ring
   $(\op{gr}\op{Diff}_k(A),[\E(A/k)], \gamma)$, where $\E(A/k)$ is the subset of all operators $D$ such that $D(1)=0$.
   
    Moreover, if $x_1,x_2,\ldots,x_n$ form coordinates for $A/k$, then the divided power structure $\gamma$ is determined by the values
    \begin{align*}
       \gamma_i\left(\left[\frac{\partial}{\partial x_j }\right]\right)=\left[\frac{1}{i!}\frac{\partial^i}{\partial x_j ^i}\right]
   \end{align*}
   for $j=1,2,\ldots,n$, and $i\ge 1$. In particular, we have an isomorphism of divided power rings
   \[
   A\left<y_1,y_2,\ldots, y_n\right> \to \op{gr}\op{Diff}_k(A)
   \]
   given by $y_i \mapsto \left[\frac{\partial}{\partial x_i }\right]$ for $i=1,2,\ldots, n$, where $A\left<y_1,y_2,\ldots, y_n\right>$ is a divided power polynomial algebra.
\end{prop}

\begin{proof}
    First, we prove that there is a natural isomorphism of rings preserving grading
    \[
    \op{gr}\op{Diff}_k(A) \simeq S(I_{A/k}/I_{A/k}^{2})^{*gr},
    \]
    where $I_{A/k}$ is the diagonal ideal for $K/k$, and $S(M)^{*gr}$ is a graded dual algebra of a symmetric algebra for an $A$-module $M$, Definition \ref{graded dual - def}.
    Let $d\ge 0$ be an integer. This isomorphism in degree $d$ is given by the composition of morphisms of $A$-modules:
    \begin{align*}
     \op{gr}^d\op{Diff}_k(A) =& \op{Diff}^{\le d}_k(A)/\op{Diff}^{\le d-1}_k(A)
    \\
    \simeq& \op{Hom}_A (I_{K/A}^d/I_{A/k}^{d+1},A)
    \\
    \simeq& \op{Hom}_A (S^d (I_{A/k}/I_{A/k}^{2}),A)\\
    =& \left(S(I_{A/k}/I_{A/k}^{2})^{*gr}\right)^d,
    \end{align*}
    where the first isomorphism is from Lemma \ref{graded d in terms of diagonal - lemma} and the second one from Lemma \ref{graded principal is symmetric algebra - lemma}.
    
    We show that the above isomorphisms for all $d\ge 0$ together define a ring map. We recall, Proposition \ref{graded dual is commutative - prop}, that the natural ring structure on any $S(M)^{*gr}$ is given by the map
    \[
    c: S(M)\to S(M)\otimes_A S(M)
    \]
    that is determined by $x \mapsto x\otimes 1 + 1\otimes x$ for $x\in M=S^1(M)$. And, though in this paper we recalled the definition of the composition for differential operators to be the composition as operators, the composition $\op{Diff}_k(A)$ may be abstractly defined by the following maps
    \[
    \delta: \mathcal{P}^{\le a+b}_{A/k} \to \mathcal{P}^{\le a}_{A/k} \otimes_K \mathcal{P}^{\le b}_{A/k},
    \]
    where $\mathcal{P}^{\le a}_{A/k}$ are defined in Definition \ref{space of principal parts: definition}.
    These maps $\delta$ are from \cite[Lemme 16.8.9.1]{EGAIV}, and, according to \cite[page 45]{EGAIV}, they are determined by 
    \[
    A\otimes_k A \to \mathcal{P}^{\le a}_{A/k} \otimes_A \mathcal{P}^{\le b}_{A/k}: \ x \otimes y \mapsto \left(x\otimes1 \text{ mod } I_{A/k}^{a+1}\right) \otimes  \left(1\otimes y \text{ mod } I_{A/k}^{b+1}\right).
    \]
    Finally, by dividing by right powers of the ideal $I_{K/k}$, we obtain that the ring structure on $\op{gr}\op{Diff}_k(K)$ is given by the following comultiplication.
    \[
    \bigoplus_{i\ge 0} I_{A/k}^i/I_{A/k}^{i+1} \to \bigoplus_{i\ge 0} I_{A/k}^i/I_{A/k}^{i+1} \otimes_A \bigoplus_{i\ge 0} I_{A/k}^i/I_{A/k}^{i+1}
    \]
    This map is determined by $x \mapsto x\otimes 1 + 1\otimes x$ for $x\in I_{A/k}/I_{A/k}^{2}$. Hence, this is the above map $c$ up to the isomorphism from Lemma \ref{graded principal is symmetric algebra - lemma}. Therefore, the natural isomorphism $\op{gr}\op{Diff}_k(A) \simeq S(I_{A/k}/I_{A/k}^{2})^{*gr}$ of $A$-modules is an isomorphism of $A$-algebras.

    Next, the isomorphism of graded rings $\op{gr}\op{Diff}_k(A) \simeq S(I_{A/k}/I_{A/k}^{2})^{*gr}$ allows to enrich the algebra $\op{gr}\op{Diff}_k(A)$ into a divided power ring. Indeed, by Proposition \ref{divided power structure for graded dual - prop}, the algebra $S(I_{A/k}/I_{A/k}^{2})^{*gr}$ is a divided power ring in a natural way. It is isomorphic to the divided power polynomial algebra $A\left<y_1,y_2,\ldots, y_n\right>$ with respect to the prescribed map from the theorem we are proving.

    We finish the proof by showing that the divided power structure $\gamma$ on $\op{gr}\op{Diff}_k(A)$ coming from $S(I_{A/k}/I_{A/k}^{2})^{*gr}$ is the one determined by the requirement that $\gamma_i\left(\left[\frac{\partial}{\partial x_j }\right]\right)=\left[\frac{1}{i!}\frac{\partial^i}{\partial x_j ^i}\right]$.
    
    We recall that the space for principal parts $\mathcal{P}^{\le m}_{A/k}$ for $A/k$ has a basis 
    $\prod_{k=1}^n \tilde{d}x_k^{a_k}$ 
    for $0\le a_k\le m$ and $\sum a_k\le m$. Thus, the classes of
    $\prod_{k=1}^n \tilde{d}x_k^{a_k}$ modulo $\mathcal{P}^{\le m-1}_{A/k}$
    for $0\le a_k\le m$ and $\sum a_k=m$ form a basis of $S^{m}(I_{A/k}/I_{A/k}^{2})$, we denote them by $\left[\prod_{k=1}^n \tilde{d}x_k^{a_k}\right]$.

    Now, we prove the operators $\gamma_i\left(\left[\frac{\partial}{\partial x_j }\right]\right)$ and $\left[\frac{1}{i!}\frac{\partial^i}{\partial x_j ^i}\right]$ have the same values on $S^{i}(I_{A/k}/I_{A/k}^{2})$ as linear operators.
    In our computation, we will freely refer to Definition \ref{symbol.differential} of the operators $\frac{1}{i!}\frac{\partial^i}{\partial x_j ^i}$, and to the definition of $\gamma_i$ from Proposition \ref{divided power structure for graded dual - prop}.
    
    Let $j\in\{1,2,\ldots,n\}$, $i\ge 1$, and $m=i$. Then, we compute
    \begin{align*}
        \gamma_i\left(\left[\frac{\partial}{\partial x_j }\right]\right)\left(\left[\prod_{k=1}^n \tilde{d}x_k^{a_k}\right]\right)
        &=
        \gamma_i\left(\left[\frac{\partial}{\partial x_j }\right]\right)\left(\bigotimes_{k=1}^n \left[\tilde{d}x_k\right]^{\otimes a_k}\right)\\
        &\coloneqq\prod_{k=1}^n \left[\frac{\partial}{\partial x_j }\right] \left( \left[\tilde{d}x_k\right]\right)^{a_k}\\
        &=\prod_{k=1}^n \frac{\partial}{\partial x_j } \left( dx_k\right)^{a_k}=
    \begin{cases}
1, \text{ if $a_j=i$,}\\
0, \text{ if $a_j\ne i$.}
\end{cases}
    \end{align*}
    And, we have
    \begin{align*}
        \left[\frac{1}{i!}\frac{\partial^i}{\partial x_j ^i}\right]\left(\left[\prod_{k=1}^n \tilde{d}x_k^{a_k}\right] \right)=
        \frac{1}{i!}\frac{\partial^i}{\partial x_j ^i}\left(\prod_{k=1}^n \tilde{d}x_k^{a_k} \right)
        =
        \begin{cases}
1, \text{ if $a_j=i$,}\\
0, \text{ if $a_j\ne i$.}
\end{cases}
    \end{align*}
    These operators have equal values on the monomials from the basis, so they are equal to each other by Corollary \ref{Determination1}. This finishes the proof.
\end{proof}

\begin{remark}
    I do not know any direct reference for the fact that the graded algebra of differential operators for $k\to A$ admits a canonical divided power structure. The only paper I know that explicitly mentions this fact is \cite{EkedahlFoliation1987}. In that paper, Ekedahl claims this is true, but he does not prove it or provide any reference.
\end{remark}

\begin{prop}\label{pd envelope - prop}
     Let $k\to A$ be a morphism of rings that admits coordinates. Let $k\to C\to A$ be ring maps.
    
    Let $\mathcal{D}=\op{Diff}_C(A)\subset  \op{Diff}_k(A)$ be a subalgebra.
    Let $\F=\mathcal{D}\cap T_{A/k}=\op{Der}_C(A)$. 

    Then, the natural inclusion
$\op{gr}\mathcal{D}\subset  \op{gr}\op{Diff}_k(A)$ of graded algebras
factorizes
\[
\op{gr}\mathcal{D}\subset {S}(\F^*)^{*gr} \subset \op{gr}\op{Diff}_k(A)
\]
through the graded dual of the symmetric algebra of $S(\F^*)$. The inclusions are ring morphisms.

Moreover, the ring ${S}(\F^*)^{*gr}$ is a divided power subring of $\op{gr}\op{Diff}_k(A)$. It is isomorphic to a divided power polynomial algebra if $\F$ is a free $A$-module. Consequently, it is generated as a subring of $\op{gr}\op{Diff}_k(A)$ by all elements $\gamma_i(D)$, where $D\in\F, i\ge 0$.
\end{prop}

\begin{proof}
    The inclusion $\op{gr}\mathcal{D}\subset  \op{gr}\op{Diff}_k(K)$ is induced by the surjection of rings
    \[
    \bigoplus_{i\ge 0} I_{A/k}^i/I_{A/k}^{i+1} \to \bigoplus_{i\ge 0} I_{A/C}^i/I_{A/C}^{i+1}
    \]
    via taking graded duals.

    By Lemma \ref{graded principal is symmetric algebra - lemma}, we have $S(I_{A/k}/I_{A/k}^{2}) \simeq \bigoplus_{i\ge 0} I_{A/k}^i/I_{A/k}^{i+1}$. Therefore, the above map is determined by the first degree $I_{A/k}/I_{A/k}^{2} \to I_{A/C}/I_{A/C}^{2}$. Consequently, it factorizes through $S(I_{A/C}/I_{A/C}^{2})$:
    \[
    S(I_{A/k}/I_{A/k}^{2})\simeq\bigoplus_{i\ge 0} I_{A/k}^i/I_{A/k}^{i+1} \to S(I_{A/C}/I_{A/C}^{2}) \to \bigoplus_{i\ge 0} I_{A/C}^i/I_{A/C}^{i+1},
    \]
    and both maps are surjective here.
    Now, by taking graded duals of the above factorization, we obtain the inclusions
    \[
\op{gr}\mathcal{D}\subset {S}(\F^*)^{*gr} \subset \op{gr}\op{Diff}_k(K).
    \]
    
    Finally, the algebra ${S}(\F^*)^{*gr}$ is a divided power polynomial ring, and it is generated by all $\gamma_i(D)$ for $D\in \F$ by Proposition \ref{divided power structure for graded dual - prop}.
\end{proof}

\begin{remark}
    The role of the algebra ${S}(\F^*)^{*gr}$ for a $p$-Lie algebra $\F$ is analogous to the universal property of the universal enveloping algebra of a Lie algebra. So, one could call it a \emph{universal divided power enveloping algebra} of the $p$-Lie algebra $\F$. 
    Actually, it is possible that there is a paper calling it precisely that way, but I do not have such a reference. 

    Moreover, we can observe that we did not use the characteristic of $A$ at all. So, accidentally, we proved the existence of ``universal enveloping algebras of Lie algebras'' of some sort.
\end{remark}

\begin{cor}\label{pd envelope - prop}
     Let $k\to A$ be a morphism of rings that admits coordinates. Let $k\to C\to A\to B$ be ring maps. We assume that $A\to B$ is surjective.
    
    Let $\mathcal{D}=\op{Diff}_C(A,B)\subset  \op{Diff}_k(A,B)$ be a submodule.
    Let $\F=\mathcal{D}\cap \op{Der}_k(A,B)=\op{Der}_C(A,B)$. We assume $\F$ is a free $A$-module. 

    Then, the natural inclusion
$\op{gr}\mathcal{D}\subset  \op{gr}\op{Diff}_k(A)$ of graded modules
factorizes
\[
\op{gr}\mathcal{D}\subset {S}(\F^*)^{*gr} \subset \op{gr}\op{Diff}_k(A,B)
\]
through the graded dual of the symmetric algebra of $S(\F^*)$.

Moreover, the graded $B$-modules $\op{gr}\mathcal{D}$, ${S}(\F^*)^{*gr}, \op{gr}\op{Diff}_k(A,B)$ are $B$-algebras.
\end{cor}

\begin{proof}
    It is analogous to the previous theorem.
    
    We recall Definition \ref{relative differential operators - def}. It follows that
    \[
    \op{Diff}_k(A,B)\coloneqq \op{Diff}_{k}(A)^{\le n}(A,B)\coloneqq \op{Hom}_{A}(\mathcal{P}^{\le n}_{A/k}\otimes A, B)\simeq \op{Hom}_{A}(\mathcal{P}^{\le n}_{A/k}\otimes A,A)\otimes_A B=\op{Diff}_k(A)\otimes B,
    \]
    where we have the isomorphism, because $\mathcal{P}^{\le n}_{A/k}$ is a free $A$-module. 

    Consequently, the graded module of $\op{Diff}_k(A)\otimes B$ is the graded dual of the ring 
    \[
    \bigoplus_{i\ge 0} \left(I_{A/k}^i/I_{A/k}^{i+1}\otimes_A B\right)
    \]
    as a $B$-module.

    Similarly, the graded module of $\op{Diff}_C(A,B)=\op{Diff}_C(A)\otimes B$ is the graded dual of the ring
    \[
    \bigoplus_{i\ge 0} \left(I_{A/C}^i/I_{A/C}^{i+1}\otimes_A B\right)\]
    as a $B$-module.

    Finally, we have a factorization of $B$-algebras:
    \[
    \bigoplus_{i\ge 0} \left(I_{A/C}^i/I_{A/C}^{i+1}\otimes_A B\right)\simeq S(I_{A/k}^1/I_{A/k}^{2})\otimes_A B
    \to S(I_{A/C}^1/I_{A/C}^{2}\otimes_A B)\to\bigoplus_{i\ge 0} \left(I_{A/C}^i/I_{A/C}^{i+1}\otimes_A B\right)
    \]
    Each map is a surjection; thus, the maps between their graded duals, as $B$-algebras, are injections. Moreover, the graded modules are graded dual algebras; thus, they are algebras.
\end{proof}

\newpage
\section{Review of Previous Galois Theories}

This section quickly reviews the Galois-type theories that we generalize in this paper. A more detailed presentation is available in my thesis \cite{Grabowski_PhD_thesis}, where I also discuss Sweedler's theory following Wechter.

\subsection{Jacobson--Bourbaki Correspondence}

In short, this correspondence says that a subfield $K\subset L$ and an algebra $\op{End}_K(L)$ of linear endomorphisms of $L$ determine each other naturally. These theorems are taken from \cite[Section 8.2]{Jacobson-BasicAlgebraII}. I added the extra adjectives ``topological'' and ``discrete''.

\begin{defin}[Finite Topology/the Krull Topology]\label{Finite Topology- Definition}
Let $X,Y$ be sets. The \emph{finite topology} on functions from $X$ to $Y$, this space is denoted by $Y^X$, is the topology with a base given by subsets of the form
$$U(f,S)\coloneqq \{ g\in Y^X: \forall_{x\in S} g(x)=f(x)  \}_{f,S},$$ where $f$ is a function from $X$ to $Y$, and $S$ is a finite subset of $X$.
\end{defin}

\begin{thm}[Topological Jacobson--Bourbaki Correspondence]\label{Topological Jacobson--Bourbaki Correspondence}
Let $L$ be a field. There is a natural order reversing correspondence between subfields $K$ of $L$, and closed in the finite topology $L$-subalgebras of $\op{End}_+(L)$. 

Explicitly,
this correspondence is given by the following formulas.

From subfields to closed subalgebras:
\[
L\supset K\mapsto \mathcal{JB}(K)\coloneqq \op{End}_K(L)\subset \op{End}_+(L).
\]

From closed subalgebras to subfields:
\[
\mathcal{JB}=\overline{\mathcal{JB}}\subset \op{End}_+(L) \mapsto \op{const}(\mathcal{JB})\coloneqq \{x\in L; \forall_{D\in \mathcal{JB}} \ [x,D]=0\}\subset L.
\]
\end{thm}

Finite subalgebras are closed in the finite topology. This gives the following result.

\begin{cor}[Discrete Jacobson--Bourbaki Correspondence]\label{Discrete Jacobson--Bourbaki Correspondence}
    Let $L$ be a field. There is a natural order reversing correspondence between subfields $K$ of $L$ with $[L:K]<\infty$, and finite dimensional $L$-subalgebras of $\op{End}_+(L)$. 

Explicitly,
this correspondence is given by the following formulas.

From subfields of finite coorder to finite-dimensional subalgebras:
\[
L\supset K\mapsto \mathcal{JB}(K)\coloneqq \op{End}_K(L)\subset \op{End}_+(L).
\]

From finite-dimensional subalgebras to subfields of finite coorder:
\[
\mathcal{JB}\subset \op{End}_+(L) \mapsto \op{const}(\mathcal{JB})\coloneqq \{x\in L; \forall_{D\in \mathcal{JB}} \ [x,D]=0\}.
\]
Moreover, we have that $\op{dim}_K(L)=\op{dim}_L(\mathcal{JB}(K))$.
\end{cor}

\begin{remark}
    For a Galois extension $L/K$ with a Galois group $G$. The subalgebra $\mathcal{JB}=\op{End}_K(L)$ is the group ring $L[G]$, because it is inside and the dimensions agree. So, $G$ is a nice choice of generators for $\mathcal{JB}$.
\end{remark}

\subsection{Jacobson Correspondence} In short, this correspondence says that a subfield of exponent one $K^p\subset W \subset K$ and a module of relative derivations $T_{K/W}\coloneqq \op{Der}_{K^p}(K, W)$ determine each other, if we have $[K: K^p]<\infty$. Originally, this theorem was proved in \cite{Jacobson1944}, but it can be obtained as a corollary of the Jacobson--Bourbaki correspondence, as it is done in the book \cite{Jacobson-BasicAlgebraII}. A part of deducing it in that way includes a lemma that gives an equivalent definition of a $p$-Lie algebra. We assume $\op{char}(K)=p>0$.

\begin{remark}
    Derivations kill $p$-powers, thus if we have $L\subset K^p\subset W$, then $\op{Der}_L(K,W)=\op{Der}_{K^p}(K,W)$, i.e., $K^p$ is a natural base to study derivations on $K$ modulo $p$.
\end{remark}

\begin{defin}
    Let $T_{K/K^p}$ be the tangent space of $K$. We define two operations:
    \begin{itemize}
        \item Lie bracket: $T_{K/K^p}\oplus T_{K/K^p}\to T_{K/K^p}: (D,E)\mapsto [D,E]\coloneqq D\circ E-E\circ D$,
        \item $p$-power: $T_{K/K^p}\to T_{K/K^p}: D \mapsto D^{\circ p}$.
    \end{itemize}
\end{defin}

Since $p=0$, we see that $v^p(fg)=\sum_{i=0}^p {{p}\choose {i}} v^{i}(f)v^{p-i}(g)=fv^p(g)+v^p(f)g$, so the above operations are well defined.

\begin{defin}\label{p-lie algebra: definition}
    A $p$-Lie algebra on a field $K$ is a $K$-vector subspace of $T_{K/K^p}$ that is closed under Lie bracket and $p$-power operations.
\end{defin}

We can provide an equivalent definition that does not use any fancy operations.

\begin{lemma}\label{p-lie algebra: definition 2}
    We assume $[K:K^p]<\infty$.
    Let $\F\subset T_{K/K^p}$ be a $K$-vector subspace. Let $\left<\F\right>$ be the smallest subalgebra of $\op{End}_k(K)$ containing $\F$, where $k=\bigcap_{n\ge 0} K^{p^n}$. Then $\F$ is a $p$-Lie algebra on $K$ if and only if the equality $\left<\F\right>\cap T_{K/K^p}=\F$ is satisfied.
\end{lemma}

\begin{thm}[Jacobson Correspondence aka Jacobson's Purely Inseparable Galois Theory for Exponent $1$]\label{Jacobson Purely Inseparable Galois Theory for Exponent $1$}
   Let $K$ be a field of positive characteristic $p>0$. Let $k$ be the perfection of $K$, i.e., $k=\bigcap_{n\ge 0} K^{p^n}$. Assume that $K/k$ is a finitely generated field extension.

   Then, there is a natural inclusion reversing correspondence between $p$-Lie algebras on $K$ and subfields $K\supset W\supset K^p$. 

Explicitly, it is given by the following formulas.

From subfields of exponent $\le 1$ to $p$-Lie subalgebras:
\[
K \supset W \supset K^{p} \mapsto T_{K/W} \subset T_{K/k}.
\]

From $p$-Lie subalgebras to subfields of exponent $\le 1$:
\[
\F \subset T_{K/k} \mapsto K \supset \op{Ann}(\F) \supset K^p,
\]
where $\op{Ann}(\F)\coloneqq\{x\in K| \forall_{D\in \F} D(x)=0\}$.
Moreover, $p^{\op{dim}_K(\F)}= \op{dim}_W(K)$.
\end{thm}

\begin{remark}
    The theorem still holds if we assume only that $[K: K^p]<\infty$. This is implied by our condition that $K\supset k$ is finitely generated as a field extension. In other words, our statement is for generic points of varieties, because we only work in this case.
\end{remark}

\subsection{Ekedahl Correspondence} Ekedahl's correspondence is the Jacobson correspondence for normal varieties. Originally, it was proved by Ekedahl in \cite[Proposition 2.2.]{EkedahlFoliation1987}. A more recent treatment is given in \cite{ZsoltJoe-1-foliations}.

\begin{defin}
Let $X$ be a normal variety over a perfect field $k$ of characteristic $p>0$.
A \emph{$p$-Lie algebra}, or a \emph{$1$-foliation}, on $X$ is a quasicoherent subsheaf $\F$ of the tangent sheaf $T_{X/k} \coloneqq \Omega_{X/k}^* = \op{Hom}_{\mathcal{O}_X}(\Omega_{X/k},\mathcal{O}_X)$ such that
    \begin{itemize}
        \item $\F$ is saturated, i.e., the quotient $T_{X/k}/\F$ is torsion-free,
        \item $\F$ is closed under Lie bracket, i.e., for every open subset $U\subset X$ and sections  $v,w\in \F(U)$ we have $[v,w]\in \F(U)$,
        \item $\F$ is closed under $p$-power, i.e., for every open subset $U\subset X$ and a section $v\in \F(U)$ we have $v^{\circ p}\in \F(U)$.
    \end{itemize}
\end{defin}

\begin{defin}
    Let $X$ be a normal variety over a perfect field $k$ of characteristic $p>0$.
    Let $\F$ be a $p$-Lie algebra on $X$.

    Then, the sheaf of annihilated elements by $\F$ is defined by
    \[
    \op{Ann}(\F)(U)=\{x\in \cO_{X/k}(U): \ \forall_{D\in \F(U)} D(x)=0\},
    \]
    where $U\subset X$ is an open subset.
\end{defin}

\begin{remark}
    Let $X$ be a normal variety over a perfect field $k$ of characteristic $p>0$.
    Let $\F$ be a $p$-Lie algebra on $X$. Then, $\op{Ann}(\F)$ is a quasicoherent sheaf of algebras on $X^{(1)}$, because vector fields kill $p$-powers.
\end{remark}

\begin{thm}[Ekedahl Correspondence]\label{Jacobson--Ekedahl Correspondence - proposition}
    Let $X$ be a normal variety over a perfect field $k$ of characteristic $p>0$. Then there is a bijection between purely inseparable morphisms $X\to Y\to X^{(1)}$ with $Y$ a normal variety over $k$ and $1$-foliations on $X$.

    Explicitly, it is given by the following formulas.

    From purely inseparable morphisms of exponent $\le 1$ to $1$-foliations:
    \[
    X\to Y\to X^{(1)} \mapsto T_{X/Y}\subset T_{X/k}.
    \]

    From $1$-foliations to purely inseparable morphisms of exponent $\le 1$:
    \[
    \F \subset T_{X/k} \mapsto X\to \op{Spec}_{X^{(1)}}(\op{Ann}(\F))\to X^{(1)}.
    \]   
\end{thm}

\newpage
\section{Power Towers on Fields}

We develop a theory of power towers on fields, primarily for fields that are generic points of varieties. We define power towers. We show their basic properties. We prove the Galois-type theorem for power towers, which includes that power towers are in an explicit bijection with subalgebras of differential operators. We show how these subalgebras can be visualized. There is a way to think about them generally that leads to the right ideas for proving things about them. We prove interactions between these algebras using differentials between them - we obtain a short exact sequence. We use that development to figure out exactly which subfields are power towers. Finally, we study a specific class of power towers: $n$-foliations, where we give examples of power towers that do not correspond to any subfields.

\subsection{Basic Theory}\label{section one - fields}

We repeat Definition \ref{Power Tower - Definition - intro} of a power tower from the introduction and introduce a compact notation.

\begin{defin}[Power Tower]\label{Power Tower - Definition}
Let $K$ be a field of characteristic $p>0$.
A \textbf{\emph{power tower}} on $K$ is a sequence of subfields $W_n\subset K$ for $n=0,1,2,\ldots$ such that $W_j = W_i \cdot K^{p^j}$ for $j\le i$. 

We denote the power tower $(W_0,W_1,W_2,\ldots)$ on $K$ by $W_\bullet$.
\end{defin}

\begin{obs}\label{observation W_n is exponent <=n}
    If $W_\bullet$ is a power tower on a field $K$, then $K\supset W_n\supset W_n\cdot K^{p^n}\supset K^{p^n}$, i.e., the subfield $W_n$ is of exponent $\le n$.
\end{obs}

It is convenient to visualize a power tower in the form of a diagram:

\begin{center}
    \begin{tikzcd}
    K=W_0 & W_1\ar[l] & K^p \ar[l]\\
    K \ar[u,equal]& W_2\ar[l]\ar[u] & K^{p^2}\ar[l]\ar[u] \\
    \ldots \ar[u,equal]& \ldots\ar[l]\ar[u] & \ldots\ar[l]\ar[u] \\
    K \ar[u,equal]& W_m\ar[l]\ar[u] & K^{p^m}\ar[l]\ar[u] \\
    \ldots \ar[u,equal]& \ldots\ar[l]\ar[u] & \ldots\ar[l]\ar[u] .
    \end{tikzcd}
\end{center}

We have already noticed two simple examples of power towers.

\begin{example}\label{Constant Power Tower}
A constant tower $W_n\coloneqq K$ is a power tower on $K$. 
\end{example}

\begin{example}\label{p-power power tower}
A tower of $p$-powers of $K$, $W_n\coloneqq K^{p^n}$, is a power tower on $K$.
\end{example}

We can strengthen Observation \ref{observation W_n is exponent <=n}.

\begin{lemma}\label{Subsequent Extensions are of exponent 1 - lemma}
Let $K$ be a field of characteristic $p>0$.
Let $W_\bullet$ be a power tower on $K$, then each extension $W_i/W_{i+1}$ is of exponent at most $1$, i.e., $W_{i+1} \supset W_i^p$, where $i\ge 0$. 
\end{lemma}

\begin{proof}
We have to prove that $W_{i+1} \supset W_i^p$. This follows from the following calculation
\[
W_i^p={\left(W_{i+1}\cdot K^{p^i}\right)}^{p}=W_{i+1}^p\cdot K^{p^{i+1}}
\subset W_{i+1} \cdot K^{p^{i+1}}=W_{i+1},
\]
where, in the second equality, we used that taking $p$-powers commutes with taking composites.
\end{proof}

More examples come from power towers that are eventually constant.

\begin{defin}\label{Finite Length - Definition}
Let $K$ be a field of characteristic $p>0$.
Let  $W_\bullet$ be a power tower on $K$. It is of \emph{length at most $n$} if $W_N=W_n$ for $N\ge n$. It is of length $n$, if the integer $n$ is minimal such. 

Furthermore, we say that a power tower is of \emph{finite length} if there is an integer $n$ such that the tower is of length at most $n$.
\end{defin}

\begin{lemma}\label{structure of inclusions - power tower - lemma}
    Let $K$ be a field of characteristic $p>0$. 
    Let  $W_\bullet$ be a power tower on $K$. 
    
    Then, for $i\ge 0$, all the inclusions $W_i\subset W_{i+1}$ are proper inclusions, or the power tower $W_\bullet$ is eventually constant.

    If the power tower $W_\bullet$ is eventually constant of length $n$, then the inclusions $K=W_0\supset W_1\supset W_2 \supset \ldots \supset W_n$ are proper, and all the inclusions $W_n\supset W_{n+1}\supset W_{n+2}\supset \ldots$ are equalities.
\end{lemma}

\begin{proof}
    Let $0\le i$ be such that $W_i=W_{i+1}$. We claim that this implies that for $i< k$ we have $W_k=W_i$, i.e., the power tower is of length $\le i$.
    Indeed, we can prove it by induction. 
    
    It is true for $k=i+1$, because $W_i=W_{i+1}$. Now, we prove that if $W_k=W_{k+1}$ for $k\ge i$, then $W_{k+2}=W_{k+1}=W_{k}$. Indeed, we have $W_{k+1}=W_k\supset K^{p^k}$. Thus, by Lemma \ref{Subsequent Extensions are of exponent 1 - lemma}, we have $W_{k+2}\supset W_{k+1}^p\supset W_k^p\supset \left(K^{p^k}\right)^p=K^{p^{k+1}}$. Consequently, $W_{k+1}=W_{k+2}\cdot K^{p^{k+1}}=W_{k+2}$. This proves that for $i< k$ we have $W_k=W_i$. 
    
    Consequently, if $n\ge 0$ is the minimal such that $W_n=W_{n+1}$, then $W_\bullet$ is of length $n$, because if it was of length $<n$, then $n$ would not be minimal.
\end{proof}

\begin{lemma}[Purely Inseparable Subfields of Finite Exponents are Power Towers of Finite Length]\label{Finite Length = finite exponent - lemma}
Let $K$ be a field of characteristic $p>0$.
Let $n$ be a positive integer. Then, there is an explicit bijection between power towers on $K$ of length $n$, and subfields of $K$ of exponent $n$. 

Explicitly, this correspondence is given by the operation $W_\bullet\mapsto W_n$.  
\end{lemma}

\begin{proof}
    We prove that the map $W_\bullet\mapsto W_n$ is injective by proving that any power tower $W_\bullet$ of length $n$ is determined by $W_n$, which is of exponent $n$. Indeed, if $i<n$, then $W_i=W_n\cdot K^{p^i}$ by Definition \ref{Power Tower - Definition}, and if $i>n$, then $W_i=W_n$ by Definition \ref{Finite Length - Definition}. And, the subfield $W_n$ is of exponent $n$, because $W_n\supset K^{p^n}$ by Observation \ref{observation W_n is exponent <=n}, and, by Lemma \ref{structure of inclusions - power tower - lemma}, we have $W_{n-1}\ne W_n$ , so $W_{n}\not \supset K^{p^{n-1}}$, i.e., it is not of exponent $\le n-1$. 

    The map is surjective, because the map $W\mapsto W_\bullet\coloneqq (W_i\coloneqq W\cdot K^{p^i})_{i\ge 0}$ from subfields of exponent $n$ is its left inverse.
    
    This proves that the map $W_\bullet\mapsto W_n$ is a bijection.
\end{proof}

By Lemma \ref{structure of inclusions - power tower - lemma}, the inclusions of subfields in a power tower are proper till they are equalities. We can compute the dimensions between these subfields - under mild assumptions on $K$.

\begin{prop}[Nonincreasing Degrees]\label{Nonincreasing Degrees Power Tower - Proposition}
Let $K$ be a field of positive characteristic $p>0$. Let $\op{dim}_{K^p}(K)$ be finite.
Let $W_\bullet$ be a power tower on $K$. 

Then, the power tower $W_\bullet$ satisfies the following sequence of inequalities:
\[\op{dim}_{W_1}(K)\ge \op{dim}_{W_2}(W_1)\ge
\ldots\ge\op{dim}_{W_n}(W_{n-1})
\ge \ldots \ge 0.
\]
\end{prop}

\begin{proof}
First, we note that all these dimensions are finite, because all dimensions $\op{dim}_{K^{p^n}}(K)=\prod_{i=0}^{n-1} \op{dim}_{K^{p^{i+1}}}(K^{i})=\op{dim}_{K^p}(K)^n$ are finite. Therefore, the dimension $\op{dim}_{W_n}(W_{n-1})$ is finite since we have $K\supset W_{n-1}\supset W_n\supset K^{p^n}$.

Now, let $n\ge 0$. We are going to show that $\op{dim}_{W_{n+1}}(W_{n})\ge \op{dim}_{W_{n+2}}(W_{n+1})$ by writing some generators of these extensions explicitly. 

Let $\{x_1,\ldots,x_r\}\subset K^{p^{n+1}}$ be a minimal subset such that $W_{n+1}=W_{n+2}(x_1,\ldots,x_r)$. We have that $x_i^p\in W_{n+2}$ for $i=1,\ldots,r$, because, by Lemma \ref{Subsequent Extensions are of exponent 1 - lemma}, the extension $W_{n+1}/W_{n+2}$ is of exponent at most $1$. Therefore, this set is a $p$-basis for $W_{n+1}/W_{n+2}$, Definition \ref{p-basis definition}.
Now, we claim that 
\[
\op{dim}_{W_{n+2}(x_i)}(W_{n+2}(x_i^{1/p}))=\op{dim}_{W_{n+2}}(W_{n+1}).
\]
Indeed, if not, then there is a nontrivial $(W_{n+2}(x_i))$-linear relation $F$ between monomials $\prod_{i=1}^r {\left(x_i^{1/p}\right)}^{a_i}$, where $0\le a_i \le p-1$, but then $F^p$ is a $W_{n+2}$-linear relation for monomials $\prod_{i=1}^r {\left(x_i\right)}^{a_i}$, because $W_{n+2}(x_i)^p\subset W_{n+2}$. This means that the set $\{x_i\}$ is not a $p$-basis. This is a contradiction. Thus, the equality of dimensions holds.

Next, we have a series of extensions 
\[
W_{n+2}\subset W_{n+1}=W_{n+2}(x_1,\ldots,x_r) \subset W_{n+1}(x_1^{1/p},\ldots,x_r^{1/p}) \subset W_{n}
\]
that gives us 
\begin{align*}
   \op{dim}_{W_{n+1}}(W_n)&=\op{dim}_{W_{n+1}(x_i^{1/p})}(W_n)\cdot \op{dim}_{W_{n+1}}(W_{n+1}(x_i^{1/p}))
   \\ & \ge \op{dim}_{W_{n+1}}(W_{n+1}(x_i^{1/p}))=\op{dim}_{W_{n+2}}(W_{n+1}). 
\end{align*}
This finishes the proof.
\end{proof}

\begin{remark}
    Proposition \ref{Nonincreasing Degrees Power Tower - Proposition} can be applied to generic points of varieties over perfect fields. Namely: Let $K$ be a field of positive characteristic $p>0$. Let $k=K^{p^\infty}$ be its perfection. If $K/k$ is a finitely generated field extension, then $\op{dim}_{K^p}(K)$ is finite.
\end{remark}

We finish this section with some equivalent conditions for being a power tower.

\begin{lemma}\label{Power Tower - alternative def}
    Let $K$ be a field of characteristic $p>0$.
Let $K=W_0\supset W_1 \supset W_2\supset \ldots$ be a sequence of subfields. The following statements are equivalent.
\begin{enumerate}
    \item $(W_0,W_1,W_2,\ldots)$ is a power tower on $K$.
    \item For every $n\ge 1$, we have $W_n= W_{n+1} \cdot K^{p^n}$.
    \item For every $n\ge 1$, we have $W_n = W_{n+1}\cdot W_{n-1}^p$.
\end{enumerate}
\end{lemma}
\begin{proof}
    We prove the lemma by proving four implications.
    
    $(1. \Longrightarrow 2.)$ The condition $W_n= W_{n+1} \cdot K^{p^n}$ is a part of the definition of a power tower, Definition \ref{Power Tower - Definition}.

    $(2. \Longrightarrow 1.)$ Let $i<j$, then 
    \[
    W_i=W_{i+1}\cdot K^{p^i}=W_{i+2}\cdot K^{p^{i+1}}\cdot K^{p^i}=W_{i+2}\cdot K^{p^i}=\ldots =W_j\cdot K^{p^i}.
    \]
    For $i=j$, we have $W_i=W_{i+1}\cdot K^{p^i}\supset K^{p^i}$, so $W_i=W_i\cdot K^{p^i}$.

    $(2. \Longrightarrow 3.)$ We have
    $
    W_n\supset W_{n+1}\cdot W_{n-1}^p \supset W_{n+1} \cdot K^{p^{n}} =W_n,
    $
    where we used $W_{n-1}=W_{n}\cdot K^{p^{n-1}}\supset K^{p^{n-1}}$. So, $W_n = W_{n+1}\cdot W_{n-1}^p$.

    $(3. \Longrightarrow 2.)$ For $n=1$, we have $W_1=W_2\cdot W_0^p=W_2\cdot K^p$. Now, let assume that $n\ge 2$ and that we already know $W_{n-1}= W_{n} \cdot K^{p^{n-1}}$. This gives
    \[
    W_n = W_{n+1}\cdot W_{n-1}^p=W_{n+1}\cdot \left(W_{n} \cdot K^{p^{n-1}}\right)^p=W_{n+1} \cdot W_n^p \cdot K^{p^n}.
    \]
    But $W_{n+1} = W_{n+2}\cdot W_{n}^p\supset W_n^p$, so $W_{n+1} \cdot W_n^p \cdot K^{p^n}=W_{n+1}\cdot K^{p^n}$. Therefore $W_n =W_{n+1}\cdot K^{p^n}$.
\end{proof}

\subsection{Galois-Type Correspondence for Power Towers}\label{section two - fields}

The fundamental theorem of Galois theory says that any algebraic separable subfield, a field theoretic datum, is equivalent to a subgroup of a Galois group, a group theoretic datum. By a Galois-type correspondence, I mean a bijection that takes field data and outputs something else: rings, groups, algebras, etc. Something that is usually more computable than abstract fields.

Let $K$ be a field of characteristic $p>0$.
Power towers on $K$ are a field-theoretic datum.
Here, we prove that two more notions are equivalent to it. Both of them are of an infinitesimal nature. The first is a subalgebra of differential operators, Section \ref{Differential operators - section}.
The second is a sequence of $p$-Lie algebras satisfying a particular compatibility condition. And, a $p$-Lie algebra is a set of vector fields/derivations, Definition \ref{p-lie algebra: definition}. 

First, we write down all required definitions.

\begin{defin}[Jacobson sequence]\label{Jacobson Sequence - definition}
Let $K$ be a field of characteristic $p>0$.
Let $k=K^{p^\infty}$ be its perfection. We assume that $K/k$ is a finitely generated field extension.

A \textbf{\emph{Jacobson sequence}} on $K$  is a sequence of $p$-Lie algebras $(\F_1,\F_2,\ldots)$ such that
\begin{itemize}
    \item $\F_1$ is a $p$-Lie algebra on $K$, 
    \item for $i\ge 1$, $\F_{i+1}$ is a $p$-Lie algebra on $\op{Ann}(\F_i)$,
    \item for $i\ge 1$, these $p$-Lie algebras satisfy 
\[
\F_{i+1}\cap T_{\op{Ann}(\F_i)/\op{Ann}(\F_{i-1})^p}=0,
\]
where we put $\op{Ann}(\F_{0})\coloneqq K$.
\end{itemize}

We denote the Jacobson sequence $(\F_1,\F_2,\ldots)$ by $\F_\bullet$.

Moreover, we say that a Jacobson sequence $\F_\bullet$ is of length at most $n$ if $\F_N=0$ for $N>n$. And, we say it is of length $n$ if this integer $n$ is minimal.
\end{defin}

\begin{remark}
    I named the above notion ``\textbf{Jacobson sequence}'' after Nathan Jacobson to celebrate his contributions to the theory of $p$-Lie algebras. This notion arises from iteratively applying his correspondence to a power tower.

    Alternatively, we could call it a ``\textbf{sequence of floors}''. Then we would call the $p$-Lie algebra $\F_1$ the ``first floor'', $\F_2$ the ``second floor'', and so on. The letter $\F$ could represent a \textit{floor}. This would be compatible with the name power \textit{tower}: a tower and its floors.
\end{remark}

\begin{defin}($p$-Filtration of a Subalgebra)\label{p-filtration for a subalgebra - def}
Let $K$ be a field of characteristic $p>0$. Let $k=K^{p^\infty}$ be its perfection.
Let $K\subset \mathcal{D}\subset \op{Diff}_k(K)$ be a $K$-subalgebra of $\op{Diff}_k(K)$. 

The \emph{\textbf{$p$-filtration}} of $\mathcal{D}$ is the family of subalgebras $\mathcal{D}_n \coloneqq \mathcal{D} \cap \op{Diff}_{K^{p^n}}(K)$ for $n\ge 0$.

Moreover, a subalgebra $\mathcal{D}$ is of \emph{height at most $n$} if $\mathcal{D}=\mathcal{D}_n$. And, of height $n$, if the number $n$ is minimal.
\end{defin}

Here is the theorem.

\begin{thm}[Galois-Type Correspondence for Power Towers]\label{MainTheorem}
Let $K$ be a field of characteristic $p>0$. Let $k\coloneqq \bigcap_{n\ge 0} K^{p^n}$ be the perfection of $K$. Let the extension $K/k$ be a finitely generated field extension.

Then, there are natural bijections between the following data on $K$:
\begin{enumerate}
    \item power towers on $K$,
    \item Jacobson sequences on $K$,
    \item $K$-subalgebras of differential operators on $K$ over $k$.
\end{enumerate}
Explicitly,
some of these bijections are given by the following operations.
\begin{enumerate}
    \item From power towers to Jacobson sequences. \emph{An iterated tangent space}:
    \[
    (W_0,W_1,W_2,\ldots) \mapsto (T_{W_0/W_1},T_{W_1/W_2},T_{W_2/W_3}, \ldots),
    \]
    \item From Jacobson sequences to power towers. \emph{An iterated annihilator}:
    \[
    (\F_1,\F_2,\F_3,\ldots)\mapsto (K,\op{Ann}(\F_1),\op{Ann}(\F_2),\ldots),
    \]
    \item From power towers to subalgebras. \emph{An algebra of differential operators relative to a power tower}:
    \[
    W_\bullet=(W_0,W_1,W_2,\ldots) \mapsto \op{Diff}_{W_\bullet}(K)\coloneqq \bigcup_{i\ge 0} \op{Diff}_{W_i}(K),
    \]
    \item From subalgebras to power towers. \emph{Fields of constants of the $p$-filtration}:
    \[
    \mathcal{D} \mapsto (K=\op{const}(\mathcal{D}_0),\op{const}(\mathcal{D}_1),\op{const}(\mathcal{D}_2),\ldots).
    \]
\end{enumerate}
\end{thm}

\begin{proof}
    This proof carefully shows that all the operations mentioned in the theorem are well defined and bijective. There are no fireworks.

    (1. $\Longleftrightarrow$ 2.)
    For convenience, we recall Jacobson Correspondence \ref{Jacobson Purely Inseparable Galois Theory for Exponent $1$}. It says that a subfield $K\supset W \supset K^p$ corresponds to a $p$-Lie algebra $T_{K/W}=\op{Der}_W(K)$ and $\op{Ann}(T_{K/W})=W$.

    First, we prove that the operations 1. and 2. are well defined. Then, they are inverse to each other.

    We prove that taking an iterated tangent space, the operation 1., is well defined. Let $W_\bullet$ be a power tower on $K$. Then we can compute a sequence of $p$-Lie algebras
    $(T_{W_0/W_1},T_{W_1/W_2},T_{W_2/W_3}, \ldots)$. 
    By Jacobson Correspondence, $T_{W_0/W_1}$ is a $p$-Lie algebra on $W_0=K$, and
    $T_{W_n/W_{n+1}}$ is a $p$-Lie algebra on $W_n=\op{Ann}(T_{W_{n-1}/W_{n}})$ for $n\ge 1$. 
    Therefore, in order to show that this sequence is a Jacobson sequence, it is sufficient to show that for $n\ge 1$ we have
    \[
T_{W_{n}/W_{n+1}}\cap T_{W_n/W_{n-1}^p}=0.
    \]
    We prove the above here. Any intersection of $p$-Lie subalgebras of a $p$-Lie algebra is a $p$-Lie subalgebra of that $p$-Lie algebra, so
    $T_{W_{n}/W_{n+1}}\cap T_{W_n/W_{n-1}^p}\subset T_{W_n/W_n^p}$ is a $p$-Lie algebra on $W_n$. By Jacobson Correspondence, it equals $T_{W_n/H}$, where $W_n \supset H \supset W_n^p$ is a subfield. By the construction, this subfield contains both $W_{n+1}$ and $W_{n-1}^p$, so it contains their composite $W_{n+1}\cdot W_{n-1}^p$ that is equal to $W_n$ by Lemma \ref{Power Tower - alternative def}.  Therefore, $H=W_n$, so 
    \[
    T_{W_{n}/W_{n+1}}\cap T_{W_n/W_{n-1}^p}=T_{W_n/H}=T_{W_n/W_n}=0.
    \]
    This proves that taking an iterated relative tangent space is a well-defined operation.

    We prove that taking an iterated annihilator, the operation 2., is well defined. Let $\F_\bullet$ be a Jacobson sequence on $K$. Then we can compute a sequence of subfields of $K$
    \[
    K\supset \op{Ann}(\F_1)\supset \op{Ann}(\F_2)\supset \ldots.
    \]
    To show that this sequence is a power tower on $K$, by Lemma \ref{Power Tower - alternative def}, it is enough to show that for every $n\ge 1$ we have
    \[
    \op{Ann}(\F_n) = \op{Ann}(\F_{n+1}) \cdot \op{Ann}(\F_{n-1})^p.
    \]
    This follows immediately from the definition of a Jacobson sequence, Definition \ref{Jacobson Sequence - definition}. Indeed, we have
    \begin{align*}
        0=\F_{n+1}\cap T_{\op{Ann}(\F_n)/\op{Ann}(\F_{n-1})^p}&=T_{\op{Ann}(\F_n)/\op{Ann}(\F_{n+1})}\cap T_{\op{Ann}(\F_n)/\op{Ann}(\F_{n-1})^p}\\
        &=T_{\op{Ann}(\F_n)/\op{Ann}(\F_{n+1})\cdot \op{Ann}(\F_{n-1})^p}.
    \end{align*}
    Now, since the subfield $\op{Ann}(\F_{n+1}) \cdot \op{Ann}(\F_{n-1})^p \subset \op{Ann}(\F_n)$  is of exponent at most $1$, we conclude the equality $\op{Ann}(\F_{n+1}) \cdot \op{Ann}(\F_{n-1})^p = \op{Ann}(\F_n)$ from Jacobson Correspondence. This proves that taking an iterated annihilator is a well-defined operation.
    
    Finally, we know these operations 1 and 2 are well defined. Now, we prove they are inverse to each other. Indeed, by iteratively applying Jacobson Correspondence, we get that $W_1$ corresponds to $\F_1$, $W_2$ corresponds to $\F_2$, and so on. Therefore, these operations are bijective.

    (1. $\Longleftrightarrow$ 3.) For convenience, we recall the discrete Jacobson--Bourbaki Correspondence \ref{Discrete Jacobson--Bourbaki Correspondence}. It says that a subfield $K\supset L$ with a finite degree $[K:L]$ corresponds to a finite algebra of linear endomorphisms $\op{End}_{L}(K)$ and $\op{const}(\op{End}_{L}(K))=L$. Moreover, if $L$ is a purely inseparable subfield of $K$, then it is of finite exponent by Lemma  \ref{Purely Inseparable is of Finite Exponent}. In particular, it is a finite extension. Therefore, we have $\op{Diff}_L(K)=\op{End}_{L}(K)$, by Lemma(Diff=End) \ref{diff=end - new}. Hence, there is a correspondence between purely inseparable subfields and finite subalgebras of differential operators given by the Jacobson--Bourbaki correspondence operations.

    First, we prove that operations 3 and 4 are well-defined and inverse to each other.

    We prove that taking an algebra of differential operators relative to a power tower, the operation 3.,  is well defined.
    Let $W_\bullet$ be a power tower on $K$. We can compute the subalgebras $\op{Diff}_{W_i}(K)\subset \op{Diff}_k(K)$ of differential operators commuting with $W_i$. Since we have inclusions $W_i\supset W_{i+1}$, therefore we have inclusions $\op{Diff}_{W_i}(K)\subset \op{Diff}_{W_{i+1}}(K)$. This means that the sum $ \bigcup_{i\ge 0} \op{Diff}_{W_i}(K)$ is an ascending sum of inclusions of subalgebras of $\op{Diff}_k(K)$, and hence it is a subalgebra of differential operators.

    We prove that taking fields of constants of the $p$-filtration, the operation 4., is well defined. Let $\mathcal{D}$ be a $K$-subalgebra (with $1$) of $\op{Diff}_k(K)$. By Definition \ref{p-filtration for a subalgebra - def}, the $p$-filtration of $\mathcal{D}$ is the sequences of subalgebras $\mathcal{D}_i=\mathcal{D}\cap \op{Diff}_{K^{p^i}}(K)$. From it, we can compute a sequence of subfields $W_i=\op{const}(\mathcal{D}_i)$. And, from the inclusions $\mathcal{D}_i\subset \mathcal{D}_{i+1}$, we get $W_i\supset W_{i+1}$. Now, we show that this sequence $(W_0,W_1,W_2,\ldots)$ is a power tower on $K$.
    The algebra $\mathcal{D}_i$ is a finitely dimensional subalgebra of linear endomorphisms, so, by Jacobson--Bourbaki Correspondence, it is equal to $\op{End}_{W_i}(K)$. And, by the definition, for $j\le i$, we have
    $\mathcal{D}_i \cap \op{Diff}_{K^{p^j}}(K)=\mathcal{D}_j$, i.e.
    \[
    \op{End}_{W_i}(K) \cap \op{End}_{K^{p^j}}(K)=\op{End}_{W_j}(K).
    \]
    The left hand side equals $\op{End}_{W_i\cdot K^{p^j}}(K)$, therefore $W_j=W_i\cdot K^{p^j}$ by Jacobson--Bourbaki Correspondence. So, the sequence is a power tower on $K$. 

    Finally, by Jacobson--Bourbaki Correspondence, the operations 3. and 4. are inverse to each other, because $W_1$ corresponds to $\mathcal{D}_1$,
    $W_2$ corresponds to $\mathcal{D}_2$, and so on. Therefore, these operations are bijective.

    This finishes the proof.
\end{proof}

\begin{remark}
    The assumption ``Let $K$ be a field of characteristic $p>0$. Let $k\coloneqq \bigcap_{n\ge 0} K^{p^n}$ be the perfection of $K$. Let the extension $K/k$ be a finitely generated field extension.'' could be replaced with ``Let $K$ be a field of characteristic $p>0$ such that $[K:K^p]<\infty$.'' No changes are required to the above proof. We only need to be able to access Jacobson--Bourbaki Correspondence and End=Diff Lemma, because Jacobson Correspondence follows from Jacobson--Bourbaki Correspondence.
    
    We frequently use the former assumption, here and there, because we will apply the theorem only to that case.
\end{remark}

One of the remaining bijections is the one from subalgebras to Jacobson sequences. 
We call this operation ``unpacking'': for a power tower, given in the form of a subalgebra, we get a sequence of floors, its Jacobson sequence. Later, we will discuss Remark \ref{packing} 's non-explicit ``inverse'' to unpacking that we call ``packing''.

\begin{thm}[Unpacking]\label{unpacking - Thm}
    Let $K$ be a field of characteristic $p>0$. Let $k\coloneqq \bigcap_{n\ge 0} K^{p^n}$ be the perfection of $K$. Let the extension $K/k$ be a finitely generated field extension.
    Let $\mathcal{D}$ be a subalgebra of differential operators on $K$ over $k$ corresponding to a power tower $W_\bullet$, and to a Jacobson sequence $\F_\bullet$. 
    
    Then, we have the following equalities that we call \emph{an \textbf{unpacking} of the subalgebra $\mathcal{D}$ into the Jacobson sequence $\F_\bullet$}:
    \begin{align*}
        \mathcal{D} \cap T_{K/K^p} &= T_{K/W_1}=\F_1,\\
        d(K/W_{1})(\mathcal{D})\cap T_{W_{1}/W_{1}^p} &= T_{W_1/W_2}=\F_2,\\
        \ldots &= \ldots ,\\
        d(K/W_{n})(\mathcal{D})\cap T_{W_{n}/W_{n}^p} &= T_{W_n/W_{n+1}}=\F_{n+1},\\
        \ldots &= \ldots ,
    \end{align*}
    where $d(K/W_i): \op{Diff}_k(K)\to \op{Diff}_k(W_i)\otimes K$ are differentials between algebras of differential operators, Definition \ref{Differential - definition}.
\end{thm}

\begin{proof}
    First, we prove the first equality from the unpacking of $\mathcal{D}$ into a Jacobson sequence: $\mathcal{D} \cap T_{K/K^p} = T_{K/W_1}$. It is proved by the following calculation.
    \[
    \mathcal{D} \cap T_{K/K^p}=\bigcup_{i\ge 0} \op{Diff}_{W_i}(K) \cap T_{K/K^p}=\bigcup_{i\ge 0} T_{K/K^p\cdot W_i}=\bigcup_{i\ge 1} T_{K/ W_1}=T_{K/W_1}.
    \]

    Now, we prove the  equalities $d(K/W_{n})(\mathcal{D})\cap T_{W_{n}/W_{n}^p}=T_{W_n/W_{n+1}}$ for $n\ge 1$.
    
    First, we explain that the left-hand side makes sense. By Corollary \ref{Fields Admit Differentials}, for every field extension $K/W/k$, we have a differential between their algebras of differential operators 
    $d(K/W):\op{Diff}_k(K)\to \op{Diff}_k(W)\otimes K$, Definition \ref{Differential - definition}, which is a restriction of a map $K\to K$ to a map $W\to K$, Lemma \ref{Differential is Restriction}. The algebra $\op{Diff}_k(W)$ contains $T_{W/W^p}$ in its first degree, and the map $\op{Diff}_k(W)\to \op{Diff}_k(W)\otimes K$ is an inclusion. Consequently, we can intersect the image of $\mathcal{D}$ with the $p$-Lie algebra $T_{W/W^p}$.

    We proceed to the proof.  Let $i>n$. We have
    \[
    d(K/W_{n})(\mathcal{D}_i)=d(K/W_{n})( \op{Diff}_{W_i}(K))= d(K/W_{n})( \op{End}_{W_i}(K))=\op{Hom}_{W_i}(W_n, K),
    \]
    where we used Lemma(Diff=End) \ref{diff=end - new}.
    At the same time, we have 
    \[
    \op{Hom}_{W_i}(W_n, K)=\op{Hom}_{W_i}(W_n, W_n)\otimes K=\op{Diff}_{W_i}(W_n)\otimes K.
    \]
    Therefore, we have
    \begin{align*}
        \op{Diff}_{W_i}(W_n)\otimes K \cap T_{W_n/W_n^p}=&\op{Diff}_{W_i}(W_n)\cap T_{W_n/W_n^p}= \\
    =& T_{W_n/W_i}=T_{W_n/W_i\cdot W_n^p}=T_{W_n/W_{n+1}}.
    \end{align*}
    This result does not depend on $i$ and therefore it holds for $\mathcal{D}$, and not just for $\mathcal{D}_i$, because we have $\mathcal{D}=\bigcup_{i>n}\mathcal{D}_i$. This finishes the proof of the unpacking.
\end{proof}

\subsection{Distinguished Generators of Subalgebras}\label{generators - subsection}

Let $K$ be a field of characteristic $p>0$. Let $K/k$ be a finitely generated field extension, where $k$ is the perfection of $K$. Then, $K/k$ admits coordinates, Definition \ref{coordinates - definition}. Let $x_1,\ldots, x_m$ be a set of coordinates for $K/k$. Then any differential operator on $K$ relative to $k$ can be described in terms of the differential operators $\frac{1}{n!}\frac{\partial^n}{\partial x_i ^n}$ for $n\ge 0$, $1\le i\le m$, Corollary \ref{explicit formulas calculus - cor}. It is not a canonical description; it depends on the coordinates. In this section, we will show that any subalgebra of $\op{Diff}_k(K)$ admits analogous generators to the symbols $\frac{1}{n!}\frac{\partial^n}{\partial x_i^n}$.

\begin{defin}\label{distinguished generators - def}
    Let $K$ be a field of characteristic $p>0$. Let $K/k$ be a finitely generated field extension, where $k$ is the perfection of $K$. Let $\mathcal{D}$ be a subalgebra of $\op{Diff}_k(K)$. Let $W_\bullet$ be a power tower that corresponds to $\mathcal{D}$. Let $p^{r_i}=[W_{i}:W_{i-1}]$ for $i>0$.

    We define a set of \textbf{\emph{distinguished generators}} of  $\mathcal{D}$ to be a subset 
    \[G=\{G^m_1,\ldots ,G^m_{r_m}\}_{m=1,2,\ldots}\subset \mathcal{D}\]
    such that
    \begin{enumerate}
        \item Every operator $G^m_i$ for $m>0$, $1\le i\le r_m$ is of order $p^{m-1}$.
        \item The leading form $[G^m_i]$ for $m>1, 1\le i\le r_m$ is a $K$-linear sum of leading forms $\gamma_{p^{m-1}}([G^{1}_{j}])$ for $1\le j\le r_{1}$.
        \item The graded subset $[G]=\{[D] : D\in G\}\subset\op{gr}(\mathcal{D})$ is $K$-linearly independent.
    \end{enumerate}
    Moreover, we say that the set of distinguished generators $G$ is \emph{well-chosen} if $d(K/W_{m-1})(G^m_i)$ belongs to $\F_m$ for every $i$, where $\F_\bullet$ is the Jacobson sequence of $W_\bullet$. 
\end{defin}

\begin{remark}
    The second condition in Definition \ref{distinguished generators - def} may be changed to an equivalent, though a priori stronger, condition:
    \begin{center}
        ``2'. The leading form $[G^m_i]$ for $m>1, 1\le i\le r_m$ is 
        
        a $K$-linear sum of leading forms $\gamma_{p}([G^{m-1}_{j}])$ for $1\le j\le r_{m-1}$.''.
    \end{center}
     This follows from the methods from the proof of Theorem \ref{saturated dist generators - exist and do the job - prop}, more specifically, the diagram there proves that.
\end{remark}

\begin{defin}\label{completion of distinguished generators - def}
    Let $K$ be a field of characteristic $p>0$. Let $K/k$ be a finitely generated field extension, where $k$ is the perfection of $K$. Let $p^n=[K:K^p]$. Let $\mathcal{D}$ be a subalgebra of $\op{Diff}_k(K)$. Let $G=\{G^m_1,\ldots , G^m_{r_m}\}_{m=1,2,\ldots}$ be a set of distinguished generators of  $\mathcal{D}$.

    We define a \emph{\textbf{completion}} of $G$ to be a subset 
    \[
    E=\{E^m_1,\ldots ,E^m_{n-r_m}\}_{m=1,2,\ldots}\subset \op{Diff}_k(K)
    \]
    such that
    \begin{itemize}
        \item Every operator $E^m_i$ for $m>0$, $1\le i\le n-r_m$ is of order $p^{m-1}$.
        \item The set $\{[G^m_1],\ldots ,[G^m_{r_m}],[E^m_1],\ldots, [E^m_{n-r_m}]\}$ for $m>0$ is a $K$-basis for the $K$-vector space spanned by $\gamma_{p^{m-1}}([T_{K/K^p}])$ inside $\op{gr}(\op{Diff}_k(K))$.
    \end{itemize}
\end{defin}

\begin{example}
    Let $K$ be a field of characteristic $p>0$. Let $K/k$ be a finitely generated field extension, where $k$ is the perfection of $K$. Let $p^n=[K:K^p]$. 
    
    Let $K\supset W\supset k$ be a subfield such that $K/W$ is a separable extension. Then, by Theorem \ref{Over Perfect is Separable}, all extensions $K/W$, $K/k$, $W/k$ are separable. Let $x_1,\ldots, x_r$ be a separable transcendental basis for $W/k$. Let $y_1,\ldots, y_{n-r}$ be a separable transcendental basis for $K/W$. Then the set $ x_1,\ldots, x_r,y_1,\ldots, y_{n-r}$ is a separable transcendental basis for $K/k$. With respect to this basis, the set $G$:
    \[
    \left\{\frac{1}{p^m!}\frac{\partial^{p^m}}{\partial x_1 ^{p^m}},\ldots , \frac{1}{p^m!}\frac{\partial^{p^m}}{\partial x_r ^{p^m}} \right\}_{m=1,2,\ldots}
    \] 
    is a set of (well-chosen) distinguished generators for $\op{Diff}_W(K)$, and the set $E$:
    \[
    \left\{\frac{1}{p^m!}\frac{\partial^{p^m}}{\partial y_1 ^{p^m}},\ldots , \frac{1}{p^m!}\frac{\partial^{p^m}}{\partial y_{n-r} ^{p^m}} \right\}_{m=1,2,\ldots}
    \] 
    is a completion of this distinguished basis. Moreover, what is not part of the definition, and it is a special property of this example, is that all these operators commute with each other.

    It is important to emphasize that the sets $G$ and $E$ are not canonical and they heavily depend on the choice of the basis $x_1,\ldots, x_r,y_1,\ldots, y_{n-r}$.
\end{example}

\begin{lemma}[Graded Differential and Graded Kernels of Differentials]\label{Graded Differential for K/K^m - lemma}
    Let $K$ be a field of characteristic $p>0$. Let $K/k$ be a finitely generated field extension, where $k$ is the perfection of $K$. 
    Let $\D$ be a subalgebra of $\op{Diff}_{k}(K)$. Let $W_\bullet$ be its power tower. Let $m\ge1$.
    
    Then the following are true:
    \begin{enumerate}
        \item The differential $d(K/K^{p^{m}})$ induces a \textbf{\textit{graded differential}} $\op{gr}d(K/K^{p^{m}})$:
        \[
        \op{gr}d(K/K^{p^{m}}): \op{gr}\op{Diff}_k(K) \to \op{gr}\op{Diff}_k(K^{p^{m}})\otimes K.
        \]
        
        Explicitly, if $x_1,\ldots,x_n$ are coordinates for $K/k$, then the graded differential is determined by, see Example \ref{symbol.differential}:
        \[
        \op{gr}d(K/K^{p^{m}}): \left[ \frac{1}{p^M!}\frac{\partial^{p^M}}{\partial x_i ^{p^{M}}}\right] \mapsto\left[\frac{1}{p^{M-m}!}\frac{\partial^{p^{M-m}}}{\partial x_i ^{p^{M-m}}}\right]
        \]
        for $i=1,2,\ldots,n$, $M\ge 0$. If $M-m<0$, then we put the value zero.

        In terms of Proposition \ref{divided power for gr diff - prop}, it is simply dividing by the ideal $I=(\gamma_{j}(y_i))_{0<j<p^m, i=1,\ldots,n}$ of $K\left<y_1,\ldots,y_n\right>$, so, up to that isomorphism, we have:
        \[
        \op{gr}d(K/K^{p^{m}}):K\left<y_1,\ldots,y_n\right>\to K\left<y_1,\ldots,y_n\right>/I.
        \]
        \item We have $\op{ker}d(K/K^{p^{m}}) \cap \D = \op{ker} d(K/W_{m}) \cap \D$.
        \item We have $\op{gr}\left(\op{ker} d(K/W_{m}) \cap \D\right)=\op{ker}\left(\op{gr}d(K/K^{p^{m}})\right)\cap \op{gr}\D$.
    \end{enumerate}
\end{lemma}

\begin{proof}
    The first item is so detailed that it proves itself.

    The second item. We will show it for $\D=\D_{m+i}$, where $i\ge 1$. it is enough, because $\D$ is a sum of its $p$-filtration, $\D=\bigcup_{N> m}\D_m$.

    We observe that $d(W_{m}/K^{p^{m}})$ is injective on $\bigcup_{N> m}\op{Diff}_{W_N}(W_{m})$, because every differential is a restriction, Lemma \ref{Differential is Restriction}, and we do not add any more commutativity. It is also true after tensoring with $K$. Since $d(K/K^{p^{m}})=d(W_{m}/K^{p^{m}})\otimes K \circ d(K/W_{m})$, we conclude that the kernels of $d(K/K^{p^{m}})$ and $d(K/W_{m})$ intersected with $\mathcal{D}_{m+i}$ are equal, i.e.,
    \[
    \op{ker}d(K/K^{p^{m-1}}) \cap \D_{m+i} = \op{ker} d(K/W_{m-1}) \cap \D_{m+i}.
    \]
    This proves the second item by taking a union over $i\ge 1$.

    The third item. Again, we work with $\D=\D_{m+i}$, $i\ge 1$. We have a sequence; the first three arrows are exact:
    \[
    0\to \mathcal{D}_{m+i}\cap \op{ker}(d(K/K^{p^{m}}))\to \mathcal{D}_{m+i} \to d(K/K^{p^{m}})(\mathcal{D}_{m+i}) \subset \op{Diff}_k(K^{p^{m}}) \otimes K.
    \]
    
    From the first item, we have that $\op{gr}d(K/K^{p^{m-1}})$ is well defined, and it fits the following short exact sequence:  
    \[
    0\to T\to \op{gr}\mathcal{D}_m \xrightarrow{\op{gr}d(K/K^{p^{m-1}})}  \op{gr}(d(K/K^{p^{m-1}})(\mathcal{D}_m))\to 0,
    \]
    where $T= \op{ker}(\op{gr}d(K/K^{p^{m-1}}))\cap  \op{gr}\mathcal{D}_m$.
    
    An elementary fact is that if $V$ is a vector subspace of a graded vector space $W$, then the dimension of $\op{gr}(V)$ equals the dimension of $V$. 
    Thus, we have that the dimension of $\op{gr}(d(K/K^{p^{m}})(\mathcal{D}_{m+i}))$ is the dimension of $d(K/K^{p^{m}})(\mathcal{D}_{m+i})$, and therefore the dimension of $T$ is the dimension of $ \mathcal{D}_{m+i}\cap \op{ker}(d(K/K^{p^{m}}))$.
    Finally, we have $\op{gr}( \mathcal{D}_m\cap \op{ker}(d(K/K^{p^{m-1}})))\subset T$, so they are equal. This proves the third item.
\end{proof}

\begin{prop}\label{dist generators exist - prop}
    Let $K$ be a field of characteristic $p>0$. Let $K/k$ be a finitely generated field extension, where $k$ is the perfection of $K$. Let $\mathcal{D}$ be a subalgebra of $\op{Diff}_k(K)$. 
    
    Then, a set of well-chosen distinguished generators for $\mathcal{D}$ exists, and every such set admits a completion.
\end{prop}

\begin{proof}
    By Theorem \ref{MainTheorem}, $\D$ corresponds to a power tower $W_\bullet$ and to a Jacobson sequence $\F_\bullet$.
    The main goal of the proof is to show that for every $i\ge 1$ we can take a basis of $\F_i$ and ``pack it back'' to the subalgebra $\D$, see Theorem \ref{unpacking - Thm}. Initially, we take any preimages with respect to the unpacking, and then we correct them to satisfy our distinguished conditions.

    Let $p^{r_i}=[W_{i}:W_{i-1}]$ for $i>0$.

    We begin with $m=1$. Let $G^1_1,\ldots, G^1_{r_1}$ be any $K$-basis of $\F_1$.

    Let $m>1$. Let $G'^m_1,\ldots, G'^m_{r_m}$ be a $W_{m-1}$-basis of $\F_{m}$. According to the unpacking, Theorem \ref{unpacking - Thm}, we have that 
    \[
    d(K/W_{m-1})(\mathcal{D})\cap T_{W_{m-1}/W_{m-1}^p}=d(K/W_{m-1})(\mathcal{D}_m)\cap T_{W_{m-1}/W_{m-1}^p}=\F_m.
    \]
    (We recall $\mathcal{D}_m=\mathcal{D}\cap \op{Diff}_{K^{p^{m}}}(K)$.)
    Therefore, there are differential operators $G''^m_1,\ldots, G''^m_{r_m} \subset \mathcal{D}_m$ that restrict to $G'^m_1,\ldots, G'^m_{r_m}$ respectively, Lemma \ref{Differential is Restriction}. However, not every choice gives distinguished generators, so we must choose them carefully. We will use Lemma \ref{Graded Differential for K/K^m - lemma}.

    We now work with $\mathcal{D}_m$, instead of the whole $\D$, to simplify some reasoning.

    We have $G''^m_1,\ldots, G''^m_{r_m} \subset \mathcal{D}_m$
    that are operators restricting to $G'^m_1,\ldots, G'^m_{r_m}$. 
    
    If operators $G'''^m_1,\ldots, G'''^m_{r_m} \in \mathcal{D}_m$ are another such choice, then we have, for $0\le i \le r_m$,
    \[
    G'''^m_i = G''^m_i + B_i,
    \]
    where  $B_i\in\mathcal{D}_m\cap \op{ker}(d(K/W_{m-1}))=\mathcal{D}_m\cap \op{ker}(d(K/K^{p^{m-1}}))$, by Lemma \ref{Graded Differential for K/K^m - lemma}. We will finish the proof by modifying operators $B_i$.

    If the order of $G'''^m_i$ is higher than $p^{m-1}$, then $\op{gr}d(K/K^{p^{m-1}}) ([G'''^m_i])=0$, so, by Lemma \ref{Graded Differential for K/K^m - lemma}, there is an operator $C_i\in \mathcal{D}$ that is also from the kernel of $d(K/K^{p^{m-1}})$ such that $[C_i]=[G'''^m_i]$.

    We substitute $G'''^m_i$ with $G'''^m_i-C_i$. The new $G'''^m_i$ still satisfies the conditions, but it is of lower order. We continue these substitutions till the order of $G'''^m_i$ becomes $p^{m-1}$. It cannot go lower, because $d(K/K^{m-1})(G'''^m_i)\ne 0$ 
    
    Now, by Proposition\ref{divided power for gr diff - prop} and Lemma \ref{divided powers - decompostion of order p^n},
    we have
    \[
    [G'''^m_i]= D_i +E_i
    \]
    where the leading form $D_i$ is a linear combination of $\gamma_{p^{m-1}}([G^1_j])$, and the leading form $E_i$ is a linear combination of compositions of $\gamma_{k}([G^1_j])$ for $k<p^{m-1}$. So, $E_i$ is killed by $d(K/K^{p^{m-1}})$. This means that there is an operator $E'_i \in \mathcal{D}_m\cap d(K/K^{p^{m-1}})=\mathcal{D}_m\cap \op{ker}(d(K/W_{m-1}))$ such that $[E'_i]=E_i$, Lemma \ref{Graded Differential for K/K^m - lemma}. The final substitution for $G'''^m_i$ is $G'''^m_i - E'_i$. We name the result: $G^m_i$.

    Finally, the leading forms $[G^m_i]$ for all admissible $m,i$ are linearly independent, because for different $m$ they are in various degrees. For the same order $p^{m-1}$, they stay independent after applying the linear map $d(K/W_{m-1})$, so they are independent.
    This finishes the construction of well-chosen distinguished generators for $\mathcal{D}$.

    Now, we prove that every set of distinguished generators admits a completion. It is trivial: just complete independent vectors $[G^m_i]$ by some vector $E'^m_j$ to a basis of the vector space spanned by $\gamma_{p^{m-1}}([T_{K/K^p}])$ inside $\op{gr}(\op{Diff}_k(K))$ and take operators $E^m_j \in \op{Diff}_k(K)$ such that $[E^m_j]=E'^m_j$. The union of all these operators $E^m_j$ is a completion of $\{G^m_i\}$.
\end{proof}

Here is the main point why these distinguished generators are useful.

\begin{prop}\label{description of any subalgebra with dist generators - prop}
    Let $K$ be a field of characteristic $p>0$. Let $K/k$ be a finitely generated field extension, where $k$ is the perfection of $K$. Let $p^n=[K:K^p]$. Let $\mathcal{D}$ be a subalgebra of $\op{Diff}_k(K)$. Let $G=\{G^m_1,\ldots , G^m_{r_m}\}_{m=1,2,\ldots}$ be a set of distinguished generators of  $\mathcal{D}$. Let $E=\{E^m_1,\ldots ,E^m_{n-r_m}\}_{m=1,2,\ldots}$ be a completion of $G$.

    Then, the set of differential operators
    \[
    \left\{
    \left(\prod_{m=1}^k\prod^{r_m}_{i=1} \left(G^m_i\right)^{b^m_i}\right): k\ge 0, 0 \le b^m_i \le p-1\right\}
    \]
    is a $K$-basis of $\mathcal{D}$.
    And, the set of differential operators
    \[
    \left\{\left(\prod_{m=1}^k \prod^{r-r_m}_{j=1} \left(E^m_j\right)^{a^m_j}\right)\cdot \left(\prod_{m=1}^k\prod^{r_m}_{i=1} \left(G^m_i\right)^{b^m_i}\right): k\ge 0, 0 \le a^m_i \le p-1, 0 \le b^m_i \le p-1 \right\},
    \]
    is a $K$-basis of $\op{Diff}_k(K)$.
\end{prop}

\begin{proof}
    By Proposition \ref{pd envelope - prop}, we know that $\op{gr}\mathcal{D}$ is a subset of a divided power polynomial algebra. Inside this algebra, it is clear that leading forms
    \[
    \left[\left(\prod_{m=1}^k\prod^{r_m}_{i=1} \left(G^m_i\right)^{b^m_i}\right)\right]\] 
    are $K$-independent, and therefore the operators 
    $\prod_{m=1}^k\prod^{r_m}_{i=1} \left(G^m_i\right)^{b^m_i}$ are $K$-independent.
    Moreover, for every $n\ge 0$, we have that $\mathcal{D}_n$ is of the dimension that is equal to the number of distinct operators $\prod_{m=1}^k\prod^{r_m}_{i=1} \left(G^m_i\right)^{b^m_i}$ inside it. Consequently, those operators form a basis. This proves the first claim.

    The second claim follows from the first one, because the set $G \cup E$ is a set of distinguished generators for $\op{Diff}_k(K)$ and the order of compositions does not matter for the leading forms argument above, because $\op{gr}\op{Diff}_k(K)$ is commutative.
      
\end{proof}

\subsubsection{Visualization of Subalgebras}

This section will provide a neat graphical way to think about any subalgebra of differential operators.

Let $K$ be a field of characteristic $p>0$. Let $K/k$ be a finitely generated field extension, where $k$ is the perfection of $K$. Let $p^n=[K:K^p]$. 

The algebra $\op{Diff}_k(K)$ is\textbf{ big and noncommutative}, but it admits many regularities, and these regularities can be turned into a semiformal way of thinking about it and its subalgebras. The main pieces are unpacking \ref{unpacking - Thm} and divided powers \ref{divided power polynomial algebra - def}. Actually, during my PhD, I had first discovered these visual representations. I had believed in them, and these pictures guided me to develop the formal theory of power towers, my PhD thesis, and now this paper. I think it is helpful to share it, though it may be initially confusing to read.

The algebra $\op{Diff}_k(K)$ has its first degree $T_{K/k}$. This first degree does not generate it, but it may determine $\op{gr}\op{Diff}_k(K)$ with a push from divided powers. The graded algebra is a divided power polynomial ring, and thus it is generated by all
\[
T_{K/k} \cup \gamma_p\left(T_{K/k}\right)\cup \gamma_{p^2}\left(T_{K/k}\right)\cup \ldots.
\]
For every $n\ge 0$, the set $\gamma_{p^n}\left(T_{K/k}\right)$ is a set of forms of order $p^n$. So, we can draw $\op{Diff}_k(K)$ as an infinite rectangular, a \textbf{tower}, and we can draw $\op{gr}\op{Diff}_k(K)$ similarly. Still, if we do it in a logarithmic scale, in a way that $+1$ higher is $p$-times higher, then we can put $\gamma_{p^n}\left(T_{K/k}\right)$ to be arithmetically growing \textbf{floors} of that tower.

Now, let $\mathcal{D}$ be a subalgebra of $\op{Diff}_k(K)$. We can observe that
this algebra looks like \textbf{stairs} inside the tower of $\op{Diff}_k(K)$. Indeed, if $\F_\bullet$ is its Jacobson sequence, Definition \ref{Jacobson Sequence - definition}, and $W_\bullet$ its power tower, then their dimensions are nonincreasing \ref{Nonincreasing Degrees Power Tower - Proposition}. Thus, how much space $\mathcal{D}$ takes inside $\op{Diff}_k(K)$ looks like \textbf{stairs}. Finite, or with the last step being infinite. This makes more sense inside $\op{gr}\D$, because we can ``pack back'' \textbf{shadows} of $\F_i$ there. Indeed, let 
\[
\overline{\F_i}\coloneqq \op{gr}\D \cap \gamma_{p^{i-1}}\left(T_{K/k}\right).
\]
This is well defined and the rank of $\F_i$ (over $W_{i-1}$) is the rank of $\overline{\F_i}$ (over $K^{p^{i-1}}$). Indeed, let $G=\{G^m_1,\ldots , G^m_{r_m}\}_{m=1,2,\ldots}$ be a set of well-chosen distinguished generators, Definition \ref{distinguished generators - def}. Then $\op{gr} G \cap \gamma_{p^{i-1}}\left(T_{K/k}\right)$ is a basis for $\overline{\F_i}$ and $d(K/W_{i-1})(G)\cap T_{W_{i-1}/W_{i-1}^p}$ is a basis of $\F_i$. So, we do not have $\F_i$, $i\ge 2$, inside $\D$, but we have their \textbf{shadows} inside $\op{gr}\D$.

The figure below, which was created using \nolinkurl{https://www.mathcha.io}, is a visual manifestation of the above discussion.

\FloatBarrier
\begin{figure}
    \centering  
    
{

\tikzset{every picture/.style={line width=0.75pt}} 

\begin{tikzpicture}[x=0.75pt,y=0.75pt,yscale=-0.7,xscale=0.7]

\draw   (50,50) -- (250,50) -- (250,300) -- (50,300) -- cycle ;
\draw   (350,50) -- (550,50) -- (550,300) -- (350,300) -- cycle ;
\draw   (275,175) -- (305,175) -- (305,170) -- (325,180) -- (305,190) -- (305,185) -- (275,185) -- cycle ;
\draw   (355,225) .. controls (355,222.24) and (357.24,220) .. (360,220) -- (540,220) .. controls (542.76,220) and (545,222.24) .. (545,225) -- (545,240) .. controls (545,242.76) and (542.76,245) .. (540,245) -- (360,245) .. controls (357.24,245) and (355,242.76) .. (355,240) -- cycle ;
\draw   (355,275) .. controls (355,272.24) and (357.24,270) .. (360,270) -- (540,270) .. controls (542.76,270) and (545,272.24) .. (545,275) -- (545,290) .. controls (545,292.76) and (542.76,295) .. (540,295) -- (360,295) .. controls (357.24,295) and (355,292.76) .. (355,290) -- cycle ;
\draw   (355,175) .. controls (355,172.24) and (357.24,170) .. (360,170) -- (540,170) .. controls (542.76,170) and (545,172.24) .. (545,175) -- (545,190) .. controls (545,192.76) and (542.76,195) .. (540,195) -- (360,195) .. controls (357.24,195) and (355,192.76) .. (355,190) -- cycle ;
\draw   (355,125) .. controls (355,122.24) and (357.24,120) .. (360,120) -- (540,120) .. controls (542.76,120) and (545,122.24) .. (545,125) -- (545,140) .. controls (545,142.76) and (542.76,145) .. (540,145) -- (360,145) .. controls (357.24,145) and (355,142.76) .. (355,140) -- cycle ;
\draw    (490,98.5) -- (490,59.5) ;
\draw [shift={(490,57.5)}, rotate = 90] [color={rgb, 255:red, 0; green, 0; blue, 0 }  ][line width=0.75]    (10.93,-3.29) .. controls (6.95,-1.4) and (3.31,-0.3) .. (0,0) .. controls (3.31,0.3) and (6.95,1.4) .. (10.93,3.29)   ;
\draw   (275,475) -- (305,475) -- (305,470) -- (325,480) -- (305,490) -- (305,485) -- (275,485) -- cycle ;
\draw   (455,525) .. controls (455,522.24) and (457.24,520) .. (460,520) -- (540,520) .. controls (542.76,520) and (545,522.24) .. (545,525) -- (545,540) .. controls (545,542.76) and (542.76,545) .. (540,545) -- (460,545) .. controls (457.24,545) and (455,542.76) .. (455,540) -- cycle ;
\draw   (405,575) .. controls (405,572.24) and (407.24,570) .. (410,570) -- (540,570) .. controls (542.76,570) and (545,572.24) .. (545,575) -- (545,590) .. controls (545,592.76) and (542.76,595) .. (540,595) -- (410,595) .. controls (407.24,595) and (405,592.76) .. (405,590) -- cycle ;
\draw   (455,475) .. controls (455,472.24) and (457.24,470) .. (460,470) -- (540,470) .. controls (542.76,470) and (545,472.24) .. (545,475) -- (545,490) .. controls (545,492.76) and (542.76,495) .. (540,495) -- (460,495) .. controls (457.24,495) and (455,492.76) .. (455,490) -- cycle ;
\draw   (505,425) .. controls (505,422.24) and (507.24,420) .. (510,420) -- (540,420) .. controls (542.76,420) and (545,422.24) .. (545,425) -- (545,440) .. controls (545,442.76) and (542.76,445) .. (540,445) -- (510,445) .. controls (507.24,445) and (505,442.76) .. (505,440) -- cycle ;
\draw    (525,398.5) -- (525,359.5) ;
\draw [shift={(525,357.5)}, rotate = 90] [color={rgb, 255:red, 0; green, 0; blue, 0 }  ][line width=0.75]    (10.93,-3.29) .. controls (6.95,-1.4) and (3.31,-0.3) .. (0,0) .. controls (3.31,0.3) and (6.95,1.4) .. (10.93,3.29)   ;
\draw   (135,321) -- (150,305) -- (165,321) -- (157.5,321) -- (157.5,345) -- (142.5,345) -- (142.5,321) -- cycle ;
\draw   (435,321) -- (450,305) -- (465,321) -- (457.5,321) -- (457.5,345) -- (442.5,345) -- (442.5,321) -- cycle ;
\draw    (100,600) -- (250,600) ;
\draw    (250,350) -- (250,600) ;
\draw    (100,600) -- (100,550) ;
\draw    (100,550) -- (150,550) ;
\draw    (150,550) -- (150,450) ;
\draw    (200,450) -- (150,450) ;
\draw    (200,450) -- (200,350) ;
\draw    (200,350) -- (250,350) ;
\draw    (400,600) -- (550,600) ;
\draw    (550,350) -- (550,600) ;
\draw    (400,600) -- (400,550) ;
\draw    (400,550) -- (450,550) ;
\draw    (450,550) -- (450,450) ;
\draw    (500,450) -- (450,450) ;
\draw    (500,450) -- (500,350) ;
\draw    (500,350) -- (550,350) ;
\draw   (55,275) .. controls (55,272.24) and (57.24,270) .. (60,270) -- (240,270) .. controls (242.76,270) and (245,272.24) .. (245,275) -- (245,290) .. controls (245,292.76) and (242.76,295) .. (240,295) -- (60,295) .. controls (57.24,295) and (55,292.76) .. (55,290) -- cycle ;
\draw    (190,98.5) -- (190,59.5) ;
\draw [shift={(190,57.5)}, rotate = 90] [color={rgb, 255:red, 0; green, 0; blue, 0 }  ][line width=0.75]    (10.93,-3.29) .. controls (6.95,-1.4) and (3.31,-0.3) .. (0,0) .. controls (3.31,0.3) and (6.95,1.4) .. (10.93,3.29)   ;
\draw    (225,398.5) -- (225,359.5) ;
\draw [shift={(225,357.5)}, rotate = 90] [color={rgb, 255:red, 0; green, 0; blue, 0 }  ][line width=0.75]    (10.93,-3.29) .. controls (6.95,-1.4) and (3.31,-0.3) .. (0,0) .. controls (3.31,0.3) and (6.95,1.4) .. (10.93,3.29)   ;
\draw   (105,575) .. controls (105,572.24) and (107.24,570) .. (110,570) -- (240,570) .. controls (242.76,570) and (245,572.24) .. (245,575) -- (245,590) .. controls (245,592.76) and (242.76,595) .. (240,595) -- (110,595) .. controls (107.24,595) and (105,592.76) .. (105,590) -- cycle ;

\draw (285,150) node [anchor=north west][inner sep=0.75pt]   [align=left] {gr};
\draw (416,75) node [anchor=north west][inner sep=0.75pt]   [align=left] {continues};
\draw (197,18.4) node [anchor=north west][inner sep=0.75pt]    {$\text{Diff}_{k}( K)$};
\draw (285,450) node [anchor=north west][inner sep=0.75pt]   [align=left] {gr};
\draw (451,375) node [anchor=north west][inner sep=0.75pt]   [align=left] {continues};
\draw (479,18.4) node [anchor=north west][inner sep=0.75pt]    {$\text{gr Diff}_{k}( K)$};
\draw (232,323.4) node [anchor=north west][inner sep=0.75pt]    {$D$};
\draw (517,323.4) node [anchor=north west][inner sep=0.75pt]    {$\text{gr} \ D$};
\draw (412,223.4) node [anchor=north west][inner sep=0.75pt]    {$\gamma _{p}( T_{K/K^{p}})$};
\draw (412,273.4) node [anchor=north west][inner sep=0.75pt]    {$\ \ \ \ \ T_{K/K^{p}}$};
\draw (412,173.4) node [anchor=north west][inner sep=0.75pt]    {$\gamma _{p^{2}}( T_{K/K^{p}})$};
\draw (412,123.4) node [anchor=north west][inner sep=0.75pt]    {$\gamma _{p^{3}}( T_{K/K^{p}})$};
\draw (112,273.4) node [anchor=north west][inner sep=0.75pt]    {$\ \ \ \ \ T_{K/K^{p}}$};
\draw (116,75) node [anchor=north west][inner sep=0.75pt]   [align=left] {continues};
\draw (151,375) node [anchor=north west][inner sep=0.75pt]   [align=left] {continues};
\draw (167,575.4) node [anchor=north west][inner sep=0.75pt]    {$F_{1}$};
\draw (492,521.4) node [anchor=north west][inner sep=0.75pt]    {$\overline{F_{2}}$};
\draw (492,471.4) node [anchor=north west][inner sep=0.75pt]    {$\overline{F_{3}}$};
\draw (517,421.4) node [anchor=north west][inner sep=0.75pt]    {$\overline{F_{4}}$};
\draw (467,575.4) node [anchor=north west][inner sep=0.75pt]    {$F_{1}$};

\FloatBarrier
\end{tikzpicture}

}  

    \caption{A visualization of a subalgebra $\mathcal{D}\subset \op{Diff}_k(K)$ for orders satisfying $n>r_1>r_2=r_3>r_4=r_5=\ldots>0$, where $p^{r_i}=[W_{i}:W_{i-1}]$ for $i>0$.}
    \label{fig:stairs}
\end{figure}
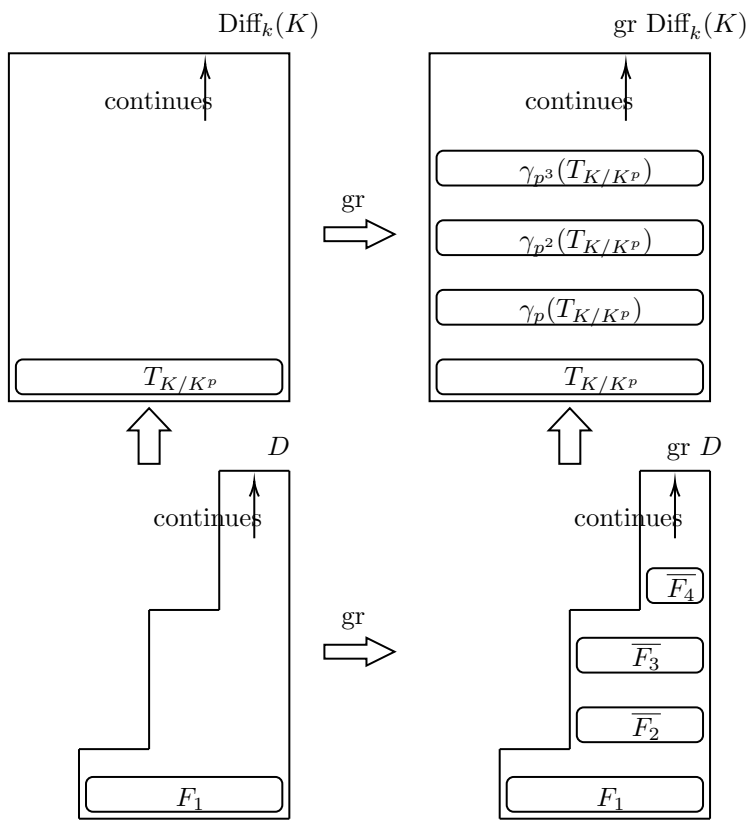
\FloatBarrier

\begin{remark}\label{packing}
    Above, we have $\overline{\F_i}\coloneqq \op{gr}\D \cap \gamma_{p^{i-1}}\left(T_{K/k}\right)$, i.e. it is canonically defined.
    Moreover, we have an isomorphism that I call ``packing''.
    \[
    \F_m \to \overline{\F_{m}}\otimes W_{m-1}: \sum_{i=1}^{r_m} a_i d(K/W_{m-1})(G^m_{i}) \mapsto \sum_{i=1}^{r_m} a_i [G^m_{i}],
    \]
    where $a_i\in W_{m-1}$. It does not depend on the chosen distinguished generators' choice $G$. This can be checked using Lemma \ref{Graded Differential for K/K^m - lemma}.

    The vector spaces $\overline{\F_{m}}$ look great, but they know less than $\F_m$. This is because different subalgebras can have the same graded subalgebras. In particular, we will see later that all $n$-foliations with the same $W_1$ have the same graded subalgebras, Proposition \ref{dist generators for n-foli - prop}.

    However, we cannot put $\F_m$ directly into $\D$. The best thing is to choose $G$ and to take vector spaces generated by its subsets with fixed top index, but this is not canonical. Therefore, I claim no ``canonical formula for computing $\D$ from $\F_\bullet$'' exists. This means I cannot really compute any good generators of $\D$ from all $\F_m$. At least I do not know how to do it, but if there were such a way, then there would be a formula for how to get a set of distinguished generators $G$ from bases of all $\F_m$, but the differentials $d(K/W_{m-1})$ are generally not injective... So, why one choice over another? Sometimes, like in Sweedler's modular theory, a group makes the choice, but generally, there is nothing.
\end{remark}

\subsection[Short Exact Sequence for Differential]{Short Exact Sequence for Differential between Algebras of Differential Operators}\label{SES for diff - subsection}

We use the differential between algebras of differential operators to unpack \ref{unpacking - Thm} any subalgebra into a Jacobson sequence. Here is more information about that function - we show it is always surjective and we describe the kernel. This should be compared with a differential for derivations, which is false. Indeed, the differential $d(K/K^p):T_{K/k}\to T_{K^p/k}\otimes K$ is zero, because derivations kill $p$-powers. Consequently, the surjectivity follows from interactions between different orders of differential operators.

\begin{thm}[SES for Diff]\label{SES for diff - fields - theorem}
     Let $K$ be a field of characteristic $p>0$. Let $k=K^{p^\infty}$ be the perfection of $K$. Let $K/k$ be a finitely generated field extension.
    Let $K\supset W\supset k$ be a subfield. 
    
    Then,
    we have a short exact sequence of $K$-vector spaces:
    \[
    0\to \op{Diff}_k(K) \circ \E(K/W)\hookrightarrow \op{Diff}_k(K) \xrightarrow{d(K/W)}  \op{Diff}_k(W) \otimes K\to 0,
    \]
    where $ \E(K/W)\coloneqq \{D\in \op{Diff}_W(K) \ : \ D(1)=0\}$, and $\op{Diff}_k(K) \circ \E(K/W)$ is a left-ideal of $\op{Diff}_k(K)$ that is spanned  as a $K$-vector space by $D_1 \circ D_2$ for $D_1\in \op{Diff}_k(K)$ and $D_2\in \E(K/W)$.
\end{thm}

\begin{proof}
    This proof consists of two parts. In the first part, we prove the surjectivity of $d(K/W)$. In the second part, we provide a description of the kernel.

    \textbf{Part 1: Surjectivity} The first part of the proof consists of three steps. First, we prove the surjectivity for separable extensions $K/W$. Second, we prove surjectivity for purely inseparable extensions $K/W$. Third, we combine the previous two steps to prove it in full generality.

    \textbf{Part 1, Step 1: Surjectivity For Separable Extensions}
    We prove that if the field extension $K/W$ is separable, then the differential $d(K/W)$ is surjective. 
    
    By the assumption, the extension $K/W$ is separable, i.e., it admits a separable transcendence basis $x_1,\ldots,x_n$. And, the extension $W/k$ admits a separable transcendence basis $y_1,\ldots,y_m$ by Theorem \ref{p-basis is separable transcedence basis}. Therefore, we can conclude that $y_1,\ldots,y_m,x_1,\ldots,x_n$ is a separable transcendence basis for the extension $K/k$.
    
    We use the bases $y_1,\ldots,y_m,x_1,\ldots,x_n$ and $y_1,\ldots,y_m$ to compute $\op{Diff}_k(K)$ and $\op{Diff}_k(W)$ respectively, Proposition \ref{Determination1}. In particular, we know that every differential operator on $W$ over $k$ is uniquely determined by its values on monomials in $y_i$'s, and every differential operator on $K$ over $k$ is determined by its values on monomials in $y_i$'s and $x_j$'s. Moreover, every choice of these values defines a unique differential operator.
    
    We prove the surjectivity. First, we observe that it is enough to show that for every $\phi \in \op{Diff}_k(W)$, there is a $\psi \in \op{Diff}_k(K)$ such that $d(K/W)(\psi)=\phi$, because $\op{Diff}_k(W)\otimes K$ is spanned by $\op{Diff}_k(W)$. Now, let $\phi \in \op{Diff}_k(W)$. We define $\psi\in \op{Diff}_k(K)$ by having the same values as $\phi$ on monomials in $y_i$'s and by being zero on monomials in $x_i$'s and $y_i$'s, including some $x_i$'s. Consequently, we have $d(K/W)(\psi)=\phi$. 
    
    This finishes the proof of the surjectivity for separable extensions.

    \textbf{Part 1, Step 2: Surjectivity For Purely Inseparable Extensions}
    We prove that if the field extension $K/W$ is purely inseparable, then $d(K/W)$ is surjective. 
    
    The extension $K/W$ is of finite exponent by Lemma \ref{Purely Inseparable is of Finite Exponent}, i.e., there is an integer $n$ such that $W\supset K^{p^n}$. This implies that the $p$-filtration of $\op{Diff}_k(K)$ partially descends to $\op{Diff}_k(W)$:
\begin{align*}
 &\op{Diff}_{K^{p^n}}(K) \subset \op{Diff}_{K^{p^{n+1}}}(K) \subset \op{Diff}_{K^{p^{n+2}}}(K) \subset \ldots \subset \op{Diff}_k(K), \\
 &\op{Diff}_{K^{p^n}}(W) \subset \op{Diff}_{K^{p^{n+1}}}(W) \subset \op{Diff}_{K^{p^{n+2}}}(W) \subset \ldots \subset \op{Diff}_k(W).
\end{align*}
    Moreover, we can show that $d(K/W)$ is compatible with these filtrations. Indeed, being compatible means that the subalgebra $\op{Diff}_{K^{p^{n+i}}}(K)$ is mapped to $\op{Diff}_{K^{p^{n+i}}}(W)\otimes K$. This holds, because the differential $d(K/W)$ is a restriction from $K$ to $W$, Lemma \ref{Differential is Restriction}, so it maps $K^{p^{n+i}}$-linear maps to $K^{p^{n+i}}$-linear maps. 
    
    Next, we prove that every map $\op{Diff}_{K^{p^{n+i}}}(K)\to \op{Diff}_{K^{p^{n+i}}}(W)\otimes K$ is surjective for $i\ge 0$.
    First, we observe that this map is the natural map 
    $\op{Hom}_{K^{p^{n+i}}}(K,K)\to \op{Hom}_{K^{p^{n+i}}}(W,K)$,
    by Proposition(Diff=End)~\ref{diff=end - new}. Now, let $\phi\in \op{Hom}_{K^{p^{n+i}}}(W,K)$. We have that $W\subset K$ is a $K^{p^{n+i}}$-linear subspace, so there exists a complementary vector space $V$ such that $W\oplus V= K$. Let $\psi\in \op{Hom}_{K^{p^{n+i}}}(K,K)$ be a linear operator that is equal $\phi$ on $W$ and $0$ on $V$, then the restriction  of $\psi$ to $W$ is equal $\phi$. Therefore, this map is surjective.

    Finally, the map $d(K/W)$ is surjective, because we have an equality 
    \[
    \bigcup_{i\ge0}\op{Diff}_{K^{p^{n+i}}}(W)\otimes K=\op{Diff}_{k}(W)\otimes K.
    \]
    This finishes the proof of the surjectivity for purely inseparable extensions.

    \textbf{Part 1, Step 3: Surjectivity In General} 
    Let $K/W/k$ be an arbitrary field extension. We prove that $d(K/W)$ is surjective.
    
    We can decompose the extension $K/W$ into two extensions $K/W'/W$, where the extension $K/W'$ is purely inseparable and $W'/W$ is separable. Indeed, let  $x_1,x_2,\ldots,x_n$ be a transcendental basis for $K/W$. (This basis is finite, because $K/W$ is finitely generated.) This means that the extension
    $K/W(x_1,\ldots, x_n)$ is finite and therefore algebraic.
    Let $W'$ be the separable algebraic closure of $W(x_1,\ldots, x_n)$ in $K$. Then, the extension $K/W'$ is purely inseparable, and the extension $W'/W$ is separable. (Caveat: this decomposition is not canonical if $n\ne0$.)
    
    Finally, we can conclude that the differential $d(K/W)$ is a composition of surjections $d(K/W')$ and $d(W'/W)\otimes K$:
    \[
\op{Diff}_k(K)\to \op{Diff}_k(W')\otimes K \to \left(\op{Diff}_k(W) \otimes W'\right) \otimes K=\op{Diff}_k(W)\otimes K.
    \]
    Therefore, it is surjective. 
    
    This finishes the proof of part 1.

    \textbf{Part 2: Kernel} Let $\mathbb{K}$ be the kernel of $d(K/W)$. We prove that it is equal to $\op{Diff}_k(K) \cdot \E(K/W)$ in two steps. First, we will show that the space $\op{Diff}_k(K) \cdot \E(K/W)$ is inside the kernel. Second, we will show that it is the whole kernel by a dimension counting argument.

    \textbf{Part 2, Step 1:} We prove $\op{Diff}_k(K) \cdot \E(K/W)\subset \mathbb{K}$. 
    
    We recall that the differential $d(K/W)$ is a restriction from $K$ to $W$, Lemma \ref{Differential is Restriction}. This means that $\mathbb{K}$ is the subspace of $\op{Diff}_k(K)$ consisting of all operators that are equal to zero on $W$.
    
    First, we prove that the space $\E(K/W)$ is in $\mathbb{K}$. Indeed, let $D\in \E(K/W)$ and $w\in W$, then $D(w)=wD(1)=0$. 
    
    Next, we prove that the kernel $\mathbb{K}$ is a left ideal. Indeed, let $D\in\mathbb{K}$, i.e., $D(W)=0$, then for every $D'\in\op{Diff}_k(K) $ we have $D'\circ D (W)=D'(D(W))=D'(0)=0$. 
    
    Finally, we can conclude that the left ideal generated by $\E(K/W)$, i.e., the space $\op{Diff}_k(K) \cdot \E(K/W)$, is in the kernel $\mathbb{K}$. 
    
    This proves the inclusion.

    \textbf{Part 2, Step 2:} We prove that the inclusion $\op{Diff}_k(K) \cdot \E(K/W)\subset \mathbb{K}$ is an equality.
    The proof goes by comparing $p$-filtrations of these spaces, and using distinguished generators to count dimensions.
    
    The $p$-filtration of the kernel $\mathbb{K}$ are the subspaces 
    \[
    \mathbb{K}_n\coloneqq \mathbb{K} \cap \op{Diff}_{K^{p^n}}(K).
    \]
    And, the $p$-filtration of $\op{Diff}_k(K) \cdot \E(K/W)$ are the subspaces 
    \[
    \left(\op{Diff}_k(K) \cdot \E(K/W)\right)\cap \op{Diff}_{K^{p^n}}(K).
    \] 
    However, we will focus on its subfiltration
    \[
    \op{Diff}_{K^{p^n}}(K)\cdot \E(K/W_n) \subset \op{Diff}_k(K) \cdot \E(K/W)\cap \op{Diff}_{K^{p^n}}(K).
    \]
    where $W_n=W\cdot K^{p^n}$ is $W$ up to $p^n$-powers. It is a filtration:
    \[
    \bigcup_{n\ge 0}\op{Diff}_{K^{p^n}}(K)\cdot \E(K/W_n)=\op{Diff}_k(K) \cdot \E(K/W).
    \]

    We observe that 
    \[
    \op{Diff}_{K^{p^n}}(K)\cdot \E(K/W_n)\subset \mathbb{K}_n= \op{ker}(d(K/W_n))\cap \op{Diff}_{K^{p^n}}(K) \subset \op{Diff}_{K^{p^n}}(K)
    \]
    because $\mathbb{K}$ is the set of operators that are zero on $W$, so $ \mathbb{K}_n$ is the set of operators that are zero on $W$ and $K^{p^n}$, i.e., on $W_n$.

    For convenience, we include a diagram:
    \begin{center}
        \begin{tikzcd}
            \{D : D(W)=0\}\ar[r] & \op{End}_k(K)\ar[r,two heads] & \op{Hom}_k(W,K)\\
            \op{ker}(d(K/W))\ar[u]\ar[r] & \op{Diff}_k(K)\ar[u] \ar[r,two heads]&  \op{Diff}_k(W) \otimes K\ar[u]\\
            \mathbb{K}_n\ar[u]\ar[r]&\op{Diff}_{K^{p^n}}(K) \ar[ru]\ar[r,two heads]\ar[u]&\op{Diff}_{K^{p^n}}(W_n) \otimes K\ar[u].
        \end{tikzcd}
    \end{center}

    We prove that, for every $n\ge 0$, we have 
    $\op{Diff}_{K^{p^n}}(K)\cdot \E(K/W_n)=\mathbb{K}_n$.

    First, we make some preparations. Let $[K:K^p]=p^r$, $[K:W_n]=p^N$, and $[W_n:K^{p^n}]=p^M$. Then, by Proposition(Diff=End)~\ref{diff=end - new} and elementary calculations, we have the following equalities
    \begin{align*}
       \op{dim}_K (\op{Diff}_{K^{p^n}}(K))=p^{N+M}=p^{rn},& \op{dim}_K (\op{Diff}_{K^{p^n}}(W_n) \otimes K)=p^{M}, \\
       \op{dim}_K(\op{Diff}_{W_n}(K))=p^N,& \op{dim}_K \E(K/W_n)= p^N-1.
    \end{align*}

    Now, we compute the dimension of $\mathbb{K}_n$. It is
    \[
    \op{dim}_K (\op{Diff}_{K^{p^n}}(K))-\op{dim}_K (\op{Diff}_{K^{p^n}}(W_n) \otimes K)=p^{N+M}-p^{M}.
    \]

    Next, we compute the dimension of $\op{Diff}_{K^{p^n}}(K)\cdot \E(K/W_n)$. 
    
    Let $G=\{G^m_1,\ldots G^m_{r_m},\}_{m=1,2,\ldots,n}$ be a set of distinguished generators for $\op{Diff}_{W_n}(K)$ and let $\{E^m_1,\ldots, E^m_{r-r_m}\}_{m=1,2,\ldots}$ be a completion of $G$, Definitions \ref{distinguished generators - def} and \ref{completion of distinguished generators - def}. Now, by Proposition \ref{description of any subalgebra with dist generators - prop}, we can identify elements of $\op{Diff}_{K^{p^n}}(K)\cdot \E(K/W_n)$ as the ones that are a $K$-linear sum of operators 
    \[
    \left(\prod_{m=1}^n \prod^{r-r_m}_{j=1} \left(E^m_j\right)^{a^m_j}\right)\left(\prod_{m=1}^n\prod^{r_m}_{i=1} \left(G^m_i\right)^{b^m_i}\right),
    \]
    where $0 \le a^m_i <p-1$ and $0 \le b^m_i <p-1$, but $\sum_{i, m}b^m_i>0$. Moreover, these operators are linearly independent, forming a basis of this subspace.
    
    We count them. First, we fix the part $\prod_{m=1}^n\prod^{r_m}_{i=1} \left(G^m_i\right)^{b^m_i}$. There are exactly $p^{\sum_{i=1}^n r-r_m}=p^{rn-\sum_{i=1}^n r_m}=p^{(N+M) -N}=p^M$ monomials with this part fixed. 
    
    Now, since the monomials $\prod_{m=1}^n\prod^{r_m}_{i=1} \left(G^m_i\right)^{b^m_i}$ are a basis of $\E(K/W_n)$, there are $p^N-1$ of them. Therefore, the dimension of $\op{Diff}_{K^{p^n}}(K)\cdot \E(K/W_n)$ is equal $p^M(p^N-1)$.

    The numbers are equal: $p^{N+M}-p^{M}=p^M(p^N-1)$. Consequently, the inclusion
    $\op{Diff}_{K^{p^n}}(K)\cdot \E(K/W_n)\subset\mathbb{K}_n$ is an equality.

    Finally, we have
    \[
    \mathbb{K}=\bigcup_{n\ge0}\mathbb{K}_n=\bigcup_{n\ge0}\op{Diff}_{K^{p^n}}(K)\cdot \E(K/W_n)=\op{Diff}_k(K) \cdot \E(K/W).
    \]
    This finishes the proof.
\end{proof}

\begin{remark}[Left Ideals, $D$-modules]
    We can observe that in Theorem \ref{SES for diff - fields - theorem} the $K$-vector space $\op{Diff}_k(K) \cdot \E(K/W)$ is only a left ideal of $\op{Diff}_k(K)$, but not a two-sided ideal. In particular, the space $\op{Diff}_k(W) \otimes K$ is not an algebra. In fact, we could easily prove that the algebra $\op{Diff}_k(K)$ is simple, i.e., it does not admit any non-trivial two-sided ideals. Indeed, the essence of a proof is that for every nonzero differential operator $D$ of order at least $1$ there exists a sequence of elements $x_1,\ldots,x_n \in K$ such that $0\ne [[\ldots[D,x_1],x_2],\ldots],x_n] \in K$. This simple lemma could be proved by writing $D$ in terms of some coordinates and then choosing elements $x_i$ from the coordinates based on the leading form of $[D]$. From this lemma, we can conclude that every non-zero two-sided ideal of $\op{Diff}_k(K)$ contains $1$, so it is the whole algebra.

    Moreover, the algebra $\op{Diff}_W(K)$ is equivalent to $\op{Diff}_k(K) \cdot \E(K/W)$. Consequently, power towers are equivalent to some left $D$-modules, where we put $D=\op{Diff}_k(K)$. Any left $D$-module can be considered a differential equation, thus power towers can be considered solutions to some differential equations.
\end{remark}

\subsection{Subfields and Power Towers}\label{section four - fields}
We determine which subfields inject into power towers.

\begin{defin}\label{Power Tower of a Subfield - definition}
Let $K$ be a field of characteristic $p>0$.
Let $W$ be a subfield of $K$. 
We define \emph{\textbf{the power tower of $W$ on $K$}} to be the power tower $W_\bullet$ defined by $W_i\coloneqq W\cdot K^{p^i}$ for $i\ge 0$.

Moreover, we say that $W_n$ is $W$ \emph{\textbf{up to $p^n$-powers}} of $K$.
\end{defin}

\begin{lemma}
    Let $K$ be a field of characteristic $p>0$.
Let $W$ be a subfield of $K$. Then the power tower of $W$ on $K$ is a power tower on $K$.
\end{lemma}

\begin{proof}
We have to show that $(W\cdot K^{p^n})\cdot K^{p^m}=W\cdot K^{p^m}$ for $0\le m\le n$. 
This follows from the fact that taking the composite of subfields is associative. Indeed, we have 
\[
(W \cdot K^{p^n})\cdot K^{p^m}=W\cdot (K^{p^n}\cdot K^{p^m})=W\cdot K^{p^m}.
\]
This finishes the proof.
\end{proof}

Here are some examples.

\begin{example}
    The constant tower on $K$, Example \ref{Constant Power Tower}, is the power tower of $K$ on $K$. It corresponds to the subalgebra $K\subset \op{Diff}_k(K)$, and it corresponds to the Jacobson sequence $(0,0,\ldots)$.
\end{example}

\begin{example}
    The tower of $p$-powers of $K$, Example \ref{p-power power tower}, is the power tower of $k$ on $K$.
    It corresponds to the subalgebra $\op{Diff}_k(K)\subset \op{Diff}_k(K)$, and it corresponds to the Jacobson sequence $(T_{K/K^p}, T_{K^p/K^{p^2}},T_{K^{p^2}/K^{p^3}},\ldots)$.
\end{example}
    
\begin{example}
    Finite length power towers are the power towers of purely inseparable subfields on $K$, Lemma \ref{Finite Length = finite exponent - lemma}.

    In particular, the extension $K/K^{p^n}$ corresponds to the subalgebra $\op{Diff}_{K^{p^n}}(K)\subset \op{Diff}_k(K)$, and to the Jacobson sequence $(T_{K/K^p}, T_{K^p/K^{p^2}},\ldots,T_{K^{p^{n-1}}/K^{p^n}},0,0,\ldots)$.
\end{example}

\begin{example}
    Let $k$ be a perfect field of characteristic $p>0$.
Let $K=k(x,y)$, and $W=k(x)$. Then the subfield $W_n=k(x,y^{p^n})$ is $W$ up to $p^n$-powers of $K$.

The power tower of $k(x)$ on $ k(x,y)$ corresponds to the subalgebra $K$-spanned by $\frac{1}{a!}\frac{\partial^{a}}{\partial y^{a}}$ for $a\ge 0$.
And, it corresponds to the Jacobson sequence 
\[
(k(x,y)\frac{\partial}{\partial y},k(x,y^p)\frac{\partial}{\partial y^p},k(x,y^{p^2})\frac{\partial}{\partial y^{p^2}},\ldots).
\]
\end{example}

The next definition introduces a name for power towers that are power towers of subfields: algebraically integrable ones. This name is inspired by the theory of foliations, where foliations corresponding to rational fibrations are called algebraically integrable. A foliation corresponds to a rational fibration if its generic leaf is an algebraic variety.

\begin{defin}\label{algebraic integrability - power tower - def}
Let $K$ be a field of characteristic $p>0$. Let  $V_\bullet$ be a power tower on $K$.
We say that the tower $V_\bullet$ is \textbf{\emph{algebraically integrable}} if there is a subfield $W$ of $K$ such that $V_\bullet=W_\bullet$, where $W_\bullet$ is the power tower of $W$ on $K$.
\end{defin}

In the theory of foliations, if we have an algebraically integrable foliation, then we can recover the rational fibration corresponding to it by computing its first integrals. This inspires the name in the following definition.

\begin{defin}\label{First Integrals Power Towers - Definition}
Let $K$ be a field of characteristic $p>0$.
Let $W_\bullet$ be a power tower on $K$. We define a field 
\[
W_\infty\coloneqq \bigcap_{n\ge 0} W_n.
\] 
We call it the field of \textbf{\emph{first integrals}} of the power tower $W_\bullet$.
\end{defin}

The following is the main theorem of this section.

\begin{thm}\label{infty=s}
Let $K$ be a field of characteristic $p>0$. Let $k=K^{p^\infty}$ be the perfection of $K$. We assume that $K/k$ is a finitely generated field extension.
Let $W$ be a subfield of $K$. 

Then, the field of first integrals of the power tower of $W$ on $K$ is the separable closure of $k\cdot W$ in $K$, i.e., $W_\infty=\left(k\cdot W\right)^s$. 

Moreover, the power tower of $W$ on $K$ and the power tower of $W_\infty$ on $K$ are equal.
\end{thm}

\begin{proof}
First, we observe that $W_\bullet=\left(W\cdot k\right)_\bullet$. Therefore, without loss of generality, we can assume that $k\subset W$.

Now, we prove that $W\subset W^s \subset W_\infty$, and their power towers on $K$ are equal. Indeed, for every $n\ge 0$, we have 
\[
K \supset \left(W^s\right)_n\supset W_n \supset K^{p^n}.
\] 
Thus $\left(W^s\right)_n\supset W_n$ is purely inseparable, but it is also separable. Therefore, we have $\left(W^s\right)_n= W_n$. Hence $W^s\subset W_\infty$, because $W^s\subset W_n$ for every $n\ge 0$.

Finally, power towers of $W^s$ and $W_\infty$ on $K$ are equal, because 
\[
W^s \subset W_\infty \subset W_n
\]
implies 
\[
W_n=W^s\cdot K^{p^n} \subset W_\infty\cdot K^{p^n} \subset W_n.
\] 
This finishes the proof of the power towers being equal.

Moreover, by the same argument as above, we have that $W_\infty^s=W_\infty$.

Next, we prove that if $W^s \ne W_\infty$, then $T_{W_\infty/W^s}\ne 0$.
Indeed, we know that we have an inclusion $W^s\subset W_\infty$ that is a finitely generated field extension, and the fields $W^s, W_\infty$ are separably closed in $K$. Now, if $W^s \ne W_\infty$, then we can find a finite sequence of elements $x_i\in W_\infty$, where $i=1,\ldots, r$, such that
\[
L_0\coloneqq W\subset L_1\coloneqq W(x_1)^s \subset L_2\coloneqq L_1(x_2)^s \subset ... \subset L_r\coloneqq L_{r-1}(x_r)^s=W_\infty
\]
is a sequence of proper extensions, and $x_i$ is purely inseparable, or purely transcendental over $L_{i-1}$. A simple calculation shows that $\Omega_{L_r/L_{r-1}}$ has dimension one, so it is non-zero. Furthermore, $\Omega_{L_r/L_0}$ surjects onto it, so it is non-zero too. Therefore, we have 
\[
T_{W_\infty/W^s}=\op{Hom}_{W_\infty}(\Omega_{L_r/L_0},W_\infty)\ne 0.
\]

Finally, we can prove that the assumption $W^s \ne W_\infty$ leads to a contradiction. Indeed, we have the following short exact sequences, Theorem \ref{SES for diff - fields - theorem}:
\begin{align*}
    0\to \op{Diff}_k(K) \cdot \E(K/W^s)\hookrightarrow &\op{Diff}_k(K) \xrightarrow{d(K/W^s)}  \op{Diff}_k(W^s) \otimes K\to 0,\\
    0\to \op{Diff}_k(K) \cdot \E(K/W_\infty)\hookrightarrow &\op{Diff}_k(K) \xrightarrow{d(K/W_\infty)}  \op{Diff}_k(W_\infty) \otimes K\to 0.\\
\end{align*}

However, every differential is a restriction of linear operators, Lemma \ref{Differential is Restriction}, so $W^s$-linear operators are mapped to $W^s$-linear operators; therefore, we get the following short exact sequences:
\begin{align*}
    0\to \left(\op{Diff}_k(K) \cdot \E(K/W^s)\right)_{W^s}\hookrightarrow &\op{Diff}_{W^s}(K) \xrightarrow{d(K/W^s)}  \op{Diff}_{W^s}(W^s) \otimes K\to 0,\\
    0\to \left(\op{Diff}_k(K) \cdot \E(K/W_\infty)\right)_{W^s}\hookrightarrow &\op{Diff}_{W^s}(K) \xrightarrow{d(K/W_\infty)}  \op{Diff}_{W^s}(W_\infty) \otimes K\to 0,\\
\end{align*}
where $\left(\op{Diff}_k(K) \cdot \E(K/W^s)\right)_{W^s}\coloneqq \op{Diff}_k(K) \cdot \E(K/W^s)\cap \op{Diff}_{W^s}(K)$ and also $\left(\op{Diff}_k(K) \cdot \E(K/W_\infty)\right)_{W^s}\coloneqq \op{Diff}_k(K) \cdot \E(K/W_\infty)\cap \op{Diff}_{W^s}(K)$.

Now, we know that power towers of $W^s$ and $W_\infty$ on $K$ are equal, therefore $\op{Diff}_{W^s}(K)=\op{Diff}_{W_\infty}(K)$ and $\E(K/W^s)=\E(K/W_\infty)$. Consequently, we have an equality of left ideals:
\[
\left(\op{Diff}_k(K) \cdot \E(K/W_\infty)\right)_{W^s}=\left(\op{Diff}_k(K) \cdot \E(K/W^s)\right)_{W^s}.
\]
This implies that 
\[
\op{Diff}_{W^s}(W_\infty) \otimes K \xrightarrow{d(W_\infty/W^s)\otimes K}\op{Diff}_{W^s}(W^s) \otimes K
\]
is an isomorphism. And this leads to a contradiction:
\[
0\ne T_{W_\infty/W^s}\otimes K \subset \op{Diff}_{W^s}(W_\infty) \otimes K \simeq\op{Diff}_{W^s}(W^s) \otimes K = W^s \otimes_{W^s} K = K,
\]
i.e., first-order operators are in the zeroth order operators. This cannot be true.

This finishes the proof of the theorem.    
\end{proof}

The next result determines all subfields determined by their approximations up to $p^n$-powers.

\begin{cor}[Subfields Into Power Towers]\label{infty=s refined - Corollary}
    Let $K$ be a field of characteristic $p>0$. Let $k=K^{p^\infty}$ be the perfection of $K$. We assume that $K/k$ is a finitely generated field extension.

    Then, there is a bijection between subfields $W\subset K$ satisfying the condition $k\subset W=W^s\subset K$ and algebraically integrable power towers on $K$. 
    
    Explicitly, this bijection is given by taking the power tower of a subfield $W$, i.e., $W\mapsto W_\bullet$. And, the inverse of this map is taking the field of first integrals of an algebraically integrable power tower, i.e,. $W_\bullet\mapsto W_\infty$.
\end{cor}

\begin{proof}
    Let $W$ be a subfield such that $k\subset W=W^s\subset K$. Let $W_\bullet$ be the power tower of $W$ on $K$. Then, by Theorem \ref{infty=s}, we have $W_\infty=\left(W\cdot k\right)^s=W^s=W$. 
    
    Let $W_\bullet$ be an algebraically integrable power tower on $K$. Let $V\subset K$ be a subfield such that $V_\bullet=W_\bullet$. Then, $W_\infty=\left(V\cdot k\right)^s$ by Theorem \ref{infty=s}. Therefore, we have $k\subset W_\infty=W_\infty^s\subset K$ and $\left(W_\infty\right)_\bullet=W_\bullet$.
    
    This finishes the proof.
\end{proof}

Non-algebraically integrable power towers exist, and it is generally complicated to know whether or not a given power tower is algebraically integrable. This resembles a situation from characteristic zero, where whether a given foliation on a variety is algebraically integrable is a nontrivial question. Furthermore, there are ``more'' non-algebraically integrable power towers than algebraically integrable ones.

\begin{example}[Not Algebraically Integrable Power Tower]\label{Not Algebraicallt Integrable power tower - example}
Let $k$ be a perfect field of characteristic $p>0$.
Let $K=k(x,y)$. We define a power tower on $K$ by
\[
W_n\coloneqq k(x+y^p+\ldots + y^{p^{n-1}},y^{p^{n}}).
\]
The power tower $W_\bullet$ is not algebraically integrable. 

Indeed, if it were a power tower of a subfield $W\subset K$, then we would have $W\subset W\cdot k\subset \bigcap_{n\ge 0} W_n=k$, and, consequently, $W\cdot k=k$. But then, we would have that $W_n=W\cdot K^{p^n}=W\cdot k\cdot K^{p^n}=k\cdot K^{p^n}=K^{p^n}$. This is a contradiction.
\end{example}

We finish this section with two remarks connecting algebraically integrable power towers with the topological Jacobson--Bourbaki correspondence \ref{Topological Jacobson--Bourbaki Correspondence}.

\begin{remark}\label{closures}
Let $K$ be a field of characteristic $p>0$. Let $k=K^{p^\infty}$ be the perfection of $K$. We assume that $K/k$ is a finitely generated field extension.

Let $\mathcal{D}\subset \op{Diff}_k(K)$ be a $K$-subalgebra. Let $W_\infty$ be the first integrals of the power tower $W_\bullet$ corresponding to $\mathcal{D}$.

We can consider $\mathcal{D}$ as a subalgebra of $\op{End}_{\mathbb{Z}}(K)$. Let $\overline{\mathcal{D}}$ be the closure of $\mathcal{D}$ in $\op{End}_\mathbb{Z}(K)$ in the finite topology, Definition \ref{Finite Topology- Definition}. This closure is of the form $\op{End}_W(K)$, where $W$ is a subfield of $K$, by Topological Jacobson--Bourbaki Correspondence \ref{Topological Jacobson--Bourbaki Correspondence}. 
We can prove that $W=W_\infty$. Indeed, $\op{const}(\mathcal{D})=W_\infty$, so $\op{End}_{W_\infty}(K)\supset\mathcal{D}$. Hence, $\op{End}_W(K)\subset \op{End}_{W_\infty}(K)$, so $W_\infty\subset W$. But, $\op{End}_{W}(K)\supset\mathcal{D}$, so $\op{const}(\op{End}_{W}(K))=W\subset \op{const}(\mathcal{D})=W_\infty$. 

However, $\mathcal{D}$ is determined by $W_\infty$ if and only if $W_\bullet$ is algebraically integrable.
\end{remark}

\begin{remark}
One could say that there are more not algebraically integrable power towers than algebraically integrable ones. Indeed, let $k$ be a countable field and let $K/k$ be a finitely generated field extension with $k$ being the perfection of $K$, then if $K/k$ has transcendental degree at least two, there are only countably many subfields between $k$ and $K$. Still, there are uncountable many power towers on $K$. Therefore, in most cases, for such $k$, a closure of a subalgebra of differential operators (in the finite topology) says little about this subalgebra because the corresponding power tower is most likely not algebraically integrable.
\end{remark}

\subsection{$n$-Foliations, $n=1,2,\ldots,\infty$}\label{section five -fields}
We develop a big class of examples of power towers.

\subsubsection{Basic Theory}
The following definition shall be compared with Proposition \ref{Nonincreasing Degrees Power Tower - Proposition}.

\begin{defin}\label{n-foliation}
Let $K$ be a field of characteristic $p>0$. Let $k=K^{p^\infty}$ be the perfection of $K$. Let $K/k$ be a finitely generated field extension.

Let $0\le n \le \infty$.
Let $W_\bullet$ be a power tower on $K$.  

The power tower $W_\bullet$ is a \textbf{\emph{$n$-foliation}} if it is of length $n$, and it satisfies 
\[
\op{dim}_{W_{i+1}}(W_{i})=\op{dim}_{W_{j+1}}(W_j)
\]
for $0\le i,j< n$. 
    
The rank of a $n$-foliation $W_\bullet$ is  $\op{log}_p(\op{dim}_{W_{1}}(K))$.
\end{defin}

In other words, a power tower $W_\bullet$ is an $n$-foliation if and only if it is a power tower such that
\[
K=W_0\supset W_1 \supset W_2 \supset \ldots\supset W_n = W_{n+1}=\ldots,
\]
\[
\op{[W_0:W_1]=[W_1:W_2]=\ldots=[W_{n-1}:W_n]},
\]
where $[K:W]=\op{dim}_W(K)$.

\begin{remark}
    The name ``$n$-foliation'' is an extension of the name ``$1$-foliation'' that in the literature refers to saturated $p$-Lie algebras on varieties, \cite{Langer1}. I believe that the origin of this name is just a short way to refer to ``foliations of height $n$'' for $n=1$, which were introduced by Ekedahl in \cite{EkedahlFoliation1987}. His definition applied to the space $\op{Spec}(K)$ is equivalent to the definition of power towers on $K$ we just gave. However, the original definition is formulated for varieties, so we will discuss it later, Proposition \ref{gr-saturared + regular = pd envelope - lemma}.
\end{remark}

The following might be considered the motivational example for the notion of an $n$-foliation. 

\begin{example}\label{transcedental => infty-foli}
Let $k$ be a perfect field.
Let $K\supset W\supset k$ be both purely transcendental extensions. 
Then, the power tower of $W$ on $K$ is an $\infty$-foliation. Indeed, we can write $W=k(x_1,\ldots,x_r)$ and $K=W(y_1,\ldots y_m)$, where $x_i$ and $y_i$ are transcendental bases for $W/k$ and $K/W$ respectively. 
Then 
\[
W_n=W(y_1^{p^n},\ldots y_m^{p^n})=k(x_1,\ldots x_r,y_1^{p^n},\ldots y_m^{p^n}).
\]
Therefore, $\op{dim}_{W_n}(W_{n-1})=p^m$ is constant. Consequently, $W_\bullet$ is a $\infty$-foliation on $K$ of rank $m$, and the power tower $(W_n)_\bullet$ is a $n$-foliation of rank $m$.
\end{example}

The visualization of an $n$-foliation begins with the following result.

\begin{prop}\label{dist generators for n-foli - prop}
Let $K$ be a field of characteristic $p>0$. Let $k=K^{p^\infty}$ be the perfection of $K$. Let $K/k$ be a finitely generated field extension.

Let $0\le n \le \infty$. Let $W_\bullet$ be a $n$-foliation on $K$ of rank $r$. Let $\mathcal{D}\subset \op{Diff}_k(K)$ be a subalgebra corresponding to $W_\bullet$. Let $\F_1 = \mathcal{D} \cap T_{K/K^p}$.

Then, there exists a set of well-chosen distinguished generators for $\mathcal{D}$:
\[
G=\{G^m_1,\ldots , G^m_{r}\}_{m=1,2,\ldots,n}
\]
such that
\[
[G^m_i] = \gamma_{p^{m-1}}([G^1_i])
\]
for $i=1,2,\ldots,r$ and $1\le m< n+1$. Consequently, the graded subalgebra satisfies
\[
\op{gr}\mathcal{D} = {S}(\F_1^*)^{*gr} \cap\op{gr}\op{Diff}_{K^{p^n}}(K),
\]
where ${S}(\F_1^*)^{*gr}$ is the divided power polynomial algebra in variables $\F_1$.
\end{prop}

\begin{proof}
    There is almost nothing to be proved, but there is still something.  Proposition \ref{dist generators exist - prop} gives us a set of well-chosen distinguished operators 
    \[
    G=\{G^m_1,\ldots , G^m_{r_m}\}_{m=1,2,\ldots}
    \]
    for $\mathcal{D}$. Their leading forms $[G^m_i]$ are linear combinations of $\gamma_{p^m}(G_j^1)$, now we want equalities, not combinations. We will do some corrections.
    
    We recall and add. The numbers $r_m$ are defined to satisfy the relation $p^{r_m}=[W_{m+1}:W_m]$, so they are equal $r$ for $m<n$ and $0$ for $m\ge n$. This means that the effective range of the index $m$ is $\{1,2,\ldots,n\}$.

    By Definition \ref{distinguished generators - def}, the leading forms $[G^m_i]$ for $i=1,2,\ldots, r$ are linearly independent inside $\gamma_{p^{m-1}}(\F_1)$. This space has dimension $r$, therefore the set $[G^m_i]$ for $i=1,2,\ldots, r$ is a basis. Consequently, we can find a $K$-linear combination of these leading forms that satisfies
    \[
    \sum_{j=1}^r a_j [G^m_j]= \gamma_{p^{m-1}}([G^1_i]),
    \]
    where $a_j \in K$. So, we can substitute $G^m_i$ with $\sum_{j=1}^r a_j G^m_j$ (a linear base change). However, it is not obvious that these new operators are well-chosen. This is true, but it requires an argument. Indeed, by simple linear algebra, we see that the derivation $d(K/W_{m-1})(\sum_{j=1}^r a_j G^m_j)$ is the unique one in $\F_{m-1}\otimes K$ to be mapped to $\op{gr}d(K/K^{m-1})[\gamma_{p^{m-1}}([G^1_i])]$, see Lemma \ref{Graded Differential for K/K^m - lemma}, but that derivation already belongs to $\F_{m-1}$ implying that $a_j\in W_{m-1}$ for $j=1,2,\ldots,r$, thus the operators $\sum_{j=1}^r a_j G^m_j$ are well-chosen. 

    Finally, by Proposition \ref{pd envelope - prop}, we have a factorization
    \[
    \op{gr} \mathcal{D} \to {S}(\F_1^*)^{*gr} \to \op{gr}\op{Diff}_{k}(K),
    \]
    and ${S}(\F_1^*)^{*gr}$ is a divided power polynomial algebra. So, the leading forms of elements from $G$ generate exactly ${S}(\F_1^*)^{*gr}\cap\op{gr}\op{Diff}_{K^{p^n}}(K)$, thus the graded algebra $\op{gr} \mathcal{D}$ is equal to this intersection. This finishes the proof. 
\end{proof}

\begin{cor}\label{graded subalgebra of n-foli - only W1 and n - corollary}
    Let $K$ be a field of characteristic $p>0$. Let $k=K^{p^\infty}$ be the perfection of $K$. Let $K/k$ be a finitely generated field extension.

    The graded subalgebra of an $n$-foliation on $K$ depends only on $W_1$ and $n$.
\end{cor}

\begin{proof}
    It directly follows from the last part of Proposition \ref{dist generators for n-foli - prop}. Indeed, let $W_\bullet$ be a $n$-foliation, then
    $\op{Diff}_{W_\bullet}(K)= {S}(\F_1^*)^{*gr} \cap\op{gr}\op{Diff}_{K^{p^n}}(K)$, and $\F_1$ depends only on $W_1$, and it is of height $n$, so the $n$ is minimal such in that formula. 
\end{proof}

\begin{remark}
    Let $W_\bullet$ be an $n$-foliation corresponding to a subalgebra $\mathcal{D}$. 
    Let $G$ be a set of distinguished generators for $\mathcal{D}$ from Proposition \ref{dist generators for n-foli - prop}.
    
    From the elementary fact that $\gamma_p^{\circ m}$ is equal to $\gamma_{p^m}$ up to an integer that is invertible modulo $p$, we see that $\gamma_{p}([G^m_j])$ for $j=1,2,\ldots,r$ is a basis for $\gamma_{p^{m+1}}(\F_1)$ for $m<n$.
\end{remark}

\begin{remark}\label{original def n-foli - fields - remark}
    The Ekedahl's definition of a ``foliation of height $n$'' applied to a field $K$ is \cite[Definition 3.1]{EkedahlFoliation1987}. We repeat it: 
    \begin{center}
    A subalgebra $\mathcal{D}\subset \op{Diff}_k(K)$ is a foliation of height $n$ 
    
    if it satisfies 
    the equality $\op{gr}\mathcal{D} = {S}(\F_1^*)^{*gr} \cap\op{gr}\op{Diff}_{K^{p^n}}(K)$.    
    \end{center}
    From this equality for the graded algebra, we can soundly read out that the power tower corresponding to $\mathcal{D}$ is of length $n$ and is of constant degree till its length, i.e., it is an $n$-foliation. And, the converse is true by Proposition \ref{dist generators for n-foli - prop}. Therefore, our definition and Ekedahl's are equivalent to each other on fields. 
\end{remark}

\FloatBarrier
\begin{remark}
    We can visualize Proposition \ref{dist generators for n-foli - prop} for an $\infty$-foliation in an analogous way to Figure \ref{fig:stairs} in the figure below. (The figure was created using \nolinkurl{https://www.mathcha.io}.) The visualization of a $n$-foliation is the truncation of the one for $\infty$-foliation at $\gamma_{p^{n-1}}(\F_1)$.
    \begin{figure}
        \centering
        {

\tikzset{every picture/.style={line width=0.75pt}} 

\begin{tikzpicture}[x=0.75pt,y=0.75pt,yscale=-0.7,xscale=0.7]

\draw   (50,50) -- (250,50) -- (250,300) -- (50,300) -- cycle ;
\draw   (350,50) -- (550,50) -- (550,300) -- (350,300) -- cycle ;
\draw   (275,175) -- (305,175) -- (305,170) -- (325,180) -- (305,190) -- (305,185) -- (275,185) -- cycle ;
\draw   (355,225) .. controls (355,222.24) and (357.24,220) .. (360,220) -- (540,220) .. controls (542.76,220) and (545,222.24) .. (545,225) -- (545,240) .. controls (545,242.76) and (542.76,245) .. (540,245) -- (360,245) .. controls (357.24,245) and (355,242.76) .. (355,240) -- cycle ;
\draw   (355,275) .. controls (355,272.24) and (357.24,270) .. (360,270) -- (540,270) .. controls (542.76,270) and (545,272.24) .. (545,275) -- (545,290) .. controls (545,292.76) and (542.76,295) .. (540,295) -- (360,295) .. controls (357.24,295) and (355,292.76) .. (355,290) -- cycle ;
\draw   (355,175) .. controls (355,172.24) and (357.24,170) .. (360,170) -- (540,170) .. controls (542.76,170) and (545,172.24) .. (545,175) -- (545,190) .. controls (545,192.76) and (542.76,195) .. (540,195) -- (360,195) .. controls (357.24,195) and (355,192.76) .. (355,190) -- cycle ;
\draw   (355,125) .. controls (355,122.24) and (357.24,120) .. (360,120) -- (540,120) .. controls (542.76,120) and (545,122.24) .. (545,125) -- (545,140) .. controls (545,142.76) and (542.76,145) .. (540,145) -- (360,145) .. controls (357.24,145) and (355,142.76) .. (355,140) -- cycle ;
\draw    (490,98.5) -- (490,59.5) ;
\draw [shift={(490,57.5)}, rotate = 90] [color={rgb, 255:red, 0; green, 0; blue, 0 }  ][line width=0.75]    (10.93,-3.29) .. controls (6.95,-1.4) and (3.31,-0.3) .. (0,0) .. controls (3.31,0.3) and (6.95,1.4) .. (10.93,3.29)   ;
\draw   (55,275) .. controls (55,272.24) and (57.24,270) .. (60,270) -- (240,270) .. controls (242.76,270) and (245,272.24) .. (245,275) -- (245,290) .. controls (245,292.76) and (242.76,295) .. (240,295) -- (60,295) .. controls (57.24,295) and (55,292.76) .. (55,290) -- cycle ;
\draw    (190,98.5) -- (190,59.5) ;
\draw [shift={(190,57.5)}, rotate = 90] [color={rgb, 255:red, 0; green, 0; blue, 0 }  ][line width=0.75]    (10.93,-3.29) .. controls (6.95,-1.4) and (3.31,-0.3) .. (0,0) .. controls (3.31,0.3) and (6.95,1.4) .. (10.93,3.29)   ;

\draw (285,150) node [anchor=north west][inner sep=0.75pt]   [align=left] {gr};
\draw (416,75) node [anchor=north west][inner sep=0.75pt]   [align=left] {continues};
\draw (232,18.4) node [anchor=north west][inner sep=0.75pt]    {$D$};
\draw (517,18.4) node [anchor=north west][inner sep=0.75pt]    {$\text{gr} \ D$};
\draw (427,223.4) node [anchor=north west][inner sep=0.75pt]    {$\gamma _{p}( F_{1})$};
\draw (447,273.4) node [anchor=north west][inner sep=0.75pt]    {$F_{1}$};
\draw (427,173.4) node [anchor=north west][inner sep=0.75pt]    {$\gamma _{p^{2}}( F_{1})$};
\draw (427,123.4) node [anchor=north west][inner sep=0.75pt]    {$\gamma _{p^{3}}( F_{1})$};
\draw (127,273.4) node [anchor=north west][inner sep=0.75pt]    {$\ \ \ \ \ F_{1}$};
\draw (116,75) node [anchor=north west][inner sep=0.75pt]   [align=left] {continues};

\end{tikzpicture}
}
        \caption{A visualisation of a subalgebra corresponding to a $\infty$-foliation.}
        \label{fig:enter-label}
    \end{figure}
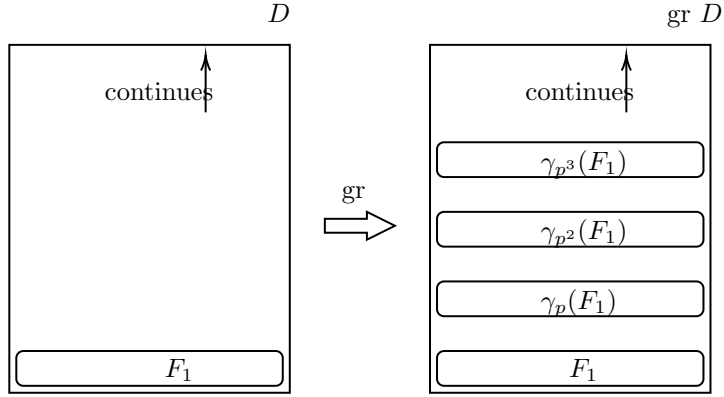
\end{remark}
\FloatBarrier

\subsubsection{Separable Field Extensions and $\infty$-Foliations}
A definition of a separable field extension can be found in Definition \ref{Separable: Definition}. Essentially, it says that such an extension admits a separable transcendental basis. It is algebraic if and only if the basis is empty.

\begin{prop}[Separable Subfields Into $\infty$-Foliations]\label{fibrations into infty-foliations - proposition}
Let $K$ be a field of characteristic $p>0$. Let $k=K^{p^\infty}$ be the perfection of $K$. Let $K/k$ be a finitely generated field extension.

Let $K/W/k$ be a subfield such that $K/W$ is a separable extension. 

Then, the power tower of $W$ on $K$ is a $\infty$-foliation on $K$.
\end{prop}

\begin{proof}
Without loss of generality, we can assume that $W=W^s$ inside $K$, because the power towers of $W$ and $W^s$ on $K$ are equal, Theorem \ref{infty=s}.

Let $y_1,\ldots,y_m$ be a $p$-basis of $W$ over $k$, see Definition \ref{p-basis definition}. Let $x_1,\ldots,x_r$ be a separating transcendence basis of $K$ over $W$. Then $y_1,\ldots,y_m,x_1,\ldots,x_r$ form a $p$-basis of $K$ over $k$.
In particular, we have 
\begin{align*}
    k(x_1,\ldots,x_r)^s&=W,\\
    W(y_1,\ldots,y_m)^s&=K,\\
    k(y_1,\ldots,y_m,x_1,\ldots,x_r)^s &=K,
\end{align*}
where the upper index $s$ means separable closure in $K$.

(To shorten the notation in the rest of the proof, we will write $(x_i)$ for $(x_1,\ldots,x_r)$, $(y_j)$ for $(y_1,\ldots,y_m)$, etc.)

Example \ref{transcedental => infty-foli} says that the power tower of $k(x_i)$ on $k(x_i,y_j)$ is a $\infty$-foliation on $k(x_i,y_j)$. The rest of the proof pulls back this information to the power tower of $W$ on $K$. 

Let $n\ge 0$ be an integer. 
We have that $k(x_i, y_j^{p^n})^s=W_n$, where the separable closure is taken in $K$. 
Moreover, we claim that
\[
    k(x_i, y_j^{p^{n}})\otimes_{k(x_i, y_j^{p^{n+1}})} W_{n+1} \simeq W_{n+1}(y_j^{p^n}) = W_{n}
\]
via natural maps to $K$. 
Of course, we have that $W_{n+1}(y_j^{p^n}) \subset W_{n}$.
We also have $W_{n+1}(y_j^{p^n})=W_{n+1}(y_j^{p^n})^s\supset k(x_i, y_j^{p^n})^s=W_n$. Therefore $ W_{n+1}(y_j^{p^n}) = W_{n}$.

Finally, we can conclude that the dimension of $ k(x_i, y_j^{p^{n}})$ over $ k(x_i, y_j^{p^{n+1}})$ is the same as the dimension of $W_n$ over $W_{n+1}$, because the later comes from the former by extending coefficients to $W_{n+1}$ and this operation is flat. This finishes the proof.
\end{proof}

\subsubsection{Examples of $\infty$-Foliations}
We compute more examples of $\infty$-foliations, their Jacobson sequences, and their subalgebras.

\begin{defin}\label{formal subfields - definition}
    Let $K=k(x,y)$ be the field of rational functions in two variables over a perfect field $k$ of characteristic $p>0$.

    Let $A_i\in k$ for $i\ge 1$. We define a power tower $W_\bullet$ of a ``formal subfield''
    \[
\overline{W_\infty}=k(x+A_1y^p+A_2y^{p^2}+A_3y^{p^3}+\ldots)
    \]
    by the following formulas, $n\ge 1$,
\[
W_n \coloneqq k(x+A_1 y^p + A_2 y^{p^2} + \ldots+ A_{n-1} y^{p^{n-1}},y^{p^n}).
\]
\end{defin}

\begin{question}[Algebraic Solutions to $D$-Modules from Power Towers]
    We do not define any ``formal subfields'' in the above definition. From the above definition, we treat the symbol $\overline{W_\infty}$ as a suggestive label for the power tower. It is not an actual notion, nor is it an actual object. It is just a name, a label.

    Nevertheless, I believe that if a good functor of solutions for algebraic $D$-modules over arbitrary bases exists, then these ``formal subfields'' could be something like ``formal first integrals on the generic point'' from that functor. It is pure speculation. 
    
    Does it make any sense? Are there actual formal subfields?
\end{question}

It is clear that the above power towers of ``formal subfields'' are power towers on $K$. Below, we present calculations of their Jacobson sequences and their subalgebras of differential operators.

\begin{example}\label{examples}
Let $K=k(x,y)$ be rational functions in two variables over a perfect field $k$ of characteristic $p>0$.

Let $W_\bullet$ be a power tower on $K$ of a ``formal subfield'', Definition \ref{formal subfields - definition},
    \[
\overline{W_\infty}=k(x+A_1y^p+A_2y^{p^2}+A_3y^{p^3}+\ldots).
    \]

We compute Jacobson sequences and subalgebras of differential operators corresponding to $W_\bullet$, Theorem \ref{MainTheorem}(Galois-type Correspondence for Power Towers).  However, we will use symbols $\frac{1}{p^{m}!}\frac{\partial^{p^m}}{\partial y ^{p^m}}$ and $\frac{\partial}{\partial \left(y^{p^m}\right)}$, and $\frac{1}{p^{m}!}\frac{\partial^{p^m}}{\partial x ^{p^m}}$ and $\frac{\partial}{\partial \left(x^{p^m}\right)}$ interchangeably on all subfields of $K$. (We can do this without losing anything, because they are mapped to each other by differentials, Proposition \ref{symbol.differential}.)

We begin with the Jacobson sequence.

The first field extension $K/W_1$ is $k(x,y)\supset k(x,y^p)$. This corresponds to a $p$-Lie algebra $\F_1$ on $K$ spanned by the vector field $\frac{\partial}{\partial y}$. 

The second extension $W_1/W_2$ is $k(x,y^p)\supset k(x+A_1y^p,y^{p^2})$.  The relative tangent space $T_{W_1/W_1^p}$ can be computed using $x,y$. Indeed, it is spanned by 
\begin{center}
    $\frac{\partial}{\partial y^p}$ and $\frac{\partial}{\partial x}$.
\end{center} 
Knowing this, we clearly see that the extension $W_1/W_2$ corresponds to a $p$-Lie algebra $\F_2$ on $W_1$ spanned by a vector field 
\[
\frac{\partial}{\partial y^p}-A_1\frac{\partial}{\partial x}.
\]
This is a $p$-Lie algebra, because the $p$-power of $\frac{\partial}{\partial y^p}-A_1\frac{\partial}{\partial x}$ is zero, because $A_1\in k$ (!). 

The third extension $W_2/W_3$ is 
\[
k(x+A_1y^p,y^{p^2})\supset k(x+A_1y^p+A_2y^{p^2}, y^{p^3}).
\]
From here, the computation of a relative tangent space $T_{W_2/W_2^p}$ in terms of $x,y$ gets harder. 
Indeed, we can use the fact that that the differential $d(K/W_2)$ is surjective, Theorem \ref{SES for diff - fields - theorem}, 
to find differential operators on $K$ that are mapped onto a basis of $T_{W_2/W_2^p}$.
It is easy to check that the operators 
\begin{center}
    $\frac{\partial}{\partial y^{p^2}}-A_1^p\frac{\partial}{\partial x^p}$ and $\frac{\partial}{\partial x}$
\end{center} 
are mapped to $T_{W_2/W_2^p}$ and span it. (My method: I have guessed them.)
In this basis, the extension $W_2/W_3$ corresponds to a $p$-Lie algebra $\F_3$ on $W_2$ spanned by a vector field 
\[
\frac{\partial}{\partial y^{p^2}}-A_1^p\frac{\partial}{\partial x^p} - A_2\frac{\partial}{\partial x}.
\]

We continue similarly for extensions $W_{n-1}/W_n$.
Let $n\ge 3$. We have that $T_{W_{n-1}/W_{n-1}^p}$ is spanned by the operators
\begin{center}
 $\frac{\partial}{\partial y^{p^{n-1}}}-A_1^{p^{n-2}}\frac{\partial}{\partial x^{p^{n-2}}} - A_2^{p^{n-3}}\frac{\partial}{\partial x^{p^{n-3}}}-\ldots-A_{n-2}^p\frac{\partial}{\partial x^{p}}$ and $\frac{\partial}{\partial x}$.   
\end{center}
 And, the extension $W_{n-1}/W_{n}$ corresponds to a $p$-Lie algebra $\F_n$ on $W_{n-1}$ spanned by the vector field 
\[
\frac{\partial}{\partial y^{p^{n-1}}}-A_1^{p^{n-2}}\frac{\partial}{\partial x^{p^{n-2}}} - A_2^{p^{n-3}}\frac{\partial}{\partial x^{p^{n-3}}}-\ldots-A_{n-1}\frac{\partial}{\partial x}.
\]

Finally, by Theorem \ref{unpacking - Thm}(Unpacking), we can conclude that $\op{Diff}_{W_\bullet}(K)$ is generated by the operators
\begin{align*}
    &\frac{\partial}{\partial y},\\
    &\frac{\partial}{\partial y^p}-A_1\frac{\partial}{\partial x},\\
    &\ldots,\\
    &\frac{\partial}{\partial y^{p^{n-1}}}-A_1^{p^{n-2}}\frac{\partial}{\partial x^{p^{n-2}}} - A_2^{p^{n-3}}\frac{\partial}{\partial x^{p^{n-3}}}-\ldots-A_{n-1}\frac{\partial}{\partial x},\\
    &\ldots
\end{align*}
that form a set of well-chosen distinguished generators for the subalgebra $\op{Diff}_{W_\bullet}(K)$, Definition \ref{distinguished generators - def}. One of its completions, Definition \ref{completion of distinguished generators - def}, is the set $\frac{\partial}{\partial x},\frac{\partial}{\partial x^p},\frac{\partial}{\partial x^{p^2}},\ldots$.

This finishes the computation of these examples.
\end{example}

\newpage
\section{Power Towers on Varieties}

We extend the theory of power towers from fields to normal varieties. We discuss why we assume ``geometrically connected'' everywhere. This is related to discussing a ``natural perfect field of definition'' for a variety. After explaining this, we set up definitions to prove the Galois-type correspondence for power towers on varieties. It follows from the case for fields. In particular, there is a bijection between the power tower on a variety and the power towers on its generic point. The next sections are about the properties of power towers, which actually depend on the variety. We study the regularity of power towers. We show that these power towers are ``trivial'' at any closed point in some formal coordinates, a Frobenius-type theorem. However, immediately after that, we show how ``trivial'' it is at one point may differ from how ``trivial'' it is at another point. This leads to a stratification of the variety. Then, we discuss a class of power towers inspired by Ekedahl's paper \cite{EkedahlFoliation1987}. We discuss when a nice morphism, a fibration, is a power tower.

\subsection{Natural Perfect Base}

 Our definition of a variety is \cite[Definition 020D]{stacks-project}. We repeat:

\begin{defin}[Variety]\label{variety - definition}
    Let $X$ be a scheme over a field $k$. The scheme $X$ is a \emph{variety} over the field $k$ if it is integral, and the morphism $X\to \op{Spec}(k)$ is separated and of finite type.
\end{defin}

\begin{remark}
        If $X$ is an integral scheme, then we denote its \textbf{field of rational functions} by $K(X)$. This $K$ is a notation from \cite{Hartshorne}, not the base field of $X$. In particular, one can use this notation for schemes that are not over any field, such as $\op{Spec}(\mathbb{Z})$, i.e., $K(\op{Spec}(\Z))=\mathbb{Q}$. This notation is compatible with the notation used in the previous chapter about power towers on fields, where we often have a big field $K$ and its perfection $k$. This chapter often has $K=K(X)$ and the base field $k$.
\end{remark}

In this chapter, we will encounter many ``geometric \emph{something}'' notions, e.g., \emph{something} may equal connected, reduced, normal, regular. However, instead of enlisting all lemmas about them here, we will organically cite them inside proofs as we go. Nevertheless, the notion of geometrically connected requires its own paragraph, as it will be almost universally present.

To develop the power tower theory for varieties, we will assume our varieties to be geometrically connected. 
We will assume this hypothesis for freely moving between a variety $X/k$ and its generic point $K(X)/k$ without doing any twists/pullback/changes, ensuring that $k$ is the perfection of $K=K(X)$. We do not lose any generality by doing so, because any variety $X/k$ over a perfect field $k$ admits a factorization through its perfection, the below Lemma \ref{natural perfect base - lemma}:
\[
X\to\op{Spec}(K(X)^{p^\infty}) \to \op{Spec}(k).
\]
So, at any moment we can replace $k$ with $K(X)^{p^\infty}$ and then the variety $X/K(X)^{p^\infty}$ is geometrically connected by the below Proposition \ref{geoemtrically connected over perfect - prop}. This discussion belongs to a classical question: ``What is a good field of definition for a given variety?''.

\begin{lemma}[Natural Perfect Base Field]\label{natural perfect base - lemma}
    Let $X/k$ be a variety over a perfect field of characteristic $p>0$. Then, we have a factorization
    \[
X\to\op{Spec}(K(X)^{p^\infty}) \to \op{Spec}(k).
    \]
    In particular, the field extension $K(X)^{p^\infty}/k$ is finite. 
    Moreover, if $U=\op{Spec}(A)$ is an affine open subset of $X$, then $K(X)^{p^\infty} \subset A$ is the greatest subfield inside $A$, and it equals $A^{p^\infty}=\cap_{n\ge 0} A^{p^n}$.
\end{lemma}

\begin{proof}
    Without loss of generality, we can assume that $X$ is affine and it is given by a ring $A$, then by the Noether normalization theorem, we have
    \[
    k[t_1,\ldots,t_r]\to A,
    \]
    where the arrow is a finite morphism, and $t_1,\ldots, t_r$ are free variables. This morphism factorizes
    \[
    k[t_1,\ldots,t_r]\to k'[t_1,\ldots,t_r] \to A,
    \]
    where $k'$ is the greatest subfield inside $A$\footnote{In general, the variables $t_i$ won't be free in $k'[t_1,\ldots,t_r]$. They will be the only generators. We just abuse the notation.} (so, the first arrow adds algebraic functions that do not involve any $t_i$, this happens finitely many times, and the second arrow adds algebraic functions involving some $t_i$ nontrivially), and $k'=A^{p^\infty}$ (all functions involving the variables $t_i$ disappear by taking the perfection), because $k'$ is perfect by being a finite extension of a perfect field. (The extension $k'/k$ is finite, because otherwise $A/k$ is not of finite type.) Finally, we can conclude that 
    \[ 
    A^{p^\infty}=k'=K(X)^{p^\infty}.
    \]
    This proves the lemma.
\end{proof}

\begin{remark}
    In Lemma \ref{natural perfect base - lemma}, the assumption that $k$ is perfect is essential. Indeed, here is a counterexample.
    
    Let $k=l(x)$, where $l$ is a perfect field. Let $A=k[t]$. Then $K(X)=l(x,t)$, so $K(X)^{p^\infty}=l$. Since it would have a nontrivial kernel, there cannot be any ring map $f:k\to l$.
\end{remark}

\begin{defin}[{\cite[Definition 0362]{stacks-project}}]
    A scheme $X$ over a field $k$ is \emph{geometrically connected} if for every field extension $k'/k$ the base change 
    \[
    X_{k'}\coloneqq X\times_{\op{Spec}(k)} \op{Spec}(k')
    \]
    is connected.
\end{defin}

\begin{prop}\label{geoemtrically connected over perfect - prop}
Let $X$ be a variety over a perfect field $k$ of characteristic $p>0$. Then, $X$ is geometrically connected if and only if $K(X)^{p^\infty}=k$.
\end{prop}

\begin{proof}
    ($\Leftarrow$) If the equality $K(X)^{p^\infty}=k$ is true, then we can apply \cite[Lemma 04KV]{stacks-project} to conclude that $X$ is geometrically connected. Indeed,
    this lemma says that: If $X$ is a scheme over a field $k$ and it has a point $x\in X$ such that $k$ is algebraically closed in $\kappa(x)$, then $X$ is geometrically connected. 
    
    We apply this lemma to the generic point of our $X$ to learn that if $k$ is algebraically closed in $K(X)$, then $X$ is geometrically connected. Now, we assume that $k$ is the perfection of $K(X)$, so it is perfect, and it is separably closed in $K(X)$ by Theorem \ref{infty=s refined - Corollary}(Subfields into Power Towers). Therefore, $k$ is algebraically closed in $K(X)$. This means that $X$ is geometrically connected. 

    ($\Rightarrow$) We have the inclusion of perfect fields $k\subset K(X)^{p^\infty}$, because $K(X)^{p^\infty}$ is the largest perfect subfield of $K(X)$, Lemma \ref{The Largest Perfect Subfield}. 
    
    By Lemma \ref{natural perfect base - lemma}, we have a factorization:
    \[
    X\to\op{Spec}(K(X)^{p^\infty}) \to \op{Spec}(k).
    \]
    
    And, we have that $\op{Spec}(K(X)^{p^\infty}) \to \op{Spec}(k)$ corresponds to a finite field extension. So, it is algebraic separable, because $k$ is perfect. 
    
    Finally, if we have $k\ne K(X)^{p^\infty}$, then $\op{Spec}(K(X)^{p^\infty}) \times_{\op{Spec}(k)} \op{Spec}(K(X)^{p^\infty})$ is disconnected. Therefore, $X_{K(X)^{p^\infty}}$ is disconnected. This is a contradiction with $X$ being geometrically connected. Therefore, we proved that $k= K(X)^{p^\infty}$.
\end{proof}

\begin{remark}[Why is ``$k=K(X)^{p^\infty}$'' everywhere?]
    Lemma \ref{natural perfect base - lemma} says that for every variety $X$ over a perfect field $k$ of characteristic $p>0$ there exists the factorization
    \[
    X\to\op{Spec}(K(X)^{p^\infty}) \to \op{Spec}(k).
    \]
    Therefore, every variety $X$ in positive characteristic is naturally a variety over $\op{Spec}(K(X)^{p^\infty})$ that is a perfect field. It means that every variety in positive characteristic admits a \emph{natural perfect base field} over which it is defined.

    Consequently, in the following sections, the assumption that $k=K(X)^{p^\infty}$ is not super essential and in most cases could probably be dropped, because we can pretty much always switch $k$ to be the natural perfect base $K(X)^{p^\infty}$. Still, it would be a drag to repeat it in every proof. Moreover, I enjoy being able to recover the base field from the tower of $p$-powers whenever I want, because this is recovering the trivial fibration $X\to \op{Spec}(k)$. Therefore, I assume it everywhere. This helps my sanity because, for me, the tower of $p$-powers is the ``foliation'' corresponding to the variety itself, and I do not wish to have any twists around that.
\end{remark}

\subsection{Basic Theory Including Galois-type Correspondence}
We define power towers on varieties and show that they are in an explicit bijection with power towers on the generic point. We use it to conclude a Galois-type correspondence for power towers on varieties from the one for fields via taking normalizations and saturations of relevant notions.

\subsubsection{Power Towers} Please, recall Definition \ref{purely inseparable - varieties - def}.

\begin{convention}[All is $k$-linear here]\label{All is k-linear - convention}
    All our purely inseparable morphisms $X\to Y$ between varieties will be twisted to be $k$-linear, i.e., morphisms of $k$-varieties, and to fit $X\to Y\to X^{(n)}$ for some $n$. See Definition \ref{relative Frob twist - def} for how to do these twists. Nothing changes, but this is always assumed, and we will not mention it again.
\end{convention}

\begin{defin}\label{Power Tower on varieties - def}
    Let $X$ be a normal variety over a perfect field $k$ of characteristic $p>0$ satisfying $k=K(X)^{p^\infty}$.

    A \emph{\textbf{power tower}} on the variety $X$ is a sequence of purely inseparable morphisms $f_i: X\to Y_i$ between normal varieties over $k$ for $i\ge 0$, such that the sequence
    \[
    (K(Y_0),K(Y_1),K(Y_2),\ldots)
    \]
    is a power tower on $K(X)$; we denote it by $K(Y_\bullet)$. And, we denote the power tower $f_0,f_1,f_2,\ldots$ by $f_\bullet$, or $Y_\bullet$.
\end{defin}

\begin{lemma}[``On $X$'' equals ``On $K(X)$'']\label{power towers on X = power towers on K(X) - lemma}
    Let $X$ be a normal variety over a perfect field $k$ of characteristic $p>0$ satisfying $k=K(X)^{p^\infty}$.
    Then, there is a bijection between power towers on $X$ and power towers on $K(X)$.
\end{lemma}

\begin{proof}
    Every power tower on $X$ induces a power tower on $K(X)$ by Definition \ref{Power Tower on varieties - def}. And, any power tower $W_\bullet$ on $K(X)$ can be ``lifted'' to a power tower on $X$. Indeed, by Proposition \ref{purely inseparable equiv on fields}, for every subfield $K(X)\supset W_i$ there exists a unique purely inseparable morphism $f_i: X\to Y_i$ between normal varieties such that $W_i=K(Y_i)$. Collecting all these morphisms defines the power tower on $X$ lifting $W_\bullet$.
\end{proof}

There are MANY arrows in a definition of a power tower. We label all of them below.

\begin{convention}[Notation for Power Towers]\label{power tower on variety - notation}
   Let $X$ be a normal variety over a perfect field $k$ of characteristic $p>0$ satisfying $k=K(X)^{p^\infty}$.
   Let $Y_\bullet$ be a power tower on $X$.

   Then, every morphism $f_i: X\to Y_i$ is of exponent at most $i$. (Its exponent is the exponent of $K(X)\supset K(Y_i)$) 
   And, they factorize through each other. We denote the between morphisms by 
   \[
   f'_i: Y_{i-1}\to Y_i.
   \]
   Each of the morphisms $f'_i$ is of exponent at most $1$ by Lemma \ref{Subsequent Extensions are of exponent 1 - lemma}. We denote the second factor of the factorization of $F_{Y_{i-1}/k}$ by 
   \[
   g_i: Y_i\to Y_{i-1}^{(1)}.
   \]
   Consequently, all these morphisms fit into the following diagram.

   \begin{center}
    \begin{tikzcd}
        X=Y_0 \ar[r,"f_1=f'_1"] & Y_{1}\ar[d,"f'_2"] \ar[r,"g_1"] & X^{(1)}\ar[d,"f^{(1)}_1"]&\\
        X \ar[r,"f_2"]\ar[u,equal] & Y_{2}\ar[d,"f'_3"] \ar[r,"g_2"] & Y_{1}^{(1)}\ar[d,"f'^{(1)}_2"]\ar[r, "g_1^{(1)}"]&X^{(2)}\ar[d,"f_1^{(2)}"]\\
        X \ar[r,"f_3"]\ar[u,equal] & Y_{3} \ar[d,"f'_4"]\ar[r,"g_3"] & Y_2^{(1)}\ar[d,"f'^{(1)}_3"]\ar[r, "g_2^{(1)}"] &Y_1^{(2)}\ar[d,"{f'}_2^{(2)}"]\\
        \cdots \ar[r]\ar[u,equal] & \cdots \ar[r] & \cdots \ar[r] &\cdots
    \end{tikzcd}
\end{center}
    There are many of them, and then there are their $p^n$-power versions with the index $\bullet^{(n)}$ for $n\in \Z$, not only for $n\ge 0$. In many cases, we can go backwards, too, which is often a critical trick for some reasoning.
\end{convention}

\begin{convention}\label{power tower is P - convention}
    Generally, we will follow a convention that a power tower $Y_\bullet$ on $X$ is $P$, where $P$ is a property of power towers on fields, if the corresponding to $Y_\bullet$ power tower $K(Y_\bullet)$ on $K(X)$ is $P$.
\end{convention}

Here is an example for Convention \ref{power tower is P - convention}.

\begin{defin}
    Let $X$ be a normal variety over a perfect field $k$ of characteristic $p>0$ satisfying $k=K(X)^{p^\infty}$.
   Let $Y_\bullet$ be a power tower on $X$.
\begin{itemize}
    \item We say that the power tower $Y_\bullet$ is of length at most $n$, where $n\ge0$ is an integer, if the morphisms $f'_m$ are identities for $m>n$. It is of length $n$ if this $n$ is minimal. And, it is of \emph{finite length}, if an integer $n\ge 0$ exists, such that $Y_\bullet$ is of length at most $n$.
    \item We say that the power tower $Y_\bullet$ is a \emph{$n$-foliation on $X$}, where $n\ge0$ is an integer, if it is of length $n$ and the ranks of finite maps $f_1,f'_2,\ldots, f'_n$ are equal.
    \item We say that the power tower $Y_\bullet$ is a \emph{$\infty$-foliation on $X$} if the rank of finite maps $f_1,f'_2,f'_3,\ldots$ are equal.
\end{itemize}
\end{defin}

Again, we have the following result: Lemma \ref{Finite Length = finite exponent - lemma}.

\begin{lemma}[Purely Inseparable Morphisms inject into Power Towers]\label{purely inseparable morphisms into power towers - lemma}
    Let $X$ be a normal variety over a perfect field $k$ of characteristic $p>0$ satisfying $k=K(X)^{p^\infty}$.
    Then, there is a bijection between purely inseparable morphisms $f: X\to Y$ between normal varieties and power towers on $X$ of finite length.
\end{lemma}

\begin{proof}
    Let $f: X\to Y$ be a purely inseparable morphism between normal varieties. It is equivalent to the field extension $K(X)\supset K(Y)$ by Proposition \ref{purely inseparable equiv on fields}. This field extension is equivalent to a finite length power tower on $K(X)$ by Lemma \ref{Finite Length = finite exponent - lemma}. This power tower on $K(X)$ is equivalent to a finite length power tower on $X$ by Lemma \ref{power towers on X = power towers on K(X) - lemma}. 
\end{proof}

\subsubsection{Jacobson Sequences}
We define Jacobson sequences on varieties by saturating Jacobson sequences from generic points, Definition \ref{Jacobson Sequence - definition}. First, for convenience, we recall a definition of a $p$-Lie algebra on a variety.

\begin{defin}\label{p-lie algebra - variety - def}
Let $X$ be a normal variety over a perfect field $k$ of characteristic $p>0$.
A \emph{\textbf{$p$-Lie algebra}} on $X$ is a quasicoherent subsheaf $\F$ of the tangent sheaf $T_{X/k} \coloneqq \Omega_{X/k}^* = \op{Hom}_{\mathcal{O}_X}(\Omega_{X/k},\mathcal{O}_X)$ such that
    \begin{itemize}
        \item $\F$ is saturated, i.e., the quotient $T_{X/k}/\F$ is torsion-free,
        \item $\F$ is closed under Lie bracket, i.e., for every open subset $U\subset X$ and sections  $v,w\in \F(U)$ we have $[v,w]\in \F(U)$,
        \item $\F$ is closed under $p$-power, i.e., for every open subset $U\subset X$ and a section $v\in \F(U)$ we have $v^{\circ p}\in \F(U)$.
    \end{itemize}
\end{defin}

We recall that Ekedahl Correspondence \ref{Jacobson--Ekedahl Correspondence - proposition} says that there is a bijection between purely inseparable morphisms of exponent at most $1$ between normal varieties from $X$ and $p$-Lie algebras on $X$. It extends the Jacobson correspondence \ref{Jacobson Purely Inseparable Galois Theory for Exponent $1$} to varieties. In particular, if $X\to Y \to X^{(1)}$ corresponds to a $p$-Lie algebra $\F$, then $Y=\op{Spec}_{X^{(1)}}(\op{Ann}(\F))$, where $\op{Ann}(\F)$ is the sheaf of functions that are killed by derivations in $\F$.

\begin{defin}\label{Jacobson sequence - variety - definition}
    Let $X$ be a normal variety over a perfect field $k$ of characteristic $p>0$ satisfying $k=K(X)^{p^\infty}$.
    A \emph{\textbf{Jacobson sequence}} on $X$ is a sequence of $p$-Lie algebras $\F_i$ for $i\ge 1$ such that:
    \begin{itemize}
        \item $\F_1$ is a $p$-Lie algebra on $X$,
        \item $\F_i$, for $i>1$, is a $p$-Lie algebra on $Y_{i-1}\coloneqq \op{Spec}_{Y_{i-2}^{(1)}}(\op{Ann}(\F_{i-1}))$,
        \item the sequence of vector spaces $\F_i \otimes K\left(Y_{i-1}\right)\subset T_{K\left(Y_{i-1}\right)/k}$ is a Jacobson sequence on $K(X)$, i.e., the following intersection is trivial, for $i>1$:
        \[
        \left(\F_i \otimes K\left(Y_{i-1}\right)\right) \cap T_{K\left(Y_{i-1}\right)/K\left(Y_{i-2}^{(1)}\right)}=0\subset T_{K\left(Y_{i-1}\right)/k}=T_{K\left(Y_{i-1}\right)/K\left(Y_{i-1}\right)^p}.
        \]
    \end{itemize}
    We denote the Jacobson sequence $\F_1,\F_2,\ldots$ by $\F_\bullet$.

    Moreover, we say that the Jacobson sequence $\F_\bullet$ is a \emph{saturation} of the Jacobson sequence $\F_1\otimes K\left(Y_{0}\right),\F_2\otimes K\left(Y_{1}\right),\ldots$ on $K(X)$, which we denote by $\F_\bullet\otimes K\left(Y_{\bullet-1}\right)$.
\end{defin}

\begin{remark}
    It is worth emphasizing that the last condition in Definition \ref{Jacobson sequence - variety - definition} is a generic point condition. It does not involve the variety. Because of this fact, measuring the quality of the saturation will bring us precious information about the variety, e.g., it will explain the behavior of canonical divisors for purely inseparable morphisms.
\end{remark}

\begin{lemma}\label{jacobson sequences on X and K(X) - lemma}
      Let $X$ be a normal variety over a perfect field $k$ of characteristic $p>0$ satisfying $k=K(X)^{p^\infty}$. Then, there is an explicit bijection between Jacobson sequences on $X$ and Jacobson sequences on $K(X)$.
    \end{lemma}

\begin{proof}
    We recall that for a torsion-free coherent sheaf $M$ on a variety $X$, there is a bijection between saturated quasicoherent subsheaves of $M$ and vector subspaces of the restriction of $M$ to the generic point that will be denoted here by $M\otimes K(X)$, Lemma \ref{Saturated are in bij with gen subspaces - lemma}. This implies that there is a bijection between $p$-Lie algebras on $X$ and $p$-Lie algebras on $K(X)$.  This immediately implies that there is a bijection between Jacobson sequences, because all conditions for a sequence to be Jacobson are essentially identical for varieties and fields.
\end{proof}

\subsubsection{Saturated Subalgebras of Differential Operators}
For a variety $X$ over a field $k$, we can apply the definition of differential operators, Definition \ref{Differential Operators: Universal Definition}, locally. (It is analogous to the definition of $T_{X/k}$ and $\Omega_{X/k}$.) This leads to the formation of a sheaf of algebras on $X$ determined by
\[
\op{Diff}_k(X)(\op{Spec}(A))\coloneqq \op{Diff}_k(A),
\]
where $\op{Spec}(A)\subset X$ is an open affine subset. The sheaf $\op{Diff}_k(X)$ is torsion-free, and it admits a natural decomposition
\[
\op{Diff}_k(X)=\mathcal{O}_{X/k} \oplus \E(X/k),
\]
where $\E(X/k)$ is the subsheaf of operators that are zero on one, i.e., the operators $D$ such that $D(1)=0$.

Analogically, in characteristic $p>0$, one can define a sheaf of algebras
\[
\op{Diff}_{X^{(r)}}(X)(\op{Spec}(A))\coloneqq \op{Diff}_{A^{p^r}}(A)
\]
that is a sheaf of subalgebras of $\op{Diff}_k(X)$. (The quotient $\op{Diff}_k(X)/\op{Diff}_{X^{(r)}}(X)$ is torsion-free.) And, the collection of all of them is a $p$-filtration of the sheaf of algebras $\op{Diff}_k(X)$.

On fields, every subalgebra corresponds to a power tower; however, this is no longer true on varieties. We have to assume that our subalgebras are saturated. 

\begin{defin}\label{saturated subalgebras - variety - definition}
    Let $X$ be a normal variety over a perfect field $k$ of characteristic $p>0$ satisfying $k=K(X)^{p^\infty}$.
A \emph{\textbf{saturated subalgebra}} of differential operators on $X/k$ is a quasicoherent subsheaf $\mathcal{O}_{X/k}\subset \mathcal{D}\subset \op{Diff}_k(X)$ such that:
    \begin{itemize}
        \item $\mathcal{D}$ is saturated, i.e., the quotient $\op{Diff}_k(X)/\mathcal{D}$ is torsion-free,
        \item $\mathcal{D}$ is closed under composition, i.e., it is a sheaf of subalgebras.
    \end{itemize}

    The \emph{$p$-filtration} of the (saturated) subalgebra $\mathcal{D}$ is defined by, for $r\ge 0$:
    \[
    \mathcal{D}_r\coloneqq \op{Diff}_{X^{(r)}}(X)\cap \mathcal{D}.
    \]
\end{defin}

\begin{lemma}\label{saturated subalg and alg - lemma}
    Let $X$ be a normal variety over a perfect field $k$ of characteristic $p>0$ satisfying $k=K(X)^{p^\infty}$. Then, there is a bijection between saturated subalgebras of $\op{Diff}_k(X)$ and subalgebras of $\op{Diff}_k(K(X))$.
\end{lemma}

\begin{proof}
    It is a special case of Lemma \ref{Saturated are in bij with gen subspaces - lemma}.
\end{proof}

We can sheafify the operation of computing constants of a subalgebra of differential operators from Theorem \ref{Topological Jacobson--Bourbaki Correspondence} to obtain a similar operation on varieties.

\begin{defin}
    Let $X$ be a normal variety over a perfect field $k$ of characteristic $p>0$ satisfying $k=K(X)^{p^\infty}$.
    Let $\mathcal{D}\subset \op{Diff}_k(K(X))$ be a saturated subalgebra on $X/k$. We define a subsheaf $\op{const}(\mathcal{D})$ of $\cO_{X/k}$ by the following values on affine open subsets $\op{Spec}(A)\subset X$:
    \[
    \op{const}(\mathcal{D})(\op{Spec}(A))\coloneqq \{f\in \cO_{X/k}(\op{Spec}(A)) : \ \forall_{D\in\mathcal{D}(\op{Spec}(A))} \ [D,f]=0\}.
    \]
\end{defin}

\begin{lemma}
    Let $X$ be a normal variety over a perfect field $k$ of characteristic $p>0$ satisfying $k=K(X)^{p^\infty}$.
    Let $\mathcal{D}\subset \op{Diff}_k(K(X))$ be a saturated subalgebra on $X/k$.
    Then, the sheaf $\op{const}(\mathcal{D}_n)$ is a quasicoherent sheaf of algebras on $X^{(n)}$.
\end{lemma}

\begin{proof}
    The sheaf $\mathcal{D}_n$ is a subsheaf of $\op{Diff}_{X^{(r)}}(X)$, by its definition. Therefore, we have an inclusion of sheaves of algebras on $|X|$:
    \[
    \op{const}(\mathcal{D}_n)\supset \cO_{X^{(n)}/k} = \op{const}(\op{Diff}_{X^{(r)}}(X)).
    \]
    This proves the lemma.
\end{proof}

\subsubsection{Galois-type Correspondence for Power Towers on Varieties} The following is the analogous statement of Theorem \ref{MainTheorem} for varieties.

\begin{thm}\label{MainTheorem - var - theorem}
    Let $X$ be a normal variety over a perfect field $k$ of characteristic $p>0$ satisfying $k=K(X)^{p^\infty}$.
    Then, there are natural bijections between the following data on $X$:
\begin{enumerate}
    \item power towers on $X$,
    \item Jacobson sequences on $X$,
    \item saturated subalgebras of $\op{Diff}_k(X)$.
\end{enumerate}
    Explicitly,
some of these bijections are given by the following operations.
\begin{enumerate}
    \item From power towers to Jacobson sequences. \emph{An iterated tangent sheaf}:
    \[
    Y_\bullet \mapsto (T_{Y_0/Y_1},T_{Y_1/Y_2},T_{Y_2/Y_3}, \ldots),
    \]
    \item From Jacobson sequences to power towers. \emph{An iterated Spec of annihilators}:
    \[
    F_\bullet \mapsto f_i : X\to \op{Spec}_{Y_{i-2}^{(1)}}(\op{Ann}(\F_{i-1})) \ \text{for $i>1$}
    \]
    \item From power towers to saturated subalgebras. \emph{A saturated algebra of differential operators relative to a power tower}:
    \[
    Y_\bullet \mapsto \op{Diff}_{Y_\bullet}(X)\coloneqq \bigcup_{i\ge 0} \op{Diff}_{Y_i}(X),
    \]
    \item From saturated subalgebras to power towers. \emph{Fields of constants of the $p$-filtration}:
    \[
    \mathcal{D} \mapsto f_i: X\to \op{Spec}_{X^{(i)}}(\op{const}(\mathcal{D}_i)) \ \text{for $i\ge 1$}.
    \]
\end{enumerate}
\end{thm}

\begin{proof}
    We already know that power towers, Jacobson sequences, and subalgebras of differential operators on $K(X)$ are naturally in bijection with each other by Theorem \ref{MainTheorem}. 
    Now, power towers on $X$ and power towers on $K(X)$ are naturally in bijection by Lemma \ref{power towers on X = power towers on K(X) - lemma}.
    Jacobson sequences on $X$ and Jacobson sequences on $K(X)$ are naturally in bijection by Lemma \ref{jacobson sequences on X and K(X) - lemma}.
    Finally, saturated subalgebras of $\op{Diff}_k(X)$ and subalgebras of $\op{Diff}_k (K(X))$ are naturally in bijection by Lemma \ref{saturated subalg and alg - lemma}. Moreover, all these bijections are given by the formulas enlisted in the theorem, because computing $\op{const}$ and $\op{Ann}$ produces sheaves of normal algebras. Hence, these operations are ``saturations'' of the ones from Theorem \ref{MainTheorem}. This finishes the proof.
\end{proof}

\begin{remark}
    In Theorem \ref{MainTheorem - var - theorem}, we provided explicit formulas for some of the bijections. However, the remaining ones are less obvious. In particular, the unpacking, Theorem \ref{unpacking - Thm}(Unpacking), is tricky, because it uses differentials. It is unclear to me that a differential $d(X/Y)$ exists since I use Criterion \ref{Criterion: admitting a differential} to know if the differential is admitted. 
    This criterion applied to $X\to Y$ requires $Y$ to be a regular variety. It looks like a quite artificial assumption.

    But, if the power tower is regular, Definition \ref{regular power tower - def}, then it admits all required differentials, and then the unpacking is true, Corollary \ref{unpacking - corollary - regular power towers}. However, in general, I do not know what happens when all kinds of singularities are present, and I believe it is likely not true then. This is just a belief, not even a strong one.
\end{remark}

\subsection{Regular Power Towers}
We develop a notion of regularity for power towers on varieties. We define the notion. We prove what it means in practice. We show that every purely inseparable morphism is regular away from a codimension at least $2$ subset.

In short, regular power towers almost behave like power towers on fields. Whenever this fails, something interesting happens at the saturation level between the variety and its generic point. Something emergent. We will see that in the next sections.
Here we collect good behaviors.

\subsubsection{Motivation} In characteristic zero, a foliation of rank $1$ on a smooth variety $X$ generated locally by a local vector field $v$ is said to be regular if every such $v$ does not admit ``singularities'', e.g., for a global vector field $v=f\frac{\partial}{\partial x}+g\frac{\partial}{\partial y}$ on a plane $\mathbb{A}^2_k$ it means that the vanishing locus of both functions $f$ and $g$ is empty. This notion of ``regularity'' of a foliation extends to foliations of all ranks (it measures how much a subsheaf $\F\subset T_{X/k}$ is a subbundle) and it is true that for every foliation there exists a subset of the variety, which is of codimension at least $2$, such that the foliation is regular away from it. One can read \cite[Chapter 2]{Brunella} for more details about this topic.

\subsubsection{Development} 
We will say that a power tower is regular if its Jacobson sequence is a sequence of subbundles. This can be formally stated using a singular locus of a subsheaf. We recall that the singular locus of a subsheaf is a singular locus of the quotient, and the singular locus of a coherent sheaf is defined in \cite[Chapter 2, Definition 18]{Friedman-bundles}. This locus is where this sheaf is \emph{not} locally free.

\begin{defin}
    Let $X$ be a normal variety over a perfect field $k$ of characteristic $p>0$ satisfying $k=K(X)^{p^\infty}$.
Let $Y_\bullet$ be a power tower on $X$, and let $\F_\bullet$ be the Jacobson sequence corresponding to $Y_\bullet$.
We define the following \emph{singular loci} as subsets of the topological space $|X|$:
    \begin{center}
        $\op{Sing}_n(Y_\bullet) \coloneqq \bigcup_{i< n+1} \op{Sing}( \F_i\subset T_{Y_{i-1}/k})$ for $n\ge 1$ or $n=\infty$.
    \end{center}
\end{defin}

\begin{defin}\label{regular power tower - def}
    Let $X$ be a normal variety over a perfect field $k$ of characteristic $p>0$ satisfying $k=K(X)^{p^\infty}$. Let $Y_\bullet$ be a power tower on $X$.

    We say that the power tower $Y_\bullet$ is \emph{\textbf{regular}} if the following conditions are satisfied:
    \begin{itemize}
        \item $X/k$ is a regular variety, \cite[Definition 07R7]{stacks-project},
        \item $\op{Sing}_\infty(Y_\bullet)=\emptyset$.
    \end{itemize}
\end{defin}

\begin{remark}
    We recall that for a variety $X/k$, it is true that $T_{X/k}=T_{X/X^{(1)}}$, because we work with tangent spaces/sheaves, and not with tangent complexes. This is nothing but ``derivations kill $p$-powers''.
    
    We say it, because there could be a confusion about why $k$ and not $X^{(1)}$, because $X^{(1)}$ is used in the definition of a $p$-Lie algebra \ref{p-lie algebra: definition}, a Jacobson sequence \ref{Jacobson sequence - variety - definition}, or \ref{Jacobson Sequence - definition}.
\end{remark}

We recall that an open subset is sometimes called big if its complement is of codimension at least $2$.

\begin{prop}\label{codim 2 regular power tower finite length- prop}
    Let $X$ be a normal variety over a perfect field $k$ of characteristic $p>0$ satisfying $k=K(X)^{p^\infty}$. Let $Y_\bullet$ be a \emph{finite length} power tower on $X$, i.e., a purely inseparable morphism.

    Then, there exists a big open subset $U\subset X$ such that the restriction of $Y_\bullet$ to $U$ is a regular power tower on $U$.
\end{prop}

\begin{proof}
    First, a big open subset $U\subset X$ exists such that $U/k$ is regular. Indeed, it follows from \cite[Exercise 6.21]{Gortz_Wedhorn_AlgGeoI}: a geometrically normal scheme of finite type over a field is smooth away from a closed subset of codimension $2$. We can apply this result to our $X/k$, because $X$ is normal over a perfect field $k$, and thus geometrically normal by \cite[Exercise 6.19: (iii)]{Gortz_Wedhorn_AlgGeoI}.

    Let $\F_\bullet$ be the Jacobson sequence corresponding to $Y_\bullet$.

    Next, we show that there is a big open subset of $V\subset U$ such that $\op{Sing}(\F_1\subset T_{X/k})\cap V=\emptyset$. (This $V$ will be big in $X$ too.) Indeed, by definition of a singular support of a coherent module, we have to show that the quotient $T_{X/k}/\F_1$ is locally free away from a closed subset of codimension $\ge 2$. It is true precisely by \cite[Proposition 22]{Friedman-bundles} that says a singular support of a finite module over a regular local Noetherian ring is of codimension at least $2$. This proposition, combined with the fact that being locally free is an open condition, implies the existence of $V$.

    Finally, we assume that the length of the power tower $Y_\bullet$ is $n$. Then, the tower is a finite sequence:
    \[
    X\to Y_1\to Y_2 \to \ldots \to Y_n.
    \]
    By the first paragraph, there are big open subsets 
    \begin{center}
        $U_0\subset X$, $U_1\subset Y_1$, $\ldots$, and $U_{n-1}\subset Y_{n-1}$
    \end{center} that are regular. By the second paragraph, for every $U_i$, there exists a big open subset $V_i\subset U_i$ such that $\op{Sing}(\F_i\subset T_{Y_{i-1}/k})\cap V_i=\emptyset$. Let $|U|\subset |X|$ be the intersection of these open sets $\bigcap_{i=0}^{n-1}|V_i|$ (we use the fact that purely inseparable morphisms are homeomorphisms), and let $U\subset X$ be the corresponding open subscheme. It is a big open subset of $X$ because it is a finite intersection of big open subsets. Thus, the power tower $Y_\bullet$ restricted to $U$ is regular. This finishes the proof.
\end{proof}

\begin{remark}
    For a power tower of infinite length, there could be no big open subset such that the tower restricted to it is regular. Indeed, we can sketch a construction of an $\infty$-foliation of rank $1$ on a plane $X=\mathbb{A}^2_k$ with a wild singular locus for $n=\infty$. We will do it below. Anyway, this means that whenever we have a power tower that is regular at a big open subset, then this power tower is somehow special.
    
    A $\infty$-foliation on $X$ is equivalent to a sequence of vector fields satisfying some conditions on the generic point, i.e., they are local generators for the Jacobson sequence. 
    
    A sketch of a construction. The main point is accumulating singular points from different $\F_i$. First, we take any vector field that defines a $1$-foliation $X\to Y_1$. We extend it to a $2$-foliation $X\to Y_2$. (We can do this by (future) Proposition \ref{always}.) This gives us a new vector field on $Y_1$. It \textit{should} be possible to choose one with a singularity. It \textit{should} be possible to decide where it is: if this singularity is shared with singularities from $\F_1$, then we can translate this singularity by a vector to a new point in a way that the translation of this vector field from $\F_2$ still defines a $2$-foliation, but maybe it is a new one. We can continue it for $\F_3,\F_4,\ldots$. At each step, we can choose a new point from the closed points of $X$ to add to the singular locus. In particular, we could produce a $\infty$-foliation on $X$ whose $\op{Sing}_\infty$ is an infinite set of closed points that is dense in $X$.

    The sketch is solid up to the explicit claims about how to work with power towers on $\mathbb{A}^2$ that we do not discuss here in this paper in that much detail. In particular, as I have not provided formal statements and proofs of those results, this sketch should be considered speculation rather than an announcement of results. As we always should do with sketches of arguments.
\end{remark}

The following result explains what being regular means in practice.

\begin{lemma}\label{regular 1-foliation = p-bases}
    Let $X$ be a normal variety over a perfect field $k$ of characteristic $p>0$ satisfying $k=K(X)^{p^\infty}$.
    Let $\F$ be a $p$-Lie algebra. Let $f:X\to Y\to X^{(1)}$ be corresponding to it $1$-foliation.

    Then, the $1$-foliation $f$ is regular if and only if $\cO_{X/k}$ locally admits a $p$-basis over $\cO_{Y/k}$ and $\cO_{Y/k}$ locally admits a $p$-basis over $\cO_{X^{(1)}/k}$.
\end{lemma}

\begin{proof}
    We recall and adapt Definition \ref{p-basis definition}~($p$-basis) to rings:
    
    A ring $A$ admits a $p$-basis over a subring $B\supset A^p$ if there are $x_1,\ldots,x_n\in A$ such that the map 
    \[
    \bigoplus_{0\le a_1,\ldots,a_n<p} \left(B\prod_{i=1}^n{x_i^{a_i}}\right)\to A
    \]
    is an isomorphism of $B$-modules, i.e., $A$ is free over $B$ and the monomials $\prod_{i=1}^n{x_i^{a_i}}$ for $0\le a_1,\ldots,a_n<p$ are a basis.

    We come back to the proof. 
    
    Of course, one way is easy. If we have the $p$-bases, then we can compute the rest, though we have to notice that in our case, smooth and regular are equivalent, because we are geometrically regular. (It is a special case of Kunz's theorem \cite[Theorem 2.1.]{kunz} to conclude that $A/k$ is regular from $A/A^p$ being flat, or having a $p$-basis.) In greater generality, we could not conclude regularity from $T_{K/k}$ being free (e.g., look at $k[x]/(x^2)$). 
    
    The other direction is hard. We work on local rings and then ``spread out'' the results to affine rings. We use Yuan's theorem \cite[12. Theorem]{Yuan-purely} to get $p$-bases, because this is precisely what that theorem does.

    This finishes the proof.
\end{proof}

\begin{prop}[Regular Power Tower in Practice]\label{regular power tower - meaning - prop}
    Let $X$ be a normal variety over a perfect field $k$ of characteristic $p>0$ satisfying $k=K(X)^{p^\infty}$. Let $Y_\bullet$ be a regular power tower on $X$, let $\F_\bullet$ be the Jacobson sequence corresponding to $Y_\bullet$, and let $\mathcal{D}$ be the saturated subalgebra of differential operators corresponding to $Y_\bullet$.

    Then, the following are true:
    \begin{enumerate}
        
        \item Every variety $X, Y_1,Y_2,\ldots$ is regular over $k$.
        \item The sheaf $\cO_{Y_i/k}$ is locally free over $\cO_{Y_{i+1}/k}$ for $i\ge 0$, and it is locally free over $\cO_{Y_{i-1}^{(1)}/k}$ for $i\ge 1$. Moreover, each of these inclusions admits a $p$-basis locally.
        \item The sheaf $\op{Diff}_{X^{(n)}}(X)$ is a vector bundle for $n\ge 0$, and $\op{Diff}_{k}(X)$ is a locally free quasicoherent sheaf.
        \item The subsheaf $\mathcal{D}_n\subset \op{Diff}_{X^{(n)}}(X)$ is a subbundle for $n\ge 0$.
        \item Let $i\ge 1$. The morphisms $f_i: X\to Y_i$ and $g_i: Y_i\to Y_{i-1}^{(1)}$ are flat, and every one of them admits a differential.
        \item The sheaves $\Omega_{Y_i/k}$, $T_{Y_i/k}$, $\Omega_{Y_i/Y_{i+1}}$, and $T_{Y_i/Y_{i+1}}=\F_{i+1}$ for $i\ge 0$ are vector bundles.
        \item The subsheaf $\F_{i+1}\subset T_{Y_i/k}$ for $i\ge 0$ is a subbundle, and the quotient is isomorphic to ${f'_{i+1}}^* \left(T_{Y_{i+1}/Y_i^{(1)}}\right)$, i.e., we have a short exact sequence of vector bundles:
        \[
        0\to \F_{i+1}\to T_{Y_i/k} \xrightarrow{d(Y_i/Y_{i+1})} {f'_{i+1}}^* \left(T_{Y_{i+1}/Y_i^{(1)}}\right) \to 0.
        \]
    \end{enumerate}
\end{prop}

\begin{proof}
    The lemma \ref{regular 1-foliation = p-bases} does all the heavy lifting. It gives us plenty of $p$-bases that we can use to compute every item of this theorem directly. (It is good to mention that one can reverse combinations of these conditions to get ``regular if and only if blah and blah'', but we do not need that.)

    Purely inseparable morphisms are finite, and all conditions can be checked locally, so we can work with affine schemes only. Let $X=\op{Spec}(A)$ and $Y_i=\op{Spec}(B_i)$ for $i\ge 0$. 
    
    Indeed, $Y_\bullet$ is a regular power tower, thus $X$ is regular, $T_{X/k}$ is a vector bundle, and $\F_1$ is a subbundle. 
    From Lemma $\ref{regular 1-foliation = p-bases}$, we conclude that locally there is a $p$-basis $t_1,\ldots,t_r$ for $A/B_1$ and a $p$-basis $l_1,l_2,\ldots,l_m$ for $B_1/A^p$. By restricting ourselves to open subsets, we can assume that these $p$-bases are exactly for these rings. Now, a trick: $l_1,\ldots,l_m,t_1^p,\ldots,t_r^p$ is a $p$-basis for $B_1/B_1^p$, so by Lemma \ref{regular 1-foliation = p-bases} applied to $X=\op{Spec}(B_1),Y=\op{Spec}(B_1^p)$, we get that $Y_1$ is regular, $T_{Y_1/k}$ is a vector bundle, and $\F_2$ is a subbundle. We can continue this reasoning over and over. This proves (1), (2), (5) by Criterion \ref{Criterion: admitting a differential}, (6).

    We prove (3), (4). Recall that Proposition \ref{diff=end - new}~(Diff=End). It says that the algebras $ \mathcal{D}_n, \op{Diff}_{X^{(n)}}(X)$ are actually isomorphic to sheaf homs that are locally given by $\op{Hom}_{B_n}(A, A)$ and $\op{Hom}_{A^{p^n}}(A, A)$. Locally, $A$ is free and finite over $B_n$ and $B_n$ is free and finite over $A^{p^n}$, this implies that these sheaves are vector bundles, and that the former is a subbundle of the latter. Finally, $\op{Diff}_{k}(X)=\bigcup_{n\ge 0} \op{Diff}_{X^{(n)}}(X)$ proves the final claim that it is quasi-coherent and locally free. (This is due to the fact how this filtration works. It adds more free variables with each $n$. Recall that $\mathbb{Q}=\bigcup_{n\ge 0}\frac{1}{n!}\Z $ is not free.)

    The final statement (7) is a result for a regular $1$-foliation $\F_{i+1}\subset T_{Y_i/k}$, so we can assume that we work with $X$ and $\F_1$. It is an equally arbitrary choice. Let $B=B_1$.

    We start with an exact sequence
    \[
    0\to \F_1\to T_{X/k} \xrightarrow{d(X/Y_1)} T_{Y_1/k}\otimes \cO_{X/k}.
    \]
    We want to compute the image of the last map.

    After applying Proposition \ref{diff=end - new}~(Diff=End), the above sequence is equal to
    \[
    0\to \op{Hom}_{B}(A,A)\to \op{Hom}_{A^p}(A,A) \to \op{Hom}_{B^p}(B,B)\otimes A.
    \]
    Since $A$ is flat over $B$, and $B$ over $A^p$. We observe that $\op{Hom}_{B^p}(B,A)=\op{Hom}_{B^p}(B,B)\otimes A$ and therefore the image is equal to $\op{Hom}_{A^p}(B,A)$. It is the local chart of ${f'_{1}}^* \left(T_{Y_{1}/X^{(1)}}\right)$. This finishes the proof.
\end{proof}

\begin{cor}[Unpacking for Regular Power Towers]\label{unpacking - corollary - regular power towers}
    Let $X$ be a normal variety over a perfect field $k$ of characteristic $p>0$ satisfying $k=K(X)^{p^\infty}$. Let $Y_\bullet$ be a regular power tower on $X$, let $\F_\bullet$ be the Jacobson sequence corresponding to $Y_\bullet$, and let $\mathcal{D}$ be the saturated subalgebra of differential operators corresponding to $Y_\bullet$.

    Then, the following equalities hold:
    \begin{align*}
        \mathcal{D} \cap T_{X/k} &= T_{X/Y_1}=\F_1,\\
        d(X/Y_1)(\mathcal{D})\cap T_{Y_1/k} &= T_{Y_1/Y_2}=\F_2,\\
        \ldots &= \ldots ,\\
        d(X/Y_n)(\mathcal{D})\cap T_{Y_{n}/k} &= T_{Y_n/Y_{n+1}}=\F_{n+1},\\
        \ldots &= \ldots ,
    \end{align*}
    We call the set of these equalities the unpacking of the saturated subalgebra $\mathcal{D}$ into the Jacobson sequence $\F_\bullet$.
\end{cor}

\begin{proof}
    First, the differentials for $f_i: X\to Y_i$ are admitted by Proposition \ref{regular power tower - meaning - prop}(Item 5).

    Second, it is enough to prove these equalities locally. Therefore, by Proposition \ref{regular power tower - meaning - prop}(Item 2), we can assume that all $Y_i$ for $i\ge 0$ are affine and free over each other. Let $X=\op{Spec}(A)$ and $Y_i=\op{Spec}(B_i)$ for $i\ge 0$.

    Let $m\ge n+1$ be a fixed integer. Then, we have, by Proposition \ref{diff=end - new}~(Diff=End),
    \begin{align*}
    \mathcal{D}_m =& \op{Hom}_{B_m}(A)=\op{Diff}_{B_m}(A)\\
    d(X/Y_n)(\mathcal{D}_m)=& \op{Hom}_{B_m}(B_n,A)  =\op{Diff}_{B_m}(B_n)\otimes A   \\
    d(X/Y_n)(\mathcal{D}_m)\cap T_{Y_n/k}=& \op{Hom}_{B_m}(B_n,A) \cap \op{Der}_{B_n^p}(B_n) = \op{Der}_{B_{n+1}}(B_n)=\F_{n+1}
    \end{align*}
    Finally, since $m$ is arbitrary, we have
    \[
    d(X/Y_n)(\mathcal{D})\cap T_{Y_n/k}=\F_{n+1}.
    \]
    This finishes the proof.
\end{proof}

\begin{prop}[Saturated Well-Chosen Generators for Regular Power Towers]\label{saturated dist generators - exist and do the job - prop}
    Let $X$ be a normal variety over a perfect field $k$ of characteristic $p>0$ satisfying $k=K(X)^{p^\infty}$. Let $Y_\bullet$ be a regular power tower on $X$, and let $\mathcal{D}$ be the saturated subalgebra of differential operators corresponding to $Y_\bullet$. Let $\op{Spec}(A)$ be an open subset of $X$ such that $A/k$ admits a set of coordinates. Let $\op{Spec}(B_i)=Y_{i} \restriction_{|\op{Spec}(A)|}$.

    Then there exists a set of well-chosen distinguished generators $G$, Definition \ref{distinguished generators - def}, for the subalgebra $\mathcal{D}\otimes K(X)$ such that
    \begin{itemize}
        \item $G \subset \D(\op{Spec}(A))$, i.e., these operators are operators $A\to A$,
        \item The set $\left\{
    \left(\prod_{m=1}^k\prod^{r_m}_{i=1} \left(G^m_i\right)^{b^m_i}\right): k\ge 0, 0 \le b^m_i \le p-1\right\}$ is a $A$-basis of $\mathcal{D}(\op{Spec}(A))$.
        \item the set $\{d(A/B_{m-1})(G^m_1),\ldots,d(A/B_{m-1})(G^m_{r_m})\}$ is a $B_{m-1}$-basis of $\F_{m}(\op{Spec}(B_{m-1}))$.
    \end{itemize}
    We call any such set $G$ by the name a set of \textbf{saturated well-chosen generators for the regular power tower} $Y_\bullet \restriction_{|\op{Spec}(A)|}$.
\end{prop}

\begin{proof}
    By Proposition \ref{regular power tower - meaning - prop}(The meaning of regularity in practice), we know that all the gadgets used in the proof of Theorem \ref{dist generators exist - prop}(Existence of Distinguished Generators) and Proposition \ref{description of any subalgebra with dist generators - prop} are available for regular power towers. Therefore, we can just repeat these proofs in this local affine chart. First, we construct a set $G$ in the same way, and for \textit{this} set we use the unpacking to conclude it is a basis, not just a generic basis. So, instead of separating the construction from the proposition, we do them together, and therefore, we do not need the counting dimensions argument at the end. We could do that for fields, but it is tricky to separate them for rings.
\end{proof}

\subsection{Frobenius Theorem for Power Towers}

\subsubsection{Context} Let $(X/\C,\F)$ be a smooth complex variety of dimension $n$ with a regular foliation $\F$ on it. The Frobenius theorem for $(X/\C,\F)$ says that locally in the analytic topology the foliation $\F$, of rank $n-m$, where $0\le m\le n$, is induced by the projection $\C^n \to \C^m$, i.e., for every closed point $x\in X$ there is an analytic open subset $x\in U\subset X$ such that $U$ is diffeomorphic to $\C^n$ and there is an analytic map $f:\C^n \to \C^m$ such that $\F_{|U}=T_f$. For a precise statement with a proof, one can consult \cite[Theorem 2.20]{Voisin1}. 

However, this version makes sense only over complex numbers. Nevertheless, it implies a formal statement that makes sense over any field of characteristic zero: $\F$ is algebraically integrable at the formal point $\widehat{x}$, i.e., there exist local coordinates $t_1,\ldots,t_n\in \widehat{\cO_{X,x}}$ such that:
\[
 \C[[t_1,\ldots,t_m]]\to \C[[t_1,\ldots,t_n]]\simeq \widehat{\cO_{X,x}} 
\]
induces $\F_{|\overline{x}}$. 
We will prove an analogous result for regular power towers. 

\subsubsection{Formal Frobenius Theorem for Power Towers}

\begin{thm}[Formal Frobenius Theorem for Regular Power Towers]\label{formal Frobenius Theorem for power towers- thm}
Let $X$ be a variety over an \emph{algebraically closed} field $k$ of characteristic $p>0$ such that $K(X)^{p^\infty}=k$. Let $d$ be the dimension of $X/k$. 
Let $Y_\bullet$ be a regular power tower on $X$.

Let $m\ge 0$ be an integer.
Let $x$ be a closed point on $X$. 
Let $y_m=f_m(x)$ be its image in $Y_m$.
Then, there exist coordinates $t_1,\ldots,t_d\in \cO_{X,x}$ at $x$ such that $f_m:X\to Y_m$ at the formal point $\widehat{x}$ is given by
\[
    \widehat{f_m}:\widehat{\mathcal{O}_{Y_m,y_m}}\simeq k[[t_1^{p^{a_1}},t_2^{p^{a_2}},t_3^{p^{a_3}},\ldots,t_d^{p^{a_d}}]]\to k[[t_1,t_2,t_3,\ldots, t_d]]\simeq \widehat{\mathcal{O}_{X,x}},
\]
where $0\le a_1\le a_2\le a_3 \le \ldots \le a_d \le m$, and $\sum_{i=1}^d a_i = N$, where we have $p^N=[K(X):K(Y_m)]$.

Moreover, for $m=\infty$, there exist coordinates $t_1,\ldots,t_d$ such that the function $\widehat{f_\infty}\coloneqq \bigcap_{m\ge 0} \widehat{f_m}$ is given by
\[
    \widehat{f_\infty}:\bigcap_{m\ge 0}\widehat{\mathcal{O}_{Y_m,y_m}}\simeq k[[t_1^{p^{a_1}},t_2^{p^{a_2}},t_3^{p^{a_3}},\ldots,t_d^{p^{a_d}}]]\to k[[t_1,t_2,t_3,\ldots, t_d]]\simeq \widehat{\mathcal{O}_{X,x}},
\]
where $0\le a_1\le a_2\le a_3 \le \ldots \le a_d \le \infty$. (Our convention is $t^{p^\infty}=0$.)
\end{thm}

\begin{proof}
    The proof has a simple structure: we choose any formal coordinates $t_1,\ldots,t_d\in \cO_{X,x}$ at $x$, and then we modify them to obtain coordinates satisfying the theorem using the sequence of natural maps
    \[
    \frac{m_x}{m_x^2} \leftarrow \frac{m_{y_1}}{m_{y_1}^2}
     \leftarrow \frac{m_{y_2}}{m_{y_2}^2}
      \leftarrow \ldots  \leftarrow \frac{m_{y_m}}{m_{y_m}^2}\leftarrow \ldots,
    \]
    where $y_i=f_i(x)\in Y_i$. Nevertheless, it requires some preparations beforehand. (We use the ``being regular'' assumption to know the dimensions of the vector spaces below, see Proposition \ref{regular power tower - meaning - prop}.)

    We introduce some notation. We put $y_0=x$, and, for $i=0,\ldots, m-1$, we denote the images and cokernels of the above maps (from cotangent complexes) by:
    \[
    \frac{m_{y_{i+1}}}{m_{y_{i+1}}^2}
     \twoheadrightarrow \mathbb{K}_{i+1} \hookrightarrow \frac{m_{y_i}}{m_{y_i}^2} \to \widetilde{\Omega_{Y_i/Y_{i+1},y_{i}}}\coloneqq \frac{m_{y_i}}{m_{y_i}^2+m_{y_{i+1}}}\to 0.
    \]
    In particular, the short exact sequence
    \[
    0 \to \mathbb{K}_{i+1} \to \frac{m_{y_i}}{m_{y_i}^2} \to \frac{m_{y_i}}{m_{y_i}^2+m_{y_{i+1}}}\to 0
    \]
    is a sequence of $k$-vector spaces.

    The map $f'_{i+1}:Y_i\to Y_{i+1}$ is a $1$-foliation, see Notation \ref{power tower on variety - notation}, so it factorizes $F_{Y_i/k}$. The map $g_{i}:Y_{i+1}\to Y^{(1)}_i$ induces a map
    \[
    \frac{m_{y_{i}^{(1)}}}{m_{{y_{i}}^{(1)}}^2}\twoheadrightarrow
    \mathbb{L}_{i+1}  \hookrightarrow \frac{m_{y_{i+1}}}{m_{y_{i+1}}^2},
    \]
    where $y_{i}^{(1)}=F_{Y_i/k}(y_i)$, and $\mathbb{L}_{i+1}$ is the image. The kernel of the surjection is ${m_{{y_{i}}^{(1)}}^2+m_{y_{i+1}^{(1)}}}$.

    We prove that the maps $\mathbb{K}_{i+1}\to  \mathbb{K}_{i}$, for $i=1,\ldots,m-1$, are surjective. Indeed, it follows from the equality of sheaves $\Omega_{Y_{i-1}/Y_{i}}=\Omega_{Y_{i-1}/Y_{i+1}}$. We can conclude this equality from comparing the cotangent complexes for $Y_{i-1}/Y_{i}$ and $Y_{i-1}/Y_{i+1}$, which define these sheaves. The image of $\mathbb{K}_{i+1}$ is the kernel of $\frac{m_{y_{i-1}}}{m_{y_{i-1}}^2} \to \frac{m_{y_{i-1}}}{m_{y_{i-1}}^2+m_{y_{i}}}$, so it is equal $\mathbb{K}_{i}$. Derivations kill $p$-powers.
    
    We define
    $
    \mathbb{K}_{i+1\to i}\coloneqq \op{Ker}(\mathbb{K}_{i+1}\twoheadrightarrow  \mathbb{K}_{i} ).$

    We observe that there is a natural function $\frac{m_{y_i}}{m_{y_i}^2}\to \frac{m_{y_i^{(1)}}}{m_{y_i^{(1)}}^2}$ given by the formula 
    \[
    t \op{ mod }m_{y_i}^2 \mapsto {\left(t \op{ mod }m_{y_i}^2\right)}^p=t^p+ \op{ mod } m_{y_i^{(1)}}^2.\] This is a bijection. 
    Moreover, we can combine it with the surjection $\frac{m_{y_{i}^{(1)}}}{m_{{y_{i}}^{(1)}}^2}\to
    \mathbb{L}_{i+1}$ to have a surjective function
    \[
    \frac{m_{y_i}}{m_{y_i}^2} \twoheadrightarrow \mathbb{L}_{i+1}.
    \]
    It is essential to point out that it is not $k$-linear. It is Frobenius twisted $k$-linear, so we can treat this function as a $k$-linear if necessary, since $k$ is perfect but will not be linear on the nose. 

    We have just proved that for each $i\ge 1$ we have the following diagram:
    \begin{center}
        \begin{tikzcd}
        \frac{m_{y_{i+1}}}{m_{y_{i+1}}^2}\ar[ddd]\ar[r, two heads] & \mathbb{K}_{i+1}\ar[ddd, bend right=45, dashed, two heads]\ar[r,hook] & \frac{m_{y_i}}{m_{y_i}^2}\ar[r]& \widetilde{\Omega_{Y_{i}/Y_{i+1},y_{i}}}\ar[r]& 0\\
        &\mathbb{K}_{i+1\to i}\ar[u,hook]\ar[r,hook,dashed]&\mathbb{L}_i \ar[u, hook]&&\\
        & & \frac{m_{y_{i-1}^{(1)}}}{m_{y_{i-1}^{(1)}}^2} \ar[u, two heads]& &\\
        \frac{m_{y_{i}}}{m_{y_{i}}^2}\ar[r, two heads] & \mathbb{K}_i\ar[ruu, "0", bend left=10, dashed]\ar[r,hook] & \frac{m_{y_{i-1}}}{m_{y_{i-1}}^2}\ar[u, dashed, "t \mapsto t^p"']\ar[r]& \widetilde{\Omega_{Y_{i-1}/Y_i,y_{i-1}}}\ar[r]\ar[luu, "\simeq'"', dashed]& 0,
        \end{tikzcd}
    \end{center}
    where  the composition $\mathbb{K}_{i}\to \mathbb{L}_i$ is zero, because if  $t\in \cO_{Y_{i+1}}$, then the class of $t^p$ is zero in $\frac{m_{y_{i}}}{m_{y_{i}}^2}$. Consequently, $\mathbb{K}_{i+1\to i}$ is inside $\mathbb{L}_i$, and there is a map $\widetilde{\Omega_{Y_{i-1}/Y_i,y_{i-1}}}\to \mathbb{L}_i$ that is a (twisted) isomorphism.

    We can link the above diagrams for unspecified $i$ into the following big diagram:

    \begin{center}
        \begin{tikzcd}
        \ldots\ar[d]&\ldots &\ldots& \ldots &\ldots\\
        \frac{m_{y_{4}}}{m_{y_{4}}^2}\ar[dd]\ar[r, two heads] & \mathbb{K}_4\ar[dd, bend right=45, two heads]\ar[r,hook] & \frac{m_{y_{3}}}{m_{y_{3}}^2}\ar[u, two heads]\ar[r]& \widetilde{\Omega_{Y_{3}/Y_4,y_{3}}}\ar[r]& 0\\
        &\mathbb{K}_{4\to3}\ar[u,hook]\ar[r,hook]&\mathbb{L}_3\ar[u,hook]\ar[uu, bend right=50, two heads]&&\\
        \frac{m_{y_{3}}}{m_{y_{3}}^2}\ar[dd]\ar[r, two heads] & \mathbb{K}_3\ar[dd, bend right=45, two heads]\ar[r,hook] & \frac{m_{y_{2}}}{m_{y_{2}}^2}\ar[u, two heads]\ar[r]& \widetilde{\Omega_{Y_{2}/Y_3,y_{2}}}\ar[r]& 0\\
        &\mathbb{K}_{3\to2}\ar[u,hook]\ar[r,hook]&\mathbb{L}_2\ar[u,hook]\ar[uu, bend right=50, two heads]&&\\
        \frac{m_{y_{2}}}{m_{y_{2}}^2}\ar[dd]\ar[r, two heads] & \mathbb{K}_2\ar[dd, bend right=45, two heads]\ar[r,hook] & \frac{m_{y_{1}}}{m_{y_{1}}^2}\ar[u, two heads]\ar[r]& \widetilde{\Omega_{Y_{1}/Y_2,y_{1}}}\ar[r]& 0\\
        &\mathbb{K}_{2\to1}\ar[u,hook]\ar[r,hook]&\mathbb{L}_1\ar[u,hook]\ar[uu, bend right=50, two heads]&&\\
          \frac{m_{y_{i}}}{m_{y_{i}}^2}\ar[r, two heads] & \mathbb{K}_1\ar[r,hook] & \frac{m_{x}}{m_{x}^2}\ar[u, two heads]\ar[r]& \widetilde{\Omega_{X/Y_1,x}}\ar[r]& 0,
        \end{tikzcd}
    \end{center}
    where the compositions $\mathbb{L}_i\to \mathbb{L}_{i+1}$ are surjective, because of the equality of sheaves $\Omega_{Y_{i+1}/X^{(i+1)}}=\Omega_{Y_{i+1}/Y_i^{(1)}}$ that follows from the condition $K(X)^{p^{i+1}}\cdot K(Y_{i+1})^p=K(Y_i)^p$ from Definition \ref{Power Tower - Definition}. Moreover, the kernel of $\mathbb{L}_i\to \mathbb{L}_{i+1}$ is equal to $\mathbb{K}_{i+1\to i}$ since $\mathbb{K}_i\to \mathbb{L}_i$ is zero, and their dimensions agree (see below).

    We can observe that the dimension of $\mathbb{K}_i$ is equal $d-r_i$, where $r_i$ is the dimension of $\mathbb{L}_i$ that is also the dimension of $\widetilde{\Omega_{Y_{i-1}/Y_i,y_{i-1}}}$. So, the dimension of $\mathbb{K}_{i+1\to i}$ is $r_{i+1}-r_i$.

    Finally, we can proceed to the proof.

    Let $t_{1,1},\ldots,t_{1,d}\in m_x\subset \cO_{X,x}$ be elements whose classes in $m_x/m_x^2$ form a basis of this vector space. Without loss of generality, we can assume that $t_{1,r_1+1},\ldots,t_{1,d}$ is a basis of $\mathbb{K}_1$.
    Therefore, $t_{1,1}^p,\ldots,t_{1,r_1}^p$ is a basis of $\mathbb{L}_1$.

    Let $b_{1},\ldots,b_{r_1-r_2}\in m_{y_2}$ be elements whose classes form a basis of $\mathbb{K}_{2\to 1}$. Without loss of generality, we can assume that $t_{1,1}^p,\ldots,t_{1,r_2}^p,b_{1},\ldots,b_{r_1-r_2}$ forms a basis of $\mathbb{L}_1$. The elements $b_j$ are $p$-powers in $\cO_{X,x}$, so we can take their $p$-roots in this ring. We denote $t_{2,j}\coloneqq b_j^{1/p}$.

    The elements $t_{1,1}^{p^2},\ldots,t_{1,r_2}^{p^2}$ form a basis for $\mathbb{L}_2$. 
    Let $b_{1},\ldots,b_{r_2-r_3}\in m_{y_3}$ be elements whose classes form a basis of $\mathbb{K}_{3\to 2}$.
    Without loss of generality, we can assume that $t_{1,1}^{p^2},\ldots,t_{1,r_3}^{p^2},b_{1},\ldots,b_{r_2-r_3}$ form a basis of $\mathbb{L}_2$.
    The elements $b_j$ are $p^{2}$-powers in $\cO_{X,x}$, 
    so we can take their $p^2$-roots in this ring. We denote $t_{3,j}\coloneqq b_j^{1/p^2}$.

    And so on and so on:

    The elements $t_{1,1}^{p^i},\ldots,t_{1,r_i}^{p^i}$ form a basis for $\mathbb{L}_i$. 
    Let $b_{1},\ldots,b_{r_i-r_{i+1}}\in m_{y_{i+1}}$ be elements whose classes form a basis of $\mathbb{K}_{i+1\to i}$.
    Without loss of generality, we can assume that $t_{1,1}^{p^i},\ldots,t_{1,r_{i+1}}^{p^i},b_{1},\ldots,b_{r_i-r_{i+1}}$ form a basis of $\mathbb{L}_i$.
    The elements $b_j$ are $p^{i}$-powers in $\cO_{X,x}$, 
    so we can take their $p^i$-roots in this ring. We denote $t_{i+1,j}\coloneqq b_j^{1/p^i}$.

    If our power tower is of finite length $m$, as it is in the first part of the theorem, then we continue the above procedure till $i=m$, and the following coordinates, in the opposite order,
    \[
    t_{m+1,1},\ldots,t_{m+1,r_m},
    t_{m,1},\ldots,t_{m+1,r_{m-1}-r_{m}},
    \ldots,
    t_{1,1},\ldots,t_{1,d-r_{1}}
    \]
    satisfy the theorem.

    If our power tower is of infinite length, then we do not stop the procedure. And, in this case, there is an integer $m\ge 1$ such that $Y_m,Y_{m+1},\ldots$ is a $\infty$-foliation of rank $r>0$ on $Y_m$, i.e., $r=r_{m+1}=r_{m+2}=r_{m+3}=\ldots>0$, because the degrees are nonincreasing by Proposition \ref{Nonincreasing Degrees Power Tower - Proposition}. By the finite length case for $m,m+1,m+2,\ldots$, we have that:
    \begin{align*}
     \widehat{f_m}: k[[t_{m+1,1}^{p^m},\ldots,t_{m+1,r}^{p^m},
    t_{m,1}^{p^{m-1}},\ldots,t_{1,d-r_{1}}]]&\to k[[t_{m+1,1},\ldots,t_{1,d-r_{1}}]], \\
    \widehat{f_{m+1}}: k[[t_{m+2,1}^{p^{m+1}},\ldots,t_{m+2,r}^{p^{m+1}},
    t_{m,1}^{p^{m-1}},\ldots,t_{1,d-r_{1}}]]&\to k[[t_{m+2,1},\ldots,t_{1,d-r_{1}}]], \\
    \widehat{f_{m+2}}: k[[t_{m+3,1}^{p^{m+2}},\ldots,t_{m+3,r}^{p^{m+2}},
    t_{m,1}^{p^{m-1}},\ldots,t_{1,d-r_{1}}]]&\to k[[t_{m+3,1},\ldots,t_{1,d-r_{1}}]], \\
    \ldots &\to \ldots .
    \end{align*}

    This implies that
    \[
    \bigcap_{i\ge 0} k[[t_{m+i,1}^{p^{m+i-1}},\ldots,t_{m+i,r}^{p^{m+i-1}},
    t_{m,1}^{p^{m-1}},\ldots,t_{1,d-r_{1}}]] = k[[
    t_{m,1}^{p^{m-1}},\ldots,t_{1,d-r_{1}}]],
    \]
    because only the first $r$ elements are modified each step.

    Consequently, we have
    \[
    \widehat{f_{\infty}}: k[[t_{m+1,1}^{p^{\infty}},\ldots,t_{m+1,r}^{p^{\infty}},
    t_{m,1}^{p^{m-1}},\ldots,t_{1,d-r_{1}}]]\to k[[t_{m+1,1},\ldots,t_{1,d-r_{1}}]].
    \]

    This finishes the proof.
\end{proof}

\begin{remark}
    In Theorem \ref{formal Frobenius Theorem for power towers- thm}, we assume that $k$ is algebraically closed instead of just perfect to avoid technicalities around variations of residue fields. I did not follow all the details for fields $k$, which are only perfect. Is the result correct in the form that we replace $k$ with its finite extension $k'=k(x)$, $k'$ is the common residue field for all $y_i$, and the proof goes the same way from there? Maybe, I don't know.
\end{remark}

\begin{remark}
    The function $\widehat{f_\infty}$ in Theorem \ref{formal Frobenius Theorem for power towers- thm} is a function from formal first integrals to all formal functions at the point $x$, i.e., it is of the same nature as the inclusion $\C[[t_1,\ldots,t_m]]\to \C[[t_1,\ldots,t_n]]\simeq \widehat{\cO_{X,x}} $ from the context paragraph. Indeed, let say that $d=3$ and $a_0=1,a_2=44, a_3=\infty$, then we get that the power tower at $x$ corresponds to
    \[
    k[[t_1^{p^{0}},t_2^{p^{44}},t_3^{p^\infty}]]=k[[t_1,t_2^{p^{44}}]]\to k[[t_1,t_2,t_3]] \simeq \widehat{\cO_{X,x}},
    \]
    and the elements in $k[[t_1,t_2^{p^{44}}]]$ are precisely the elements from $ k[[t_1,t_2,t_3]]$ that are constants for the subalgebra $\op{Diff}_{Y_\bullet}(X)_{| \overline{x}}$ (of \emph{continuous/formal} differential operators).
\end{remark}

\begin{question}
    Is the above theorem \ref{formal Frobenius Theorem for power towers- thm} true for every subalgebra of ``continuous/formal'' differential operators? Or is there an extra condition that needs to be added?

    I believe that all ``formal power towers'' are of the form from the above theorem. Is it actually true? If so, then, in the language of Sweedler's Galois theory \cite{Sweedler-modular}, we could say that every ``formal power tower'' is ``modular'', which is nice, because not every power tower on a field or on a variety is ``modular''. Some examples are not such.
\end{question}

\subsection{Stratifications Induced by Regular Power Towers}

The formal Frobenius theorem for power towers, Theorem \ref{formal Frobenius Theorem for power towers- thm}, says that every regular power tower on a variety $X/k$, where $k=\overline{k}$, at every closed point $x\in X(k)$ is locally in a formal neighborhood given by
\[
k[[t_1^{p^{a_1}},t_2^{p^{a_2}},t_3^{p^{a_3}},\ldots,t_d^{p^{a_d}}]]\to k[[t_1,t_2,t_3,\ldots, t_d]]\simeq \widehat{\mathcal{O}_{X,x}}.
\]
The tuple $(a_1,\ldots,a_d)$ is determined by the point $x$; however, it may differ at other points of $X$. Nevertheless, the function $x\mapsto (a_1,\ldots,a_d)$ is constructible, meaning that the sets of points with the same tuple are locally closed. This induces a stratification of the variety. This stratification may be nontrivial; an example is given later, Example \ref{not Ekedahl power tower}.

\begin{thm}\label{stratification exists - thm}
Let $X$ be a variety over an \emph{algebraically closed} field $k$ of characteristic $p>0$ such that $K(X)^{p^\infty}=k$. Let $d$ be the dimension of $X/k$. 
Let $Y_\bullet$ be a regular power tower on $X$.

Then, there exists a constructible function 
\begin{align*}
    G(Y_\bullet):X(k)&\to \{\infty,0,1,2\ldots\}^d\\
    x &\mapsto (a_1,\ldots,a_d),
\end{align*}
where the $d$-tuple $(a_1,\ldots,a_d)$ is the unique one satisfying:
\[
\widehat{f_\infty}: k[[t_1^{p^{a_1}},t_2^{p^{a_2}},t_3^{p^{a_3}},\ldots,t_d^{p^{a_d}}]]\to k[[t_1,t_2,t_3,\ldots, t_d]]\simeq \widehat{\mathcal{O}_{X,x}}
\]
with $0\le a_1\le a_2\le a_3 \le \ldots \le a_d \le \infty$, where the elements $t_1,\ldots, t_d$ are some coordinates at the point $x$ according to Theorem \ref{formal Frobenius Theorem for power towers- thm}.

In particular, the sets $X(k)_{(a_1,\ldots,a_d)}\coloneqq G(Y_\bullet)^{-1}((a_1,\ldots,a_d))$ are locally closed in $X(k)$, only finitely many of them are not empty, and their union is the set of all the closed points, i.e., $\bigcup_{a\in \{\infty,0,1,2\ldots,\}^d} X(k)_a=X(k)$.
\end{thm}

\begin{remark}[Avalanche]\label{avalanche - remark}
    Before the proof of \ref{stratification exists - thm}, we will explain the main idea behind the proof and why it is true. The actual proof executes this idea, but it does so with some machinery that might hide the idea.
    
    Let's say that we have a regular power tower of length $2$ with $d=4$, $r_1=3$, and $r_2=2$ on $k[x_1,x_2,x_3,x_4]$. Then it admits a set of saturated, well-chosen distinguished generators \ref{saturated dist generators - exist and do the job - prop}: 
    \begin{center}
        $3$ in degree $1$: $G_1,G_2,G_3$ and $2$ in degree $p$: $H_1$, $H_2$.
    \end{center}
    Let say that the leading form $[H_1]$ admits an invertible function in its formula in the coordinates, e.g., $H_1=1 \cdot\frac{1}{p!}\frac{\partial^p}{\partial x_4^p}$, and let say that $[H_2]$ does not, e.g., $H_2=x_1\cdot \frac{1}{p!}\frac{\partial^p}{\partial x_3^p}+\frac{\partial}{\partial x_1}$. \footnote{I do not claim that a power tower with such $H_1, H_2$ exists! This is simply an explanation of what is important to focus on.}

    Then, at every closed point $x$ outside $V(x_1)$ the reductions of $G_1,G_2,G_3,H_1,H_2$ form a set of distinguished generators for $\D_{|x}$. So, the tuple is $(0,0,1,2,2)$ at these points.

    However, if the point $x$ belongs to $V(x_1)$, then the reductions of $G_1,G_2,G_3,H_1,H_2$ are still generators of the algebra, but not necessarily distinguished generators. Indeed, the degree of the operator $H_2$ drops after assuming $x_1=0$, then $\left(\mathcal{D}_{|V(x_1)}\right)_1$ has greater rank than $\left(\mathcal{D}_1\right)_{|V(x_1)}$. 
    (In general, the $p$-filtration does not commute with a restriction to a subvariety!) 
    However, $H_1$ stays in degree $p$. Finally, by the proof, we will know that $\mathcal{D}_{|V(x_1)}$ admits a set of distinguished generators, so for the rank reasons it must be:
    \begin{center}
        $4$ in degree $1$: $G'_1,G'_2,G'_3, G'_4$ and $1$ in degree $p$: $H'_1$,
    \end{center}
    where most likely we can assume that $G'_1,G'_2,G'_3$ are reductions of $G_1,G_2,G_3$ and $H'_1$ is the reduction of $H_1$, and then there is an extra $G'_4$ that emerges somehow from $H_2$ modulo $x_1$ interacting with operators of order $<p$ taken modulo $x_1$ too. Consequently, the tuple is $(0,1,1,1,2)$ for all points $x \in V(x_1)$. 
    
    In this explanation, a part of one of the twos fell into the lower order after $x_1=0$. Such cascades explain all changes in the stratification. And, if the power tower's length is great, there could be many of them co-occurring, and we could call this phenomenon an avalanche. This is a picturesque explanation, I like!

    In a less poetic language, there is an order on partitions, our tuples. We can only get a smaller or equal partition after restricting our power tower to a subvariety, and whether we got a smaller one depends solely on the leading forms of saturated distinguished generators. It is independent of the generators chosen.
\end{remark}

\begin{proof}
        First, we prove that the function $G(Y_\bullet)$ is well defined. Indeed, let $x\in X(k)$ be a closed point. Then, according to Theorem \ref{formal Frobenius Theorem for power towers- thm}, there exist some coordinates $t_1,\ldots,t_d$ at $x$ such that we have
    \[
\widehat{f_\infty}: k[[t_1^{p^{a_1}},t_2^{p^{a_2}},t_3^{p^{a_3}},\ldots,t_d^{p^{a_d}}]]\to k[[t_1,t_2,t_3,\ldots, t_d]]\simeq \widehat{\mathcal{O}_{X,x}}.
    \]
    Let $m\ge 0$ be an integer. By observing that $\widehat{f_m}$ is given by
    \[
    \widehat{f_m}: \widehat{\mathcal{O}_{Y_m,y_m}}\simeq k[[t_1^{p^{\op{min}(m,a_1)}},\ldots,t_d^{p^{\op{min}(m,a_d)}}]]\to k[[t_1,\ldots, t_d]]\simeq \widehat{\mathcal{O}_{X,x}}
    \]
    and $k[[t_1^{p^{\op{min}(m,a_1)}},\ldots,t_d^{p^{\op{min}(m,a_d)}}]]$ is a ``composite'' of the complete rings $k[[t_1^{p^{a_1}},\ldots,t_d^{p^{a_d}}]]$ and $k[[t_1^{p^m},\ldots, t_d^{p^m}]]$,
    thus we conclude that the indices $a_i$ are determined by the dimensions of $k[[t_1,\ldots, t_d]]$ over $k[[t_1^{p^{\op{min}(m,a_1)}},\ldots,t_d^{p^{\op{min}(m,a_d)}}]]$ that are finite, and they do not depend on the coordinates. Therefore, the indices $a_i$ do not depend on the coordinates. This proves that the function $G(Y_\bullet)$ is well defined.

    Next, we prove that $\bigcup_{a\in \{\infty,0,1,2\ldots,\}^d} X(k)_a=X(k)$. This is a property of every function: the union of all inverse images is the whole domain.

    Finally, we proceed to prove that the function $G$ is constructible.
    
    Let $\D$ be a saturated subalgebra of differential operators corresponding to $Y_\bullet$, Theorem \ref{MainTheorem - var - theorem}. And let $\F_\bullet$ be its saturated Jacobson sequence. We know that $\D_n=\op{Diff}_{Y_n}(X)$, and that Unpacking \ref{unpacking - corollary - regular power towers} is true for $\D$. 

    Let $Z\subset X$ be a closed subscheme/subvariety.
    We will prove that a version of Unpacking works for the pullback $\D_{| Z}$ of $\D$ to $Z$. This will be enough to prove that the graded (commutative) algebra (!) $\op{gr}\D_{| Z}$ admits generators that look like distinguished ones. ($\D_{| Z}$ is not an algebra, just a module.) This will be enough to conclude that generically $\D_{| Z}$ allows only for one tuple, which means that $G$ is constructible, because outside the locus of vanishing of any of (the leading forms of) these generators, all is the same at every point.

    We can assume that $X,Y_1,\ldots$ are affine and given by rings $A,B_1,B_2,\ldots$. Let $Z$ be a subvariety of $X$ corresponding to a ring $C$. Let $N\ge 1$ be an integer.

    First, by the definition we have $\D_{N| Z}=\op{Diff}_{B_N}(A,A)\otimes C$. And, by Proposition \ref{pd envelope - prop}, we have that $\op{Diff}_{B_N}(A,A)\otimes C=\op{Diff}_{B_N}(A,C)$, $\op{Diff}_{k}(A,A)\otimes C=\op{Diff}_{k}(A,C)$, and the inclusion $\op{gr}\op{Diff}_{B_N}(A,C)\subset \op{gr}\op{Diff}_k(A,C)$ factorizes through ${S}(\F_1^*)^{*gr}$, where $\F_1=\mathcal{D}_N\cap \op{Der}_k(A,C)=\op{Der}_{B_N}(A,C)=\op{Der}_{B_1}(A,C)$.
    We define $\F_i\coloneqq \op{Der}_{B_i}(B_{i-1},C)$, for $i\ge 1$.
    By iterating Proposition \ref{pd envelope - prop}, we can obtain the following diagram:

    \hspace{-0.5cm}
    \begin{tikzcd}
        \op{gr}\op{Diff}_{B_N}(A,C) \ar[r] \ar[d,two heads]& {S}(\F_1^*)^{*gr} \ar[r]\ar[rd]& \op{gr}\op{Diff}_k(A,C)\\
        \op{gr}\op{Diff}_{B_N}(A,C)/(\F_1^*)\simeq \op{gr}\op{Diff}_{B_N}(B_1,C)\ar[r] \ar[d,two heads]& {S}(\F_2^*)^{*gr}\ar[r] \ar[rd]& {S}(\F_1^*)^{*gr}/(\F_1^*)\\
        \op{gr}\op{Diff}_{B_N}(B_1,C)/((\F_2)^*)\simeq \op{gr}\op{Diff}_{B_N}(B_2,C)\ar[r] \ar[d,two heads]& {S}(\F_3^*)^{*gr}\ar[r] \ar[dr]& {S}(\F_2^*)^{*gr}/((\F_2)^*)\\
        \ldots\ar[r]&\ldots \ar[r]&\ldots
    \end{tikzcd}
   
By Proposition \ref{diff=end - new}, we can check the surjectivity of the vertical maps before taking $\op{gr}$ and by replacing $\op{Diff}$ with $\op{End}$. The surjectivity is clear using $\op{End}$. 

The number $N$ was just a scaffolding to use $\op{End}$ instead of $\op{Diff}$. We can choose it in an arbitrary way. We will use it to get a basis for the whole $\op{gr}\D_{| Z}$.

Now, we can repeat the proof of Proposition \ref{saturated dist generators - exist and do the job - prop}, or \ref{dist generators exist - prop}: for each $m\ge 1$ we take a basis of $G_1^m,\ldots, G_{r_m}^m$ of $\F_m$, and we take their preimages from $\gamma_{p^{m-1}}\F_1 \subset \D_{N| Z}$ for $N>m$. This is possible because of the basic arithmetic of divided power rings.

The set $G=\{G_1^m,\ldots, G_{r_m}^m\}_{m\ge 1}$ is like a set of leading forms of a set distinguished generators (we can call any its preimage $G'$ with respect to grading ``\textbf{a set of distinguished generators $G'$ for $\D_{| Z}$}''), and by an ``Unpacking'' along the vertical lines they determine the dimensions of free modules $\F_i$, so
if we take them modulo another ideal, corresponding to $Z_1\subset Z$. All of them stay nonzero, then they will keep this role for $\D_{| Z_1}$ and its $\F_i$'s (in this proof, $\F_i$'s depended on $\D$ and $Z$). Generators will stay generators. Bases will stay bases. 

In particular, it is true for closed points, where we will get that $r_i=\op{dim}_k(\mathbb{L}_i)$, where $\mathbb{L}_i$ follows the notation from the proof of \ref{formal Frobenius Theorem for power towers- thm}. Consequently, there is an open subset, namely, $U\coloneqq Z\setminus \bigcup_{i,m} V(G^i_m)\subset Z$, where $V(G^i_m)$ is the vanishing locus of $G^i_m$, such that all points $x\in U(k)$ have the same value $G(x)$ that is determined by integers $r_1,r_2,\ldots$.
\end{proof}

\begin{question}
    Can we use power towers to study partitions? 
    
    Or partitions to study power towers?
\end{question}

\begin{prop}\label{the tuple for the constant - prop}
    In Theorem \ref{stratification exists - thm}, if the function $G(Y_\bullet)$ is constant, then the value depends only on the degrees $p^{r_{i+1}}=[K(Y_{i}): K(Y_{i+1})]$ for $i\ge 0$.

    In particular, the tuple is equal to $d-r_1$ of zeros, $r_1-r_2$ of ones, $r_2-r_3$ of threes, and so on till we reach for the first time the moment $r_m=r_{m+1}=\ldots$. Then, we have $r_{m-1}-r_{m}$ of $m-1$, and finally $r=r_m$ of $\infty$.

    Moreover, the same formula works for the generic value of the function $G$ on a subvariety $Z\subset X$, where we take $r_i$ to be the ranks of $\op{Der}_{Y_i}(Y_{i-1}, Z)$.
\end{prop}

\begin{proof}
    If the function is constant, then any set of saturated distinguished generators remains a set of saturated distinguished generators at every closed point, since the number of generators from each degree $p^m$ determines the tuple, thus the tuple is constant, and given by the values in the statement.

    The ``moreover'' part is proved in the same way by using ``a set of distinguished generators $G$ for $\D_{| Z}$'' from the proof of Theorem \ref{stratification exists - thm} instead of $G$ for $\D$.
\end{proof}

\begin{remark}[Visualization is... Natural?]
    The proof of Theorem \ref{stratification exists - thm} implies that every reduction $\op{D}_{|Z}$ of a saturated subalgebra $\op{D}$ of the algebra of differential operators $\op{Diff}_k(X)$ to a subscheme $Z$ of $X$, assuming that $\D$ corresponds to a regular power tower, admits a set of distinguished generators. In other words, the visualization \ref{fig:stairs} of a subalgebra $\op{Diff}_{W_\bullet}(K)$ is true for all modules $\op{Diff}_{Y_\bullet}(X,Z)$. This seems to be a deep fundamental property of ``regular $\op{Diff}$ objects''. I do not know how far this observation goes, and I do not know how to formulate it better. 
    
    However, it is a promising observation because power towers are close to being infinitesimal groupoids, and this observation strengthens that. It shows that structures on $\op{Diff}$ pullback quite well. Likely, power towers are avatars of something more categorical.
\end{remark}

\begin{remark}
    We can combine Theorem \ref{codim 2 regular power tower finite length- prop}: every purely inseparable morphism is a regular power tower away from a codim $2$ subset (that depends on the morphism), and Theorem \ref{stratification exists - thm}: regular power towers induce stratifications. The result is:
    \begin{center}
        Every purely inseparable morphism induces a stratification of its source away from a codim $2$ subset that depends on the morphism. Each piece is determined by the local structure of the purely inseparable morphism according to the formal Frobenius theorem for PTs \ref{formal Frobenius Theorem for power towers- thm}
    \end{center}
\end{remark}

\subsection{Ekedahl Power Towers}
In this section, we study saturated subalgebras of differential operators that remain saturated after taking the gradation with respect to orders. Not all subalgebras satisfy this; therefore, we introduce names for those that do. This leads to gr-saturated subalgebras, and their more loose version we call \textit{Ekedahl power towers}. We provide criteria. We compute not-examples. 

\begin{example}\label{simple example - grs-aturated fails - example}
    We provide a simple example of how gradation can disturb being saturated.
    
    Let $G=\Z e_1\oplus \Z e_2$ be an abelian group. It admits a filtration $F_0 G \coloneqq 0$, $F_1 G\coloneqq\Z e_1\oplus 0 $, $F_2 G\coloneqq \Z e_1\oplus \Z e_2$.
    
    A free subgroup $M=\Z (e_1+2e_2) \subset G$ is saturated in $G$, however its gradation with respect to the filtration $F$ is equal $\op{gr} M=2[e_2]$, which is not saturated inside $\op{gr} G=\Z [e_1]\oplus \Z [e_2]$.
\end{example}

\subsubsection{Definition and Criteria}

\begin{defin}\label{gr-saturated - def}
    Let $X$ be a normal variety over a perfect field $k$ of characteristic $p>0$ satisfying $K(X)^{p^\infty}=k$.
    Let $Y_\bullet$ be a power tower on $X$, and let $\mathcal{D}$ be the saturated subalgebra of differential operators that corresponds to it.

    We say that the power tower $Y_\bullet$ is \textbf{\emph{gr-saturated}} if the sheaf $\op{gr}\mathcal{D}$ is saturated in $\op{gr}\op{Diff}_k(X)$, i.e., the quotient $\op{gr}\op{Diff}_k(X)/\op{gr}\mathcal{D}$ is torsion-free.
\end{defin}

\begin{example}
    Every $1$-foliation is gr-saturated. 
    If the length of a power tower is at least $2$, then it may not be gr-saturated. 
    
    Therefore, being gr-saturated is a meaningful notion only for exponents $>1$.
\end{example}

\begin{defin}\label{Ekedahl - def}
    Let $X$ be a normal variety over a perfect field $k$ of characteristic $p>0$ satisfying $K(X)^{p^\infty}=k$.
    Let $Y_\bullet$ be a power tower on $X$. 

    We say that the power tower $Y_\bullet$ is an \textbf{\emph{Ekedahl power tower}} on $X$ if there is a big open subset $U\subset X$, i.e., $X\setminus U$ is of codimension at least $2$, such that the power tower $Y_\bullet \cap U$ on $U$ is gr-saturated. 
\end{defin}

\begin{remark}
    An Ekedahl $n$-foliation is an $n$-foliation that is an Ekedahl power tower.
\end{remark}

\begin{remark}\label{interpretation - ekedahl foli of height n - remark}
    We named power towers on a variety $X$ that are gr-saturated away from a codimension $\ge 2$ closed subset \emph{``Ekedahl''} to honor Torsten Ekedahl's contributions to the theory of $n$-foliations, mostly present in the paper \cite{EkedahlFoliation1987}. 
    In that paper, he defined ``foliations of height $n$'' on $X$, which we can reinterpret in our vocabulary in the following way:
    \begin{center}
      Ekedahl's ``foliations of height $n$'' on $X$ = gr-saturated regular $n$-foliations on $X$,
    \end{center} and he defined ``foliations of height $\infty$'' to be gr-saturated regular $\infty$-foliations on $X$. Our Ekedahl power towers are a slight generalization of his notion that is compatible with how he uses it in his proofs. Indeed, in his paper, he primarily works with ``singular foliations of height $n$'' defined as purely inseparable morphisms that are ''foliations of height $n$'' away from a codimension $\ge 2$ closed subset. Therefore, we took his working definition and put it into a proper definition, in which we dropped the assumption of being an $n$-foliation or regular. The result is a notion that is much more versatile.

    So, the adjectives \textit{regular} and \textit{gr-saturated} are separated, though they will often come together.
\end{remark}

\begin{lemma}\label{n-foli torsion sheaves - lemma}
    Let $X$ be a normal variety over a perfect field $k$ of characteristic $p>0$ satisfying $K(X)^{p^\infty}=k$.
    Let $Y_\bullet$ be a $n$-foliation on $X$. Let $\F_\bullet$ be its Jacobson sequence.

    Then $\frac{T_{Y_i/k}}{\F_{i+1}+\G_i}$ for $i=1,2,\ldots, n-1$ are torsion sheaves, where we put $\G_i\coloneqq T_{Y_i/Y_{i-1}^{(1)}}$.
\end{lemma}

\begin{proof}
    We prove that the sheaves $\frac{T_{Y_i/k}}{\F_{i+1}+\G_i}$ for $i=1,2,\ldots, n-1$ are always torsion for $n$-foliations. Indeed, we have
    \[
    \left(\F_{i+1}\otimes K(Y_i)\right) \oplus \left(\G_i\otimes K(Y_i) \right)= T_{K(Y_i)/k},
    \]
    so $\frac{T_{Y_i/k}}{\F_{i+1}+\G_i}\otimes K(Y_i)=0$, i.e., it is a torsion sheaf. The equality follows from two facts. First, on the generic point $K(Y_i)$, the vector spaces $\F_{i+1}\otimes K(Y_i) $ and $ \G_i\otimes K(Y_i)$ are disjoint by Definition \ref{Jacobson Sequence - definition}. Second, the ranks of the sheaves $\F_{i+1}$ and $\G_i$ are complementary, because $Y_\bullet$ is a $n$-foliation, Definition \ref{n-foliation}.
\end{proof}

\begin{prop}[Criteria for Being gr-Saturated/Ekedahl]\label{saturated in terms of Jacobson sequences- Proposition}\label{criterion for being ekedahl}
Let $X$ be a normal variety over a perfect field $k$ of characteristic $p>0$ satisfying $K(X)^{p^\infty}=k$.
    Let $Y_\bullet$ be a \textit{regular} power tower on $X$. 
    Let $\F_\bullet$ be the Jacobson sequence corresponding to $Y_\bullet$.
    Let $\mathcal{D}$ be the saturated subalgebra corresponding to $Y_\bullet$. Then:

    \begin{enumerate}
     \item The following are equivalent:
    \begin{enumerate}
        \item The power tower $Y_\bullet$ is gr-saturated.
        \item For any/every set $G$ is a set of saturated well-chosen distinguished generators for $\mathcal{D}$ (on an open subset around any point of $X$) the submodule spanned by  $[G^m_j]$ inside $\op{gr}\op{Diff}_k(X)$ is saturated, i.e., all the submodules, for $m\ge 1$,
        \[
        \F'_m\coloneqq \op{Span}_A([G^m_1],\ldots,[G^m_{r_m}])\subset \op{gr}\D.
        \]
        are saturated. Equivalently, each $\F'_m$ is saturated in $\left(\gamma_{p^{m-1}}(\F_1)\otimes A\right)$.
        \item The stratification from Theorem \ref{stratification exists - thm} induced by $Y_\bullet$ on $X$ is constant, i.e., the image of $G(Y_\bullet)$ is one point.
        \item All the subsheaves 
    \[
    d(Y_i/X^{(i)})(\F_{i+1})\subset T_{X^{(i)}/k} \otimes \cO_{Y_i/k}
    \]
    are saturated for $i\ge 1$.
        \item All the subsheaves
    \[
    d(Y_i/Y^{(1)}_{i-1})(\F_{i+1})\subset d(Y_i/Y^{(1)}_{i-1})(T_{Y_i/k})=\F_i^{(1)}\otimes \cO_{Y_i/k}
    \]
    are saturated for $i\ge 1$.
    \end{enumerate}
    \item Let $Y_\bullet$ be a $n$-foliation, $n<\infty$. Let $\G_i\coloneqq T_{Y_i/Y_{i-1}^{(1)}}$ for $i\ge 1$.
    \begin{enumerate}
        \item The $n$-foliation $Y_\bullet$ is gr-saturated if and only if
        the quotient sheaves $\frac{T_{Y_i/k}}{\F_{i+1}+\G_i}$ for $1\le i<n$ are zero.
        \item The $n$-foliation $Y_\bullet$ is an Ekedahl power tower
    if and only if
    the quotient sheaves $\frac{T_{Y_i/k}}{\F_{i+1}+\G_i}$ for $1\le i<n$ are torsion with support of codimension at least $2$.
    \end{enumerate}
    \end{enumerate}
\end{prop}

\begin{proof}
    We prove (1). The main difficulty is getting that (a) is equivalent to (b). After that, the rest will just unpack itself.

    All conditions are local - it is checking torsion, so we can assume that all $X, Y_i$ are affine, given by rings $A, B_i$, and the local distinguished generators are given on these rings already, not only on their localizations. Let $G$ be any set of saturated well-chosen distinguished generators for $\mathcal{D}$, Proposition \ref{saturated dist generators - exist and do the job - prop}.

    We prove that if (a) then (b). Indeed, if $\D$ is gr-saturated, then it is true for submodules spanned by elements from $[G]$ from different orders, because 
    \[
    \F'_m\coloneqq \op{Span}_A([G^m_1],\ldots,[G^m_{r_m}])=\left(\gamma_{p^{m-1}}(\F_1)\otimes A\right) \cap \op{gr}\D.
    \]  
    which is an intersection of two saturated submodules of $\op{gr}\op{Diff}_k(X)$.

    We prove that if (b) then (a). We know that the modules $\op{Span}_A([G^m_1],\ldots,[G^m_{r_m}])$ are saturated in $\op{gr}\op{Diff}_k(X)$, so they are saturated in $\left(\gamma_{p^{m-1}}(\F_1)\otimes A\right)$, because it is a saturated submodule containing the span. Thus, if we write each of $[G^m_{r_j}]$ in terms of some coordinates $x_1,\ldots,x_d$ for $A/k$, then, maybe after a $A$-base change of $\F'_m$, we can assume that $[G^m_{r_j}]$ has $\frac{1}{p^{m-1}!}\frac{\partial^{p^{m-1}}}{\partial x_{c(j)}^{p^{m-1}}}$ in its expansion with a coefficient from $k$ that is not zero, and other $[G^m_{r_i}]$, $i\ne j$, do not have that element in their formulas, so $j\mapsto c(j)$ will be injective. 
    
    Consequently, we can use that to conclude that leading forms of all monomials in $[G^i_j]$ with exponents from $0$ up to $p-1$, that are a basis of $\op{gr}\D$ by Proposition \ref{saturated dist generators - exist and do the job - prop}, include products of $\frac{1}{p^{m-1}!}\frac{\partial^{p^{m-1}}}{\partial x_{c(j)}^{p^{m-1}}}$ with exponents from $0$ up to $p-1$ up to invertible function in their explicit formulas in terms of the coordinates, Corollary \ref{explicit formulas calculus - cor} and Proposition \ref{Determination1}. Therefore, there cannot be any graded operator $D\in \op{gr}\op{Diff}_k(A)$, that after multiplying by non-invertible $a\in A$ gets into $\op{gr}\D$ without being there already, because that would mean that its coefficient in front of that product of $\frac{1}{p^{m-1}!}\frac{\partial^{p^{m-1}}}{\partial x_{c(j)}^{p^{m-1}}}$ would be from $\frac{1}{a}A\setminus A$, which cannot be true.
    
    We prove (b) iff (c). The changes of values of $G(Y_\bullet)$ happen, i.e., not-(c), if and only if a set $G$ stops being a set of distinguished generators after reducing it to a subscheme. This ``stop'' happens if and only if the rank of $\F'_i$ drops after restriction, which is precisely not-(b). This proves not-(c) iff not-(b).

    We prove (b) iff (d) iff (e). All these conditions in terms of subsheaves are equivalent, because there are some explicit isomorphisms around.

    First, we focus on (b) iff (d). Let $T'_{X^{(m)}/k}=\left(\gamma_{p^{m}}(T_{X/k})\otimes A\right)$, i.e., these are $\F'_i$ for the power tower of $p$-powers, Example \ref{p-power power tower}. We have
    $
    \F'_m \subset T'_{X^{(m)}/k} 
    $
    and $\left(\gamma_{p^{m}}(T_{X/k})\otimes A\right)$ is saturated in $\op{gr}\op{Diff}_k(A)$. So, we have (b) iff $\F'_m$ is saturated in $T'_{X^{(m)}/k}$.  Now, we have an isomorphism
    \[
    d(X/X^{(m)}):\frac{T'_{X^{(m)}/k}}{\F'_{m+1}} \simeq \frac{ T_{X^{(m)}/k} \otimes \cO_{X/k}}{d(Y_m/X^{(m)})(\F_{m+1})}
    \]
    that is induced by Packing, see Remark \ref{packing}:
    \begin{center}
        \begin{tikzcd}
            \F'_{m+1} \ar[r]\ar[d, "\simeq"]&  T'_{X^{(m)}/k}\ar[d, "\simeq"]\\
            \F_{m+1} \otimes \mathcal{O}_{X/k} \ar[r] & T_{X^{(m)}/k} \otimes \mathcal{O}_{X/k},
        \end{tikzcd}
    \end{center}
    so the left-hand side is torsion if and only if the right-hand side is torsion. 
    
    We prove (d) iff (e). We take 
    $\F_{i+1}\subset T_{Y_i/k}$. We apply $d(Y_i/Y^{(1)}_{i-1})$ to it. We get
    \[
    d(Y_i/Y^{(1)}_{i-1})(\F_{i+1}) \to d(Y_i/Y^{(1)}_{i-1})(T_{Y_i/k}) \to T_{Y^{(1)}_{i-1}/k} \otimes \cO_{Y_i/k}.
    \]
    By Proposition \ref{regular power tower - meaning - prop}, we have that $d(Y_i/Y^{(1)}_{i-1})(T_{Y_i/k})=T_{Y^{(1)}_{i-1}/Y_i^{(1)}}\otimes \cO_{Y_i/k}=\F_i^{(1)}\otimes \cO_{Y_i/k}$. Since we know that $\F_1$ is saturated in $T_{X=Y_0/k}$, thus by an induction we can proceed to conclude that $d(Y_i/Y^{(1)}_{i-1})(T_{Y_i/k})=\F_i^{(1)}\otimes \cO_{Y_i/k} \subset T_{Y^{(1)}_{i-1}/k} \otimes \cO_{Y_i/k}$ is a saturated subsheaf. Finally, $d(Y^{(1)}_{i-1}/X^{(i)})\otimes \cO_{Y_i/k}$ is injective on $\F_i^{(1)}\otimes \cO_{Y_i/k} $ and the image is saturated in $T_{X^{(i)}/k} \otimes \cO_{Y_i/k}$, so $d(Y_i/Y^{(1)}_{i-1})(\F_{i+1})$ is saturated in $\F_i^{(1)}\otimes \cO_{Y_i/k}$ if and only if $d(Y_i/X^{(i)})(\F_{i+1})$ is saturated in $T_{X^{(i)}/k} \otimes \cO_{Y_i/k}$.
    
    We prove (2). The quotient sheaves are torsion by Lemma \ref{n-foli torsion sheaves - lemma}.

    (a) implies (b), because $Y_\bullet$ is Ekedahl if and only if $\frac{T_{Y_i/k}}{\F_{i+1}+\G_i}$ is zero away from codimension $2$ subset. So, after leaving that subset, it is enough to prove that $Y_\bullet$ is gr-saturated if and only if $\frac{T_{Y_i/k}}{\F_{i+1}+\G_i}$ is zero. 
    
    We can use (1): (a) iff (e) that we already proved.

    Indeed, $Y_\bullet$ is gr-saturated if and only if
    $d(Y_i/Y^{(1)}_{i-1})(\F_{i+1})\subset\F_i^{(1)}\otimes \cO_{Y_i/k}$ are saturated for $i=1,2,\ldots,n-1$. These sheaves are of the same rank. This means they must be equal. However, the quotient satisfies:
    \[
    \frac{\F_i^{(1)}\otimes \cO_{Y_i/k}}{d(Y_i/Y^{(1)}_{i-1})(\F_{i+1})}=\frac{d(Y_i/Y^{(1)}_{i-1})(T_{Y_i/k})}{d(Y_i/Y^{(1)}_{i-1})(\F_{i+1})}\simeq \frac{T_{Y_i/k}}{\F_{i+1}+\op{ker}d(Y_i/Y^{(1)}_{i-1})\cap T_{Y_i/k}}=\frac{T_{Y_i/k}}{\F_{i+1}+\G_i}.
    \]
    So, the equality holds if and only if $\frac{T_{Y_i/k}}{\F_{i+1}+\G_i}$ is zero. This finishes the proof.
\end{proof}

\begin{prop}[Ekedahl's ``foliations of height $n$'']\label{gr-saturared + regular = pd envelope - lemma}
    Let $X/k$ be a \emph{regular} variety over a perfect field $k$ of characteristic $p>0$ satisfying $K(X)^{p^\infty}=k$. Let $Y_\bullet$ be an $n$-foliation on $X$, $n=1,2,\ldots,\infty$. Let $\F_\bullet$ be its Jacobson sequence and let $\mathcal{D}$ be its saturated subalgebra. We assume that the $1$-foliation $\F_1$ is regular.

    Then, the $n$-foliation is gr-saturated and regular if and only if the equality 
    \[
    \op{gr}\mathcal{D}={S}({\F_1}^*)^{*gr} \cap \op{gr}\op{Diff}_{X^{(n)}}(X)
    \]
    holds. 

    In particular, the $\infty$-foliation is gr-saturated and regular if and only if the equality 
    \[
    \op{gr}\mathcal{D}={S}({\F_1}^*)^{*gr}
    \]
    holds. 
\end{prop}

\begin{proof}
    Let it be gr-saturated and regular. For $n$-foliations, we always have $\op{gr}\mathcal{D}\subset {S}({\F_1}^*)^{*gr} \cap \op{gr}\op{Diff}_{X^{(n)}}(X)$, see Proposition \ref{pd envelope - prop}. And the equality is true on the generic point: $\op{gr}\mathcal{D}\otimes K(X)= {S}(({\F_1}\otimes K(X))^*)^{*gr} \cap \op{gr}\op{Diff}_{X^{(n)}}(X)\otimes K(X)$, Proposition \ref{dist generators for n-foli - prop}. So, if it is gr-saturated, then they are equal.

    We prove that the equality implies being gr-saturated and regular.
    We show it for the finite $n$ - the case of $\infty$ follows from it.

    Being gr-saturated is easy, both  ${S}({\F_1}^*)^{*gr}$ and $\op{gr}\op{Diff}_{X^{(n)}}(X)$ are saturated in $\op{gr}\op{Diff}_{k}(X)$. The first is because it is a divided power polynomial subring, and the second is almost by definition. So, the harder part is to get regularity.

    We can work affine locally.
    We know $\F_1$ is regular. Let assume that we enough local that we can use Lemma \ref{regular 1-foliation = p-bases} directly: Let $x_1,\ldots, x_r$, and $y_1,\ldots,y_d$ be $p$-bases fo $A$ over $B$ and $B$ over $A^p$, where $X=\op{Spec}(A)$, and $Y_1=\op{Spec}(B)$. In particular, $X$ and $Y_1$ are regular. Together, $x_i,y_j$ are coordinates for $A/k$.

    Let $G=\{G^m_1,\ldots,G^m_r\}_{m=1,\ldots, n}\subset \D$ be a choice of operators such that $G^1_i=\frac{\partial}{\partial x_i}$ is a basis of $\F_1$, and $[G^m_i]=\gamma_{p^{m-1}}[G^1_i]$.

    By Proposition \ref{regular power tower - meaning - prop}, we know $d(X/Y_1)$ exists. We can unpack: $d(X/Y_1)\D=\op{Diff}_{Y_n}(Y_1)\otimes A$. In particular, $d(X/Y_1)(G^2_i)\in\F_2=d(X/Y_1)\op{Span}_A(G^2_1,\ldots,G^2_r)\cap T_{Y_1/k}$. And, $\frac{\partial}{\partial y_j}$ restricts to be derivations on $Y_1$. They are a basis of $T_{Y_1/X^{(1)}}$. Consequently, $\F_2\oplus T_{Y_1/X^{(1)}}=T_{Y_1/k}$, i.e., $\F_2$ is a regular $1$-foliation. Moreover, we have
    \begin{align*}
    \op{Diff}_{Y_n}(Y_1)\otimes A=d(X/Y_1)\D&=d(X/Y_1)({S}({\F_1}^*)^{*gr} \cap \op{gr}\op{Diff}_{X^{(n)}}(X))\\
    &={S}({\F_2}^*)^{*gr}\otimes A \cap \op{gr}\op{Diff}_{Y_1^{(n-1)}}(Y_1))\otimes A   \\
    &=({S}({\F_2}^*)^{*gr} \cap \op{gr}\op{Diff}_{Y_1^{(n-1)}}(Y_1)))\otimes A,
    \end{align*}
    but $A$ is faithfully flat over $B$, so $\op{Diff}_{Y_n}(Y_1)={S}({\F_2}^*)^{*gr} \cap \op{gr}\op{Diff}_{Y_1^{(n-1)}}(Y_1)$. Consequently, we can repeat the reasoning for $\F_3$, then for $\F_4$, and so on till $\F_n$. This proves that $Y_\bullet$ is regular.
\end{proof}

\subsubsection{Not-Ekedahl Power Tower}

In this section, we describe a $2$-foliation that is not Ekedahl. Moreover, we draw pictures to visualize this $2$-foliation. 

\begin{example}[A non-Ekedahl $2$-foliation on a variety.]\label{not Ekedahl power tower}
    Let $k$ be a perfect field of characteristic $p>0$.
Let $X=\mathbb{A}^2_k=\op{Spec}(k[x,y])$.
    
    Let $Y_1=\op{Spec}(k[x^p,y])$, $Y_2=\op{Spec}( k[x^{p^2}, y^{p}, y^{p+1}-x^p])$.

    Then $X,Y_1,Y_2,Y_2,Y_2,\ldots$ is a $2$-foliation on $X$, and its Jacobson sequence is:
    \begin{center}
        $\F_1$ is spanned by the vector field $\frac{\partial}{\partial x}$ on $X$,

        $\F_2$ by spanned by the vector field $y^p\frac{\partial}{\partial x^p}+\frac{\partial}{\partial y}$ on $Y_1$, 
        
        and $T_{Y_1/k}$ is spanned by $\frac{\partial}{\partial x^p}$ and $\frac{\partial}{\partial y}$.

        Also, $\G_1 = T_{Y_1/X^{(1)}}$ is spanned by the vector field $\frac{\partial}{\partial y}$ on $Y_1$.
    \end{center}
    By the definition of a Jacobson sequence, we have that the vector fields $y^p\frac{\partial}{\partial x^p}+\frac{\partial}{\partial y}$ and $\frac{\partial}{\partial y}$ on $Y_1$ are perpendicular on the generic point. However, this is not true at every closed point of $Y_1$. 
    Indeed, on the vanishing locus $V(y)$, i.e., if $y=0$, then we have
    $y^p\frac{\partial}{\partial x^p}+\frac{\partial}{\partial y}=\frac{\partial}{\partial y}$, so they are equal and therefore parallel on this locus. Now, we can observe that the locus where these two generators are parallel to each other is precisely $\frac{T_{Y_1/k}}{\F_{2}+\G_1}$, so in our case, its support is of codimension $1$. Therefore, by Proposition \ref{criterion for being ekedahl}, the $2$-foliation $f_2: X\to Y_2$ is not Ekedahl.

    The figures \ref{fig:non-example Ekedahl1} and \ref{fig:non-example Ekedahl2} contain pictures (These pictures were made with Wolfram Cloud \nolinkurl{https://www.wolframcloud.com/}.) of the above vector fields with an extra image for the support of $\frac{T_{Y_1/k}}{\F_{2}+\G_1}$.

    Furthermore, it is worth to notice that the sheaves $T_{Y_1/k},\F_2,\G_1$, and $\frac{T_{Y_1/k}}{\F_{2}+\G_1}$ fit the following diagram, where we put $R=k[x^p,y]$.

    \begin{center}
        \begin{tikzcd}
            &0&0&0&\\
            0\ar[r]&R\ar[r]\ar[u]&R\ar[r]\ar[u]&R/y^p R=\frac{T_{Y_1/k}}{\F_2+\G_1}\ar[r]\ar[u]&0\\
            0\ar[r]& R=\G_1\ar[r,"1\mapsto (0; 1)"]\ar[u]&R\frac{\partial}{\partial x^p}\oplus R\frac{\partial}{\partial y}=T_{Y_1/k}\ar[r]\ar[u]&R\ar[r]\ar[u]&0\\
            0\ar[r]&0\ar[r]\ar[u]&R=\F_2\ar[r]\ar[u,"1\mapsto (y^p; 1)"']&R\ar[r]\ar[u]&0\\
            &0\ar[u]&0\ar[u]&0\ar[u]&
        \end{tikzcd}
    \end{center}

    Finally, if $k$ is algebraically closed, then the stratification from Theorem \ref{stratification exists - thm} for this power tower is given by
    \begin{align*}
        X(k)_{(0,2)} &= \{(x,y): \ y \ne 0  \}\subset k\times k = X(k),\\
        X(k)_{(1,1)} &= \{(x,y): \ y=0 \}\subset k\times k = X(k).
    \end{align*}

    In other words, the function $G(Y_\bullet):k^2\to \{\infty,0,1,2,\ldots\}^2$ is given by
    \[
    G(Y_\bullet)(x,y) = \begin{cases}
        (0,2) , \ \text{if $y\ne 0$}\\
        (1,1) , \ \text{if $y=0$}
    \end{cases}.
    \]

    Consequently, it is not constant, which is another reason why it is not an Ekedahl power tower, Proposition \ref{criterion for being ekedahl}.
    
\FloatBarrier
\begin{figure}%
    \centering
    \subfloat[\centering $\F_1$: $\frac{\partial}{\partial x}$ on $X$]{{\includegraphics[width=5.5cm]{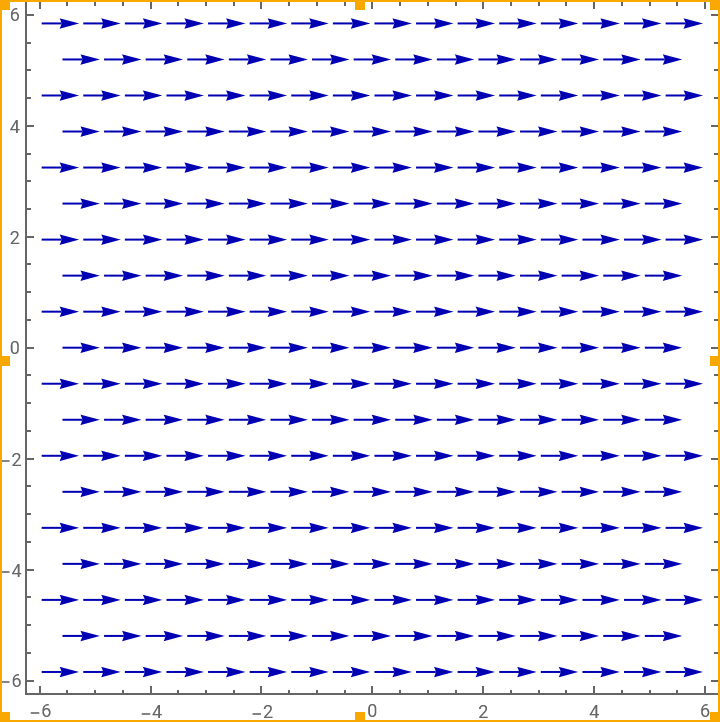} }}%
    \qquad
    \subfloat[\centering $\F_2$: $y^p\frac{\partial}{\partial x^p}+\frac{\partial}{\partial y}$ on $Y_1$]{{\includegraphics[width=5.5cm]{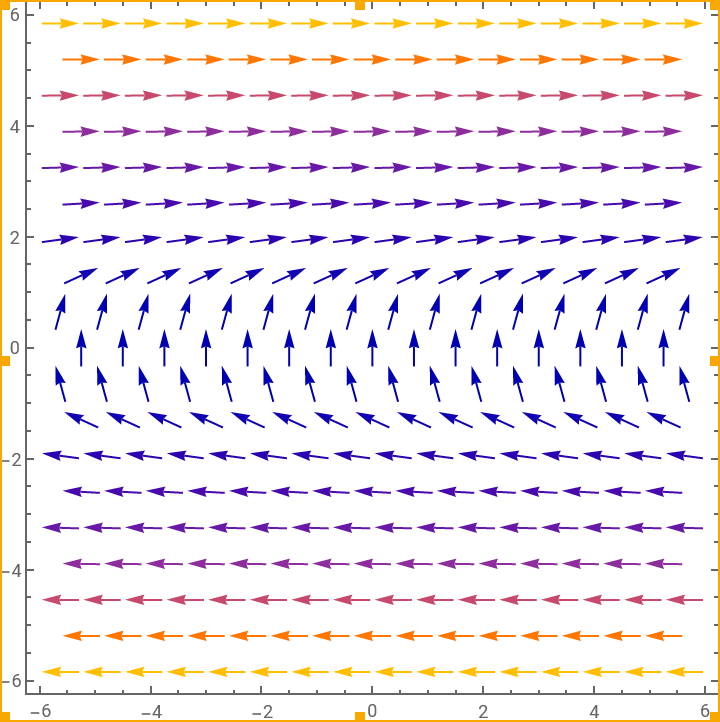} }}%
    \caption{The Jacobson sequence for $k[x^{p^2}, y^p, y^{p+1}-x^p]\to k[x,y]$.}%
    \label{fig:non-example Ekedahl1}%
\end{figure}

\begin{figure}%
    \centering
    \subfloat[\centering $\G_1$: $\frac{\partial}{\partial y}$ on $Y_1$]{{\includegraphics[width=5.5cm]{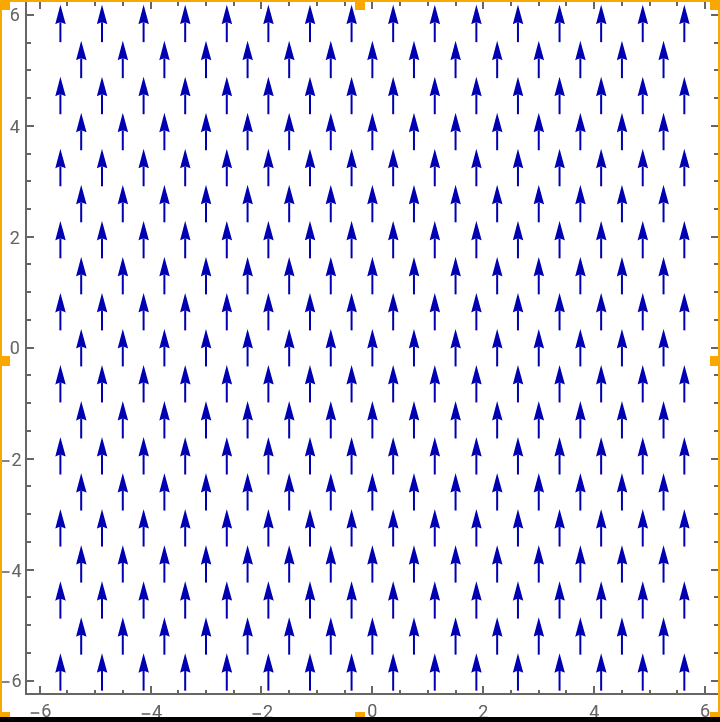} }}%
    \qquad
    \subfloat[\centering The support of $\frac{T_{Y_1/k}}{\F_{2}+\G_1}$ on $Y_1$]{{\includegraphics[width=5.5cm]{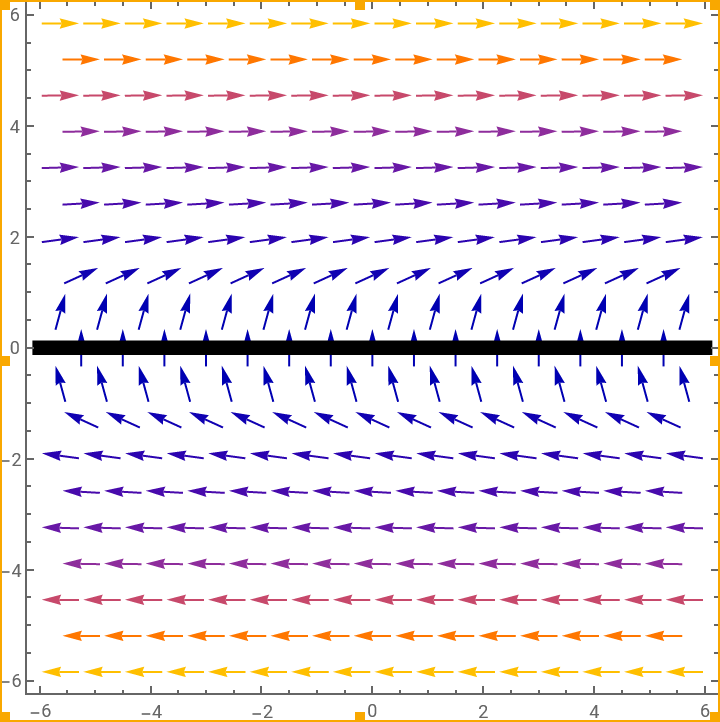} }}%
    \caption{A visualization of being non-Ekedahl for $k[x^{p^2},y^p, y^{p+1}-x^p]\to k[x,y]$.}%
    \label{fig:non-example Ekedahl2}%
\end{figure}
\FloatBarrier

\end{example}

\begin{remark}
    In general, it is true, that working with power towers on surfaces is (essentially) equivalent to working with a series of vector fields (if we discard the first (boring?) part of $p$-powers, $p$-powers,... and we only focus on the tail that is a $n$-foliation, for $n\in\{0,1,2,\ldots,\infty\}$). (See Proposition \ref{Nonincreasing Degrees Power Tower - Proposition}.)
    
    Therefore, one can draw pictures, like the ones in \ref{fig:non-example Ekedahl1} and \ref{fig:non-example Ekedahl2}, of these vector fields and \textbf{see} the geometrical interactions between different $Y_i$.
\end{remark}

\subsection{Morphisms and Power Towers}

Subfields were a good source of power towers. In this section, we use morphisms to get power towers. We primarily focus on fibrations and smooth morphisms.

\subsubsection{Basic Definitions}

\begin{defin}
    Let $X$ be a normal variety over a perfect field $k$ of characteristic $p>0$ satisfying $K(X)^{p^\infty}=k$. 

    Let $f: X\to Y$ be a dominant morphism, or a dominant rational map over $k$. 
    Then, the power tower on $X$ that corresponds to the power tower of $K(Y)$ on $K(X)$ is called \emph{the power tower of $f$ on $X$}. We denote this power tower by $\left(f_1: X\to Y_1, f_2:X\to Y_2, \ldots\right)$, $f_\bullet$, or $Y_\bullet$.
\end{defin}

\begin{defin}
    Let $X$ be a normal variety over a perfect field $k$ of characteristic $p>0$ satisfying $K(X)^{p^\infty}=k$. 
    Let $Y_\bullet$ be a power tower on $X$.

    We say that the power tower $Y_\bullet$ is \emph{algebraically integrable} if there is a dominant rational map $f: X\to Z$ such that $Y_\bullet = Z_\bullet$.
\end{defin}

\begin{remark}
    It is generally hard to reconstruct a nice morphism from an algebraically integrable power tower. This is similar to the GIT theory's study of quotients. The similarities for foliations are discussed in more detail in Bongiorno's preprint \cite{bongiorno2021foliation}. Power towers are analogous to foliations, so much of this discussion carries over to power towers of infinite length.

    On the other hand, the generic point could be figured out completely, Theorem \ref{infty=s}. So, the problem is in ``spreading out'' what we know at the generic point to rings.
\end{remark}

\begin{defin}
    Let $X$ be a normal variety over a perfect field $k$ of characteristic $p>0$ satisfying $K(X)^{p^\infty}=k$. 
    Let $Y_\bullet$ be a power tower on $X$. Let $\mathcal{D}$ be the saturated subalgebra of differential operators that corresponds to $Y_\bullet$.

    Then, the \textbf{\emph{sheaf of first integral}} of $Y_\bullet$ is defined by
    \[
    \cO_{Y_\infty/k}(U)\coloneqq \{f\in \cO_{X/k}(U): \ \forall_{ D\in \mathcal{D}(U)} [D,f]=0 \}= \op{const}(\mathcal{D}(U)),
    \]
    where $U\subset X$ is an open subset.
\end{defin}

\subsubsection{Fibrations Inject into Power Towers}

In this section, we define fibrations and prove that a particular class of fibrations is determined by its rational equivalence class.

The following definition is a version of \cite[Definition 0.1.]{Fibration_Voisin}, which is a definition of a fibration due to Claire Voisin from her survey paper. However, we dropped an assumption ``$X$ is a projective/compact complex variety'' to allow affine fibrations. 
I did this because I want the projections $\mathbb{A}^n\to\mathbb{A}^m$ for $n\ge m$ to be fibrations, as I consider these maps to be a basic example of what a fibration is, since in differential geometry every fibration is such locally.

\begin{defin}\label{Fibration: voisin modified definition}
Let $k$ be a field.
  A (rational) \emph{fibration} on a variety
$X/k$ is a dominant (rational) map $X\to Y$ with connected general fiber, where $Y/k$ is a variety.  
\end{defin}

\begin{defin}
    Let $X$ be a variety over a field $k$. Let $f: X\to Y$ be a fibration on $X$. Let $P$ be a property of a morphism of varieties, e.g., surjective, flat, faithfully flat, smooth. We say that the fibration $f$ is $P$ if the morphism $f$ is $P$. 
\end{defin}

\begin{defin}\label{sep closed/separable fibration}
    Let $X$ be a scheme over a field $k$. Let $f: X\to Y$ be a dominant morphism over $k$.
    \begin{itemize}
        \item We say that $f$ is \emph{separably closed} if $K(Y)$ is separably closed in $K(X)$, i.e., $K(Y)=K(Y)^s\subset K(X)$,
        \item We say that $f$ is \emph{separable} if the field extension $K(X)/K(Y)$ is separable, Definition \ref{Separable: Definition}.
    \end{itemize}
\end{defin}

\begin{remark}
    I am aware that the definition \ref{sep closed/separable fibration} is clumsy, and there are probably better words to be used, or these things already have names. If so, send it to me.
\end{remark}

\begin{defin}
    Let $g: X\to Y$ and $g: X\to Y'$ be morphisms between varieties over a field $k$. We say that $f$ and $g$ are \emph{rationally equivalent} if there exists a commutative diagram:
        \begin{center}
            \begin{tikzcd}
                &&&Y\\
               X \ar[rrru,"f"', bend left=20]\ar[rrrd,"g",bend right =20]&U\ar[r]\ar[l,"j_U"']& V\ar[ru,"j_V"]\ar[rd,"j'_{V}"']&\\
               &&&Y',
            \end{tikzcd}
        \end{center}
        where the maps $j_U, j_V, j'_{V}$ are nonempty open immersions.
\end{defin}

\begin{prop}\label{recovering flat separable fibrations - prop}
    Let $X,Y,Y'$ be varieties over a field $k$.
    Let $f: X\to Y$ and $g: X\to Y'$ be faithfully flat, separable, dominant morphisms of finite type.
    
    If $f$ and $g$ are rationally equivalent, then there exists an isomorphism $\rho: Y\to Y'$ such that $\rho \circ f =g$.
\end{prop}

\begin{proof}
    First, we prove that a closed subscheme $X\times_Y X\subset X\times X$ is a subvariety. To do this, it is enough to prove that it is reduced. 

    We prove that $X\times_Y X$ is reduced. Being reduced is a local property; therefore, we can check it on an open affine covering of $X\times_Y X$. We can produce a convenient covering by observing that $f$ is flat and of finite type, therefore it is open. Let $U=\op{Spec}(B)$ be an open affine subset of $X$, then the restriction of $f$ to this open subset defines a morphism 
    \[
    f: U \to f(U),
    \]
    where $f(U)$ is an open subscheme of $Y$. Now, since $Y$ is separated, it means that this morphism is affine, see e.g. \cite{affine_Lurie}. Therefore, if $\op{Spec}(A)\subset f(U)$ is an open affine subset, then $f^{-1}(\op{Spec}(A))$ is an open affine subset of $U$. Consequently, we can conclude that there is an open affine covering $\op{Spec}(A_i)$ of $Y$, for some indices $i$, such that the set $f^{-1}(\op{Spec}(A_i))=\op{Spec}(B_i)$ is an open affine covering of $X$.
    Now, the set $\op{Spec}(B_i)\times_{\op{Spec}(A_i)}\op{Spec}(B_i)$ is an open affine covering of $X\times_Y X$. Consequently, it is enough to prove that $X\times_Y X$ is reduced, assuming that $X$ and $Y$ are affine.

    Let $X, Y$ be affine schemes corresponding to rings $B$ and $A$ respectively. Let $f^{\#}: A\to B$ be the ring map corresponding to $f$. 
    Then, there exists a nonzero $c\in A$ and algebraically independent over $A_c$ elements $x_1,\ldots,x_n\in B_c$ such that
    \[
    A_c[x_1,\ldots,x_n] \subset B_c
    \]
    is a finite {\'e}tale ring extension, i.e., it is a smooth ring map. Indeed, since $(A)\subset (B)$ is separable, then there is an open subset $U$ of $Y$ such that $f: f^{-1}(U)\to U$ is smooth by \cite[Chapter V, Corollary 4.3.]{Mumford_II}. So, we can take $U=Y-V(c)$, and then we have the above decomposition by \cite[Lemma 00T7]{stacks-project}.
    
    Now, we can conclude that we have an inclusion
    \[
    B\otimes_A B \hookrightarrow \left(B\otimes_A B\right)_c \simeq  B_c \otimes_{A_c} B_c.
    \]
    Indeed, the inclusion follows from the element $c$ not being a zero divisor on $B\otimes_A B$, so the localizing $c$ is an inclusion. (This is an element of this ring as $c\coloneqq c\otimes 1 = 1\otimes c$.) The $c$ is not a zero divisor, because we have that the short exact sequence
    \[
    0\to B \xrightarrow{c} B \to B/cB \to 0
    \]
    tensored with $B$ over $A$ is still exact, by $A\to B$ being flat, and equal to the short exact sequence
    \[
    0\to B\otimes_A B \xrightarrow{c} B\otimes_A B \to B/cB \otimes_A B \to 0.
    \]
    The isomorphism follows from an elementary property of tensoring:
    \[
    B_c \otimes_{A_c} B_c\simeq (B\otimes_A A_c) \otimes_{A_c} (B\otimes_A A_c) \simeq B\otimes_A B \otimes_A A_c \simeq \left(B\otimes_A B\right)_c.
    \]
    
    From this inclusion, we can deduce that $B\otimes_A B$ is reduced if $B_c \otimes_{A_c} B_c$ is reduced. The last statement can be directly computed from the above form for $B_c/A_c$. This proves that $X\times_Y X$ is reduced, and therefore a variety.

    Next, let the following be a diagram given by $f$ and $g$ being rationally equivalent to each other.
    \begin{center}
            \begin{tikzcd}
                &&&Y\\
               X \ar[rrru,"f"', bend left=20]\ar[rrrd,"g",bend right =20]&U\ar[r]\ar[l,"j_U"']& V\ar[ru,"j_V"]\ar[rd,"j'_{V}"']&\\
               &&&Y',
            \end{tikzcd}
        \end{center}
    We have that $U \times_V U \subset X\times_Y X$ is an open dense subset. Now, since $X\times_Y X$ is a closed subvariety of $X\times X$, we get that the schematic closure $\overline{U \times_V U}$ of $U \times_V U $ in $X\times X$ is equal $X\times_Y X$. This reasoning also applies to $X\times_{Y'} X$. Consequently, we have the following equalities of groupoids on $X$:
    \[
    X\times_Y X = \overline{U \times_V U} = X\times_{Y'} X.
    \]

    Finally, we know that $f,g$ are fpqc coverings since they are faithfully flat morphisms between varieties, so these morphisms are effective epimorphisms
    by \cite[Lemma 023Q]{stacks-project}. This means precisely that $Y$ is a categorical quotient of $X\times_Y X$ and $Y'$ is a categorical quotient of $X\times_{Y'} X$. Now, we know that these groupoids are equal. Therefore, from the universal property of a groupoid quotient, there is an isomorphism $\rho: Y \to Y'$ such that $\rho \circ f = g$. This finishes the proof.
\end{proof}

The above was a step towards the following.

\begin{thm}[Flat Separable Fibrations into $\infty$-Foliations]\label{fibrations into infty-folaitions - varieties - theorem}
    Let $X$ be a normal variety over a perfect field $k$ of characteristic $p>0$ satisfying $K(X)^{p^\infty}=k$.

    Then, there is an injection from a set of flat, separably closed, separable, surjective fibrations on $X$ up to rational equivalence to $\infty$-foliations on $X$.

    In particular, if $f: X\to Y$ is a flat, separably closed, separable, surjective fibration on $X$, then the first integrals of the power tower $Y_\bullet$ are equal to $f^{-1}\cO_{Y/k}$.
\end{thm}

\begin{proof}
    Let $f: X\to Y$ be a flat separably closed separable surjective fibration on $X$.
    This morphism is faithfully flat, because it is flat and surjective. It is dominant because it is surjective. And, $X, Y$ are varieties over $k$, so $f$ is of finite type. Consequently, we can use Proposition \ref{recovering flat separable fibrations - prop} to conclude that the operation
    \[
    f:X\to Y \mapsto k\subset K(Y)\subset K(X)
    \]
    is injective precisely up to rational equivalence of flat separably closed separable surjective fibrations on $X$.

    Next, by the assumptions on $f$, we have that $K(Y)$ is separably closed in $K(X)$, i.e., $K(Y)=K(Y)^s$, and the field extension $K(X)/K(Y)$ is separable. Therefore, by Proposition \ref{fibrations into infty-foliations - proposition}, the subfield $K(Y)\subset K(X)$ corresponds to the power tower of $K(Y)$ on $K(X)$ that is a $\infty$-foliation on $K(X)$.

    Finally, by Lemma \ref{power towers on X = power towers on K(X) - lemma}, power towers on $X$ and on $K(X)$ are equivalent, thus $f$ injects into power towers on $X$ and its image is a $\infty$-foliation on $X$. This proves the injection.

    Now, we prove that the sheaf of first integrals equals $f^{-1}\cO_{Y/k}$. We can observe that the groupoid $X\times_Y X$ admits a quotient and it is equal to $f$, by \cite[Lemma 023Q]{stacks-project}. This means that $Y$ represents the sheaf of invariant functions for this groupoid, \cite[Section 02VE]{stacks-project}. By the definition, a local section $g\in \cO_{X/k}(U)$ is invariant with respect to $X\times_Y X$, i.e., it belongs to $f^{-1}\cO_{Y/k}(U)$, if and only if $g$ commutes with $\op{Diff}_Y{X}(U)$, i.e., it is a first integral of the power tower $Y_\bullet$, Definition \ref{First Integrals Power Towers - Definition}. This finishes the proof.
\end{proof}

\subsubsection{Smooth Morphisms Inject into Power Towers}

\begin{prop}[Smooth Fibrations into gr-Saturated Regular $\infty$-Foliations]\label{smooth fibrations into ekedahl - Proposition}
    Let $X$ be a regular variety over a perfect field $k$ of characteristic $p>0$ satisfying $K(X)^{p^\infty}=k$. 
    Let $Y$ be a regular variety over $k$.
    Let $f: X\to Y$ be a smooth surjective morphism such that $K(Y)=K(Y)^s\subset K(X)$.
    
    Then, the power tower of $f$ is a gr-saturated regular $\infty$-foliation.

    Moreover, the sheaf of first integrals of the power tower $Y_\bullet$ is equal to the pullback sheaf
    $
    f^{-1}\mathcal{O}_{Y/k} \subset  \mathcal{O}_{X/k}.
    $

    Therefore, there is an injection from separably closed smooth surjective morphisms $X\to Y$ with $Y$ regular into gr-saturated regular $\infty$-foliations on $X$.
\end{prop}

\begin{proof}
    A separably closed smooth surjective morphism is a flat separably closed separable surjective fibration, so, by Theorem \ref{fibrations into infty-folaitions - varieties - theorem}, we conclude that the power tower $Y_\bullet$ of $f: X\to Y$ is a $\infty$-foliation, the sheaf of first integrals is equal $f^{-1}\mathcal{O}_{Y/k}$, and we have an injection from separably closed smooth surjective morphisms into $\infty$-foliations.
    Therefore, the only thing that remains to be proved is that the power tower $Y_\bullet$ is gr-saturated and regular.

    First, we prove that $Y_\bullet$ is gr-saturated. This is a local question, so we can assume that $X$ and $Y$ are affine. The morphism $f:X\to Y$ is smooth, so, by \cite[Theorem 039Q.]{stacks-project}, we can assume that $X=\op{Spec}(A)$, $Y=\op{Spec}(B)$, and that there exist $x_1,\ldots, x_n\in A$ such that we have
    \[
    B \subset B[x_1,\ldots, x_n] \subset A,
    \]
    where $B[x_1,\ldots, x_n]$ is isomorphic to a polynomial ring, and the inclusion of rings $B[x_1,\ldots, x_n] \subset A$ is {\'e}tale. Moreover, $Y$ is regular over a perfect field $k$, so it is smooth by \cite[Lemma 038V (3), Lemma 038X]{stacks-project}. So, we can apply \cite[Theorem 039Q.]{stacks-project} again, this time to $k\to B$, meaning that there exist $y_1,\ldots,y_m\in B$ such that
    \[
    k\subset k[y_1,\ldots,y_m]\subset B
    \]
    with $k[y_1,\ldots,y_m]\subset B$ being {\'e}tale.
    Now, the set $y_1,\ldots, y_m, x_1,\ldots, x_n$ is a set of coordinates for $A/k$. Therefore, by using explicit formulas for differential operators, we conclude that $\op{Diff}_B(A)=\op{Diff}_{Y_\bullet}(X)$ is generated by $\frac{1}{n!}\frac{\partial^n}{\partial x_i^n}$. This observation implies the power tower is gr-saturated as $\op{gr}\op{Diff}_B(A)$ is generated by all $\left[\frac{1}{n!}\frac{\partial^n}{\partial x_i^n}\right]$ and the whole graded algebra $\op{gr}\op{Diff}_k(A)$ by all $\left[\frac{1}{n!}\frac{\partial^n}{\partial x_i^n}\right]$ and all $\left[\frac{1}{n!}\frac{\partial^n}{\partial y_j^n}\right]$.

    Now, we prove that the power tower $Y_\bullet$ is regular. First, $X$ is regular based on the assumption. So, we only have to prove that all subsheaves $\F_i \subset T_{Y_{i-1}/k}$ are subbundles. We put $Y_i=\op{Spec}(B_i)$ and use the previous paragraph's assumptions. We have that $B_i$ is equal to the integral closure of $B[x_1^{p^i},\ldots, x_n^{p^i}]$ in $(B[x_1^{p^i},\ldots, x_n^{p^i}])^s\subset (A)=K(X)$ by the Galois-type Correspondences \ref{MainTheorem - var - theorem}. In particular, the rings $B_i$ are regular with coordinates $y_1,\ldots,y_m,x_1^{p^i},\ldots, x_n^{p^i}$, and therefore we have, for $i\ge 0$,
    \[
    \F_{i+1}=\op{Span}(\frac{\partial}{\partial x_1^{p^i}},\ldots,\frac{\partial}{\partial x_1^{p^i}}) \subset 
    \op{Span}(\frac{\partial}{\partial x_1^{p^i}},\ldots,\frac{\partial}{\partial x_n^{p^i}},\frac{\partial}{\partial y_1},\ldots,\frac{\partial}{\partial y_m})=T_{Y_{i}/k}.
    \]
    Consequently, all of them are subbundles, meaning that $\op{Sing}(Y_\bullet)=\emptyset$. This finishes the proof.
\end{proof}

\newpage
\section{Foliations, $1$-Foliations, and $\infty$-Foliations}

The main source of power towers so far has been subfields and morphisms. In particular, if we are given a dominant map $f: X\to Y$, then we can take its
power tower $X\to Y_1\to Y_2\to \ldots$, where $Y_i$ is an approximation of $Y$ up to $p^i$-powers. We showed that nice classes of morphisms and subfields give rise to nice power towers. However, we can reverse this procedure. This leads to a question:
\begin{center}
    How to grow power towers?
\end{center}
Indeed, we can start with a $p$-Lie algebra/$1$-foliation on $X$, that is equivalent to $X\to Y_1\to X^{(1)}$, and then... what do we do next? How to get more $Y_2,Y_3,\ldots$ for this $Y_1$ systematically? Is it even always possible to continue, or maybe there are power towers that cannot be prolonged? If so, what is the obstruction? In general, these questions lead to:

\begin{question}
    What is the moduli of power towers? 
\end{question}

The above question is not necessarily tricky, but it is definitely perplexing. Consequently, this chapter will focus on much more minor problems.

\vspace{0.3cm}
\begin{center}
    \textbf{Problem 1:} Given a $n$-foliation on $X$, can we find $(n+1)$-foliation on $X$ that extends it?
\end{center}
\vspace{0.3cm}

This question depends only on $K(X)$, which is about fields. The answer is: yes, it is always possible to take an $n$-foliation: $K=W_0, W_1,\ldots, W_n$, and to find $W_{n+1}$ such that $K=W_0, W_1,\ldots, W_n, W_{n+1}$ is an $(n+1)$-foliation. This is proved in Theorem \ref{always}.

However, if we want to make it about a problem about varieties, then we have to add a variety-only adjective about power towers to the problem.

\vspace{0.3cm}
\begin{center}
    \textbf{Problem 2:} Given a \textbf{regular/gr-saturated/Ekedahl} $n$-foliation on $X$, can we find \textbf{regular/gr-saturated/Ekedahl} $(n+1)$-foliation on $X$ that extends it?
\end{center}
\vspace{0.3cm}

Now, it turns out that the answer might be negative! It was already shown by Ekedahl: there exists a variety $X$ with a $1$-foliation on $X$ given by $f_1:X\to Y_1$ such that there is no Ekedahl $2$-foliation on $X$ given by $f_1:X\to Y_1, f_2: X\to Y_2$.
Indeed, in the paper \cite[pages 145-146]{EkedahlFoliation1987}, Ekedahl constructed a ``$p$-cyclic cover of an abelian surface'' and showed that this $1$-foliation cannot be extended to any Ekedahl $2$-foliation.

Curiously, $1$-foliations that are ``$p$-cyclic covers of abelian surfaces'' tend to be counterexamples to many statements that are inspired by characteristic zero. Indeed, the same non-extendable $1$-foliations from  Ekedahl's paper show up in Miyaoka's paper \cite{Miyaoka1987}, and they are counterexamples to a ``bend and break for special $1$-foliations'' there.
These $1$-foliations are also counterexamples to bend and break results from Langer's paper \cite[Section 5.1.]{Langer2}.
And, another example is Bernasconi's preprint \cite{bernasconi2024counterexamples} in which he shows that the main ingredients of a minimal model program fail for $1$-foliations. His counterexamples are again ``$p$-cyclic covers of abelian surfaces'', i.e., $1$-foliations that do not extend to Ekedahl $2$-foliations.

Consequently, I believe that a problem of studying extensions of $1$-foliations to Ekedahl $n$-foliations and Ekedahl $\infty$-foliations is vital for positive characteristic algebraic geometry.

For example, Miyaoka already studied this problem in his paper \cite{Miyaoka1987}. Namely, I interpret his paper that he (not explicitly) uses a fact that we prove in the below Theorem \ref{lifting => extension - thm}, that a lifting of a $1$-foliation to a foliation induces a unique extension to an Ekedahl $\infty$-foliation. Indeed, my understanding is that he actually performed ``deformations along \emph{special} Ekedahl $\infty$-foliations'', but he did not have a theory of such objects, so he manipulated them using the liftings instead of interacting with them directly.

Therefore, Miyaoka's paper \cite{Miyaoka1987} is the inspiration for the main result of this chapter, Theorem \ref{lifting => extension - thm}: liftings of $1$-foliations to foliations extend them to nice $\infty$-foliations. So, there is an interaction between moduli of foliations and moduli of power towers.

\subsection{Extensions of $n$-Foliations on Fields}

We prove that every $1$-foliation extends to a $n$-foliation for $n=2,3,\ldots,\infty$.

\begin{defin}
   Let $K$ be a field of characteristic $p>0$. Let $k=K^{p^\infty}$ be the perfection of $K$. Let $K/k$ be a finitely generated field extension. 

   Let $(K, W_1, W_2,\ldots)$ be a power tower of finite length ($\le n$) on $K$, then we denote it by $W=W_n$, i.e., by the subfield it corresponds to. The index is at least the length, if we write the index, then we assume that the length is not greater then that.

   We say that a power tower $W_\bullet$ on $K$ \emph{\textbf{extends}} a power tower of finite length $V_n$ on $K$ if $W_n=V_n$.
\end{defin}

The following lemma reduces the problem of extending $1$-foliation to an $n$-foliation to a problem of extending $1$-foliation to a $2$-foliation.

\begin{lemma}\label{2 and 2 is 3 foliation - lemma}
Let $K$ be a field of characteristic $p>0$. Let $k=K^{p^\infty}$ be the perfection of $K$. Let $K/k$ be a finitely generated field extension.

Let $(K,W_1,W_2,W_2,\ldots)$ be a $2$-foliation on $K$, and let $(W_1,W_2,W_3,W_3,\ldots)$ be a $2$-foliation on $W_1$. Then the tower $(K,W_1,W_2,W_3,W_3,\ldots)$ is a $3$-foliation on $K$.
\end{lemma}

\begin{proof}
    If the tower $(K,W_1,W_2,W_3,W_3,\ldots)$ is a power tower on $K$, then it is a $3$-foliation, because
    $[K:W_1]=[W_1:W_2]$ and $[W_1:W_2]=[W_2:W_3]$ by the definition of a $2$-foliation, Definition \ref{n-foliation}. Therefore, we only have to prove this tower is a power tower.

    By Lemma \ref{Power Tower - alternative def}, it is enough to prove that $W_3\cdot K^{p^2}=W_2$. We prove this equality below. 
    We have that $W_1=W_2\cdot K^p$, so $W_1^p=W_2^p\cdot K^{p^2}$. Consequently, we have 
    \[
    W_2=W_3\cdot W_1^p=W_3\cdot (W_2^p\cdot K^{p^2})=(W_3\cdot W_2^p)\cdot K^{p^2}=W_3\cdot K^{p^2},\] because $W_3\supset W_2^p$ by Lemma \ref{Subsequent Extensions are of exponent 1 - lemma}. This finishes the proof.
\end{proof}

\begin{thm}\label{always}
    Let $K$ be a field of characteristic $p>0$. Let $k=K^{p^\infty}$ be the perfection of $K$. Let $K/k$ be a finitely generated field extension.

    Every $1$-foliation on $K$ can be extended to a $2$-foliation on $K$.
\end{thm}

\begin{proof}
The proof is a construction of a specific $p$-basis, and then using it to define a $2$-foliation extending the given $1$-foliation. 

    Let $\op{tr.deg}(K/k)=N$.

    Let $W_1$ be a $1$-foliation on $K$ of rank $r$.

    There exist elements $t_1,t_2,\ldots,t_r\in K$ such that $W_1(t_1,t_2,\ldots,t_r)=K$, i.e., a minimal set of generators. These elements are $p$-independent, Definition \ref{p-basis definition}, in $K$, because monomials $\prod_{i=1}^r t_i^{a_i}$, where $0\le a_i<p$, are $W_1$-linearly independent. Therefore, they are $K^p$-linearly independent, because $W_1\supset K^p$. 
    Moreover, we can prove that their $p$-powers $t_i^p\in W_1$ for $1\le i\le r$ are $p$-independent in $W_1$. Indeed, by taking $p$-roots, it follows from that the monomials $\prod_{i=1}^r t_i^{p \cdot a_i}$, where $0\le a_i<p$, are $W_1^p$-linearly independent. This means that differentials $d(t_i^p)$ are $W_1$-linearly independent in $\Omega_{W_1/k}$.

    Let $a_1,a_2,\ldots, a_N$ be a $p$-basis for $W_1/k$, Definition \ref{p-basis definition}. This means that differentials $da_1,da_2,\ldots,da_N$ are a $W_1$-linear basis of $\Omega_{W_1/k}$.
    Consequently, the notion of a $p$-basis admits an exchange lemma, i.e., any $p$-independent set can be completed to a $p$-independent set of maximal size by adding some elements from another maximal size independent set. Therefore, there are $N-r$ elements among $a_1,a_2,\ldots, a_N$, say $a_1,a_2,\ldots,a_{N-r}$, such that $a_1,a_2,\ldots,a_{N-r}, t_1^p, t_2^p, \ldots, t_r^p$ is a $p$-basis for $W_1/k$.

    We can prove that $a_1,a_2,\ldots,a_{N-r}, t_1, t_2, \ldots, t_r$ forms a $p$-basis for the extension $K/k$.
    In this situation, being a $p$-basis is equivalent to being a separable transcendental basis, Theorem \ref{p-basis is separable transcedence basis}, so we will prove that this set is such a basis.
    
    First, these elements are algebraically independent, because if not, then there would be a relation $F\in k[x_1,\ldots,x_N]$ between them, but then $F^p$ would give such a relation for $a_1,a_2,\ldots, a_{N-r}, t_1^p, t_2^p, \ldots, t_r^p$ over $k^p=k$. This is a contradiction with the fact that $a_1, \ldots, a_{N-r}, t_1^p, \ldots, t_r^p$ is a $p$-basis for $W_1/k$.
    
    Next, we show that
    \[
    k(a_1, \ldots, a_{N-r}, t_1, \ldots, t_r)^s=K.
    \]
    Indeed, we have a diagram of finite field extensions
    \begin{center}
        \begin{tikzcd}
        &K=W_1(t_1, \ldots, t_r)&\\
        k(a_1, \ldots, a_{N-r}, t_1, \ldots, t_r)^s\ar[ur]&&W_1\ar[ul]\\
        &k(a_1, \ldots, a_{N-r}, t_1^p, \ldots, t_r^p)^s\ar[ur, equal]\ar[ul],&
        \end{tikzcd}
    \end{center}
    wherein the separable closures are taken in $K$.
    Now, we see that the extensions 
    
    $k(a_1, \ldots, a_{N-r}, t_1, \ldots, t_r)^s/k(a_1, \ldots, a_{N-r}, t_1^p, \ldots, t_r^p)^s$ and $K/W_1$ are of order $p^r$. 
    
    So, the extension $K/k(a_1, \ldots, a_{N-r}, t_1, \ldots, t_r)^s$ must be of order $1$. This proves the equality.
    
    Consequently, we conclude that $a_1, \ldots, a_{N-r}, t_1, \ldots, t_r$ form a separating transcendental basis for $K/k$, so it is a $p$-basis for $K/k$.

    Finally, we can define a $2$-foliation extending $W_1$. Indeed, we consider the following diagram of fields (in which $(a_\bullet,t_\bullet)$ stands for $(a_1, \ldots, a_{N-r}, t_1, \ldots, t_r)$, etc):

    \begin{center}
        \begin{tikzcd}
            K=k(a_\bullet,t_\bullet)^s & W_1=k(a_\bullet,t^p_\bullet)^s\ar[l] & W_2\coloneqq k(a_\bullet,t^{p^2}_\bullet)^s\ar[l]\\
            k(a_\bullet,t_\bullet)\ar[u]&k(a_\bullet,t^p_\bullet)\ar[u]\ar[l]&k(a_\bullet,t^{p^2}_\bullet)\ar[u]\ar[l]
        \end{tikzcd}
    \end{center}
    Now, we claim that 
    \[
    K\supset W_1\supset W_2\supset W_2\supset W_2\supset \ldots
    \] 
    is a $2$-foliation on $K$.
    Indeed, we have that $W_1\supset K^{p}$, $W_2\supset K^{p^2}$, and $[K:W_1]=[W_1:W_2]=p^r$. We want to show it is a power tower, i.e., that the equality $W_2\cdot K^p=W_1$ holds.
    It follows from the following calculation:
    \[
    W_1\supset W_2\cdot K^p\supset k(a_\bullet,t^{p^2}_\bullet)^s(a_\bullet,t^p_\bullet)=k(a_\bullet,t^{p^2}_\bullet)^s(t^p_\bullet)=k(a_\bullet^p,t^p_\bullet)^s=W_1
    \]
    where we are allowed to move the separable closure and generators around, because all subfields involved are of finite exponent, this implies that separable closures of these subfields are these subfields.
    This finishes the construction of $W_2$ extending $W_1$.
\end{proof}

\begin{cor}\label{1-foli extends to a infty-foli - corollary}
    Let $K$ be a field of characteristic $p>0$. Let $k=K^{p^\infty}$ be the perfection of $K$. Let $K/k$ be a finitely generated field extension. Let $n>0$ be an integer.

    Every $n$-foliation on $K$ can be extended to a $\infty$-foliation on $K$.
\end{cor}

\begin{proof}
    Let $W_n$ be a $n$-foliation on $K$. 
    Then $W_n$ is a $1$-foliation on $W_{n-1}=W_n\cdot K^{p^{n-1}}$.
    Let $W_{n+1}$ be an extension of $W_n$ to a $2$-foliation on $W_{n-1}$. It exists by Proposition \ref{always}.

    Then $W_{n+1}$ is a $1$-foliation on $W_{n}$.
    Let $W_{n+2}$ be an extension of $W_{n+1}$ to a $2$-foliation on $W_{n}$.

    Then $W_{n+2}$ is a $1$-foliation on $W_{n+1}$.
    Let $W_{n+3}$ be an extension of $W_{n+2}$ to a $2$-foliation on $W_{n+1}$.

    And so on and so on, we obtain a tower of subfields of $K$: $(K, W_1, W_2,\ldots)$. This tower is a $\infty$-foliation on $K$ extending $W_n$ by Lemma \ref{2 and 2 is 3 foliation - lemma}.
\end{proof}

\begin{question}
    Is there an abstract $p$-Lie algebra cohomology? A purely inseparable cohomology? Assuming yes, and that we could use that cohomology to have a tangent obstruction theory for the problem of extending power towers, then we just proved that some elements in it are zero.

An \emph{abstract $p$-Lie algebra} can be defined, see e.g. \cite{Hochschild-p-lie-cohomology}.

Let $W_1$ be a subfield of $K$ of exponent one; it corresponds to a $p$-Lie algebra $\F_1=T_{K/W_1}$. Then, we have a short exact sequence of abstract $p$-Lie algebras on $W_1$:
\[
0 \to T_{W_1/K^p} \to T_{W_1/W_1^p} \xrightarrow{d(W_1/K^p)} T_{K^P/W_1^p}\otimes W_1 \to 0.
\]
This sequence admits a right splitting of abstract $p$-Lie algebras (there is an $\F_2$) if and only if $W_1$ can be extended to a $2$-foliation on $K$. 
We can conclude that some \emph{Ext groups of $p$-Lie algebras}, whatever they are, probably will control the extension problem for $n$-foliations for fields, for varieties, for whatever. However, it could be weirder than that. I do not know.
\end{question}

\subsection{Extensions via Liftings}
We prove that a lifting of a $1$-foliation to a foliation induces an extension to an Ekedahl $\infty$-foliation for this $1$-foliation, and that this $\infty$-foliation is regular on a big open subset.

We denote ($p$-typical) Witt vectors over $k$ by $W(k)$. For the basic theory of these rings, one can consult the following paper \cite{Hazewinkel_witt_vectors_1}. (If one does not know Witt vectors, one can think they are just $p$-adic integers.) And, a flat lifting of a variety $X$ over $k$ to $W(k)$ is an integral saturated scheme $\overline{X}$ flat over $W(k)$ whose pullback $\overline{X}\otimes_{\op{Spec}(W(K))} \op{Spec}(k)$, i.e., its reduction modulo $p$, is equal to $X$.

\begin{defin}
    Let $X/k$ be a normal variety over a perfect field $k$ of characteristic $p>0$ satisfying $K(X)^{p^\infty}=k$. Let $\F$ be a $p$-Lie algebra corresponding to a $1$-foliation $f_1: X\to Y_1$ on $X$.

    Then, a \emph{lifting of the $1$-foliation} $f_1$ is a pair $(\overline{X},\overline{\F})$ such that
    \begin{itemize}
        \item $\overline{X}$ is a flat lifting of $X$ to  $W(k)$,
        \item $\overline{\F}\subset T_{\overline{X}/W(k)}$ is a Lie subsheaf, i.e., a foliation on $\overline{X}/W(k)$, whose reduction modulo $p$ is equal $\F \subset T_{X/k}$.
    \end{itemize} 
\end{defin}

\begin{defin}
    Let $X/k$ be a normal variety over a perfect field $k$ of characteristic $p>0$ satisfying $K(X)^{p^\infty}=k$. Let $\F$ be a $p$-Lie algebra corresponding to a $1$-foliation $f_1: X\to Y_1$ on $X$. Let $(\overline{X},\overline{\F})$ be a lifting of $f_1$.

    Then, we define a \emph{saturation of a subalgebra of differential operators on $\overline{X}/W(k)$ generated by $\overline{\F}$} to be the following sheaf:
    \[
    \op{Sat}\left<\overline{\F}\right>(U)\coloneqq  \left(\left<\overline{\F}\right>\otimes  K(\overline{X})\right)\cap \op{Diff}_{W(k)}(\overline{X})(U),
    \]
    where $U\subset \overline{X}$ is an open subset, $\left<\overline{\F}\right>\subset \op{Diff}_{W(k)}(\overline{X})$ is the smallest subalgebra containing $\overline{\F}$, and $K(\overline{X})$ is the field of rational functions on $\overline{X}$ (It includes $\frac{1}{p}$.).
\end{defin}

\begin{thm}\label{lifting => extension - thm}
    Let $X/k$ be a normal variety over a perfect field $k$ of characteristic $p>0$ satisfying $K(X)^{p^\infty}=k$. Let $\F$ be a $p$-Lie algebra corresponding to a $1$-foliation $f_1: X\to Y_1$ on $X$. Let $(\overline{X},\overline{\F})$ be a lifting of $f_1$.

    Then, the lifting of the $1$-foliation induces a unique extension of this $1$-foliation to an Ekedahl $\infty$-foliation that is compatible with the lifting.

    Explicitly, the extension is given by the following saturated subalgebra:
    \[
    \mathcal{D}(\overline{X},\overline{\F})\coloneqq \op{Sat}\left<\overline{\F}\right> \otimes k \subset \op{Diff}_{W(k)}(\overline{X})\otimes k = \op{Diff}_k(X).
    \]

    Moreover, this $\infty$-foliation is regular on a big open subset.
\end{thm}

\begin{proof}
    The proof is an explicit construction of a set of liftings of saturated distinguished generators $G$ for the extension, Proposition \ref{saturated dist generators - exist and do the job - prop}, on a big open subset of $X$, on which the given $1$-foliation is regular, and then using these liftings to show that
    \[
    \op{gr}\op{Sat}\left<\overline{\F}\right>  \cong {S}(\overline{\F}^*)^{*gr}.
    \]
    From this equality, the whole theorem follows. Indeed, this equality modulo $p$, i.e., after tensoring with $k$, gives
    \[
    \op{gr}\left(\op{Sat}\left<\overline{\F}\right>\otimes k\right)  \cong {S}({\F}^*)^{*gr}.
    \]
    And, since a reduction modulo $p$ of a subalgebra is a subalgebra, thus by Proposition \ref{gr-saturared + regular = pd envelope - lemma} we get that the subalgebra $\op{Sat}\left<\overline{\F}\right>\otimes k$ is gr-saturated (in particular, it is saturated) and regular. Hence, we will be done. 

    The main trick is an induction. We take already constructed $G^i_j$, we take its lifting  $\overline{G^i_j}$, a lifting $C$ of ${G^i_j}^{\circ p}$, and we put $\overline{G^{i+1}_j}=\frac{\overline{G^i_j}^{\circ p}-C}{p!}$ up to a scalar. The technical part is to get control over the leading forms of these liftings.
    
    Here are the details: 

    Without loss of generality, we can assume that the $1$-foliation $f_1$ is regular, Definition \ref{regular power tower - def}, because there exists a big open subset $U\subset X$ such that $f_1 \cap U$ is regular, Proposition \ref{codim 2 regular power tower finite length- prop}, and being Ekedahl means being gr-saturated on a big open subset, Definition \ref{Ekedahl - def}, so it is enough to prove that a lifting of a regular $1$-foliation induces a gr-saturated regular $\infty$-foliation, which is characterized by Proposition \ref{gr-saturared + regular = pd envelope - lemma}.

    Let $f_1$ be a regular $1$-foliation. This implies that $X$ is regular and $\F\subset T_{X/k}$ is a subbundle, Proposition \ref{regular power tower - meaning - prop}. Moreover, by a simple deformation theory, it implies that $\overline{X}$ is regular over $W(k)$ and that $\overline{\F}\subset T_{\overline{X}/W(k)}$ is a subbundle as well.

    Without loss of generality, we can assume that $X$, $Y_1$, and $\overline{X}$ are affine, because being gr-saturated and regular are local properties. We put $X=\op{Spec}(A)$, $ Y_1=\op{Spec}(B_1)$, and $\overline{X}=\op{Spec}(\overline{A})$. (Because a flat lifting of an affine scheme is affine.) Moreover, we can assume that both subbundles $\F\subset T_{X/k}$ and $\overline{\F}\subset T_{\overline{X}/W(k)}$ are actually free submodules.

    Let $D_1,\ldots, D_r$ be a basis of $\F$.

    We construct the following data:

    \begin{itemize}
        \item For $j\ge 1$, operators $\overline{D_{i,(j)}}\in \op{Sat}\left<\overline{\F}\right>$, where $i=1,2,\ldots,r$, such that:
        \begin{itemize}
            \item For $j\ge 1$ and $i=1,2,\ldots,r$, we have $\left[\overline{D_{i,(j)}}\right]=\gamma_{p^{j-1}}(\overline{D_{i,(1)}})$.
        \item For $i=1,2,\ldots,r$, we have $\overline{D_{i,(1)}}\in \overline{\F}$ and $\left(\overline{D_{i,(1)}} \ \op{mod} \ p\right)= D_i$.
        \end{itemize}
    \end{itemize}

    We begin the construction.

    Let $j=1$. We put $\overline{D_{i,(1)}}\coloneqq \overline{D_{i}}$, where $\overline{D_i}\in \overline{\F}$ is a lifting of $D_i$, i.e., an operator that modulo $p$ is equal $D_i$.

    Let $j=2$. Let $\overline{D_i^p}\in \overline{\F}$ be a lifting of the $p$-power $D_i^p$. We have that $\overline{D_i}^{\circ p}$ mod $p$ is equal $D_i^p$, however its order is $p$, not $1$. Consequently, the difference $\overline{D_i}^{\circ p}-\overline{D_i^p}$ is divisible by $p$. Therefore, the following operator is well defined:
    \[
    \overline{D_{i,(2)}}\coloneqq \frac{\overline{D_i}^{\circ p}-\overline{D_i^p}}{p!} \in \op{Sat}\left<\overline{\F}\right>.
    \]
    By explicit computation in any coordinates, we can justify the following equality for the leading form:
    \[
    [\overline{D_{i,(2)}}]=\gamma_p(\overline{D_i}).
    \]

    Let $n\ge 2$. Let us assume that the construction is already done for $j\le n$. We construct the operators for $j=n+1$.

    Let $\mathbb{M}_n\subset \op{Sat}\left<\overline{\F}\right>$ be the $\overline{A}$-submodule spanned by the operators 
    \[
    \prod_{i=1,\ldots,r;\ j=1,\ldots,n} \overline{D_{i,(j)}}^{\circ a_{i,j}},
    \]
    where $0\le a_{i,j}\le p-1$. (The order of compositions may matter, and we fix one order for each choice of indices.) These operators are a $\overline{A}$-basis of this module, because their leading forms are independent. Moreover, this module is gr-saturated, i.e., $\op{gr}\mathbb{M}_n \subset \op{gr}\op{Diff}_{W(k)}(\overline{X})$ is saturated, because the leading forms $[\overline{D_{i,(j)}}]$ are not divisible by any non-invertible element of $\overline{A}$.

    The reduction modulo $p$ of the module $\mathbb{M}_n$ is a saturated subalgebra on $X$. We denote it by $\mathcal{D}_n \coloneqq  \mathbb{M}_n\otimes k$. Indeed, we have that
    \[
    \op{gr}\op{Sat}\left<\overline{\F}\right> \subset {S}(\overline{\F}^*)^{*gr}
    \]
    by the universal property of ${S}(\overline{\F}^*)^{*gr}$ as in Proposition \ref{pd envelope - prop}. So, we have that the gradation of the subalgebra generated by all of $\overline{D_{i,(j)}}$ for $j\le n$ is inside ${S}(\overline{\F}^*)^{*gr}$ too, i.e., $\op{gr}\left< \overline{D_{i,(j)}}\right>_{i=1,\ldots,r; 1\le j\le n}\subset {S}(\overline{\F}^*)^{*gr}$, and it is saturated.
    Consequently, the reduction modulo $p$ of $\left< \overline{D_{i,(j)}}\right>_{i=1,\ldots,r; 1\le j\le n}$ is a gr-saturated subalgebra whose gradation is inside ${S}({\F}^*)^{*gr}$. Actually, because the reduction of the operator $\overline{D_{i,(j)}}$ is of order $p^{j-1}$, we get that
    \[
    \op{gr}\left(\left< \overline{D_{i,(j)}}\right>_{i=1,\ldots,r; 1\le j\le n}\otimes k\right) = {S}({\F}^*)^{*gr} \cap \op{Diff}_{X^{(n)}}(X).
    \]
    Therefore, by Proposition \ref{gr-saturared + regular = pd envelope - lemma}, we get that $\left< \overline{D_{i,(j)}}\right>_{i=1,\ldots,r; 1\le j\le n}\otimes k$ corresponds to a gr-saturated regular $n$-foliation. The first order is $\F$, so it extends $\F$. Moreover, the set of the reductions $\{\overline{D_{i,(j)}}\otimes k\}$ is a set of saturated distinguished generators of this subalgebra, Proposition \ref{saturated dist generators - exist and do the job - prop}. So, the monomials $\prod_{i=1,\ldots,r;\ j=1,\ldots,n} (\overline{D_{i,(j)}}\otimes k)^{\circ a_{i,j}}$,  where $0\le a_{i,j}\le p-1$, are a basis of this algebra, Proposition \ref{saturated dist generators - exist and do the job - prop}. Finally, this shows that this is an algebra.

    Let $i\in \{1,2,\ldots,r\}$. Let $E_i$ be $\overline{D_{i,(n)}}^{\circ p}$ modulo $p$. It is equal $\left( \overline{D_{i,(n)}} \ \op{mod} \ p\right)^{\circ p}$, so it belongs to $\mathcal{D}_n$ since it is an algebra. Let $\overline{E_i}\in \mathbb{M}_n$ be a lifting of $E_i$. (It is important that the lifting is taken from this module! Otherwise, we could enter a computational nightmare.)
    Consequently, the following operator is well defined:
    \[
    \overline{D_{i,(n+1)}}'\coloneqq a\cdot \frac{\overline{D_{i,(n)}}^{\circ p}-\overline{E_i}}{p!} \in \op{Sat}\left<\overline{\F}\right>,
    \]
    where $a=\frac{p!\left(p^n!\right)^p}{p^{n+1}!}$ is an invertible modulo $p$ integer. 
    We put this extra $a$ into the formula to in a moment get $\gamma_{p^n}(D_i)$ instead of $\gamma_p(\gamma_{p^{n-1}}(D_i))$. These two values are the same up to multiplication by this integer. Still, we have to modify this operator a bit to achieve this leading form, because this operator may have too high order.

    If the leading form of $\overline{D_{i,(n+1)}}'$ is of order $> p^n$, then it is contributed solely by $\frac{\overline{E_i}}{p!}$, i.e., $[\overline{D_{i,(n+1)}}' ]=a[\frac{\overline{E_i}}{p!}]=\frac{1}{p!}[\overline{E_i}]$. However, since $\op{gr}\mathbb{M}_n$ is saturated, it means that the form $\frac{1}{p!}[\overline{E_i}]$ is well defined over $W(k)$ if and only if the form $[\overline{E_i}]$ is divisible by $p$.
    In our case, it is well defined, so there exists an operator $L_i\in \mathbb{M}_n$ such that $p![L_i]=a[\overline{E_i}]$. We make a correction:
    \[
    \overline{D_{i,(n+1)}}''\coloneqq \overline{D_{i,(n+1)}}' - L_i.
    \]
    The new operator has a lower order. If it is still $> p^n$, then we repeat the correction. So, we can assume that the order of $\overline{D_{i,(n+1)}}''$ is $\le p^n$. Actually, in this situation, it is equal $p^n$, because by the construction it has a nonzero coefficient next to an operator $\frac{1}{p^n!}\frac{\partial ^{p^n}}{\partial x_i^{p^n}}$, where $x_1,\ldots, x_d$ is any set of coordinates for $\overline{A}/W(k)$. It is contributed from $\frac{\overline{D_{i,(n)}}^{\circ p}}{p!} $. (An explicit computation.)
    
    Next, we observe that there exists a canonical decomposition, Lemma \ref{divided power polynomial is PD lemma}:
    \[
    {S}(\overline{\F}^*)^{*gr} \cap \op{gr}^{p^n}\op{Diff}_{W(k)}(\overline{A}) = \left(\gamma_{p^n}(\overline{\F}) \otimes \overline{A}\right) \oplus \text{REST},
    \]
    where the submodule $\text{REST}$ is given by
    \[
    \text{REST}\coloneqq {S}(\overline{\F}^*)^{*gr} \cap \op{gr}\mathbb{M}_n \cap \op{gr}^{p^n}\op{Diff}_{W(k)}(\overline{A})\subset \op{gr}\mathbb{M}_n,
    \]
    so this submodule depends only on $\overline{\F}$, because $\op{gr}\mathbb{M}_n$ is determined by the leading forms that are compositions of divided powers of elements from $\overline{\F}$. With respect to this decomposition, a direct computation reveals that
    \[
    \left[\overline{D_{i,(n+1)}}''\right]=\left(\gamma_{p^n}(\overline{{D_i}}), [G_i] \right),
    \]
    where $G_i\in \mathbb{M}_n$.
    Now, we can make the last correction
    \[
    \overline{D_{i,(n+1)}}\coloneqq \overline{D_{i,(n+1)}}'' -{G_i}.
    \]

    This finishes the construction of the liftings of saturated distinguished generators.

    Finally, we can prove the theorem. Indeed, by the universal property of symmetric algebras, we have a natural factorization (compare it with Proposition \ref{pd envelope - prop}):
    \[
    \op{gr}\op{Sat}\left<\overline{\F}\right>  \subset {S}(\overline{\F}^*)^{*gr} \subset \op{gr} \op{Diff}_{W(k)}(\overline{X}).
    \]
    Also, the leading forms of operators $\overline{D_{i,(j)}}$ generate ${S}(\overline{\F}^*)^{*gr}$, therefore we have $\op{gr}\op{Sat}\left<\overline{\F}\right>  = {S}(\overline{\F}^*)^{*gr} $.

    Let $\mathcal{D}\coloneqq \op{Sat}\left<\overline{\F}\right> \otimes k$. It is a subalgebra of $\op{Diff}_k(X)$, because it is a reduction modulo $p$ of a subalgebra. Moreover, taking a gradation commutes with reduction modulo $p$, so
    we have
    \[
    \op{gr}\mathcal{D}=\op{gr}\op{Sat}\left<\overline{\F}\right> \otimes k
    ={S}(\overline{\F}^*)^{*gr}\otimes k
    ={S}({\F}^*)^{*gr}.
    \]
    Therefore, by Proposition \ref{gr-saturared + regular = pd envelope - lemma}, we get that $\mathcal{D}$ corresponds to a gr-saturated regular $\infty$-foliation. 

    This $\infty$-foliation extends $\F$. This finishes the proof.
\end{proof}

\begin{remark}
    In Theorem \ref{lifting => extension - thm}, we only used that the algebras are flat over $W(k)$, so we could conclude divisibility by $p$ in some calculations.
    
    Consequently,  a lifting of a $1$-foliation to a foliation modulo $p^2$ (i.e., a saturated subsheaf of $T_{\overline{X}/W_2(k)}$ closed under Lie brackets), i.e., a lifting to Witt vectors of length $2$ denoted by $W_2(k)$, is enough to give a $\infty$-foliation that extends it. (Where the saturation must be defined differently using a substitute for the field of fractions.) Nevertheless, psychologically, it was easier to prove it for $W(k)$.
    
    This observation is analogous to the famous Deligne--Illusie theorem that a lifting modulo $p^2$ of a variety induces a degeneration of an important spectral sequence on it, see their paper \cite{deligne-illusie-p^2}.
\end{remark}

\begin{remark}
    In Theorem \ref{lifting => extension - thm}, we have to take the subalgebra $\op{Sat}\left<\overline{\F}\right> $, because $\left<\overline{\F}\right>$ is too small. Indeed, we can observe the essential behavior behind this choice on the simplest example: the vector field $\frac{\partial}{\partial x}$ on $\mathbb{A}^1_k=\op{Spec}(k[x])$, which spans the $p$-Lie algebra $\F\coloneqq T_{\mathbb{A}^1_k/k}\subset T_{\mathbb{A}^1_k/k}$. This $p$-Lie algebra admits a lifting to $\overline{\F}\coloneqq T_{\mathbb{A}^1_{W(k)}/{W(k)}}\subset T_{\mathbb{A}^1_{W(k)}/{W(k)}}$ that is also spanned by $\frac{\partial}{\partial x}$.

    Now, the above algebras for this example are given by:
    \begin{align*}
        \left<\overline{\F}\right> &= \bigoplus_{n\ge 0} \frac{\partial}{\partial x}^{\circ n} W(k)[x],\\
        \op{Sat}\left<\overline{\F}\right> &= \bigoplus_{n\ge 0} \frac{1}{n!}\frac{\partial^n}{\partial x ^n} W(k)[x].
    \end{align*}
    The main, and essentially only, difference is that we have 
    \begin{center}
      $\frac{1}{p^n!}\frac{\partial^{p^n}}{\partial x ^{p^n}}\in \op{Sat}\left<\overline{\F}\right>$,
    but $\frac{1}{p^n!}\frac{\partial^{p^n}}{\partial x ^{p^n}}\not\in \left<\overline{\F}\right>$, for $n\ge 1$.
    \end{center}
    Consequently, in order to get the right leading forms, we must saturate our subalgebras. However, the reductions modulo $p$ of both of these algebras are saturated anyway.
\end{remark}

\newpage
\section{Pullbacks of Canonical Divisors}

\subsection{Formula for Pullbacks of Canonical Divisors}

\begin{thm}[Canonical Divisor Formula]\label{Canonical Divisor Formula - theorem}
    Let $f: X\to Y\to X^{(n)}$ be a purely inseparable morphism of exponent $n$ between normal varieties over a perfect field $k$ of characteristic $p>0$ such that $k=K(X)^{p^\infty}$.
    Let $\mathcal{F}_\bullet$ be the Jacobson sequence corresponding to the power tower of $f$ on $X$. Let $K_X, K_Y$ be canonical divisors of $X, Y$ respectively.
    
    Then, there exist divisors $D_2,D_3,\ldots, D_n$ on $X$ such that
    \[
    f^* K_Y -K_X = (p^n - 1) (-c_1(\mathcal{F}_1)) + \sum_{i=2}^n (p^{n+1-i} - 1)D_i,
    \]
    where $c_1$ is the first Chern class.
    Explicitly, these divisors are given by, for $i\ge 2$, 
    \[
    D_i\coloneqq f_{i-1}^*c_1\left(\frac{T_{Y_{i-1}/k}}{\F_i+\G_{i-1}}\right),
    \]
    where $\G_{i}\coloneqq T_{Y_i/Y_{i-1}^{(1)}}$.
\end{thm}

\begin{proof}
    Without loss of generality, we can assume that the power tower $Y_\bullet$ of $f$ on $X$ is regular. Indeed, by Proposition \ref{codim 2 regular power tower finite length- prop}, there exists a big open subset $U\subset X$ such that the power tower $Y_\bullet \cap U$ is regular. And, any identity between divisors on a variety is true if it is true on a big open subset by the excision property \cite[Chapter II, Proposition 6.5 (b)]{Hartshorne}.

    Let the power tower $Y_\bullet$ be regular. We recall that we put $Y_0=X$.

    We prove the formula in \textbf{four steps}.

   \textbf{Step 1}, we prove that the formula holds for $1$-foliations. Indeed, let $f: X\to Y$ be a $1$-foliation.
    From Proposition \ref{regular power tower - meaning - prop}, we conclude that we have the following short exact sequences of vector bundles on $X$:
    \begin{align*}
    0\to \F_1 \to T_{X/k} &\to f_1^* \G_1 \to 0,\\
    0 \to f_1^* \G_1 \to f_1^*T_{Y_1/k} &\to f_1^*\left(\frac{T_{Y_1/k}}{\G_1}\right) \to 0,
    \end{align*}
    where $\G_{1}\coloneqq T_{Y_1/Y_{0}^{(1)}}=T_{Y_1/X^{(1)}}$.
    So, from the linearity of the first Chern class, we get
    \begin{align*}
    c_1(\F_1) + c_1(f_1^* \G_1)&= -K_X,\\
    c_1(f_1^* \G_1) +c_1(f_1^*\left(\frac{T_{Y_1/k}}{\G_1}\right)) &= -f_1^* K_{Y_1}.
    \end{align*}
    And, from this, we get
    \[
    f_1^* K_{Y_1} -K_X = c_1(\F_1) - c_1(f_1^*\left(\frac{T_{Y_1/k}}{\G_1}\right)).
    \]
    Next, we can observe that
    \[
    f_1^*\left(\frac{T_{Y_1/k}}{\G_1}\right) = F_{X/k}^*\left(\F_1^{(1)}\right),
    \]
    where $\F_\bullet^{(1)}$ is the Jacobson sequence corresponding to $f^{(1)}:X^{(1)}\to Y^{(1)}$.
    Indeed, this follows from the isomorphism
    \[
    \frac{T_{Y_1/k}}{\G_1} \xrightarrow{d(Y_1/X^{(1)})} \F_1^{(1)} \otimes \cO_{Y_1/k}=g_1^*\left(\F_1^{(1)}\right).
    \]
    Thus, we have
    \[
    c_1(f_1^*\left(\frac{T_{Y_1/k}}{\G_1}\right))=c_1(F_{X/k}^*(\F_1^{(1)}))=pc_1(\F_1),
    \]
    because the line bundle $c_1(F_{X/k}^*(\F_1^{(1)}))$ is defined locally by $p$-powers of local generators of $c_1(\F_1)$. Finally, we conclude
    \[
    f_1^* K_{Y_1} -K_X = c_1(\F_1) - c_1(f_1^*\left(\frac{T_{Y_1/k}}{\G_1}\right))=c_1(\F_1)-pc_1(\F_1)=(p-1)(-c_1(\F_1)).
    \]
    This finishes the proof of the formula for a $1$-foliation.

    \textbf{Step 2}, we prove that, for a power tower $Y_\bullet$ of length $n$, we have\footnote{Actually, this formula holds for any sequence of $1$-foliations.}
    \[
    f_n^* K_{Y_n} - K_X=(1-p)\left(c_1(\F_1)+f_1^* c_1(\F_2)+\ldots+f_{n-1}^* c_1(\F_{n})\right).
    \]
    Indeed, we already know the formula from the theorem for $1$-foliations, so we can apply it to the $1$-foliations $X\to Y_1, Y_1\to Y_2, \ldots, Y_{n-1}\to Y_n$, and then we can pull back all these formulas to $X$:
    \begin{align*}
    (1-p)c_1(\F_1) &= f_1^* K_{Y_1} - K_X,\\
    (1-p)f_1^* c_1(\F_2) &= f_2^* K_{Y_2} - f_1^* K_{Y_1},\\
    &\ldots,\\
    (1-p)f_{m-1}^* c_1(\F_m) &= f_m^* K_{Y_m} - f_{m-1}^* K_{Y_{n-1}}.
\end{align*}
    The statement follows from adding them all up.

    \textbf{Step 3}, we prove a relation:
    \[
    {f'}_{i-2}^{*}\left(c_1(\F_i)\right)=  pc_1(\F_{i-1}) - {f'}_{i-2}^{*}c_1\left(\frac{T_{Y_{i-1}/k}}{\F_i+\G_{i-1}}\right),
    \]
    where we used the functions $f'_i$ from Notation \ref{power tower on variety - notation}. We recall $\G_{i}\coloneqq T_{Y_i/Y_{i-1}^{(1)}}$.
    
    Indeed, it follows from the following diagram:
    \begin{center}
        \begin{tikzcd}
            &0&0&0&\\
            0\ar[r]&\G_{i-1}\ar[r]\ar[u]&{f'_{i}}^* \G_{i}\ar[r]\ar[u]& \frac{T_{Y_{i-1}/k}}{\F_i+\G_{i-1}}\ar[u]\ar[r]&0\\
            0\ar[r]& \G_{i-1}\ar[r]\ar[u]&T_{Y_{i-1}/k}\ar[r]\ar[u]&g_{i-1}^* \F_{i-1}^{(1)}\ar[r]\ar[u]&0\\
            0\ar[r]&0\ar[r]\ar[u]&\F_{i}\ar[r]\ar[u]&\F_{i}\ar[r]\ar[u]&0\\
            &0\ar[u]&0\ar[u]&0\ar[u]&
        \end{tikzcd}
    \end{center}
    From which, we deduce an equality on $Y_{i-1}$:
    \[
    c_1(\F_i) + c_1\left(\frac{T_{Y_{i-1}/k}}{\F_i+\G_{i-1}}\right) = c_1 (g_{i-1}^{*} \F_{i-1})= g_{i-1}^{*}c_1 ( \F_{i-1}),
    \]
    where we used $g_i$ from Notation \ref{power tower on variety - notation}.
    If we apply ${f'}_{i-2}^{*}$ to this equality, then we are done.

    \textbf{Last Step 4}, we can combine the previous steps to finish the proof. We recall that $D_i\coloneqq f_{i-1}^*c_1\left(\frac{T_{Y_{i-1}/k}}{\F_i+\G_{i-1}}\right)$.

    By iterating the relation from step 3, we obtain
    \[
    f_{i-1}^* c_1(\F_i) = p^{i-1} c_1(\F_1)-p^{i-2}D_2-  \ldots  -p^2D_{i-2}-pD_{i-1} -D_i.
    \]
    
    Finally, we can combine the developments from the second and third steps:
    \begin{align*}
        \frac{1}{1-p}(f_{m}^*K_{Y_m}-K_X)=& \sum_{i=1}^{m} f_{i-1}^* c_1(\F_{i} )\\
        =& \sum_{i=1}^{m}\left( p^{i-1} c_1(\F_1)-\sum_{j=2}^{i} p^{i-j}D_j \right)\\
        =& \sum_{i=1}^{m} p^{i-1} c_1(\F_1) - \sum_{i=1}^{m}\sum_{j=2}^{i} p^{i-j}D_j \\
        =& \frac{1-p^m}{1-p} c_1(\F_1) - \sum_{i=1}^{m}\sum_{j=2}^{i} p^{i-j}D_j \\
        =& \frac{1-p^m}{1-p} c_1(\F_1) -\sum_{j=2}^{m} \frac{1-p^{m-j+1}}{1-p}D_j.
    \end{align*}
    The proof is finished by multiplying the above equation by $1-p$.
\end{proof}

Some control over the divisors for $n$-foliations.

\begin{prop}\label{canonical divisor for n-foli and ekedahl - prop}
    Let $f: X\to Y$ be a purely inseparable morphism from Theorem \ref{Canonical Divisor Formula - theorem}.
    
    If $f: X\to Y$ is a $n$-foliation on $X$, then $D_i\ge 0$, for $i\ge 2$. 
    
    If $f: X\to Y$ is an Ekedahl $n$-foliation on $X$, then $D_i= 0$, for $i\ge 2$.
\end{prop}

\begin{proof}
    Let $f:X\to Y$ be a $n$-foliation. Then, for $i=2,\ldots,n$, the sheaf $\frac{T_{Y_{i-1}/k}}{\F_i+\G_{i-1}}$ is torsion, Lemma \ref{n-foli torsion sheaves - lemma}. Therefore, from functorial properties of Chern classes of coherent sheaves: a definition by using resolutions by vector bundles, \cite[Section 0ESY]{stacks-project} and a proper pushforward \cite[Section 02R3]{stacks-project} applied to the closed immersion $\op{Supp}(\frac{T_{Y_{i-1}/k}}{\F_i+\G_{i-1}})\to Y_{i-1}$ and a gradation of $\frac{T_{Y_{i-1}/k}}{\F_i+\G_{i-1}}$ that is coherent on $\op{Supp}(\frac{T_{Y_{i-1}/k}}{\F_i+\G_{i-1}})$, we deduce that $c_1(\frac{T_{Y_{i-1}/k}}{\F_i+\G_{i-1}})$ is a sum of divisors $\sum a_i E_i$, where $E_1,\ldots,E_k$ are irreducible components of the support of $\frac{T_{Y_{i-1}/k}}{\F_i+\G_{i-1}}$ that are of codimension $1$ in $X$, and $a_i\ge 0$ are integers. Consequently, the pullback satisfies $f_{i-1}^* c_1(\frac{T_{Y_{i-1}/k}}{\F_i+\G_{i-1}})=\sum a_i f_{i-1}^*(E_i)$, where $f_{i-1}^*(E_i)$ are effective divisors on $X$ that satisfy $|f_{i-1}^*(E_i)|=|E_i|$, because $f$ is a finite morphism and a homeomorphism. This proves that
    \[
    D_i\coloneqq f_{i-1}^* c_1(\frac{T_{Y_{i-1}/k}}{\F_i+\G_{i-1}}) \ge 0.
    \]

    Let $f:X\to Y$ be an Ekedahl $n$-foliation. Then, for $i=2,\ldots,n$, the supports of the sheaves $\frac{T_{Y_{i-1}/k}}{\F_i+\G_{i-1}}$ are of codimension at least $2$ by Proposition \ref{criterion for being ekedahl}, therefore $c_1(\frac{T_{Y_{i-1}/k}}{\F_i+\G_{i-1}})=0$, and so $D_i=f_{i-1}^* c_1(\frac{T_{Y_{i-1}/k}}{\F_i+\G_{i-1}})=f_{i-1}^* 0=0$.
\end{proof}

\begin{question}
    I believe that if I have a fibration $f: X\to Y$ that can be recovered from its power tower, then the formulas for $K_X-f_i^*K_{Y_i}$ converge (possibly in the $p$-adic sense) to a formula for $K_X-f^*K_{Y}$. Is it true? Is it a special case of a more general fact that there exists a ``canonical divisor of a power tower'' $K_{Y_\bullet}$ and this divisor satisfies that formula?
    
    I do not have any proof for it; however, there is evidence for this claim in Marta Benozzo's research. She studied formulas for $K_X-f^*K_{Y}$ in her PhD thesis, which is partially covered in her preprints \cite{benozzo2023canonicalbundleformulapositive},\cite{benozzo2024superadditivityanticanonicaliitakadimension}, and her paper \cite{benozzo2022iitakaconjectureanticanonicaldivisors_paper!}, and she communicated to me that it seems that supports of her ``extra divisors''  in those formulas coincide with the supports of extra divisors appearing in ``the limit'' of my formulas, at least in very well-behaved cases. More research is needed.
\end{question}

\subsection{Application: Not-L{\" u}roth Theorems}

L{\" u}roth-type theorems say that if a variety $Y/k$ of dimension $n$ is dominated by $\mathbb{P}^n_k$, then, under some extra conditions, $Y/k$ is birational to $\mathbb{P}^n_k$. One such condition that works is $n=1$ and $Y$ being smooth. This is \emph{the} L{\" u}roth theorem, see e.g. \cite[Chapter IV, Example 2.5.5.]{Hartshorne}. 

In this section, we prove that if a morphism $f:\mathbb{P}^n_k \to Y$ is purely inseparable, then in most cases $Y$ is \textit{not} like $\mathbb{P}^n_k$, and we can be more precise about what it means. This result is an application of the formula from Theorem \ref{Canonical Divisor Formula - theorem}.

The following definition is \cite[Definition 2.1.3.]{positivity1}; it is applied to a variety over an arbitrary field, not necessarily $\mathbb{C}$.

\begin{defin}\label{kodaira definiton}
    Let $D$ be a divisor on a normal variety $X/k$. The \emph{Iitaka dimension} of $D$ is defined to be an integer $d$
    such that there are positive real numbers $a,b>0$ satisfying
    \[
    am^d \le h_0(X, mD) \le b m^d
    \]
    for $m>>0$. In other words, $h_0(X, mD)\sim m^d$. We denote it by $\kappa(X, D)=d$.
\end{defin}

\begin{example}\label{iitaka dim for a div on P^n - Example}
    Let $X/k=\Proj_k^n$. Let $D$ be a divisor on $X$. Then there are only three options for $\kappa(X,D)$:
    \begin{itemize}
        \item If $d<0$, then we have $h_0(\mathbb{P}^2_k, md)=0$, for $m>0$. So, $\kappa(X,D)=-\infty$.
        \item If $d=0$, then we have $h_0(\mathbb{P}^2_k, md)=1$, for $m>0$. So, $\kappa(X,D)=0$.
        \item If $d>0$, then we have $h_0(\mathbb{P}^2_k, md) \sim m^n$, for $m>>0$. So, $\kappa(X,D)=n$.
    \end{itemize}
    The computation follows from \cite[Chapter III, Theorem 5.1]{Hartshorne}.
\end{example}

\begin{lemma}\label{iitaka is preserved under purely insep - lem}
    Let $f: X\to Y$ be a purely inseparable morphism between normal varieties over a perfect field $k$ such that $k=K(X)^{p^\infty}$. Let $D$ be a divisor on $Y$. Then, we have 
    \[
    \kappa(Y,D)=\kappa(X, f^* D).
    \]
\end{lemma}

\begin{proof}
    Let $n$ be the exponent of $f$. It is an integer, not $\infty$, by Lemma \ref{Purely Inseparable is of Finite Exponent}. Therefore, there is a factorization of $F_{Y^{(-n)}/k}$:
    \[
    Y^{(-n)}\to X \xrightarrow{f} Y.
    \]
    (We can take the negative exponent $-n$ in the following way: we take $K(Y)$, and we take the normalization of $\cO_{X/k}$ in $K(Y)^{p^{-n}}$.)
    From this factorization, we obtain the following sequence.
    \[
    H^0(Y,mD)\to H^0(X,mf^* D)\to H^0(Y^{(-n)},m\left(F_{Y^{(-n)/k}}^n\right)^*D)\simeq H^0(Y,mp^n D).
    \]
    All the above arrows are injections, because the composition takes a section to its $p^n$-power, i.e., $a\mapsto a^{p^n}$, and it is an injection. (For the second arrow, we apply the same argument for $X^{(-n)}\to Y^{(-n)} \to X$.)

    Finally, if $\kappa(Y,D)=d$, then there are positive real numbers $a,b\ge 0$ such that
    \[
   a m^d \le h_0(Y,mD) \le h_0(X,mf^* D) \le h_0(Y,mp^n D) \le (bp^{nd}) m^d,
   \]
   where $m\ge 0$ is an integer. Therefore, we have $\kappa(X, f^* D)=d$. (The opposite implication follows from applying the same argument to $Y^{(-n)}\to X$ and $f^* D$ on $X$.) This proves the lemma.
\end{proof}

\begin{thm}[A Not-L{\" u}roth Theorem]\label{big not-luroth theorem}
    Let $n>1$ be an integer. 
    Let $\mathbb{P}^n_k$ be the $n$-dimensional projective space over a perfect field $k$ of characteristic $p>0$. 
    Let $p-1$ not divide $n+1$.
    
    If $f: \mathbb{P}^n_k\to Y$ is a purely inseparable morphism between normal varieties over $k$, then the Iitaka dimension of a canonical divisor of $Y$ is equal $-\infty$ or $n$, i.e., we have
    \[
    \kappa(Y, K_Y) = -\infty, \text{or } n.
    \]
\end{thm}

\begin{proof}
    By Lemma \ref{iitaka is preserved under purely insep - lem}, we know that 
    \[
    \kappa(Y, K_Y)=\kappa(\mathbb{P}^n_k, f^* K_Y).
    \]
    By Example \ref{iitaka dim for a div on P^n - Example}, we know that $\kappa(\mathbb{P}^n_k, f^* K_Y)= -\infty$, or $ n$ if and only if $f^* K_Y\ne 0$. So, we have to prove that $f^* K_Y\ne 0$. Let us assume that it is false, i.e., $f^* K_Y=0$. Let $m$ be the exponent of $f$. Then, the formula from Theorem \ref{Canonical Divisor Formula - theorem} applied to $f$ gives:
    \begin{align*}
        f^*K_{Y}-K_{\mathbb{P}^n_k} &= (1-p^m)c_1(\F_1) - \sum_{j=2}^{m} (1-p^{m-j+1})D_j,\\
        n+1 &=(1-p^m)c_1(\F_1) - \sum_{j=2}^{m} (1-p^{m-j+1})D_j,\\
        n+1 &=(1-p)\left(\frac{1-p^m}{1-p}c_1(\F_1) - \sum_{j=2}^{m} \frac{1-p^{m-j+1}}{1-p} D_j\right).
    \end{align*}
    This is a contradiction, because $p-1$ does not divide $n+1$. This proves the theorem.
\end{proof}

\begin{question}
    The conclusion of Theorem \ref{big not-luroth theorem} is most likely false if $p-1$ divides $n+1$. Indeed, I believe that counterexamples exist. In particular, the ``normalisations of K3-like coverings'' from \cite[Section 4]{Fanelli-Schroer-k3-like} look like a promising direction for finding these counterexamples for $n=2$. 
    
    Are these morphisms counterexamples? Are all counterexamples like them?
\end{question}

\begin{remark}
    The condition ``$p-1$ does not divide $n+1$'' is true if $p>n+2$, so if the characteristic is great enough with respect to the dimension of the projective space, then Theorem \ref{big not-luroth theorem} holds.

    Moreover, for $n=2$, the assumption on $p$ in Theorem \ref{big not-luroth theorem} can be replaced with $p>2$, because a prime number satisfies $p>2$ if and only if $p-1$ does not divide $3$.
\end{remark}

\newpage
\bibliographystyle{amsalpha} 
\bibliography{bib}

\end{document}